\documentclass{amsart}
\usepackage{amssymb,amsmath,stmaryrd,mathrsfs}
\def\definetac{\newif\iftac}    
\ifx\tactrue\undefined
  \definetac
  \ifx\state\undefined\tacfalse\else\tactrue\fi
\fi

\iftac\else\usepackage{amsthm}\fi
\usepackage[all,2cell]{xy}
\UseAllTwocells
\usepackage[neveradjust]{paralist}
\usepackage{xcolor}
\definecolor{darkgreen}{rgb}{0,0.45,0}
\usepackage[pagebackref,colorlinks,citecolor=darkgreen,linkcolor=darkgreen]{hyperref}
\usepackage{mathtools}
\usepackage{braket}
\let\setof\Set
\usepackage{url}
\usepackage{xspace}
\usepackage{comment}
\usepackage{caption}
\usepackage{enumitem}

\usepackage{framed}
\definecolor{shadecolor}{HTML}{EEEEEE}

\usepackage{tikz}
\usetikzlibrary{arrows,matrix,calc,backgrounds,patterns,snakes,shapes.geometric}
\usepackage{tikz-cd}

\usepackage{aliascnt}
\usepackage[nameinlink,capitalize]{cleveref} 

\makeatletter
\let\ea\expandafter

\def\mdef#1#2{\ea\ea\ea\gdef\ea\ea\noexpand#1\ea{\ea\ensuremath\ea{#2}\xspace}}
\def\alwaysmath#1{\ea\ea\ea\global\ea\ea\ea\let\ea\ea\csname your@#1\endcsname\csname #1\endcsname
  \ea\def\csname #1\endcsname{\ensuremath{\csname your@#1\endcsname}\xspace}}

\newcount\foreachcount

\def\foreachletter#1#2#3{\foreachcount=#1
  \ea\loop\ea\ea\ea#3\@alph\foreachcount
  \advance\foreachcount by 1
  \ifnum\foreachcount<#2\repeat}

\def\foreachLetter#1#2#3{\foreachcount=#1
  \ea\loop\ea\ea\ea#3\@Alph\foreachcount
  \advance\foreachcount by 1
  \ifnum\foreachcount<#2\repeat}

\def\definescr#1{\ea\gdef\csname s#1\endcsname{\ensuremath{\mathscr{#1}}\xspace}}
\foreachLetter{1}{27}{\definescr}
\def\definecal#1{\ea\gdef\csname c#1\endcsname{\ensuremath{\mathcal{#1}}\xspace}}
\foreachLetter{1}{27}{\definecal}
\def\definebold#1{\ea\gdef\csname b#1\endcsname{\ensuremath{\mathbf{#1}}\xspace}}
\foreachLetter{1}{27}{\definebold}
\def\definebb#1{\ea\gdef\csname l#1\endcsname{\ensuremath{\mathbb{#1}}\xspace}}
\foreachLetter{1}{27}{\definebb}
\def\definefrak#1{\ea\gdef\csname f#1\endcsname{\ensuremath{\mathfrak{#1}}\xspace}}
\foreachletter{1}{9}{\definefrak} 
\foreachletter{10}{27}{\definefrak}
\foreachLetter{1}{26}{\definefrak}
\def\definebar#1{\ea\gdef\csname #1bar\endcsname{\ensuremath{\overline{#1}}\xspace}}
\foreachLetter{1}{27}{\definebar}
\foreachletter{1}{8}{\definebar} 
\foreachletter{9}{15}{\definebar} 
\foreachletter{16}{27}{\definebar}
\def\definetil#1{\ea\gdef\csname #1til\endcsname{\ensuremath{\widetilde{#1}}\xspace}}
\foreachLetter{1}{27}{\definetil}
\foreachletter{1}{27}{\definetil}
\def\definehat#1{\ea\gdef\csname #1hat\endcsname{\ensuremath{\widehat{#1}}\xspace}}
\foreachLetter{1}{27}{\definehat}
\foreachletter{1}{27}{\definehat}
\def\definechk#1{\ea\gdef\csname #1chk\endcsname{\ensuremath{\check{#1}}\xspace}}
\foreachLetter{1}{27}{\definechk}
\foreachletter{1}{27}{\definechk}
\def\defineul#1{\ea\gdef\csname u#1\endcsname{\ensuremath{\underline{#1}}\xspace}}
\foreachLetter{1}{27}{\defineul}
\foreachletter{1}{27}{\defineul}

\def\autofmt@n#1\autofmt@end{\mathrm{#1}}
\def\autofmt@b#1\autofmt@end{\mathbf{#1}}
\def\autofmt@l#1#2\autofmt@end{\mathbb{#1}\mathsf{#2}}
\def\autofmt@c#1#2\autofmt@end{\mathcal{#1}\mathit{#2}}
\def\autofmt@s#1#2\autofmt@end{\mathscr{#1}\mathit{#2}}
\def\autofmt@f#1\autofmt@end{\mathfrak{#1}}
\def\autofmt@u#1\autofmt@end{\underline{\smash{\mathsf{#1}}}}
\def\autofmt@U#1\autofmt@end{\underline{\underline{\smash{\mathsf{#1}}}}}
\def\autofmt@uu#1\autofmt@end{\underline{\underline{\smash{\mathsf{#1}}}}}
\def\autofmt@h#1\autofmt@end{\widehat{#1}}
\def\autofmt@r#1\autofmt@end{\overline{#1}}
\def\autofmt@t#1\autofmt@end{\widetilde{#1}}
\def\autofmt@k#1\autofmt@end{\check{#1}}

\def\auto@drop#1{}
\def\autodef#1{\ea\ea\ea\@autodef\ea\ea\ea#1\ea\auto@drop\string#1\autodef@end}
\def\@autodef#1#2#3\autodef@end{%
  \ea\def\ea#1\ea{\ea\ensuremath\ea{\csname autofmt@#2\endcsname#3\autofmt@end}\xspace}}
\def\autodefs@end{blarg!}
\def\autodefs#1{\@autodefs#1\autodefs@end}
\def\@autodefs#1{\ifx#1\autodefs@end%
  \def\autodefs@next{}%
  \else%
  \def\autodefs@next{\autodef#1\@autodefs}%
  \fi\autodefs@next}

\DeclareSymbolFont{bbold}{U}{bbold}{m}{n}
\DeclareSymbolFontAlphabet{\mathbbb}{bbold}

\mdef\alhat{\widehat{\alpha}}

\let\del\partial
\mdef\delbar{\overline{\partial}}

\mdef\hf{\textstyle\frac12 }
\mdef\thrd{\textstyle\frac13 }
\mdef\qtr{\textstyle\frac14 }

\SelectTips{cm}{}
\newdir{ >}{{}*!/-10pt/@{>}}    

\mdef\Id{\mathrm{Id}}
\mdef\id{\mathrm{id}}
\alwaysmath{ell}
\alwaysmath{infty}
\alwaysmath{odot}
\alwaysmath{sim}
\def\frc#1/#2.{\frac{#1}{#2}}   
\mdef\ten{\mathrel{\otimes}}

\mdef\sqten{\mathrel{\boxtimes}}
\def\pow(#1,#2){\mathop{\pitchfork}(#1,#2)} 

\DeclareRobustCommand\widecheck[1]{{\mathpalette\@widecheck{#1}}}
\def\@widecheck#1#2{%
    \setbox\z@\hbox{\m@th$#1#2$}%
    \setbox\tw@\hbox{\m@th$#1%
       \widehat{%
          \vrule\@width\z@\@height\ht\z@
          \vrule\@height\z@\@width\wd\z@}$}%
    \dp\tw@-\ht\z@
    \@tempdima\ht\z@ \advance\@tempdima2\ht\tw@ \divide\@tempdima\thr@@
    \setbox\tw@\hbox{%
       \raise\@tempdima\hbox{\scalebox{1}[-1]{\lower\@tempdima\box
\tw@}}}%
    {\ooalign{\box\tw@ \cr \box\z@}}}

\DeclareMathOperator\Hom{Hom}

\mdef{\sk}{\mathop{\mathsf{sk}}}
\mdef{\cosk}{\mathop{\mathsf{cosk}}}

\newcommand{\too}[1][]{\ensuremath{\overset{#1}{\longrightarrow}}}

\let\toto\rightrightarrows
\let\into\hookrightarrow

\mdef\we{\overset{\sim}{\longrightarrow}}
\mdef\leftwe{\overset{\sim}{\longleftarrow}}
\let\mono\rightarrowtail

\def\rightarrowtailfill@{\arrowfill@{\Yright\joinrel\relbar}\relbar\rightarrow}
\newcommand\xrightarrowtail[2][]{\ext@arrow 0055{\rightarrowtailfill@}{#1}{#2}}

\def\twoheadrightarrowfill@{\arrowfill@{\relbar\joinrel\relbar}\relbar\twoheadrightarrow}
\newcommand\xtwoheadrightarrow[2][]{\ext@arrow 0055{\twoheadrightarrowfill@}{#1}{#2}}

\def\slashedarrowfill@#1#2#3#4#5{%
  $\m@th\thickmuskip0mu\medmuskip\thickmuskip\thinmuskip\thickmuskip
   \relax#5#1\mkern-7mu%
   \cleaders\hbox{$#5\mkern-2mu#2\mkern-2mu$}\hfill
   \mathclap{#3}\mathclap{#2}%
   \cleaders\hbox{$#5\mkern-2mu#2\mkern-2mu$}\hfill
   \mkern-7mu#4$%
}
\def\rightslashedarrowfill@{%
  \slashedarrowfill@\relbar\relbar\mapstochar\rightarrow}
\newcommand\xslashedrightarrow[2][]{%
  \ext@arrow 0055{\rightslashedarrowfill@}{#1}{#2}}
\mdef\hto{\xslashedrightarrow{}}
\mdef\htoo{\xslashedrightarrow{\quad}}

\def\defthm#1#2#3{%
  \newaliascnt{#1}{thm}
  \newtheorem{#1}[#1]{#2}
  \aliascntresetthe{#1}
  \crefname{#1}{#2}{#3}
}

\newtheorem{thm}{Theorem}[section]
\crefname{thm}{Theorem}{Theorems}
\defthm{cor}{Corollary}{Corollaries}
\defthm{prop}{Proposition}{Propositions}
\defthm{lem}{Lemma}{Lemmas}
\defthm{sch}{Scholium}{Scholia}
\defthm{assume}{Assumption}{Assumptions}
\defthm{claim}{Claim}{Claims}
\defthm{conj}{Conjecture}{Conjectures}
\defthm{hyp}{Hypothesis}{Hypotheses}
\defthm{fact}{Fact}{Facts}
\theoremstyle{definition}
\defthm{defn}{Definition}{Definitions}
\defthm{notn}{Notation}{Notations}
\theoremstyle{remark}
\defthm{rmk}{Remark}{Remarks}
\defthm{eg}{Example}{Examples}
\defthm{egs}{Examples}{Examples}
\defthm{ex}{Exercise}{Exercises}
\defthm{ceg}{Counterexample}{Counterexamples}
\defthm{warn}{Warning}{Warnings}

\crefname{figure}{Figure}{Figures}

\def\thmqedhere{\expandafter\csname\csname @currenvir\endcsname @qed\endcsname}

\iftac
  \let\c@equation\c@subsection
\else
  \let\c@equation\c@thm
\fi
\numberwithin{equation}{section}

\@ifpackageloaded{mathtools}{\mathtoolsset{showonlyrefs,showmanualtags}}{}

\newcommand{\drpullback}[1][dr]{\ar[#1,phantom,near start,"\lrcorner"]}

\newlength\oldleftmargini       
\newlength\oldleftmarginii
\newlength\oldleftmarginiii
\newlength\oldleftmarginiv
\newlength\oldleftmarginv
\newlength\oldleftmarginvi
\newcount\maxenum
\newif\ifkillspacing
\def\@adjust@enum@labelwidth{%
  \advance\@listdepth by 1\relax
  \ifkillspacing                
    \csname c@\@enumctr\endcsname\maxenum
    \settowidth{\@tempdima}{%
      \csname label\@enumctr\endcsname\hspace{\labelsep}}%
    \csname leftmargin\romannumeral\@listdepth\endcsname
      \@tempdima
  \else                         
    \csname fixspacing\romannumeral\@listdepth\endcsname
  \fi
  \advance\@listdepth by -1\relax}
\def\fixspacingi{\ifnum\oldleftmargini=0\setlength\oldleftmargini\leftmargini\else\setlength\leftmargini\oldleftmargini\fi}
\def\fixspacingii{\ifnum\oldleftmarginii=0\setlength\oldleftmarginii\leftmarginii\else\setlength\leftmarginii\oldleftmarginii\fi}
\def\fixspacingiii{\ifnum\oldleftmarginiii=0\setlength\oldleftmarginiii\leftmarginiii\else\setlength\leftmarginiii\oldleftmarginiii\fi}
\def\fixspacingiv{\ifnum\oldleftmarginiv=0\setlength\oldleftmarginiv\leftmarginiv\else\setlength\leftmarginiv\oldleftmarginiv\fi}
\def\fixspacingv{\ifnum\oldleftmarginv=0\setlength\oldleftmarginv\leftmarginv\else\setlength\leftmarginv\oldleftmarginv\fi}
\def\fixspacingvi{\ifnum\oldleftmarginvi=0\setlength\oldleftmarginvi\leftmarginvi\else\setlength\leftmarginvi\oldleftmarginvi\fi}

\def\pl@label#1#2{%
  \edef\pl@the{\noexpand#1{\@enumctr}}%
  \pl@lab\expandafter{\the\pl@lab\csname yourthe\@enumctr\endcsname}%
  \advance\@tempcnta1
  \pl@loop}
\def\@enumlabel@#1[#2]{%
  \@plmylabeltrue
  \@tempcnta0
  \pl@lab{}%
  \let\pl@the\pl@qmark
  \expandafter\pl@loop\@gobble#2\@@@
  \ifnum\@tempcnta=1\else
    \PackageWarning{paralist}{Incorrect label; no or multiple
      counters.\MessageBreak The label is: \@gobble#2}%
  \fi
  \expandafter\edef\csname label\@enumctr\endcsname{\the\pl@lab}%
  \expandafter\edef\csname the\@enumctr\endcsname{\the\pl@lab}%
  \expandafter\let\csname yourthe\@enumctr\endcsname\pl@the
  #1}

\alwaysmath{alpha}
\alwaysmath{beta}
\alwaysmath{gamma}
\alwaysmath{Gamma}
\alwaysmath{delta}
\alwaysmath{Delta}
\alwaysmath{epsilon}
\mdef\ep{\varepsilon}
\alwaysmath{zeta}
\alwaysmath{eta}
\alwaysmath{theta}
\alwaysmath{Theta}
\alwaysmath{iota}
\alwaysmath{kappa}
\alwaysmath{lambda}
\alwaysmath{Lambda}
\alwaysmath{mu}
\alwaysmath{nu}
\alwaysmath{xi}
\alwaysmath{pi}
\alwaysmath{rho}
\alwaysmath{sigma}
\alwaysmath{Sigma}
\alwaysmath{tau}
\alwaysmath{upsilon}
\alwaysmath{Upsilon}
\alwaysmath{phi}
\alwaysmath{Pi}
\alwaysmath{Phi}
\mdef\ph{\varphi}
\alwaysmath{chi}
\alwaysmath{psi}
\alwaysmath{Psi}
\alwaysmath{omega}
\alwaysmath{Omega}

\let\th\theta

\makeatother

\title{Doubly weak double categories}
\author{Aaron David Fairbanks \and Michael Shulman}
\thanks{The second author was supported by the Air Force Office of Scientific Research under award number FA9550-21-1-0009.}

\autodefs{\bCat\bCptd\bDblCat\bDblCptd\bSet\bGph\bDblGph\bDblCptd\bdbl\cLax\nob\bWDblCat\bBiDblGph\bPosDblCptd\bSuDblCptd\bPsDblCat\bCubBicat\bMoDblGph\bMnd\bRel\bRing\bBimod\bProf\cCat\cProf\bAdj\cSpan\cPoly\cCatAna}
\mdef\ovz{{1\mathord{\vee}0}}
\mdef\zvo{{0\mathord{\vee}1}}
\mdef\ovo{{1\mathord{\vee}1}}
\mdef\wovo{{\mathbf{w}1\mathord{\vee}1}}
\mdef\obo{\mathbf{d}}
\mdef\ovocat{\ovo\text{-}\bCat}

\mdef\tcat{2\text{-}\bCat}
\mdef\ocat{1\text{-}\bCat}
\mdef\bBicat{\mathbf{W}\text{-}2\text{-}\bCat}
\mdef\bDblBicat{\mathbf{DblBicat}}
\mdef\bDblBicatst{\mathbf{DblBicat}_{\mathbf{st}}}
\mdef\bWDblCatst{\bWDblCat_{\mathbf{st}}}
\mdef\cWDblCatst{\mathscr{W\!D}\mathit{bl}\mathscr{C}\!\mathit{at}_{\mathit{st}}}
\mdef\cWDblCat{\mathscr{W\!D}\mathit{bl}\mathscr{C}\!\mathit{at}}
\mdef\bPsDblCatst{\bPsDblCat_{\mathbf{st}}}
\mdef\bBicatst{\bBicat_{\mathbf{st}}}
\mdef\cBicatst{\mathscr{W}\!\text{-}2\text{-}\mathscr{C}\!\mathit{at}_{\mathit{st}}}
\mdef\cBicat{\mathscr{W}\!\text{-}2\text{-}\mathscr{C}\!\mathit{at}}
\mdef\bDblBicattidy{\bDblBicat_{\mathbf{st,tidy}}}
\mdef\ovocptd{\ovo\text{-}\bCptd}
\let\obocptd\bDblCptd

\mdef\tcptd{2\text{-}\bCptd}
\mdef\ocptd{1\text{-}\bCptd}
\mdef\tgph{2\text{-}\bGph}
\mdef\ocattgph{\ocat\text{-}\tgph}
\mdef\ovocatbDblGph{\ovocat\bDblGph}
\mdef\ordblgph{\mathbf{1RDblGph}}
\mdef\sq{\square}
\mdef\SQ{\blacksquare}
\mdef\uovo{\cU_{\ovo}}
\mdef\fovo{\cF_{\ovo}}
\mdef\tovo{T_{\ovo}}
\mdef\d{\mathbf{d}}
\mdef\mono{\mathbf{m}}
\mdef\dw{\mathbf{wd}}
\mdef\wt{{\mathbf{w}2}}
\mdef\wo{{\mathbf{w}1}}
\mdef\dpl{{\mathbf{d}+}}
\mdef\wdpl{{\mathbf{wd}+}}
\mdef\tflat{\iota_2}
\mdef\dflat{\iota_\d}
\mdef\tz{0}
\mdef\ton{1}
\mdef\th{\ton^H}
\mdef\tv{\ton^V}
\mdef\ttw{2}
\def\td#1#2#3#4{\ensuremath{\ttw^{#1,#2}_{#3,#4}}}
\mdef\simd{\backsim}
\mdef\dblcatpl{\bDblCat_+}
\mdef\wdblcatpl{\bWDblCat_+}
\mdef\posd{{\mathbf{posd}}}
\mdef\posdpl{{\mathbf{posd}+}}

\mdef\leovo{\le 1}

\mdef\choseniso{\cong}

\DeclareFontFamily{U}{dmjhira}{}
\DeclareFontShape{U}{dmjhira}{m}{n}{ <-> dmjhira }{}

\mdef\bcdots{\hspace{-.01cm}\cdots} 
\mdef\vcdots{\rotatebox{90}{\hspace{.05cm}$\cdots$}} 
\mdef\dcdots{\rotatebox{-45}{\hspace{.05cm}$\cdots$}}
\mdef\icdots{\rotatebox{45}{\hspace{.05cm}$\cdots$}}
\mdef\NEarrow{\mathrel{\rotatebox[origin=c]{45}{$\Rightarrow$}}}
\mdef\NWarrow{\mathrel{\rotatebox[origin=c]{135}{$\Rightarrow$}}}
\mdef\SWarrow{\mathrel{\rotatebox[origin=c]{-135}{$\Rightarrow$}}}
\mdef\SEarrow{\mathrel{\rotatebox[origin=c]{-45}{$\Rightarrow$}}}

\def\grpair#1#2{\ensuremath{(#1,#2)}}
\mdef\grempty{()}

\def\justi{1}

\DeclareMathOperator\st{st}

\def\ubcom#1#2{\ensuremath{#1 #2}}
\def\bcom#1#2{\ensuremath{(\ubcom{#1}{#2})}}
\def\bi#1{\ensuremath{\justi_{#1}}}

\def\dcom#1#2{\bcom{#1}{#2}}
\def\udcom#1#2{\ubcom{#1}{#2}}
\def\di#1{\bi{#1}}
\def\hcom#1#2{\bcom{#1}{#2}}
\def\uhcom#1#2{\ubcom{#1}{#2}}
\def\hi#1{\bi{#1}}

\def\uvcom#1#2{\ensuremath{\rotatebox{-90}{$\rotatebox{90}{\!$#1$}\;\rotatebox{90}{\!$#2$}$}\!}}
\def\vcoms#1#2{\ensuremath{\rotatebox{-90}{$(\hspace{.08em}\rotatebox{90}{\!$#1$}\;\rotatebox{90}{\!$#2$}\hspace{.08em})$}\!}}
\def\uvcoms#1#2{\ensuremath{\rotatebox{-90}{$\rotatebox{90}{\!$#1$}\;\rotatebox{90}{\!$#2$}$}\!}}

\def\ti#1{\bi{#1}}

\mdef\Homl{\Hom_\mathbf{co/lax}}

\def\tproj{\scalebox{.7}{\rotatebox{45}{\SQ}}}
\def\oproj{\scalebox{.7}{\rotatebox{45}{\sq}}}

\newcommand{\pb}{\ar@{}[dr]|<<*+{\lrcorner}}
\newcommand{\po}{\ar@{}[dr]|>>*+{\ulcorner}}

\def\cat#1{\ensuremath{\mathbf{#1}}}
\def\bicat#1{\ensuremath{\mathcal{#1}}}
\def\bicatf#1{\ensuremath{\mathcal{#1}}}

\mdef\paral{\toto}
\mdef\paralfill{\Downarrow}

\mdef\dom{\mathop{\text{dom}}}
\mdef\cod{\mathop{\text{cod}}}
\mdef\hdom{\mathop{\overset{H}{\text{dom}}}}
\mdef\hcod{\mathop{\overset{H}{\text{cod}}}}
\mdef\vdom{\mathop{\overset{V}{\text{dom}}}}
\mdef\vcod{\mathop{\overset{V}{\text{cod}}}}
\mdef\hcomp{\mathrel{\overset{H}{\cdot}}}
\mdef\vcomp{\mathrel{\overset{V}{\cdot}}}
\mdef\hid{\mathop{\overset{H}{\id}}}
\mdef\vid{\mathop{\overset{V}{\id}}}
\mdef\hbicomp{\mathrel{\underset{\text{bi}}{\overset{H}{\cdot}}}}
\mdef\vbicomp{\mathrel{\underset{\text{bi}}{\overset{V}{\cdot}}}}
\mdef\hsqcomp{\mathrel{\underset{\text{sq}}{\overset{H}{\cdot}}}}
\mdef\vsqcomp{\mathrel{\underset{\text{sq}}{\overset{V}{\cdot}}}}
\mdef\hbiid{\mathop{\underset{\text{bi}}{\overset{H}{\id}}}}
\mdef\vbiid{\mathop{\underset{\text{bi}}{\overset{V}{\id}}}}
\mdef\hsqid{\mathop{\underset{\text{sq}}{\overset{H}{\id}}}}
\mdef\vsqid{\mathop{\underset{\text{sq}}{\overset{V}{\id}}}}
\mdef\hbia{\underset{\text{bi}}{\overset{H}{a}}}
\mdef\vbia{\underset{\text{bi}}{\overset{V}{a}}}
\mdef\hbil{\underset{\text{bi}}{\overset{H}{l}}}
\mdef\vbil{\underset{\text{bi}}{\overset{V}{l}}}
\mdef\hbi{\mathop{\overset{H}{\text{bi}}}}
\mdef\vbi{\mathop{\overset{V}{\text{bi}}}}
\mdef\hsq{\mathop{\overset{H}{\text{sq}}}}
\mdef\vsq{\mathop{\overset{V}{\text{sq}}}}

\def\aan{an}
\def\Aan{An}
\def\adjective{implicit}

\def\doubleword{{\adjective} double category}
\def\doublewords{{\adjective} double categories}
\def\twoword{{\adjective} 2-category}
\def\twowords{{\adjective} 2-categories}

\mdef\ad{\mathbf{I}}

\mdef\wk{\mathbf{w}}

\mdef{\bDblWord}{\mathbf{IDblCat}}
\mdef{\cDblWord}{\mathscr{I\!D}\mathit{bl}\mathscr{C}\!\mathit{at}}
\mdef\tword{\mathbf{I}\text{-}2\text{-}\bCat}
\mdef\ctword{\mathscr{I}\!\text{-}2\text{-}\mathscr{C}\!\mathit{at}}

\mdef{\cCat}{\mathscr{C}\!\mathit{at}}

\tikzset{every picture/.style={>=stealth,baseline={([yshift=-.8ex]current bounding box.center)}}}
\tikzset{arr/.style={auto,font=\scriptsize}} 
\tikzset{e/.style={fill=white,inner sep=1pt}}
\tikzset{ed/.style={rectangle, rounded corners,fill=white,inner sep=1pt,font=\scriptsize}} 
\tikzset{nated/.style={rectangle, rounded corners,fill=white,inner sep=1pt}}
\tikzset{rect/.style={rounded corners}} 
\tikzset{ov/.style={coordinate}} 
\tikzset{iv/.style={circle,draw,inner sep=1pt,font=\scriptsize}} 
\tikzset{ivs/.style={regular polygon,regular polygon sides=4,draw,minimum size=8pt,inner sep=1pt,font=\scriptsize}} 
\tikzset{ivo/.style={regular polygon,regular polygon sides=8,border rotate=22.5,draw,minimum size=8pt,inner sep=1pt,font=\scriptsize}} 
\tikzset{fun/.style={pattern color=gray}} 
\tikzset{ivf/.style={circle,fill=white,draw=white,double=black,very thick,inner sep=1pt,font=\scriptsize}} 
\tikzset{edf/.style={draw=white,double=black,very thick}} 
\tikzset{nat/.style={draw=white,double=black,very thick,decorate,decoration={zigzag,segment length=1mm,amplitude=.5mm},thick}} 
\tikzset{mod/.style={diamond,fill=white,draw=white,double=black,very thick,inner sep=1pt,font=\scriptsize}}
\tikzset{modc/.style={iv,fill=white,draw=white,double=black,very thick}} 
\tikzset{mods/.style={ivs,fill=white,draw=white,double=black,very thick}} 
\tikzset{modo/.style={regular polygon,regular polygon sides=8,shape border rotate=22.5,fill=white,draw=white,double=black,very thick,minimum size=8pt,inner sep=1pt}} 

\DeclareFontFamily{U}{FdSymbolC}{}
\DeclareFontShape{U}{FdSymbolC}{m}{n}{<-> s * FdSymbolC-Book}{}
\DeclareSymbolFont{fdarrows}{U}{FdSymbolC}{m}{n}
\DeclareMathSymbol{\Nearrow}{\mathrel}{fdarrows}{12}
\DeclareMathSymbol{\Nwarrow}{\mathrel}{fdarrows}{13}
\DeclareMathSymbol{\Swarrow}{\mathrel}{fdarrows}{14}
\DeclareMathSymbol{\Searrow}{\mathrel}{fdarrows}{15}

\begin{document}

\begin{abstract}
  We propose a definition of double categories whose composition of
  1-cells is weak in both directions.  Namely, a doubly weak double
  category is a double computad --- a structure with 2-cells of all
  possible double-categorical shapes --- equipped with all possible
  composition operations, coherently.  We also characterize them using
  ``implicit'' double categories, which are double computads having
  all possible compositions of 2-cells, but no compositions of
  1-cells; doubly weak double categories are then obtained by a simple
  representability criterion.  We also show that they are equivalent
  to Verity's double bicategories satisfying a simple additional
  condition that has appeared previously in the literature, and to a
  similar enhancement of Garner's cubical bicategories.
\end{abstract}

\maketitle
\setcounter{tocdepth}{1}
\tableofcontents

\killspacingtrue

\section{Introduction}
\label{sec:introduction}

\subsection{The problem of doubly weak double categories}
\label{sec:doubly-weak-double}

A double category is a structure like a 2-category but with two different sorts of 1-cells, \emph{horizontal} and \emph{vertical}, and 2-cells shaped like squares (with two 1-cells of each sort on their boundaries):
\[
  \begin{tikzpicture}[->, scale=.4]
    \node (a) at (-1, 1) {};
    \node (b) at (1, 1) {};
    \node (c) at (-1, -1) {};
    \node (d) at (1, -1) {};
    \node at (a) {$\cdot$};
    \node at (b) {$\cdot$};
    \node at (c) {$\cdot$};
    \node at (d) {$\cdot$};
    \draw (a) -- (b);
    \draw (b) -- (d);
    \draw (a) -- (c);
    \draw (c) -- (d);
    \node at (0,0) {$\alpha$};
  \end{tikzpicture}
\]

Just as there are strict and weak versions of 2-categories, there are
strict and weak versions of double categories.  Strict double
categories are easy to define, as internal categories in the category
\cat{Cat} of categories (whereas 2-categories are \emph{enriched}
categories in \cat{Cat}). The two different sorts of 1-cell are then,
respectively, the morphisms in the category-of-objects and the objects
in the category-of-morphisms.  Now just as a bicategory is a ``weakly
enriched category'' in the 2-category \cCat of categories, the
definition of internal category can be weakened so that it satisfies
the usual associativity and unit laws only up to coherent isomorphism
(a so-called ``internal pseudo-category''
\cite{ferreira:pseudo}). This results in the \emph{pseudo double
  categories} from~\cite{gp:double-limits}.

However, pseudo double categories are weak in only one direction:
composition of morphisms in the category-of-objects is still
strict. Many of the weak double categories arising naturally do
satisfy this constraint (e.g.\ the double category of categories,
whose two sorts of 1-cells are functors, which compose in a strict
way, and profunctors, which do not). But there are some situations in
which one would like a notion of double category where composition is
weak in both directions.  For example:
\begin{itemize}
\item Every strict 2-category $\bicat{C}$ has a strict double category of ``squares'' a.k.a.\ ``quintets'',\footnote{This unlovely term arises from the fact that to determine a 2-cell in this double category requires five data: a 2-cell in $\bicat{C}$ and four 1-cells in $\bicat{C}$ that form its boundary (the decomposition of its source and target as composites not being determined by the 2-cell itself).}
  where both sorts of 1-cells are those of $\bicat{C}$, and the squares are 2-cells in $\bicat{C}$ of the form \[
    \begin{tikzpicture}[->, scale=.4]
      \node (a) at (-1, 1) {};
      \node (b) at (1, 1) {};
      \node (c) at (-1, -1) {};
      \node (d) at (1, -1) {};
      \node at (a) {$\cdot$};
      \node at (b) {$\cdot$};
      \node at (c) {$\cdot$};
      \node at (d) {$\cdot$};
      \draw (a) -- (b);
      \draw (b) -- (d);
      \draw (a) -- (c);
      \draw (c) -- (d);
      \node at (0,0) {$\Swarrow$};
    \end{tikzpicture}
  \]
  But if $\bicat{C}$ is a \emph{bicategory}, then this would have to be a double category that is weak in both directions.
\item As shown in~\cite{brown:homotopy}, any topological space has a fundamental double groupoid consisting of points as 0-cells, continuous paths as both kinds of 1-cells, and homotopy classes of homotopies as 2-cells.
  The double groupoid constructed in~\cite{brown:homotopy} is made strict by quotienting the paths by ``thin homotopy'', but it would be more natural to have weak composition in both directions, since concatenation of paths is not strictly associative.
\item A \emph{proarrow equipment}~\cite{wood:pro1} can be defined as a pseudofunctor of bicategories $\bicat{C} \to \bicat{D}$ that is bijective on objects, locally full and faithful, and such that every 1-cell in its image is a left adjoint.
  This is intended as an abstraction of examples such as the pseudofunctor $\cCat \to \cProf$ assigning to each functor its representable profunctor.
  As observed in~\cite{verity:base-change,shulman:frbi}, a proarrow equipment gives rise to a double category, whose objects are those shared by $\bicat{C}$ and $\bicat{D}$, whose two sorts of 1-cell are those of $\bicat{C}$ and $\bicat{D}$ respectively, and whose 2-cells come from $\bicat{D}$.
  However, this is only a pseudo double category if $\bicat{C}$ is a strict 2-category.
  When $\bicat{C}$ and $\bicat{D}$ are both bicategories, this double category should be weak in both directions.

  In practice, often $\bicat{C}$ is strict, but not always.
  Two examples where it is not are the inclusion $\cSpan(\cat{E}) \to \cPoly(\cat{E})$ of the bicategory of spans in the bicategory of polynomials~\cite{kg:polynomials,weber:poly-pb}, for any locally cartesian closed category $\cat{E}$; and the inclusion $\cCatAna(\cat{E}) \to \cProf(\cat{E})$ of internal anafunctors~\cite{bartels:hgt,roberts:ana} into internal profunctors, for any topos $\cat{E}$.
\item A special case of an equipment is when the 1-cells of $\bicat{C}$ are defined to be adjunctions in $\bicat{D}$ (pointing in the direction of the left adjoints; these are sometimes called \emph{maps}).
  The resulting double category was used in~\cite{ks:r2cats} to formalize the functoriality of the ``mates'' correspondence in $\bicat{D}$.
  To do the same when $\bicat{D}$ is a bicategory would require a doubly weak double category.
\item If $\bicat{C}$ and $\bicat{D}$ are strict 2-categories, there is a strict double category that we denote $\Homl(\bicat{C}, \bicat{D})$ whose objects are functors $\bicat{C}\to\bicat{D}$, whose horizontal and vertical 1-cells are \emph{lax} and \emph{colax} transformations respectively, and whose 2-cells are a general notion of modification.
  This should also be true if $\bicat{C}$ and $\bicat{D}$ are bicategories, but in that case this double category would be weak in both directions.
\item Similarly, if $T$ is a 2-monad on a 2-category $\bicat{C}$, there is a strict double category whose objects are $T$-algebras and whose horizontal and vertical 1-cells are lax and colax $T$-morphisms respectively.
  (Such double categories were first considered by~\cite{gp:double-adjoints}.)
  This should also be true if $T$ is a pseudomonad on a bicategory, but in that case this double category would again be weak in both directions.
\end{itemize}

We evidently cannot define doubly weak double categories as any sort
of internal category in categories (since the arrows of a category
compose strictly associatively). But we can write out the definition
of a double category explicitly, with sets of 0-cells, vertical and
horizontal 1-cells, and squares, and then try to insert coherence
isomorphisms relating compositions of 1-cells.  However, it is
surprisingly tricky to make this work, for the following reason.

Note first that the usual associativity and unit constraint isomorphisms in a bicategory are \emph{globular}:\[
  \begin{tikzpicture}[scale=.6]
    \node[inner sep=1] (x1) at (-1,0) {$\cdot$};
    \node[inner sep=1] (x2) at (1,0) {$\cdot$};
    \draw[->,shorten <=.5pt] (x1) to[bend left=50] node[above] {\small$\uhcom{\justi}{f}$} (x2);
    \draw[->,shorten <=.5pt] (x1) to[bend right=50] node[below] {\small$f$} (x2);
    \node at (0,0) {$\lambda$};
  \end{tikzpicture}
  \qquad\qquad\qquad
  \begin{tikzpicture}[scale=.6]
    \node[inner sep=1] (x1) at (-1,0) {$\cdot$};
    \node[inner sep=1] (x2) at (1,0) {$\cdot$};
    \draw[->,shorten <=.5pt] (x1) to[bend left=50] node[above] {\small$\uhcom{f}{\justi}$} (x2);
    \draw[->,shorten <=.5pt] (x1) to[bend right=50] node[below] {\small$f$} (x2);
    \node at (0,0) {$\rho$};
  \end{tikzpicture}
  \qquad\qquad\qquad
  \begin{tikzpicture}[scale=.6]
    \node[inner sep=1] (x1) at (-1,0) {$\cdot$};
    \node[inner sep=1] (x2) at (1,0) {$\cdot$};
    \draw[->,shorten <=.5pt] (x1) to[bend left=50] node[above] {\small$\uhcom{f}{\hcom{g}{h}}$} (x2);
    \draw[->,shorten <=.5pt] (x1) to[bend right=50] node[below] {\small$\uhcom{\hcom{f}{g}}{h}$} (x2);
    \node at (0,0) {$\alpha$};
  \end{tikzpicture}
\]
In a pseudo double category, and presumptively in a doubly weak double category, the corresponding requirement would be that they are squares bordered by
vertical identity 1-cells, simulating globular 2-cells:
\[
  \begin{tikzpicture}[scale=.6]
    \node[inner sep=1] (a1) at (-1,1) {$\cdot$};
    \node[inner sep=1] (b1) at (1,1) {$\cdot$};
    \node[inner sep=1] (c1) at (-1,-1) {$\cdot$};
    \node[inner sep=1] (d1) at (1,-1) {$\cdot$};
    \draw[->,shorten <=.5pt] (a1) -- node[above] {\small$\uhcom{\justi}{f}$} (b1);
    \draw[->,shorten <=.5pt] (c1) -- node[below] {\small$f$} (d1);
    \draw[->,shorten <=.5pt] (a1) -- node[left] {\small$\justi$} (c1);
    \draw[->,shorten <=.5pt] (b1) -- node[right] {\small$\justi$} (d1);
    \node at (0,0) {$\lambda$};
  \end{tikzpicture}
  \qquad\qquad\,
  \begin{tikzpicture}[scale=.6]
    \node[inner sep=1] (a1) at (-1,1) {$\cdot$};
    \node[inner sep=1] (b1) at (1,1) {$\cdot$};
    \node[inner sep=1] (c1) at (-1,-1) {$\cdot$};
    \node[inner sep=1] (d1) at (1,-1) {$\cdot$};
    \draw[->,shorten <=.5pt] (a1) -- node[above] {\small$\uhcom{f}{\justi}$} (b1);
    \draw[->,shorten <=.5pt] (c1) -- node[below] {\small$f$} (d1);
    \draw[->,shorten <=.5pt] (a1) -- node[left] {\small$\justi$} (c1);
    \draw[->,shorten <=.5pt] (b1) -- node[right] {\small$\justi$} (d1);
    \node at (0,0) {$\rho$};
  \end{tikzpicture}
  \qquad\qquad\,
  \begin{tikzpicture}[scale=.6]
    \node[inner sep=1] (a1) at (-1,1) {$\cdot$};
    \node[inner sep=1] (b1) at (1,1) {$\cdot$};
    \node[inner sep=1] (c1) at (-1,-1) {$\cdot$};
    \node[inner sep=1] (d1) at (1,-1) {$\cdot$};
    \draw[->,shorten <=.5pt] (a1) -- node[above] {\small$\uhcom{f}{\hcom{g}{h}}$} (b1);
    \draw[->,shorten <=.5pt] (c1) -- node[below] {\small$\uhcom{\hcom{f}{g}}{h}$} (d1);
    \draw[->,shorten <=.5pt] (a1) -- node[left] {\small$\justi$} (c1);
    \draw[->,shorten <=.5pt] (b1) -- node[right] {\small$\justi$} (d1);
    \node at (0,0) {$\alpha$};
  \end{tikzpicture}
\]
In order to state the usual coherence conditions that these globular
2-cells should satisfy, we must be able to compose them. But when
\emph{vertical} composition of 1-cells is not strictly unital,
vertical composition of squares takes squares that are bordered by
vertical identities to squares that are not; thus the usual coherence
conditions on these squares are not well-typed (the vertical
boundaries of the two sides of the equation are not equal).
\[
  \begin{tikzpicture}[xscale=.8,yscale=.5]
    \begin{scope}[shift={(0,1)}]
      \node[inner sep=1] (a1) at (-1,1) {$\cdot$};m
      \node[inner sep=1] (b1) at (1,1) {$\cdot$};
      \node[inner sep=1] (c1) at (-1,-1) {$\cdot$};
      \node[inner sep=1] (d1) at (1,-1) {$\cdot$};
      \draw[->,shorten <=.5pt] (a1) -- node[above] {\scriptsize$\ubcom{f}{\bcom{\justi}{g}}$} (b1);
      \draw[->,shorten <=.5pt] (c1) -- node[below] {\scriptsize$\ubcom{f}{g}$} (d1);
      \draw[->,shorten <=.5pt] (a1) -- node[left] {\small$\justi$} (c1);
      \draw[->,shorten <=.5pt] (b1) -- node[right] {\small$\justi$} (d1);
      \node at (0,0) {$\ubcom{1}{\lambda}$};
    \end{scope}
  \end{tikzpicture}
  \stackrel{?}{=}
  \begin{tikzpicture}[xscale=.8,yscale=.5]
    \begin{scope}[shift={(0,1)}]
      \node[inner sep=1] (a1) at (-1,1) {$\cdot$};
      \node[inner sep=1] (b1) at (1,1) {$\cdot$};
      \node[inner sep=1] (c1) at (-1,-1) {$\cdot$};
      \node[inner sep=1] (d1) at (1,-1) {$\cdot$};
      \draw[->,shorten <=.5pt] (a1) -- node[above] {\scriptsize$\ubcom{f}{\bcom{\justi}{g}}$} (b1);
      \draw[->,shorten <=.5pt] (c1) -- node[e] {\scriptsize$\ubcom{\bcom{f}{\justi}}{g}$} (d1);
      \draw[->,shorten <=.5pt] (a1) -- node[left] {\small$\justi$} (c1);
      \draw[->,shorten <=.5pt] (b1) -- node[right] {\small$\justi$} (d1);
      \node at (0,0) {$\alpha$};
    \end{scope}
    \begin{scope}[shift={(0,-1)}]
      \node[inner sep=1] (c) at (-1,-1) {$\cdot$};
      \node[inner sep=1] (d) at (1,-1) {$\cdot$};
      \draw[->,shorten <=.5pt] (c) -- node[below] {\scriptsize$\ubcom{f}{g}$} (d);
      \draw[->,shorten <=.5pt] (c1) -- node[left] {\small$\justi$} (c);
      \draw[->,shorten <=.5pt] (d1) -- node[right] {\small$\justi$} (d);
      \node at (0,0) {$\ubcom{\rho}{\justi}$};
    \end{scope}
  \end{tikzpicture}
  \qquad\quad\qquad
  \begin{tikzpicture}[xscale=.8,yscale=.5]
    \begin{scope}[shift={(0,1)}]
      \node[inner sep=1] (a1) at (-1,1) {$\cdot$};
      \node[inner sep=1] (b1) at (1,1) {$\cdot$};
      \node[inner sep=1] (c1) at (-1,-1) {$\cdot$};
      \node[inner sep=1] (d1) at (1,-1) {$\cdot$};
      \draw[->,shorten <=.5pt] (a1) -- node[above] {\scriptsize$\ubcom{f}{\bcom{g}{\bcom{h}{k}}}$} (b1);
      \draw[->,shorten <=.5pt] (c1) -- node[e] {\scriptsize$\ubcom{\bcom{f}{g}}{\bcom{h}{k}}$} (d1);
      \draw[->,shorten <=.5pt] (a1) -- node[left] {\small$\justi$} (c1);
      \draw[->,shorten <=.5pt] (b1) -- node[right] {\small$\justi$} (d1);
      \node at (0,0) {$\alpha$};
    \end{scope}
    \begin{scope}[shift={(0,-1)}]
      \node[inner sep=1] (c) at (-1,-1) {$\cdot$};
      \node[inner sep=1] (d) at (1,-1) {$\cdot$};
      \draw[->,shorten <=.5pt] (c) -- node[below] {\scriptsize$\ubcom{\bcom{\bcom{f}{g}}{h}}{k}$} (d);
      \draw[->,shorten <=.5pt] (c1) -- node[left] {\small$\justi$} (c);
      \draw[->,shorten <=.5pt] (d1) -- node[right] {\small$\justi$} (d);
      \node at (0,0) {$\alpha$};
    \end{scope}
  \end{tikzpicture}
  \stackrel{?}{=}
  \begin{tikzpicture}[xscale=.8,yscale=.5]
    \begin{scope}[shift={(0,1)}]
      \node[inner sep=1] (a1) at (-1,1) {$\cdot$};
      \node[inner sep=1] (b1) at (1,1) {$\cdot$};
      \node[inner sep=1] (c1) at (-1,-1) {$\cdot$};
      \node[inner sep=1] (d1) at (1,-1) {$\cdot$};
      \draw[->,shorten <=.5pt] (a1) -- node[above] {\scriptsize$\ubcom{f}{\bcom{g}{\bcom{h}{k}}}$} (b1);
      \draw[->,shorten <=.5pt] (c1) -- node[e] {\scriptsize$\ubcom{f}{\bcom{\bcom{g}{h}}{k}}$} (d1);
      \draw[->,shorten <=.5pt] (a1) -- node[left] {\small$\justi$} (c1);
      \draw[->,shorten <=.5pt] (b1) -- node[right] {\small$\justi$} (d1);
      \node at (0,0) {$\ubcom{\justi}{\alpha}$};
    \end{scope}
    \begin{scope}[shift={(0,-1)}]
      \node[inner sep=1] (c) at (-1,-1) {$\cdot$};
      \node[inner sep=1] (d) at (1,-1) {$\cdot$};
      \draw[->,shorten <=.5pt] (c) -- node[e] {\scriptsize$\ubcom{\bcom{f}{\bcom{g}{h}}}{k}$} (d);
      \draw[->,shorten <=.5pt] (c1) -- node[left] {\small$\justi$} (c);
      \draw[->,shorten <=.5pt] (d1) -- node[right] {\small$\justi$} (d);
      \node at (0,0) {$\alpha$};
    \end{scope}
    \begin{scope}[shift={(0,-3)}]
      \node[inner sep=1] (c2) at (-1,-1) {$\cdot$};
      \node[inner sep=1] (d2) at (1,-1) {$\cdot$};
      \draw[->,shorten <=.5pt] (c2) -- node[below] {\scriptsize$\ubcom{\bcom{\bcom{f}{g}}{h}}{k}$} (d2);
      \draw[->,shorten <=.5pt] (c) -- node[left] {\small$\justi$} (c2);
      \draw[->,shorten <=.5pt] (d) -- node[right] {\small$\justi$} (d2);
      \node at (0,0) {$\ubcom{\alpha}{\justi}$};
    \end{scope}
  \end{tikzpicture}
\]

We might
try to correct this by horizontally composing with vertical unitors,
but this in turn affects the bordering horizontal 1-cells; and so on,
\textit{ad infinitum}.  For instance, we cannot even compose a
putative isomorphism $\alpha$ with its putative inverse and other
coherence cells to yield an identity on the source or target:

\[\!\!\!
  \begin{tikzpicture}[scale=.5]
    \begin{scope}[shift={(0,1)}]
      \node[inner sep=1] (a1) at (-1,1) {$\cdot$};m
      \node[inner sep=1] (b1) at (1,1) {$\cdot$};
      \node[inner sep=1] (c1) at (-1,-1) {$\cdot$};
      \node[inner sep=1] (d1) at (1,-1) {$\cdot$};
      \draw[->,shorten <=.5pt] (a1) -- node[above] {\small$s$} (b1);
      \draw[->,shorten <=.5pt] (c1) -- node[e] {\small$t$} (d1);
      \draw[->,shorten <=.5pt] (a1) -- node[left] {\scriptsize$\justi$} (c1);
      \draw[->,shorten <=.5pt] (b1) -- node[right] {\scriptsize$\justi$} (d1);
      \node at (0,0) {$\alpha$};
    \end{scope}
    \begin{scope}[shift={(0,-1)}]
      \node[inner sep=1] (c) at (-1,-1) {$\cdot$};
      \node[inner sep=1] (d) at (1,-1) {$\cdot$};
      \draw[->,shorten <=.5pt] (c) -- node[below] {\small$s$} (d);
      \draw[->,shorten <=.5pt] (c1) -- node[left] {\scriptsize$\justi$} (c);
      \draw[->,shorten <=.5pt] (d1) -- node[right] {\scriptsize$\justi$} (d);
      \node at (0,0) {$\alpha\scalebox{.75}{${}^{-1}$}$};
    \end{scope}
  \end{tikzpicture}
  \!\mapsto\;
  \begin{tikzpicture}[xscale=.5,yscale=.5]
    \node[inner sep=1] (a1) at (-1,1) {$\cdot$};
    \node[inner sep=1] (b1) at (1,1) {$\cdot$};
    \node[inner sep=1] (c1) at (-1,-1) {$\cdot$};
    \node[inner sep=1] (d1) at (1,-1) {$\cdot$};
    \draw[->,shorten <=.5pt] (a1) -- node[above] {\small$s$} (b1);
    \draw[->,shorten <=.5pt] (c1) -- node[below] {\small$s$} (d1);
    \draw[->,shorten <=.5pt] (a1) -- node[left] {\tiny$\uvcom{\!\justi}{\!\justi}$} (c1);
    \draw[->,shorten <=.5pt] (b1) -- node[right] {\,\tiny$\uvcom{\!\justi}{\!\justi}$} (d1);
    \node at (0,0) {\tiny\;\;$\uvcom{\alpha}{\!\alpha\scalebox{.75}{${}^{-1}$}}$};
  \end{tikzpicture}
  \qquad\quad
  \begin{tikzpicture}[xscale=.5,yscale=.5]
    \node[inner sep=1] (a1) at (-1,1) {$\cdot$};
    \node[inner sep=1] (b1) at (1,1) {$\cdot$};
    \node[inner sep=1] (c1) at (-1,-1) {$\cdot$};
    \node[inner sep=1] (d1) at (1,-1) {$\cdot$};
    \draw[->,shorten <=.5pt] (a1) -- node[above] {\small$s$} (b1);
    \draw[->,shorten <=.5pt] (c1) -- node[below] {\small$s$} (d1);
    \draw[->,shorten <=.5pt] (a1) -- node[left] {\tiny$\uvcom{\!\justi}{\!\justi}$} (c1);
    \draw[->,shorten <=.5pt] (b1) -- node[e] {\tiny$\uvcom{\!\justi}{\!\justi}$} (d1);
    \node at (0,0) {\tiny$\uvcom{\alpha}{\!\alpha\scalebox{.75}{${}^{-1}$}}$};
    \begin{scope}[shift={(2,0)}]
      \node[inner sep=1] (b2) at (1,1) {$\cdot$};
      \node[inner sep=1] (d2) at (1,-1) {$\cdot$};
      \draw[->,shorten <=.5pt] (b1) -- node[above] {\tiny$\justi$} (b2);
      \draw[->,shorten <=.5pt] (d1) -- node[below] {\tiny$\justi$} (d2);
      \draw[->,shorten <=.5pt] (b2) -- node[right] {\tiny$\justi$} (d2);
      \node at (0,0) {$\lambda$};
    \end{scope}
  \end{tikzpicture}
  \mapsto\;
  \begin{tikzpicture}[xscale=.5,yscale=.5]
    \node[inner sep=1] (a1) at (-1,1) {$\cdot$};
    \node[inner sep=1] (b1) at (1,1) {$\cdot$};
    \node[inner sep=1] (c1) at (-1,-1) {$\cdot$};
    \node[inner sep=1] (d1) at (1,-1) {$\cdot$};
    \draw[->,shorten <=.5pt] (a1) -- node[above] {\tiny$\uhcom{s}{\justi}$} (b1);
    \draw[->,shorten <=.5pt] (c1) -- node[below] {\tiny$\uhcom{s}{\justi}$} (d1);
    \draw[->,shorten <=.5pt] (a1) -- node[left] {\tiny$\uvcom{\!\justi}{\!\justi}$} (c1);
    \draw[->,shorten <=.5pt] (b1) -- node[right] {\tiny$\justi$} (d1);
    \node at (-.15,-.01) {\tiny$\uvcom{\alpha}{\!\alpha\scalebox{.75}{${}^{-1}$}}$};
    \node at (-.75,0) {\footnotesize$($};
    \node at (.3,0) {\footnotesize$)$};
    \node at (.65,0) {\footnotesize$\lambda$};
  \end{tikzpicture}
  \quad\;\;\cdots
\]

At least two ways around this problem have been proposed to date.
\begin{itemize}
\item In~\cite{verity:base-change} Verity defined a \emph{double
    bicategory} to consist of horizontal and vertical bicategories
  with the same set of objects, together with sets of squares that are
  acted on by the 2-cells of the bicategories and can be composed with
  each other horizontally and vertically.

  This includes the important examples, but it does not quite capture
  all their structure, since nothing in a double bicategory allows us
  to identify the 2-cells in the horizontal and vertical bicategories
  with the squares bordered by identities, whereas in examples these
  two are always the same.

  It is possible to correct this problem by assuming an additional
  axiom.  This axiom was already mentioned by Verity~\cite[Lemma
  1.4.9]{verity:base-change}, and was called ``saturation''
  in~\cite{rwan:univalent-doublecats}.  We will discuss this further
  in \cref{sec:computads,sec:tidy}.

\item In~\cite{garner:cubical-bicats} Garner proposed a definition of
  \emph{cubical bicategory} that consists of the data of a double
  category (objects, horizontal and vertical 1-cells, and squares)
  with 1-cell composition and identities (satisfying no axioms), plus
  a way to compose any grid of squares along any way of composing up
  its boundaries, satisfying appropriate coherence axioms.

  This also describes the important examples, but also does not
  capture all of their structure.  In particular, with this definition
  there is no obvious way to extract (say) a horizontal bicategory
  consisting of objects, horizontal arrows, and squares bordered by
  vertical identities.
\end{itemize}

In this paper we propose a new definition of doubly weak double
category, which is closely related to the above approaches but is not
subject to either of their problems.  We will show that our doubly
weak double categories are equivalent to double bicategories
satisfying the ``saturation'' condition (which we call ``tidiness''),
and also to cubical bicategories satisfying an analogous condition.
Furthermore, from a certain perspective, our doubly weak double
categories are simply the double-categorical analogue of bicategories,
as we will explain next.


\subsection{Implicit structures}
\label{sec:implicit-structures}

Bicategories are typically regarded as more complicated than strict
2-categories. But from another point of view, bicategories are simpler
than strict 2-categories. Roughly, a bicategory is like a strict
2-category but \emph{without} equalities between compositions of 1-cells.


From this perspective, just like a group has ``fewer ingredients''
than a ring, a bicategory has ``fewer ingredients'' than a strict
2-category. In particular, when a definition of a 2-categorical
shape (e.g.\ the shape of an adjunction, a monad, or a module) makes
no reference to equality between compositions of 1-cells, it actually
belongs in the more general setting of bicategories.

Let us make this more precise. We start with a \textbf{2-computad}
(introduced by Street in~\cite{street:catval-2lim}\footnote{Street's
  computads were later generalized to $n$- and $\infty$-computads by
  Burroni~\cite{burroni:polygraphs} (who introduced them independently, calling them
  \emph{polygraphs}), Batanin~\cite{batanin:cptd-fin,batanin:cptd-slices}, and
  Makkai~\cite{hmz:computads-multitopes}.}), a ``2-category without
composition''.  Explicitly, this consists of
\begin{itemize}
\item
  a collection of 0-cells,
\item
  a collection of 1-cells, each with a source and a target
  0-cell, and
\item a collection of 2-cells, each with a source and a target string
  of 1-cells (where these 1-cells match along 0-cells as appropriate).
\end{itemize}

A 2-computad is the sort of structure that generates a free
2-category, just as a directed graph (a.k.a.\ 1-computad, a ``category
without composition'') is the sort of structure that generates a free
category; indeed, Street observed in~\cite{street:catval-2lim} that
2-categories are monadic over 2-computads. We can draw a 2-cell either
as a pasting diagram or a string diagram (the topological dual):
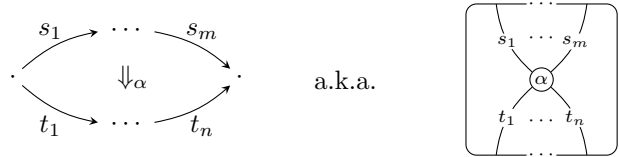
\begin{center}
  \begin{tikzpicture}[->, scale=.25]
    \node (a) at (-6, 0) {};
    \node (b) at (6, 0) {};
    \node at (a) {$\cdot$};
    \node at (b) {$\cdot$};
    \node (z) at (0, 0) {$\!\!\phantom{{}_\alpha}\Downarrow_\alpha$};
    \node (s1) at (-1, 2.5) {};
    \node (s2) at (1, 2.5) {};
    \node (t1) at (-1, -2.5) {};
    \node (t2) at (1, -2.5) {};
    \node (sh) at (0,2.5) {$\bcdots$};
    \node (th) at (0,-2.5) {$\bcdots$};
    \draw (a) [bend left=15] to node [auto,anchor=south,pos=.4] {$s_1$} (s1);
    \draw (s2) [bend left=15] to node [auto,anchor=south,pos=.6] {$s_m$} (b);
    \draw (a) [bend right=15] to node [auto,anchor=north,pos=.4] {$t_1$} (t1);
    \draw (t2) [bend right=15] to node [auto,anchor=north,pos=.6] {$t_n$} (b);
  \end{tikzpicture}
  \qquad a.k.a.\ \qquad\quad
  \begin{tikzpicture}[scale=.2]
    \draw [rect] (-5,-5) rectangle (5,5);
    \node [ov] (u) at (-3, 5) {};
    \node [ov] (v) at (3, 5) {};
    \node [iv] (a) at (0, 0) {$\alpha$};
    \node [ov] (x) at (-3, -5) {};
    \node [ov] (y) at (3, -5) {};
    \draw (u) [bend right=20] to node [ed] {$s_1$} (a);
    \draw (v) [bend left=20] to node [ed] {$s_m$} (a);
    \draw (a) [bend right=20] to node [ed] {$t_1$} (x);
    \draw (a) [bend left=20] to node [ed] {$t_n$} (y);
    \node [ed] at (0,-2.75) {$\bcdots$};
    \node [ed] at (0,2.75) {$\bcdots$};
    \node [ed] at (0,5) {$\bcdots$};
    \node [ed] at (0,-5) {$\bcdots$};
  \end{tikzpicture}
\end{center}

There is also an intermediate notion between a 2-computad and a
2-category: a structure in which the 2-cells can be composed, but the
1-cells cannot. We call this essentially algebraic structure {\aan}
\textbf{{\twoword}}.  It consists of
\begin{itemize}
\item a 2-computad,
\item 2-cell composition and identity operations (horizontal and vertical), and
\item associativity, unit, and interchange
  laws.
\end{itemize}
In other words, it has 0-cells, 1-cells, 2-cells with composition, and equalities between
compositions of 2-cells.
The compositions of 2-cells can be drawn for example as follows:


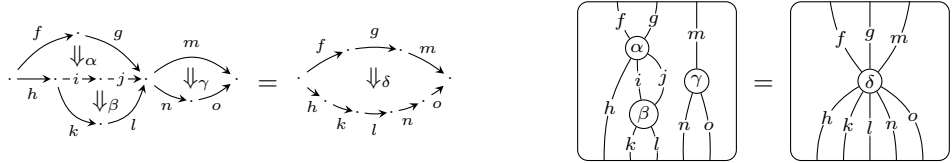
\begin{center}
  \begin{tikzpicture}[->,shorten <=-2pt,shorten >=-3pt,xscale=.15,yscale=.2,font=\tiny]
    \node (a) at (-10, 0) {$\cdot$};
    \node (b) at (2, 0) {$\cdot$};
    \node (c) at (10, 0) {$\cdot$};
    \node (s) at (-4,3) {$\cdot$};
    \node (m1) at (-6,0) {$\cdot$};
    \node (m2) at (-2,0) {$\cdot$};
    \node (t1) at (-2,-3) {$\cdot$};
    \node (t2) at (6,-1.5) {$\cdot$};
    \draw (a) [bend left=15] to node [auto,anchor=south,pos=.4] {$f$} (s);
    \draw (s) [bend left=15] to node [auto,anchor=south,pos=.6] {$g$} (b);
    \draw (a) -- node [auto,anchor=north] {$h$} (m1);
    \draw (m1) -- node [fill=white,inner sep=.75,pos=.4] {$i$} (m2);
    \draw (m2) -- node [fill=white,inner sep=.75,pos=.4] {$j$} (b);
    \draw (m1) [bend right=30] to node [auto,anchor=north,pos=.6] {$k$} (t1);
    \draw (t1) [bend right=30] to node [auto,anchor=north,pos=.6] {$l$} (b);
    \draw (b) [bend left=30] to node [auto,anchor=south] {$m$} (c);
    \draw (b) [bend right=10] to node [auto,anchor=north,pos=.4] {$n$} (t2);
    \draw (t2) [bend right=10] to node [auto,anchor=north,pos=.6] {$o$} (c);
    \node at (-4, 1.5) {\small$\!\!\phantom{{}_\alpha}\Downarrow_\alpha$};
    \node at (-2, -1.5) {\small$\!\!\phantom{{}_\beta}\Downarrow_\beta$};
    \node at (6, 0) {\small$\!\!\phantom{{}_\gamma}\Downarrow_\gamma$};
  \end{tikzpicture}
  $=$
  \begin{tikzpicture}[->,shorten <=-2pt,shorten >=-3pt,xscale=.1,yscale=.15,font=\tiny]
    \node (a) at (-10, 0) {$\cdot$};
    \node (c) at (10, 0) {$\cdot$};
    \node (s1) at (-3.5, 2.5) {$\cdot$};
    \node (s2) at (3.5, 2.5) {$\cdot$};
    \node (t1) at (-6.5, -2) {$\cdot$};
    \node (t2) at (-2.25, -3) {$\cdot$};
    \node (t3) at (2.25, -3) {$\cdot$};
    \node (t4) at (6.5, -2) {$\cdot$};
    \draw (a) [bend left=5] to node [auto,anchor=south,pos=.4] {$f$} (s1);
    \draw (s1) [bend left=5] to node [auto,anchor=south] {$g$} (s2);
    \draw (s2) [bend left=5] to node [auto,anchor=south,pos=.6] {$m$} (c);
    \draw (a) [bend right=5] to node [auto,anchor=north] {$h$} (t1);
    \draw (t1) [bend right=5] to node [auto,anchor=north] {$k$} (t2);
    \draw (t2) [bend right=5] to node [auto,anchor=north] {$l$} (t3);
    \draw (t3) [bend right=5] to node [auto,anchor=north] {$n$} (t4);
    \draw (t4) [bend right=5] to node [auto,anchor=north] {$o$} (c);
    \node at (0, 0) {\small$\!\!\phantom{{}_\delta}\Downarrow_\delta$};
  \end{tikzpicture}
  \qquad\qquad
  \begin{tikzpicture}[scale=.175]
    \draw [rect] (-6,-6) rectangle (6,6);
    \node [ov] (u) at (-3, 6) {};
    \node [ov] (v) at (0, 6) {};
    \node [ov] (w) at (3, 6) {};
    \node [iv] (a) at (-1.5, 2.5) {$\alpha$};
    \node [iv] (b) at (-1, -2.5) {$\beta$};
    \node [iv] (d) at (3, 0) {$\gamma$};
    \node [ov] (xx) at (-4, -6) {};
    \node [ov] (x) at (-2, -6) {};
    \node [ov] (y) at (0, -6) {};
    \node [ov] (z) at (2, -6) {};
    \node [ov] (zz) at (4, -6) {};
    \draw (u) [bend right=15] to node [ed] {$f$} (a);
    \draw (v) [bend left=15] to node [ed] {$g$} (a);
    \draw (a) [bend right=15] to node [ed] {$h$} (xx);
    \draw (a) [bend right=5] to node [ed] {$i$} (b);
    \draw (a) [bend left=55] to node [ed] {$j$} (b);
    \draw (b) [bend right=15] to node [ed] {$k$} (x);
    \draw (b) [bend left=15] to node [ed] {$l$} (y);
    \draw (w) -- node [ed] {$m$} (d);
    \draw (d) [bend right=10] to node [ed] {$n$} (z);
    \draw (d) [bend left=10] to node [ed] {$o$} (zz);
  \end{tikzpicture}
  \;$=$\;
  \begin{tikzpicture}[scale=.175]
    \draw [rect] (-6,-6) rectangle (6,6);
    \node [ov] (u) at (-3, 6) {};
    \node [ov] (v) at (0, 6) {};
    \node [ov] (w) at (3, 6) {};
    \node [iv] (a) at (0, 0) {$\delta$};
    \node [ov] (xx) at (-4, -6) {};
    \node [ov] (x) at (-2, -6) {};
    \node [ov] (y) at (0, -6) {};
    \node [ov] (z) at (2, -6) {};
    \node [ov] (zz) at (4, -6) {};
    \draw (u) [bend right=15] to node [ed] {$f$} (a);
    \draw (v) -- node [ed] {$g$} (a);
    \draw (w) [bend left=15] to node [ed] {$m$} (a);
    \draw (a) [bend right=30] to node [ed] {$h$} (xx);
    \draw (a) [bend right=15] to node [ed] {$k$} (x);
    \draw (a) -- node [ed] {$l$} (y);
    \draw (a) [bend left=15] to node [ed] {$n$} (z);
    \draw (a) [bend left=30] to node [ed] {$o$} (zz);
  \end{tikzpicture}
\end{center}

Equivalently, an implicit 2-category can be defined as a \emph{strict}
2-category whose underlying 1-category is freely generated; the
1-cells of the implicit 2-category then being the \emph{generating}
1-cells of this free category.

An implicit 2-category is already quite close to a bicategory, but one more detail is
required. {\Aan} {\twoword} is called
\textbf{representable}\footnote{This usage of ``representable'' traces
  back to the representable  multicategories of~\cite{hermida}.} if each
string of compatible 1-cells is isomorphic to a single 1-cell. (It is
sufficient to require this for binary and nullary strings.) This
allows the 1-cells to be ``composed'', where a ``composite'' 1-cell is
defined up to isomorphism only.

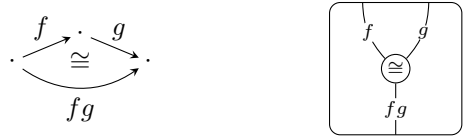
\begin{center}
  \begin{tikzpicture}[->,shorten <=.5pt, scale=.15]
    \node (a) at (-6, 0) {};
    \node (b) at (6, 0) {};
    \node (s) at (0,2.5) {};
    \node at (a) {$\cdot$};
    \node at (b) {$\cdot$};
    \node at (s) {$\cdot$};
    \node (z) at (0, 0) {$\cong$};
    \draw (a) -- node [auto,anchor=south,pos=.4] {$f$} (s);
    \draw (s) -- node [auto,anchor=south,pos=.6] {$g$} (b);
    \draw (a) [bend right=35] to node [auto,anchor=north] {$\ubcom{f}{g}$} (b);
  \end{tikzpicture}
  \qquad\qquad\qquad
  \begin{tikzpicture}[scale=.175]
    \draw [rect] (-5,-5) rectangle (5,5);
    \node [ov] (u) at (-2.5, 5) {};
    \node [ov] (v) at (2.5, 5) {};
    \node [iv] (a) at (0, 0) {$\cong$};
    \node [ov] (x) at (0, -5) {};
    \draw (u) [bend right=20] to node [ed] {$f$} (a);
    \draw (v) [bend left=20] to node [ed] {$g$} (a);
    \draw (a) -- node [ed] {$\ubcom{f}{g}$} (x);
  \end{tikzpicture}
\end{center}

In \cref{sec:twowords} we will show that the category of bicategories
and pseudofunctors is equivalent to that of representable {\twowords}
and {\twoword} functors (homomorphisms of the essentially
algebraic structure).  This alternative definition of bicategory is
appealing for several reasons. First of all, there are no coherence
axioms. Secondly, there is no extraneous structure present that is not
respected by isomorphism of bicategories; it is not possible to even
express equality between compositions of 1-cells, which is
conceptually clarifying.





Having considered the situation for 2-categories, we proceed to treat
double categories in just the same way.
A \textbf{double computad} is the sort of structure that generates a
free double category: it has 0-cells, horizontal and vertical 1-cells,
and 2-cells bordered by strings of compatible 1-cells.
We can draw 2-cells in a double computad either as
pasting diagrams or string diagrams (string diagrams for double
categories are discussed in~\cite{myers:string-dbl}):
\begin{center}
  \begin{tikzpicture}[->, scale=.2]
    \node (a) at (-5, 5) {};
    \node (b) at (5, 5) {};
    \node (c) at (-5, -5) {};
    \node (d) at (5, -5) {};
    \node at (a) {$\cdot$};
    \node at (b) {$\cdot$};
    \node at (c) {$\cdot$};
    \node at (d) {$\cdot$};
    \node (z) at (0, 0) {$\alpha$};
    \node (sh) at (0,5) {$\bcdots$};
    \node (th) at (0,-5) {$\bcdots$};
    \node (tv) at (5,0) {$\vcdots$};
    \node (sv) at (-5,0) {$\vcdots$};
    \draw (a) -- node [auto,anchor=south] {$s^H_1$} (sh);
    \draw (sh) -- node [auto,anchor=south] {$s^H_a$} (b);
    \draw (b) -- node [auto,anchor=west] {$t^V_1$} (tv);
    \draw (tv) -- node [auto,anchor=west] {$t^V_b$} (d);
    \draw (a) -- node [auto,anchor=east] {$s^V_1$} (sv);
    \draw (sv) -- node [auto,anchor=east] {$s^V_c$} (c);
    \draw (c) -- node [auto,anchor=north] {$t^H_1$} (th);
    \draw (th) -- node [auto,anchor=north] {$t^H_d$} (d);
  \end{tikzpicture}
  \qquad a.k.a.\ \quad\qquad
  \begin{tikzpicture}[scale=.25]
    \draw [rect] (-5,-5) rectangle (5,5);
    \node [ov] (sh1) at (-2, 5) {};
    \node [ov] (sh2) at (2, 5) {};
    \node [ov] (th1) at (-2, -5) {};
    \node [ov] (th2) at (2, -5) {};
    \node [ov] (tv2) at (5, -2) {};
    \node [ov] (tv1) at (5, 2) {};
    \node [ov] (sv2) at (-5, -2) {};
    \node [ov] (sv1) at (-5, 2) {};
    \node [ivs] (a) at (0, 0) {$\alpha$};
    \draw (a) [bend left=10,pos=0.75] to node [ed] {$s^H_1$} (sh1);
    \draw (a) [bend right=10,pos=0.75] to node [ed] {$s^H_a$} (sh2);
    \draw (a) [bend right=10,pos=0.75] to node [ed] {$t^H_1$} (th1);
    \draw (a) [bend left=10,pos=0.75] to node [ed] {$t^H_b$} (th2);
    \draw (a) [bend right=10,pos=0.75] to node [ed] {$s^V_1$} (sv1);
    \draw (a) [bend left=10,pos=0.75] to node [ed] {$s^V_c$} (sv2);
    \draw (a) [bend left=10,pos=0.75] to node [ed] {$t^V_1$} (tv1);
    \draw (a) [bend right=10,pos=0.75] to node [ed] {$t^V_d$} (tv2);
    \node [ed,font=\tiny] at (0,-3.75) {$\bcdots$};
    \node [ed,font=\tiny] at (0,3.75) {$\bcdots$};
    \node [ed,font=\tiny] at (0,5) {$\bcdots$};
    \node [ed,font=\tiny] at (0,-5) {$\bcdots$};
    \node [ed,font=\tiny] at (-3.75,0) {$\vcdots$};
    \node [ed,font=\tiny] at (3.75,0) {$\vcdots$};
    \node [ed,font=\tiny] at (5,0) {$\vcdots$};
    \node [ed,font=\tiny] at (-5,0) {$\vcdots$};
  \end{tikzpicture}
\end{center}
{\Aan}
\textbf{\doubleword} is then a double computad with composition operations on 2-cells like in a double category, but \emph{without}
any composition of 1-cells (neither horizontal nor vertical).
We can then define a \textbf{doubly weak double category} to be {\aan}
{\doubleword} that is representable, i.e.\ every string of compatible
1-cells (horizontal or vertical) has a composite.  Thus defined,
doubly weak double categories are the algebras for a finitary monad on
double computads.


\begin{rmk}
  Implicit structures are related to the \emph{virtual} structures
  of~\cite{cruttwell-shulman} (generalized multicategories). For
  instance, a \emph{virtual 2-category} is like an implicit 2-category
  but requires the targets of all 2-cells to be length-1 paths (and
  restricts compositions to those that preserve this property).  A
  \emph{virtual double category} likewise restricts the lower
  boundaries of 2-cells to be length-1 paths, but as with pseudo
  double categories, the vertical 1-cells compose strictly, breaking
  the horizontal/vertical symmetry.

  In fact, after this paper was completed, we learned that Keisuke
  Hoshino has characterized virtual double categories, and also
  augmented virtual double categories~\cite{koudenburg:aug-vdc} (which
  also allow length-0 paths as lower boundaries), as implicit double
  categories with appropriate composition and decomposition
  operations.

  Like implicit structures, virtual structures allow the
  characterization of weak structures without explicit coherence
  axioms, but there are two main differences. Firstly, in {\aan} {\twoword} we can
  define composites simply in terms of isomorphisms internal to the
  structure, whereas in a virtual 2-category composites must be
  defined by way of universal properties (since an inverse of a
  many-to-one morphism would be a one-to-many morphism, which virtual
  structures do not have). Secondly, both representable {\twowords}
  and representable virtual 2-categories can be identified with
  bicategories; however, the maps of {\twowords} (using the definition
  of homomorphism that is automatic from the essentially algebraic
  presentation) correspond to pseudofunctors, whereas the maps of
  virtual 2-categories correspond to \emph{lax} functors.
  
  Moreover, virtual structures apparently cannot be used to define
  doubly weak double categories, since there does not seem to be a
  sensible notion of a double category that is both horizontally and
  vertically virtual.
\end{rmk}

\subsection{Other definitions}
\label{sec:computads}

The presentation outlined in \cref{sec:implicit-structures} gives a
monad on double computads whose algebras are doubly weak double
categories (with chosen composites). But we may also describe the
algebras of this monad more directly, without factoring through the
intermediate step of implicit double categories: a doubly weak double
category is a double computad equipped with 1-cell composition and
identities (satisfying no axioms), plus a way of composing any formal
diagram of 2-cells\footnote{Here we refer only to formal diagrams that
  ``ought to be composable in a double category''.  There are other
  formal diagrams, such as
  ``pinwheels''~\cite{dp:gencomp-dbl,dawson:compdiag-dbl}, that are
  not even composable in a strict double category, and these are not
  composable in our doubly weak double categories either.  (Although
  one can also consider (strict or weak) double categories in which
  such diagrams are composable, e.g.~\cite{lr:pinwheel-dbl}.)}
along any way of composing up its boundaries, satisfying appropriate
coherence axioms (see \cref{cor:freedescription}).

This is similar to Garner's definition of cubical bicategory as
described above in \cref{sec:doubly-weak-double}; the only difference
is that our definition uses a double computad, whereas the 2-cells in
Garner's definition are all \emph{squares}, i.e.\ the horizontal and
vertical sources and targets are length-1 paths.  Indeed, we will show
that Garner's and Verity's definitions both can be derived from ours
by simply ignoring some of the structure of a double computad.

More precisely, the forgetful functor from doubly weak
double categories to \emph{double graphs} (double computads consisting
of only 0-cells, 1-cells, and squares) induces a monad whose algebras
are precisely Garner's cubical bicategories.  Likewise, the forgetful
functor to \emph{double graphs with bigons} (double computads
consisting of only 0-cells, 1-cells, squares, and horizontal and vertical bigons\footnote{By a ``bigon'' we mean a globular 2-cell, having two opposite boundary paths of length 1 and the other two opposite paths of length 0.}) induces a
monad whose algebras are precisely Verity's double bicategories.

In particular, our doubly weak double categories are \emph{not}
monadic over double graphs or double graphs with bigons; additional
shapes featured in a double computad are necessary. (This is perhaps
surprising, since bicategories \emph{are} monadic over 2-graphs,
a.k.a.\ 2-globular sets.)  However, these forgetful functors are ``the
next best thing'' to monadic: they are ``of descent type'', which in
this case means that the comparison functors from doubly weak double
categories to double bicategories \emph{and} to cubical bicategories
are fully faithful.  Thus we can indeed describe a doubly weak double
category as \emph{structure} on a double graph with bigons, or on a
double graph, though these structures are not monadic.

We refer to the resulting equivalent notions of doubly weak double
category respectively as \emph{tidy double bicategories} and
\emph{tidy cubical bicategories}.  Tidiness in both cases is a similar
condition: it says that the operations of composing a square or bigon
with an identity square are bijections.  As noted above, tidiness for double
bicategories is not a new condition; it already appeared without a
name in Verity's thesis~\cite[Lemma 1.4.9]{verity:base-change}, and
in~\cite{rwan:univalent-doublecats} it is called \emph{saturation}.
Our general theory shows that this apparently \emph{ad hoc} condition
does indeed yield a ``correct'' definition, in any reasonable sense.

With that said, an advantage of tidy double bicategories is that they
yield an entirely \emph{finite} presentation of doubly weak double
categories, which we will show can be reduced to a double graph with
binary composition and identity operations, and associator and unitor
coherence squares, and appropriate axioms.  This is perhaps the
simplest definition, and the most amenable to checking all the pieces
by hand in an example.

Finally, we give one last equivalent finite presentation, exhibiting
doubly weak double categories as \emph{monadic} over the category of
double computads containing only 0-cells, 1-cells, squares, and all
four kinds of \emph{monogons}.

\subsection{Outline}
\label{sec:outline}

The structure of the paper is as follows.  In \cref{sec:twowords}, we
spell out in detail the correspondence between bicategories and
representable {\twowords}, using a quick definition of {\twowords} as
strict 2-categories with free underlying 1-category.  Then in
\cref{sec:shortcut}, we by analogy quickly define implicit double
categories, doubly weak double categories, and pseudofunctors between
them, and give some examples (one with proofs postponed to
\cref{sec:twowords-hom}).

Then we move on to the computadic definitions.  In
\cref{sec:double-computads}, we introduce double computads.  In
\cref{sec:doublewords}, we present implicit structures, weak
structures, and strict structures as monads on computads.  And in
\cref{sec:2-monads}, we upgrade the categories of implicit and weak
structures to 2-categories, upgrade the monads to 2-monads, and prove
coherence theorems.

Finally we consider alternative definitions and finite presentations: we discuss tidy double
bicategories in \cref{sec:tidy}, tidy cubical bicategories in
\cref{sec:cubical}, and monogons in \cref{sec:finax}.


\subsection{Acknowledgments}
\label{sec:acknowledgments}

We are grateful to Nathanael Arkor for a careful reading and several
helpful suggestions and to Bob Par\'e for helpful discussions.


\section{Bicategories}
\label{sec:twowords}

We first spell out the equivalence between bicategories and
representable {\twowords}, alluded to in the introduction
(\cref{sec:implicit-structures}).  Although it is helpful to view
{\twowords} as \emph{prior} to 2-categories, to get the main ideas
across as quickly as possible, we start with a definition of
{\twowords} in terms of strict 2-categories.  Later we will give an
alternative definition without reference to strict 2-categories, and
describe 2-categories as extra structure on top of it.

\begin{defn}\label{defn:twoword}
  {\Aan} \textbf{{\twoword}} is a strict 2-category whose
  category of 1-cells is free (i.e.\ freely generated by a directed
  graph).

  We call the generating 1-cells simply \textbf{1-cells}, and we do
  not use this word for their compositions, which we rather call
  \textbf{paths of 1-cells}. The arrows and strings
  in our pasting diagrams and string diagrams always refer to
  generating 1-cells, and we draw these arrows with a distinguished arrowhead
  \tikz{ \draw[->] (0,0) -- (.5,0); }.
  We call a 2-cell whose source and target are both length 1 paths a
  \textbf{bigon}.
\end{defn}



A \textbf{functor} of {\twowords} is a strict 2-functor \emph{that
 sends 1-cells to 1-cells}.  We write \tword for the
category of {\twowords} and such functors. 

For clarity, we may call the strict 2-category associated to {\aan}
{\twoword} its \textbf{path 2-category}. (1-cells in the path
2-category are paths of 1-cells in the {\twoword}.)

When a path of 1-cells is isomorphic to a single 1-cell, we call the
latter a \textbf{composite} of the path. We call {\aan} {\twoword}
\textbf{representable} if each path of 1-cells has a composite.
\begin{center}
  $\forall$
  \begin{tikzpicture}[->, scale=.15]
    \node (a) at (-6, 0) {};
    \node (b) at (6, 0) {};
    \node at (a) {$\cdot$};
    \node at (b) {$\cdot$};
    \node (s1) at (-1, 0) {};
    \node (s2) at (1, 0) {};
    \node (sh) at (0,0) {$\bcdots$};
    \draw (a) -- node [auto,anchor=south] {$f_1$} (s1);
    \draw (s2) -- node [auto,anchor=south] {$f_a$} (b);
    \path (0,6) -- (0, -6);
  \end{tikzpicture}\!\!,
  \;\;$\exists$
  \begin{tikzpicture}[->, scale=.15]
    \node (a) at (-6, 0) {};
    \node (b) at (6, 0) {};
    \node at (a) {$\cdot$};
    \node at (b) {$\cdot$};
    \node (z) at (0, 0) {$\cong$};
    \node (s1) at (-1, 2.5) {};
    \node (s2) at (1, 2.5) {};
    \node (sh) at (0,2.25) {$\bcdots$};
    \draw (a) -- node [auto,anchor=south,pos=.4] {$f_1$} (s1);
    \draw (s2) -- node [auto,anchor=south,pos=.6] {$f_a$} (b);
    \draw (a) [bend right=35] to node [auto,anchor=north] {$f$} (b);
    \path (0,6) -- (0, -6);
  \end{tikzpicture}
  \;\; a.k.a.\ \;\;\;
  $\forall$
  \begin{tikzpicture}[scale=.175]
    \draw [rect] (-5,-5) rectangle (5,5);
    \node [ov] (u) at (-3, 5) {};
    \node [ov] (v) at (3, 5) {};
    \node [ov] (x) at (-3, -5) {};
    \node [ov] (y) at (3, -5) {};
    \draw (u) -- node [ed] {$f_1$} (x);
    \draw (v) -- node [ed] {$f_n$} (y);
    \node [ed] at (0,5) {$\bcdots$};
    \node [ed] at (0,0) {$\bcdots$};
    \node [ed] at (0,-5) {$\bcdots$};
  \end{tikzpicture}\;,
  \;\; $\exists$
  \begin{tikzpicture}[scale=.175]
    \draw [rect] (-5,-5) rectangle (5,5);
    \node [ov] (u) at (-3, 5) {};
    \node [ov] (v) at (3, 5) {};
    \node [iv] (a) at (0, 0) {$\cong$};
    \node [ov] (x) at (0, -5) {};
    \draw (u) [bend right=20] to node [ed] {$f_1$} (a);
    \draw (v) [bend left=20] to node [ed] {$f_n$} (a);
    \draw (a) -- node [ed] {$f$} (x);
    \node [ed] at (0,5) {$\bcdots$};
    \node [ed] at (0,2.75) {$\bcdots$};
  \end{tikzpicture}\;.
\end{center}

\begin{rmk} {\Aan} {\twoword} with one 0-cell and one 1-cell is known
  elsewhere as a \textbf{PRO}; {\aan} {\twoword} with one 0-cell
  (which we might call {\aan} ``{\adjective} monoidal category'') is
  often called a \textbf{colored PRO}.


  The result to be shown, that bicategories are equivalent to
  representable {\twowords}, specializes to that monoidal categories
  are equivalent to representable colored PROs.
\end{rmk}

\begin{defn}
  {\Aan} {\twoword} is \textbf{represented} if it has a
  \emph{chosen} isomorphism between each length 2 or 0 path of 1-cells
  and a composite 1-cell.
  
  It follows that every path of 1-cells has a composite (i.e.\
  represented implies representable).
  We denote the chosen composite of 1-cells $f \colon A \to B$ and
  $g \colon B \to C$ by $\ubcom{f}{g} \colon A \to C$ and we denote the
  chosen nullary composite at the 0-cell $A$ by
  $\bi{A} \colon A \to A$. A functor between represented {\twowords} is
  called \textbf{strict} if it preserves the chosen composition
  isomorphisms.
\end{defn}

\begin{rmk}
  One could alternatively suppose a chosen
  composition isomorphism for \emph{every} path of 1-cells, instead of
  just binary and nullary paths. This would be equivalent to an
  \emph{unbiased} bicategory.
\end{rmk}

To translate from bicategories to represented {\twowords} is the
construction known as \emph{strictification}. (Strictification of
bicategories is typically described as a functor $\bBicat \to \tcat$,
but it may be described slightly more precisely as a functor
$\bBicat \to \tword$.) 
For proof that this indeed defines a functor,
we refer to e.g.~\cite[Chapter 2]{gurski};\footnote{The definition of
  strictification in \cite{gurski} makes choices of parenthesizations
  whereas we use ``cliques'' of parenthesizations following
  \cite{nlab:strict}; this makes no essential difference but we find
  the presentation with cliques helpful.} showing this from the
definitions below amounts to a series of straightforward
verifications.


By a \emph{bracketing} of a path of 1-cells in a bicategory, we mean a
single 1-cell produced from the path by composing and introducing
units. For example, a path $f,g,h$ could be bracketed as $(fg)h$,
$(f1)(gh)$, $((fg)h)(11)$, or infinitely many other ways. By the
coherence theorem for bicategories, for any two bracketings of a path
of 1-cells there is a canonical \emph{rebracketing} isomorphism
between them built from coherence isomorphism.

\begin{prop}\label{prop:bitotwo}
  Given a bicategory $\bicat{C}$, the following data amount to a
  represented {\twoword}:
  \begin{itemize}
  \item The 0-cells and 1-cells are as in $\bicat{C}$.
  \item A 2-cell from $s_1, \ldots, s_m$ to $t_1, \ldots, t_n$ is a
    family consisting of a 2-cell in $\bicat{B}$ for every possible
    bracketing of the source and target, such that these 2-cells are
    related by composing with the appropriate rebracketing coherence
    isomorphisms (a.k.a.\ a \emph{clique morphism}).
  \item Composition of 2-cells (including identities) is induced by
    composition of 2-cells in $\bicat{C}$.
  \item The composition isomorphisms are given by identities.
  \end{itemize}
\end{prop}
\begin{proof}
  The coherence theorem for bicategories guarantees that each
    2-cell from a bracketed form of $s_1 \cdots s_m$ to a bracketed
    form of $t_1 \cdots t_n$ determines, by composing with coherence
    isomorphisms, a unique corresponding 2-cell for every rebracketing
    of the source and target.
    Thus composition is well-defined, since rebracketing then
    composing 2-cells is the same as composing then rebracketing as
    appropriate.
    The axioms follow from coherence and the bicategory axioms.
\end{proof}


We call this the ``underlying {\twoword}'' of a bicategory.
Similarly, using coherence for pseudofunctors, we have:

\begin{prop}\label{prop:pstofun}
  A pseudofunctor between bicategories $\bicatf{F} \colon \bicat{C} \to \bicat{D}$
  induces a functor (\emph{not} necessarily preserving chosen
  composition isomorphisms) between the underlying {\twowords} as
  follows:
  \begin{itemize}
  \item The maps of 0-cells and 1-cells are as in $\bicatf{F}$.
  \item The map on 2-cells is by applying $\bicatf{F}$ and composing
    with pseudofunctor coherence isomorphisms. (2-cells in $\bicat{C}$
    between $\bicat{C}$-bracketed paths of 1-cells map to 2-cells in
    $\bicat{D}$ between $\bicat{D}$-bracketed paths of corresponding
    1-cells.)
  \end{itemize}

  Moreover, this defines a functor $\bBicat \to \tword$.\qed
\end{prop}



Next we see this functor $\bBicat \to \tword$ is fully faithful, and
its image consists of the representable {\twowords}.

\begin{prop}\label{prop:twotobi}
  Given a represented {\twoword} $\cat{C}$, the following data amount
  to a bicategory:
  \begin{itemize}
  \item The 0-cells are the 0-cells in $\cat{C}$.
  \item The category $\Hom(A, B)$ is the category of bigons between $A$
    and $B$ in $\cat{C}$.
  \item Composition and identity for 1-cells is as in $\cat{C}$.
  \item Horizontal composition of 2-cells is by horizontally composing
    bigons in $\cat{C}$, and converting to a bigon (by vertically
    composing with composition isomorphisms):
    
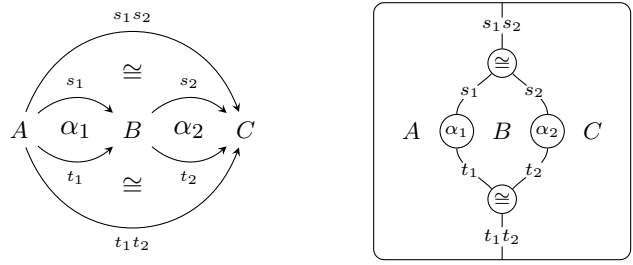
\[
      \begin{tikzpicture}[->,xscale=.75,yscale=.65]
        \node (a) at (-2,0) {\small$A$};
        \node (b) at (0,0) {\small$B$};
        \node (c) at (2,0) {\small$C$};
        \node at (-1,0) {\large$\alpha_1$};
        \node at (1,0) {\large$\alpha_2$};
        \node at (0,1.15) {$\choseniso$};
        \node at (0,-1.15) {$\choseniso$};
        \draw (a) [bend left=19] .. controls +(.75,.75) and +(-.75,.75) .. node [above] {\tiny$s_1$} (b);
        \draw (b) [bend left=19] .. controls +(.75,.75) and +(-.75,.75) .. node [above] {\tiny$s_2$} (c);
        \draw (a) [bend right=19] .. controls +(.75,-.75) and +(-.75,-.75) .. node [below] {\tiny$t_1$} (b);
        \draw (b) [bend right=19] .. controls +(.75,-.75) and +(-.75,-.75) .. node [below] {\tiny$t_2$} (c);
        \draw (a) .. controls +(.5,1.5) and +(-.75,0) .. (0,2) node [above] {\tiny$\ubcom{s_1}{s_2}$} .. controls +(.75,0) and +(-.5,1.5) .. (c);
        \draw (a) .. controls +(.5,-1.5) and +(-.75,0) .. (0,-2) node [below] {\tiny$\ubcom{t_1}{t_2}$} .. controls +(.75,0) and +(-.5,-1.5) .. (c);
      \end{tikzpicture}
      \qquad
      \qquad
      \begin{tikzpicture}[scale=.4,looseness=.75]
        \node (a) at (-3,0) {\small$A$};
        \node (b) at (0,0) {\small$B$};
        \node (c) at (3,0) {\small$C$};
        \draw [rect] (-4.25,-4.25) rectangle (4.25,4.25);
        \node [ov] (top) at (0, 4.25) {};
        \node [ov] (bottom) at (0, -4.25) {};
        \node [iv] (a) at (0, 2.25) {$\choseniso$};
        \node [iv] (b) at (0, -2.25) {$\choseniso$};
        \node [iv] (2cella) at (-1.5, 0) {$\alpha_1$};
        \node [iv] (2cellb) at (1.5, 0) {$\alpha_2$};
        \draw (top) -- node [ed] {$\ubcom{s_1}{s_2}$} (a);
        \draw (bottom) -- node [ed] {$\ubcom{t_1}{t_2}$} (b);
        \draw (a) [out=-135, in=90] to node [ed] {$s_1$} (2cella);
        \draw (2cella) [out=-90, in=135] to node [ed] {$t_1$} (b);
        \draw (a) [out=-45, in=90] to node [ed] {$s_2$} (2cellb);
        \draw (2cellb) [out=-90, in=45] to node [ed] {$t_2$} (b);
      \end{tikzpicture}
    \]
  \item The components of left and right unitors and associators are
    induced by the composition isomorphisms (by de-composing then
    re-composing):
    
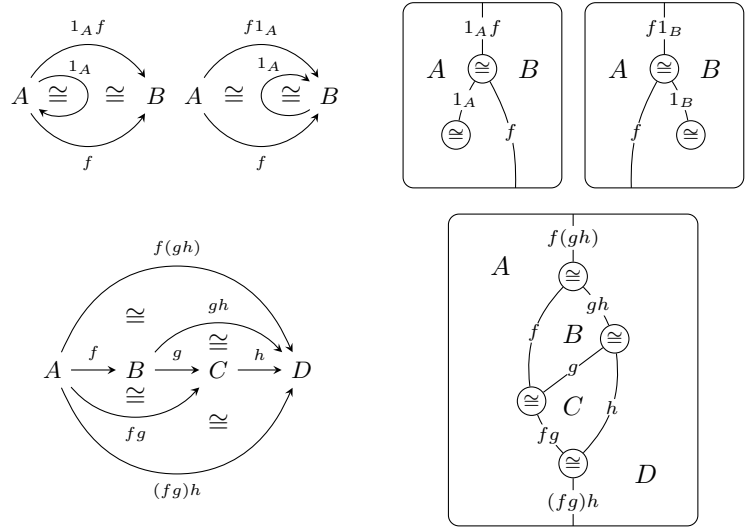
\[
      \begin{tikzpicture}[->,rotate=-90,scale=.9]
        \begin{scope}[xscale=.3]
          \node (a) at (0,0) {\small$A$};
          \node (b) at (0,2) {\small$B$};
          \node at (0,1.4) {$\choseniso$};
          \begin{scope}[yscale=.8]
            \node at (0,.7) {$\choseniso$};
            \draw (a) .. controls +(-1.25,.65) and +(-1,0) .. (0,1.25) node [pos=.45,above right=-3] {\tiny$\bi{A}$} .. controls +(1,0) and +(1.25,.65) .. (a);
          \end{scope}
          \draw (a) .. controls +(-1.35,.25) and +(0,-.45) .. (-2.5,1) node [above] {\tiny$\ubcom{\bi{A}}{f}$} .. controls +(0,.45) and +(-1.35,-.25) .. (b);
          \draw (a) .. controls +(1.35,.25) and +(0,-.45) .. (2.5,1) node [below] {\tiny$f$} .. controls +(0,.45) and +(1.35,-.25) .. (b);
        \end{scope}
      \end{tikzpicture}
      \begin{tikzpicture}[<-,rotate=90,scale=.9]
        \begin{scope}[xscale=.3]
          \node (a) at (0,0) {\small$B$};
          \node (b) at (0,2) {\small$A$};
          \node at (0,1.4) {$\choseniso$};
          \begin{scope}[yscale=.8]
            \node at (0,.7) {$\choseniso$};
            \draw [->] (a) .. controls +(-1.25,.65) and +(-1,0) .. (0,1.25) .. controls +(1,0) and +(1.25,.65) .. node [pos=.6,above left=-2] {\tiny$\bi{A}$} (a);
          \end{scope}
          \draw (a) .. controls +(-1.35,.25) and +(0,-.45) .. (-2.5,1) node [below] {\tiny$f$} .. controls +(0,.45) and +(-1.35,-.25) .. (b);
          \draw (a) .. controls +(1.35,.25) and +(0,-.45) .. (2.5,1) node [above] {\tiny$\ubcom{f}{\bi{A}}$} .. controls +(0,.45) and +(1.35,-.25) .. (b);
        \end{scope}
      \end{tikzpicture}
      \qquad
      \begin{tikzpicture}[scale=.175]
        \draw [rect] (-6,-7) rectangle (6,7);
        \node [ov] (v) at (2.5, -7) {};
        \node (0cella) at (-3.5, 2) {$A$};
        \node (0cellb) at (3.5, 2) {$B$};
        \node [iv] (a) at (-2, -3) {$\choseniso$};
        \node [iv] (c) at (0, 2) {$\choseniso$};
        \node [ov] (y) at (0, 7) {};
        \draw (v) [bend right=15] to node [ed] {$f$} (c);
        \draw (a) [bend left=10] to node [ed] {$\bi{A}$} (c);
        \draw (c) -- node [ed] {$\ubcom{\bi{A}}{f}$} (y);
      \end{tikzpicture}
      \;\;\;
      \begin{tikzpicture}[scale=.175]
        \draw [rect] (-6,-7) rectangle (6,7);
        \node [ov] (v) at (-2.5, -7) {};
        \node (0cella) at (-3.5, 2) {$A$};
        \node (0cellb) at (3.5, 2) {$B$};
        \node [iv] (a) at (2, -3) {$\choseniso$};
        \node [iv] (c) at (0, 2) {$\choseniso$};
        \node [ov] (y) at (0, 7) {};
        \draw (v) [bend left=15] to node [ed] {$f$} (c);
        \draw (a) [bend right=10] to node [ed] {$\bi{B}$} (c);
        \draw (c) -- node [ed] {$\ubcom{f}{\bi{B}}$} (y);
      \end{tikzpicture}
    \]
    \[\hspace{-.5em}
      \begin{tikzpicture}[->,xscale=.55,yscale=-.5]
        \node (a) at (-3,0) {\small$A$};
        \node (b) at (-1,0) {\small$B$};
        \node (c) at (1,0) {\small$C$};
        \node (d) at (3,0) {\small$D$};
        \node at (1,1.4) {$\choseniso$};
        \node at (-1,.7) {$\choseniso$};
        \node at (1,-.7) {$\choseniso$};
        \node at (-1,-1.4) {$\choseniso$};
        \draw (a) -- node [above] {\tiny$f$} (b);
        \draw (b) -- node [above] {\tiny$g$} (c);
        \draw (c) -- node [above] {\tiny$h$} (d);
        \draw (a) .. controls +(.375,.8) and +(-2,0) .. (0,2.75) node [below] {\tiny$\ubcom{\bcom{f}{g}}{h}$} .. controls +(2,0) and +(-.375,.8) .. (d);
        \draw (a) .. controls +(.5,.5) and +(-1,0) .. (-1,1.25) node [below] {\tiny$\ubcom{f}{g}$} .. controls +(1,0) and +(-.5,.5) .. (c);
        \draw (b) .. controls +(.5,-.5) and +(-1,0) .. (1,-1.25) node [above] {\tiny$\ubcom{g}{h}$} .. controls +(1,0) and +(-.5,-.5) .. (d);
        \draw (a) .. controls +(.375,-.8) and +(-2,0) .. (0,-2.75) node [above] {\tiny$\ubcom{f}{\bcom{g}{h}}$} .. controls +(2,0) and +(-.375,-.8) .. (d);
      \end{tikzpicture}
      \qquad
      \,\hspace{.5em}
      \qquad
      \begin{tikzpicture}[xscale=.275,yscale=-.275]
        \draw [rect] (-6,-7.5) rectangle (6,7.5);
        \node [ov] (v) at (0, 7.5) {};
        \node (0cella) at (-3.5, -5) {$A$};
        \node (0cellb) at (0, -1.75) {$B$};
        \node (0cellc) at (0, 1.75) {$C$};
        \node (0celld) at (3.5, 5) {$D$};
        \node [iv] (a) at (0, 4.5) {$\choseniso$};
        \node [iv] (b) at (-2, 1.5) {$\choseniso$};
        \node [iv] (c) at (2, -1.5) {$\choseniso$};
        \node [iv] (d) at (0, -4.5) {$\choseniso$};
        \node [ov] (y) at (0, -7.5) {};
        \draw (v) -- node [ed] {$\ubcom{\bcom{f}{g}}{h}$} (a);
        \draw (a) [bend right=10] to node [ed] {$\ubcom{f}{g}$} (b);
        \draw (a) [bend left=25] to node [ed] {$h$} (c);
        \draw (b) -- node [ed] {$g$} (c);
        \draw (b) [bend right=25] to node [ed] {$f$} (d);
        \draw (c) [bend left=10] to node [ed] {$\ubcom{g}{h}$} (d);
        \draw (d) -- node [ed] {$\ubcom{f}{\bcom{g}{h}}$} (y);
      \end{tikzpicture}
    \]
  \end{itemize}
\end{prop}
\begin{proof}
  Functoriality, naturality, pentagon, and triangle follow from
  composition isomorphisms cancelling with their inverses.
\end{proof}

We call this the ``underlying bicategory'' of a represented
{\twoword}.

\begin{prop}\label{prop:funtops}
  A functor between represented {\twowords}
  $F \colon \cat{C} \to \cat{D}$ (\emph{not} necessarily preserving
  the chosen composition isomorphisms) induces a pseudofunctor between
  the underlying bicategories as follows:
  \begin{itemize}
  \item The functor is $F$ on 0-cells, 1-cells, and 2-cells
    (bigons).
  \item The coherence isomorphisms $\bi{FA} \to F\bi{A}$ and
    $\ubcom{(Ff)}{(Fg)} \to F\bcom{f}{g}$ are built from the chosen
    composition isomorphisms (by de-composing in $\cat{D}$ and
    re-composing in $\cat{C}$):
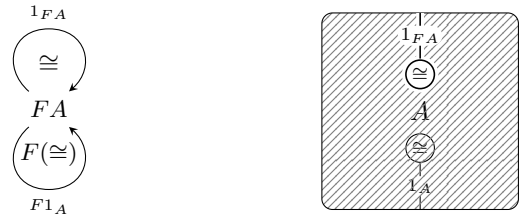
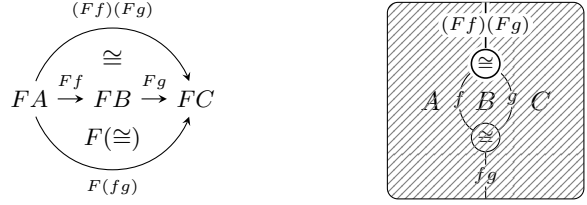
    \[
      \quad\quad\;\,
      \begin{tikzpicture}[->,xscale=.5,yscale=.85]
        \node (a) at (0,0) {\small$FA$};
        \node at (0,.7) {$\choseniso$};
        \node at (0,-.7) {\small$F(\choseniso)$};
        \draw (a) .. controls +(-1.25,.65) and +(-1,0) .. (0,1.25) node [above] {\tiny$\bi{FA}$} .. controls +(1,0) and +(1.25,.65) .. (a);
        \draw (a) .. controls +(-1.25,-.65) and +(-1,0) .. (0,-1.25) node [below] {\tiny$F\bi{A}$} .. controls +(1,0) and +(1.25,-.65) .. (a);
      \end{tikzpicture}
      \;\qquad\qquad\qquad\qquad
      \begin{tikzpicture}[scale=.325]
        \node (0cell) at (0, 0) {$A$};
        \node [iv] (2cell) at (0, -1.5) {$\choseniso$};
        \node [ov] (1cell) at (0, -4) {};
        \draw (2cell) -- node [ed] {$\bi{A}$} (1cell);
        \draw [rect, fun, pattern=north east lines] (-4,-4) rectangle (4,4);
        \node [ov] (1cellout) at (0, 4) {};
        \node [ivf] (2cellout) at (0, 1.5) {$\choseniso$};
        \draw [edf] (1cellout) -- node [ed] {$\bi{FA}$} (2cellout);
      \end{tikzpicture}
    \]
    \[
      \begin{tikzpicture}[->,xscale=.55,yscale=.75]
        \node (a) at (-2,0) {\small$FA$};
        \node (b) at (0,0) {\small$FB$};
        \node (c) at (2,0) {\small$FC$};
        \node at (0,.7) {$\choseniso$};
        \node at (0,-.7) {\small$F(\choseniso)$};
        \draw (a) -- node [above] {\tiny$Ff$} (b);
        \draw (b) -- node [above] {\tiny$Fg$} (c);
        \draw (a) .. controls +(.25,.5) and +(-1.25,0) .. (0,1.25) node [above] {\tiny$\ubcom{(Ff)}{(Fg)}$} .. controls +(1.25,0) and +(-.25,.5) .. (c);
        \draw (a) .. controls +(.25,-.5) and +(-1.25,0) .. (0,-1.25) node [below] {\tiny$F\bcom{f}{g}$} .. controls +(1.25,0) and +(-.25,-.5) .. (c);
      \end{tikzpicture}
      \qquad\qquad\qquad
      \begin{tikzpicture}[scale=.325]
        \node at (-2.25, 0) {$A$};
        \node at (0, 0) {$B$};
        \node at (2.25, 0) {$C$};
        \node [ov] (y) at (0, -4) {};
        \node [iv] (b) at (0, -1.5) {$\choseniso$};
        \draw (y) -- node [ed] {$\ubcom{f}{g}$} (b);
        \node (a) at (0, 1.5) {};
        \draw (a) [bend right=60] to node [ed] {$f$} (b);
        \draw (a) [bend left=60] to node [ed] {$g$} (b);
        \draw [rect, fun, pattern=north east lines](-4,-4) rectangle (4,4);
        \node [ov] (v) at (0, 4) {};
        \node [ivf] (a) at (0, 1.5) {$\choseniso$};
        \draw [edf] (v) -- node [ed] {$\ubcom{(Ff)}{(Fg)}$}(a);
      \end{tikzpicture}
    \]
  \end{itemize}
\end{prop}
\begin{proof}
  Naturality and coherence axioms follow from composition isomorphisms
  cancelling with their inverses.
\end{proof}


Any represented {\twoword} is canonically identified with the
underlying {\twoword} of its underlying bicategory: by composing with
chosen isomorphisms, the 2-cells with arbitrary boundary are in
composition-respecting correspondence with bracketed bigons. Likewise,
any {\twoword} functor is recovered from its underlying pseudofunctor:
the underlying {\twoword} functor is defined in the same way on bigons
and composition isomorphisms, and therefore on all 2-cells. Hence, we
obtain:

\begin{prop}\label{prop:main-two}
  The category of bicategories (and pseudofunctors) is equivalent to the
  category of representable {\twowords} (and {\twoword} functors).\qed
\end{prop}

Moreover, by construction, a pseudofunctor having identities as the
coherence isomorphisms corresponds to {\aan} {\twoword} functor preserving
chosen composition isomorphisms on the nose, so we also obtain:

\begin{cor}\label{cor:stricttwo}
  The category of bicategories and strict functors is equivalent to
  the category of represented {\twowords} and strict functors (functors that
  preserve the chosen composition isomorphisms).\qed
\end{cor}

\begin{rmk}\label{rmk:comonadic}
 

  

  Other characterizations of {\twowords} as structure on 2-categories
  are as follows: they are the flexible algebras of the strict
  2-category 2-monad on $\bCat$-enriched graphs (this can be deduced
  from~\cite[Theorem 4.8]{lack:model-2}); they are also the ``pie''
  algebras of this 2-monad in the terminology of~\cite{bourke-garner};
  and they are the cofibrant objects in the canonical model structure
  on 2-categories from~\cite{lack:model-2,lack:model-bi}. Moreover the
  evident (path 2-category) functor $\tword \to \tcat$ is comonadic,
  as shown in~\cite[Proposition 2.5]{amar}.  In particular,
  pseudofunctors are \emph{weak maps} of 2-categories in the sense
  of~\cite{garner:hom-hcat,bg:awfs-ii}.

  We also note that results analogous to those in this section appear
  in~\cite[Section 5]{hadzihasanovic} about a structure similar to an
  implicit 2-category, except not allowing 2-cells with nullary inputs
  or outputs or parallel composition, and with a different treatment
  of nullary composites. Results similar to our
  \cref{sec:twowords-hom} (in which we discuss transformations and
  modifications) are covered there as well in the same context.
\end{rmk} 


\section{Doubly weak double categories}
\label{sec:shortcut}

Now we quickly define doubly weak double categories, using strict double categories by analogy to \cref{sec:twowords}.
(Later in \cref{sec:double-computads} and \cref{sec:doublewords} we will use a more systematic approach, building the essentially algebraic implicit structures from the ground up.)

\begin{defn}\label{defn:impl-dbl}
  {\Aan} \textbf{\doubleword} is a strict double category
  whose horizontal and vertical categories of 1-cells are free (i.e.\
  each is freely generated by a directed graph).

  We call the generating 1-cells simply \textbf{1-cells}, and we do
  not use this word for their compositions, which we rather call
  \textbf{paths of 1-cells}. (In particular, a length zero path of
  1-cells consists of an object.)  The arrows and strings shown
  in our pasting diagrams and string diagrams always refer to
  1-cells.

  We call a 2-cell whose horizontal and vertical sources and targets are
  all length 1 paths a \textbf{square}.
  If its horizontal sources and targets are length 1 and its vertical ones are length 0, we call it a \textbf{horizontal bigon}; dually we have \textbf{vertical bigons}.
\end{defn}

%


A \textbf{functor} of {\doublewords} is a strict double functor
\emph{that moreover sends 1-cells to 1-cells}. We write \bDblWord for
the category of {\doublewords} and such functors.

When a path of 1-cells (horizontal or vertical) is isomorphic to a
single 1-cell, we call the latter a \textbf{composite} of the path.

\begin{defn}
  A \textbf{doubly weak double category} is {\aan} {\doubleword} in which
  each path of 1-cells (horizontal or vertical) has a
  composite.

  We also use the adjective \textbf{representable} to describe such
  {\doublewords}. We write \bWDblCat for this full subcategory of
  \bDblWord.
\end{defn}



We will often additionally assume our doubly weak double categories
are equipped with specific choices of composites, just as it is
customary to assume bicategories are equipped with specific choices of
composites:

\begin{defn}
  {\Aan} {\doubleword} is \textbf{represented} when it is equipped with a
  \emph{chosen} isomorphism between each horizontal or vertical length 2
  or 0 path of 1-cells and a single composite 1-cell.
  It follows that every path of 1-cells has a composite (i.e.\
  represented implies representable). We will refer to this too as
  simply a \textbf{doubly weak double category} where it is clear from
  context that we intend to have chosen composites.

  We denote the chosen composite of 1-cells $f \colon A \to B$ and
  $g \colon B \to C$ by $\udcom{f}{g} \colon A \to C$ (diagrammatic
  order). We denote the chosen nullary composite at the 0-cell $A$ by
  $\di{A} \colon A \to A$ (and it will be clear from context whether
  we mean the horizontal or vertical one).
  
  A functor between doubly weak double categories is
  \textbf{horizontally strict} if it preserves chosen horizontal
  composition isomorphisms. Similarly, it is \textbf{vertically
    strict} if it preserves chosen vertical composition isomorphisms,
  and it is simply \textbf{strict} if it preserves both.  We denote by
  \bWDblCatst the category of doubly weak double categories and strict
  functors.
\end{defn}

\begin{rmk}
  One could give an alternative definition that supposes a chosen
  composition isomorphism for every path of 1-cells, instead
  of just binary and nullary paths. This would provide an
  \emph{unbiased} definition of doubly weak double category, analogous
  to unbiased definitions of monoidal category or bicategory.
\end{rmk}

\begin{rmk}\label{rmk:dbl-bi}
  Just as every strict double category has underlying horizontal and
  vertical strict 2-categories (comprising the 2-cells that are
  respectively vertically and horizontally degenerate), every
  {\doubleword} has underlying horizontal and vertical
  {\twowords}. Note that if {\aan} {\doubleword} is representable,
  then so are its underlying {\twowords}. Hence every doubly weak
  double category has underlying horizontal and vertical bicategories.
\end{rmk}

\begin{eg}\label{eg:bi-dbl}
  In the other direction, just as every strict 2-category has an
  associated double category of squares --- a.k.a.\ the quintet
  construction --- every {\twoword} has an associated {\doubleword}.
  (Indeed, if a strict 2-category has a free underlying 1-category,
  then its strict double category of quintets also has free underlying
  1-categories.)  If {\aan} {\twoword} is representable, then so is
  its associated {\doubleword}. Hence every bicategory has an
  associated doubly weak double category of squares/quintets.
\end{eg}

\begin{eg}\label{eg:fundamental}
  Let $X$ be a topological space. There is an associated doubly weak
  double category, the \emph{fundamental (doubly weak) double
    groupoid} of $X$. The 0-cells are points, 1-cells are continuous
  paths\footnote{In this example we use ``path'' with the topological
    meaning, rather than the categorical one of \cref{defn:impl-dbl}.}
  $p:[0,1] \to X$, and the 2-cells with a given boundary loop
  correspond to relative homotopy classes of disks with that boundary.
  More precisely, given the boundary of a 2-cell, we compose each of the
  four sequences of paths to get a single path defined on $[0,1]$, and
  then a 2-cell with that boundary is a homotopy class of continuous
  maps $[0,1]\times [0,1] \to X$ relative to those four paths as the
  boundary.  (Composing an empty sequence of paths yields a
  constant path.)  Composing 2-cells is done as usual, plus we have to
  compose with reparametrizing homotopies to make the boundaries
  correct. Later we will construct this example in a more finitary
  way, in terms of composition of square 2-cells only, in
  \cref{eg:fundamental2}.

  Note that a doubly weak double category is more directly fitted to
  describing this structure than a strict double category (as
  in~\cite{brown:homotopy}), since composition of paths in a
  topological space is not strictly associative. Note also that
  although this example can be seen as a special case of squares in a
  bicategory, describing the composition of squares in a topological
  space is arguably simpler than describing the composition of 2-cells
  of globular shape (bigons), as discussed in~\cite{brown:homotopy}.
\end{eg}

\begin{eg}
  Given any strict double category, for each symmetry of the square we
  obtain a related strict double category. In particular, we obtain
  the \emph{horizontal opposite} by interchanging the horizontal
  sources and targets of cells, the \emph{vertical opposite} by
  interchanging the vertical sources and targets of cells, and the
  \emph{transpose} by interchanging horizontal and vertical cells.
  Likewise, {\doublewords} and doubly weak double categories are
  closed under these constructions. This makes the theory symmetric,
  so that any concept defined for horizontal arrows also makes sense
  for vertical arrows and vice versa.

  In contrast, the traditional notion of (singly) weak double
  category, a.k.a.\ pseudo double category~\cite{gp:double-limits}, is
  asymmetric: it has strict composition in one direction but weak
  composition in the other. Hence traditionally, a weak double
  category has no transpose. However, as we will see soon in
  \cref{prop:main}, a pseudo double category is a special case of a
  doubly weak double category, so its transpose exists in the form of
  another doubly weak double category.
\end{eg}

\begin{eg}
  Suppose $F\colon\cat{C}\to \cat{D}$ is a functor of implicit 2-categories
  that is bijective on objects.  Then there is an implicit double
  category whose horizontal 1-cells are those of $\cat{D}$, whose
  vertical 1-cells are those of $\cat{C}$, and whose 2-cells are those
  of $\cat{D}$ with $F$ applied to their vertical boundaries.  Indeed,
  this construction can be performed on strict 2-categories and strict
  double categories, and preserves freeness of 1-cells.  And if $\cat{C}$ and
  $\cat{D}$ are representable, so is the resulting implicit double category.

  In particular, a \textbf{(proarrow)
    equipment}~\cite{wood:pro1,wood:pro2} is a bijective on objects
  and locally full and faithful pseudofunctor of bicategories
  $\bicat{C} \to \bicat{D}$ such that every 1-cell in the image is a
  left adjoint.  This serves as an abstraction of e.g.\
  \begin{itemize}
  \item sets, functions, and relations;
  \item rings, homomorphisms, and bimodules; and
  \item categories, functors, and profunctors.
  \end{itemize}
  Thus, any proarrow equipment gives rise to a doubly weak double
  category.  Analogous results were shown in~\cite{verity:base-change}
  using double bicategories, and in~\cite{shulman:frbi} using pseudo
  double categories which requires $\bicat{C}$ to be a strict
  2-category.  As in the latter case, the doubly weak double
  categories arising from equipments can be characterized as those
  where each vertical 1-cell has a horizontal \emph{companion} and
  \emph{conjoint}.
\end{eg}

\begin{eg}
  For any strict 2-category $\bicat{C}$, there are two double
  categories $\bAdj(\bicat{C})$ and $\bAdj'(\bicat{C})$ both of whose objects
  and horizontal 1-cells are those of $\bicat{C}$ and both of whose vertical
  1-cells are adjunctions $f^* \dashv f_*$ in $\bicat{C}$ pointing in
  the direction of the left adjoint.  The 2-cells in the two cases are
  as shown below, one involving the left adjoints and the other the
  right adjoints:
  \[
    \begin{tikzcd}[arrow style=tikz]
      A \ar[r,"f"] \ar[d,"g^*"'] \ar[dr,phantom,"\Swarrow"] & B \ar[d,"h^*"]\\
      C \ar[r,"k"'] & D
    \end{tikzcd}
    \hspace{2cm}
    \begin{tikzcd}[arrow style=tikz]
      A \ar[r,"f"] \ar[dr,phantom,"\Searrow"] & B \\
      C \ar[r,"k"']  \ar[u,"g_*"] & D\ar[u,"h_*"']
    \end{tikzcd}
  \]
  The \emph{mates correspondence}~\cite{ks:r2cats} then yields an
  isomorphism $\bAdj(\bicat{C})\cong \bAdj'(\bicat{C})$ that is the
  identity on 0-cells and 1-cells.

  If instead $\cat{C}$ is an implicit 2-category, we have implicit
  double categories $\bAdj(\cat{C})$ and $\bAdj'(\cat{C})$ and an
  isomorphism between them defined in the same way, using the fact
  that adjunctions in a 2-category compose.  And, if $\cat{C}$ is
  representable, so are $\bAdj(\cat{C})$ and $\bAdj'(\cat{C})$.  Thus,
  we obtain a formalization of the mates correspondence for
  bicategories using double weak double categories.
\end{eg}

\begin{eg}
  In \cref{sec:twowords-hom}, we will show that for any two
  bicategories $\bicat{C}$ and $\bicat{D}$, there is a doubly weak
  double category $\Homl(\bicat{C}, \bicat{D})$ in which the objects
  are functors from $\bicat{C}$ to $\bicat{D}$, the horizontal and
  vertical 1-cells are lax and colax transformations, and the 2-cells
  are an appropriate kind of modification. More generally, for any two
  {\twowords} $\cat{C}$ and $\cat{D}$, there is {\aan} {\doubleword}
  $\Homl(\cat{C}, \cat{D})$, which is representable if $\cat{D}$ is.

  Special cases of this general construction produce more
  examples. Taking $\cat{C}$ to be freely generated by a 1-cell, we
  obtain a doubly weak double category where the 1-cells are lax and
  colax squares in the bicategory $\bicat{D}$. Taking $\cat{C}$ to be
  freely generated by a monad, we obtain a doubly weak double category
  of monads in $\bicat{D}$, where the 1-cells are lax and colax monad
  maps.
\end{eg}





\begin{defn}
  A doubly weak double category is \textbf{horizontally strict} if its
  underlying horizontal bicategory is a strict 2-category.

  Equivalently, for all horizontal $f \colon A \to B$,
  $g \colon B \to C$, and $h \colon A \to B$, we have
  $\udcom{\dcom{f}{g}}{h} = \udcom{f}{\dcom{g}{h}}$ and
  $\udcom{\hi{A}}{f} = f = \udcom{f}{\hi{B}}$, and likewise
  \[
    \begin{tikzpicture}[xscale=.15,yscale=.225]
      \draw [rect] (-6.5,-6) rectangle (6.5,6);
      \node [ov] (u) at (-4, 6) {};
      \node [ov] (v) at (0, 6) {};
      \node [ov] (w) at (4, 6) {};
      \node [ivs, inner sep=0] (a) at (-2, 2) {$\choseniso$};
      \node [ivs, inner sep=0] (b) at (0, -2) {$\choseniso$};
      \node [ov] (x) at (0, -6) {};
      \draw (u) [bend right=20] to node [ed] {$f$} (a);
      \draw (v) [bend left=20] to node [ed] {$g$} (a);
      \draw (w) [bend left=20] to node [ed] {$h$} (b);
      \draw (a) [bend right=20] to node [ed] {$\udcom{f}{g}$} (b);
      \draw (b) -- node [ed] {$\udcom{\dcom{f}{g}}{h}$} (x);
    \end{tikzpicture}
    \; = \;
    \begin{tikzpicture}[xscale=.15,yscale=.225]
      \draw [rect] (-6.5,-6) rectangle (6.5,6);
      \node [ov] (u) at (-4, 6) {};
      \node [ov] (v) at (0, 6) {};
      \node [ov] (w) at (4, 6) {};
      \node [ivs, inner sep=0] (a) at (2, 2) {$\choseniso$};
      \node [ivs, inner sep=0] (b) at (0, -2) {$\choseniso$};
      \node [ov] (x) at (0, -6) {};
      \draw (u) [bend right=20] to node [ed] {$f$} (b);
      \draw (v) [bend right=20] to node [ed] {$g$} (a);
      \draw (w) [bend left=20] to node [ed] {$h$} (a);
      \draw (a) [bend left=20] to node [ed] {$\udcom{g}{h}$} (b);
      \draw (b) -- node [ed] {$\udcom{f}{\dcom{g}{h}}$} (x);
    \end{tikzpicture}
    \qquad \text{and} \qquad
    \begin{tikzpicture}[xscale=.15,yscale=-.225]
      \draw [rect] (-5,-6) rectangle (5,6);
      \node [ov] (v) at (2.5, -6) {};
      \node [ivs, inner sep=0] (a) at (-2, -2) {$\choseniso$};
      \node [ivs, inner sep=0] (c) at (0, 2) {$\choseniso$};
      \node [ov] (y) at (0, 6) {};
      \draw (v) [bend right=15] to node [ed] {$f$} (c);
      \draw (a) [bend left=10] to node [ed] {$\hi{A}$} (c);
      \draw (c) -- node [ed] {$\ubcom{\hi{A}}{f}$} (y);
    \end{tikzpicture}
    \; = \;
    \begin{tikzpicture}[xscale=.15,yscale=-.225]
      \draw [rect] (-5,-6) rectangle (5,6);
      \node [ov] (v) at (0, -6) {};
      \node [ov] (y) at (0, 6) {};
      \draw (v) -- node [ed] {$f$} (y);
    \end{tikzpicture}
    \; = \;
    \begin{tikzpicture}[xscale=.15,yscale=-.225]
      \draw [rect] (-5,-6) rectangle (5,6);
      \node [ov] (v) at (-2.5, -6) {};
      \node [ivs, inner sep=0] (a) at (2, -2) {$\choseniso$};
      \node [ivs, inner sep=0] (c) at (0, 2) {$\choseniso$};
      \node [ov] (y) at (0, 6) {};
      \draw (v) [bend left=15] to node [ed] {$f$} (c);
      \draw (a) [bend right=10] to node [ed] {$\hi{B}$} (c);
      \draw (c) -- node [ed] {$\ubcom{f}{\hi{B}}$} (y);
    \end{tikzpicture}\,.
  \]
  
  Similarly, it is \textbf{vertically strict} if its underlying
  vertical bicategory is strict, and it is \textbf{strict} if it is
  both horizontally and vertically strict.
\end{defn}

\begin{prop}\label{prop:main}
  The category of vertically strict doubly weak double categories and
  vertically strict functors (resp.\ strict functors) is equivalent to
  the category of pseudo double categories and pseudofunctors (resp.\
  strict functors).
\end{prop}
\begin{proof}
  The proof follows the same blueprint as \cref{prop:main-two},
  which we walk through again in this case.

  Every pseudo double category $\bicat{C}$ has an underlying
  vertically strict doubly weak double category with the same 0-cells
  and 1-cells, and where a 2-cell with any boundary is a family
  consisting of a choice of square in $\bicat{C}$ for every possible
  bracketing of the source and target in the weak (horizontal)
  direction, such that these squares are related by composing with the
  relevant coherence isomorphisms (a.k.a.\ a \emph{clique
    morphism}). Composition is as in $\bicat{C}$, and composition
  isomorphisms are given by identities, as in \cref{prop:bitotwo}.

  Likewise every pseudo double functor $\bicat{F}$ has an underlying
  vertically strict functor of {\doublewords}, defined as $\bicat{F}$
  on 0-cells and 1-cells, and with the map on 2-cells induced by
  composing with pseudofunctor coherence isomorphisms, as in
  \cref{prop:pstofun}. (Note that coherence for pseudofunctors of
  bicategories 
  applies just as
  well here, since a pseudo double functor in particular includes
  pseudofunctors between underlying bicategories.)

  Conversely, every vertically strict doubly weak double category
  $\cat{C}$ has an underlying pseudo double category with the same
  0-cells, 1-cells, and \emph{square} 2-cells (those bordered by
  length one paths), and with identities and compositions derived from
  those in $\cat{C}$:
  \[
    \begin{tikzpicture}[->,xscale=.75,yscale=.65]
      \node (au) at (-2,1) {$\cdot$};
      \node (bu) at (0,1) {$\cdot$};
      \node (cu) at (2,1) {$\cdot$};
      \node (ad) at (-2,-1) {$\cdot$};
      \node (bd) at (0,-1) {$\cdot$};
      \node (cd) at (2,-1) {$\cdot$};
      \node at (-1,0) {\large$\alpha_1$};
      \node at (1,0) {\large$\alpha_2$};
      \node at (0,1.5) {$\choseniso$};
      \node at (0,-1.5) {$\choseniso$};
      \draw (au) -- node [above] {\footnotesize$s_1$} (bu);
      \draw (bu) -- node [above] {\footnotesize$s_2$} (cu);
      \draw (ad) -- node [below] {\footnotesize$t_1$} (bd);
      \draw (bd) -- node [below] {\footnotesize$t_2$} (cd);
      \draw (au) -- node [left] {\footnotesize$f$} (ad);
      \draw (bu) -- node [left] {\footnotesize$g$} (bd);
      \draw (cu) -- node [right] {\footnotesize$h$} (cd);
      \draw (au) .. controls +(.5,.75) and +(-.75,0) .. (0,2) node [above] {\scriptsize$\ubcom{s_1}{s_2}$} .. controls +(.75,0) and +(-.5,.75) .. (cu);
      \draw (ad) .. controls +(.5,-.75) and +(-.75,0) .. (0,-2) node [below] {\scriptsize$\ubcom{t_1}{t_2}$} .. controls +(.75,0) and +(-.5,-.75) .. (cd);
    \end{tikzpicture}
    \qquad
    \qquad
    \qquad
    \begin{tikzpicture}[scale=.425]
      \draw [rect] (-4,-4) rectangle (4,4);
      \node [ov] (top) at (0, 4) {};
      \node [ov] (bottom) at (0, -4) {};
      \node [ov] (left) at (-4, 0) {};
      \node [ov] (right) at (4, 0) {};
      \node [ivs] (a) at (0, 2.25) {$\choseniso$};
      \node [ivs] (b) at (0, -2.25) {$\choseniso$};
      \node [ivs] (2cella) at (-1.5, 0) {$\alpha_1$};
      \node [ivs] (2cellb) at (1.5, 0) {$\alpha_2$};
      \draw (top) -- node [ed] {$\ubcom{s_1}{s_2}$} (a);
      \draw (bottom) -- node [ed] {$\ubcom{t_1}{t_2}$} (b);
      \draw (left) -- node [ed] {$f$} (2cella);
      \draw (2cella) -- node [ed] {$g$} (2cellb);
      \draw (2cellb) -- node [ed] {$h$} (right);
      \draw (a) [out=-110, in=90] to node [ed] {$s_1$} (2cella);
      \draw (2cella) [out=-90, in=110] to node [ed] {$t_1$} (b);
      \draw (a) [out=-70, in=90] to node [ed] {$s_2$} (2cellb);
      \draw (2cellb) [out=-90, in=70] to node [ed] {$t_2$} (b);
    \end{tikzpicture}
  \]
  \[
    \qquad\;\;\;
    \begin{tikzpicture}[->,xscale=.3,yscale=.45,shorten <=-1.5pt,shorten >=-1.5pt]
      \node (a) at (0,1) {$\cdot$};
      \node (b) at (0,-1) {$\cdot$};
      \node at (0,1.7) {$\choseniso$};
      \node at (0,-1.7) {$\choseniso$};
      \draw (a) -- node [left] {\footnotesize$f$} (b);
      \draw (a) .. controls +(-1.25,.65) and +(-1,0) .. (0,2.25) node [above] {\footnotesize$\justi$} .. controls +(1,0) and +(1.25,.65) .. (a);
      \draw (b) .. controls +(-1.25,-.65) and +(-1,0) .. (0,-2.25) node [below] {\footnotesize$\justi$} .. controls +(1,0) and +(1.25,-.65) .. (b);
    \end{tikzpicture}
    \quad\;\;\,
    \qquad
    \qquad
    \qquad
    \qquad
    \qquad
    \begin{tikzpicture}[scale=.3]
      \draw [rect] (-4,-4) rectangle (4,4);
      \node [ov] (top) at (0, 4) {};
      \node [ov] (bottom) at (0, -4) {};
      \node [ov] (left) at (-4, 0) {};
      \node [ov] (right) at (4, 0) {};
      \node [ivs] (up) at (0, 2) {$\choseniso$};
      \node [ivs] (down) at (0, -2) {$\choseniso$};
      \draw (top) -- node [ed] {$\justi$} (up);
      \draw (down) -- node [ed] {$\justi$} (bottom);
      \draw (left) -- node [ed] {$f$} (right);
    \end{tikzpicture}
  \]
  (There are analogous diagrams for vertical identities and
  compositions.) The coherence data are built from the chosen
  composition isomorphisms just as in \cref{prop:twotobi}.

  Likewise every vertically strict functor $F$ between vertically
  strict doubly weak double categories has an underlying pseudo double
  functor (see~\cite{gp:double-limits} for a precise definition of
  pseudo double functor), defined as $F$ on all cells, and with
  coherence data built from the chosen composition isomorphisms, just
  as in \cref{prop:funtops}.

  That these assignments constitute an equivalence of categories, as
  in \cref{prop:main-two}, is a series of straightforward
  verifications. Moreover, strict functors of doubly weak double
  categories correspond to strict functors of pseudo double categories
  because preservation of chosen composition isomorphisms amounts to
  triviality of coherence isomorphisms, as in \cref{cor:stricttwo}.
\end{proof}

\begin{cor}\label{cor:strict-double}
  The category of strict doubly weak double categories and strict
  functors is equivalent to the category of strict double categories.\qed
\end{cor}

\begin{rmk}
  Keisuke Hoshino has shown that there is an analogue of
  \cref{rmk:comonadic} for double categories as well.  That is, the
  category of implicit double categories is comonadic over that of
  strict double categories, with the comonad being a cofibrant
  replacement; thus double pseudofunctors are the \emph{weak maps} of
  double categories in the sense of~\cite{garner:hom-hcat,bg:awfs-ii}.
\end{rmk}

\section{Double computads}
\label{sec:double-computads}

We next embark on a more algebraic treatment of implicit and doubly
weak double categories, starting with the definition of double
computads.
For comparison and later use, we first recall some details about computads
for 1-categories and 2-categories. By a \textbf{1-computad} we will
mean simply a directed (multi)graph, a.k.a.\ quiver. The category \ocptd of
1-computads is a functor category $[\lC_1,\bSet]$ with domain $\lC_1$
given by the category
\[
  \ton \toto \tz.
\]

The category \ocat of (small) 1-categories is monadic over
1-computads, via an adjunction which we write
\[
  \vcenter{\xymatrix@C=12mm{\ocptd \ar@{}[r]|{\;\;\bot}
      \ar@<2mm>[r]^-{\cF_{1}} & \ocat \ar@<2mm>[l]^-{\cU_{1}}}}
\]
with induced monad $T_1 = \cU_1 \cF_1$. When $X$ is a 1-computad, the
0-cells in $T_1 X$ are the same as in $X$, and the 1-cells in $T_1 X$
are paths in $X$.  We denote by \paral the 1-computad containing two
objects and two parallel arrows between them.


\begin{defn}
  A \textbf{2-computad} consists of a 1-computad $X_{\le 1}$, together
  with a set $X_2$ of \emph{2-cells} and a function $\del$ sending each
  2-cell to a parallel pair of paths in $X_{\le 1}$ (its boundary):
  \[\del\colon X_2 \too \ocptd(\paral, T_1 X_{\le 1}). \]

  We denote by \tcptd the category of 2-computads, defined as the comma
  category of \bSet over $\ocptd(\paral,T_1 -)$.
\end{defn}

The following theorem allows us to quickly deduce that \tcptd is
itself a presheaf category.\footnote{This fact was apparently first
  observed by Schanuel, as mentioned in~\cite{cj:clfrag}.}  Recall
that a functor $G\colon \bC\to \bD$ is a \textbf{parametric right
  adjoint} if \bC has a terminal object $1$ and the induced
$\Ghat\colon \bC \to \bD/G1$ has a left adjoint.
\begin{thm}[\cite{cj:clfrag}]\label{thm:prafunctors}
  Given a functor between presheaf categories
  $G\colon [\lC,\bSet] \to [\lD,\bSet]$, the comma category (a.k.a.\
  Artin gluing) $([\lD,\bSet] / G)$ is again a presheaf category
  $[\lE,\bSet]$ if and only if $G\colon [\lC,\bSet] \to [\lD,\bSet]$
  is a parametric right adjoint.
\end{thm}
For functors between well-behaved categories such as presheaf
categories $\bC = [\lC, \bSet]$ and $\bD = [\lD, \bSet]$, parametric
right adjoints are equivalently the functors preserving connected
limits. When moreover $\bD = \bSet$, parametric right adjoints are
simply coproducts of representable functors.

Indeed, $T_1$ and $\ocptd(\paral, -)$ are both parametric right
adjoints, thus so is their composite; hence by \cref{thm:prafunctors}
there is a category $\lC_2$ such that $\tcptd \cong [\lC_2,
\bSet]$. Moreover the proof of this theorem in~\cite{cj:clfrag} tells
us how to explicitly describe the domain category, giving us the
definition of $\lC_2$ written below and shown graphically in \cref{fig:C2}. (It is also not difficult to verify
directly from the definition that functors $\lC_2 \to \bSet$ are
identified with 2-computads.)

\begin{figure}
  \begin{tikzpicture}[scale=.75]
    \begin{scope}[shift={(-7.5,-1.5)},scale=.75]
      \node(x1) at (0,0) {$\cdot$};
    \end{scope}
    \begin{scope}[shift={(-5.5,-1.5)},scale=.75]
      \node[inner sep=0] (x1) at (-.5,0) {$\cdot$};
      \node[inner sep=0](x2) at (.5,0) {$\cdot$};
      \draw[->,shorten <=.5pt] (x1) -- (x2);
    \end{scope}
    \begin{scope}[shift={(-.2,-.2)},xscale=-.9,rotate=90]
      \node (x1) at (0,0) {$\cdot$};
      \draw[-implies,double equal sign distance,shorten <=-1.5pt,shorten >=-3pt](.15,.135)arc(-45:210:.2);
    \end{scope}
    \begin{scope}[shift={(-1,-1)},xscale=-.75, yscale=.75]
      \node (x1) at (0,-.5) {$\cdot$};
      \draw[->,shorten <=2.5pt](x1.center)arc(-90:255:.5);
      \node at (0,0) {$\Downarrow$};
    \end{scope}
    \begin{scope}[shift={(1,-1)},xscale=-.75, yscale=-.75]
      \node (x1) at (0,-.5) {$\cdot$};
      \draw[->,shorten <=2.5pt](x1.center)arc(-90:255:.5);
      \node at (0,0) {$\Downarrow$};
    \end{scope}
    \begin{scope}[shift={(0,-2)},scale=.75]
      \node[inner sep=0] (x1) at (-.75,0) {$\cdot$};
      \node[inner sep=0] (x2) at (.75,0) {$\cdot$};
      \draw[->,shorten <=.5pt] (x1) to[bend left=50] (x2);
      \draw[->,shorten <=.5pt] (x1) to[bend right=50] (x2);
      \node at (0,0) {$\Downarrow$};
    \end{scope}
    \begin{scope}[shift={(-2,-2)},scale=.75]
      \node[inner sep=0] (x1) at (0,-.5) {$\cdot$};
      \node[inner sep=0] (x2) at (0,.5) {$\cdot$};
      \draw[->,shorten <=.5pt](x1) to[out=135, in=180,looseness=2] (x2);
      \draw[->,shorten <=.5pt](x2) to[out=0, in=45,looseness=2] (x1);
      \node at (0,.05) {$\Downarrow$};
    \end{scope}
    \begin{scope}[shift={(2,-2)},scale=.75]
      \node[inner sep=0] (x1) at (0,.5) {$\cdot$};
      \node[inner sep=0] (x2) at (0,-.5) {$\cdot$};
      \draw[->,shorten <=.5pt](x1) to[out=-135, in=180,looseness=2] (x2);
      \draw[->,shorten <=.5pt](x2) to[out=0, in=-45,looseness=2] (x1);
      \node at (0,-.05) {$\Downarrow$};
    \end{scope}
    \begin{scope}[shift={(-3,-3)},scale=.75]
      \node[inner sep=0] (x1) at (0,-.5) {$\cdot$};
      \node[inner sep=0] (x2) at (-.4,.5) {$\cdot$};
      \node[inner sep=0] (x3) at (.4,.5) {$\cdot$};
      \draw[->,shorten <=.5pt,shorten >=-.5pt](x1) to[out=155, in=-140] (x2);
      \draw[->,shorten <=.5pt](x2) to[bend left=10] (x3);
      \draw[->](x3) to[out=-40, in=35] (x1);
      \node at (0,.075) {$\Downarrow$};
    \end{scope}
    \begin{scope}[shift={(3,-3)},scale=.75]
      \node[inner sep=0] (x1) at (0,.5) {$\cdot$};
      \node[inner sep=0] (x2) at (-.4,-.5) {$\cdot$};
      \node[inner sep=0] (x3) at (.4,-.5) {$\cdot$};
      \draw[->,shorten <=.5pt,shorten >=-.5pt](x1) to[out=-155, in=140] (x2);
      \draw[->,shorten <=.5pt](x2) to[bend right=10] (x3);
      \draw[->](x3) to[out=40, in=-35] (x1);
      \node at (0,-.075) {$\Downarrow$};
    \end{scope}
    \begin{scope}[shift={(-1,-3)},scale=.75]
      \node[inner sep=0] (x1) at (-.75,0) {$\cdot$};
      \node[inner sep=0] (x2) at (0,.45) {$\cdot$};
      \node[inner sep=0] (x3) at (.75,0) {$\cdot$};
      \draw[->,shorten <=.5pt] (x1) to[bend left=20] (x2);
      \draw[->,shorten <=.5pt] (x2) to[bend left=20] (x3);
      \draw[->,shorten <=.5pt] (x1) to[bend right=50] (x3);
      \node at (0,0) {$\Downarrow$};
    \end{scope}
    \begin{scope}[shift={(1,-3)},scale=.75]
      \node[inner sep=0] (x1) at (-.75,0) {$\cdot$};
      \node[inner sep=0] (x2) at (0,-.45) {$\cdot$};
      \node[inner sep=0] (x3) at (.75,0) {$\cdot$};
      \draw[->,shorten <=.5pt] (x1) to[bend right=20] (x2);
      \draw[->,shorten <=.5pt] (x2) to[bend right=20] (x3);
      \draw[->,shorten <=.5pt] (x1) to[bend left=50] (x3);
      \node at (0,0) {$\Downarrow$};
    \end{scope}
    \node at (0,-4) {$\vdots$};
  \end{tikzpicture}
  \caption{$\lC_2$ consists of the ``shapes of cell'' in a 2-computad.}
  \label{fig:C2}
\end{figure}
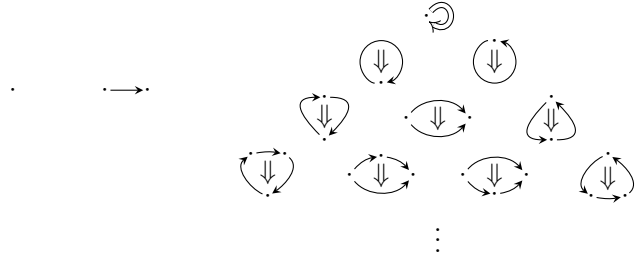

The category $\lC_2$ has objects $\tz$, $\ton$, and $\ttw^m_n$
for natural numbers $m,n \in \mathbb{N}$, and the morphisms are as
follows:
\begin{itemize}
\item The full subcategory of objects $\tz$ and $\ton$ is $\lC_1$.
\item The only arrows into the objects $\ttw^m_n$ are identities.
\item For each $m,n \in \mathbb{N}$, the homsets from $\ttw^m_n$ into
  $\tz$ and $\ton$, acted on by composing arrows in $\lC_1$, determine
  the 1-computad representing a pair of parallel paths of
  lengths $m$ and $n$: 
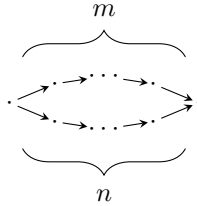
\[
    \begin{tikzpicture}[scale=.5]
      \begin{scope}[->,shorten <=-2pt,shorten >=-3pt]
        \node (x1) at (0,0) {$\cdot$};
        \node (x5) at (5,0) {$\cdot$};
        \begin{scope}
          \node (x2) at (1.2,.5)  {$\cdot$};
          \draw (x1) -- (x2);
          \node (x3) at (2.5,.7)  {$\bcdots$};
          \draw (x2) -- (x3);
          \node (x4) at  (3.8,.5) {$\cdot$};
          \draw (x3) -- (x4);
          \draw (x4) -- (x5);
        \end{scope}
        \begin{scope}[yscale=-1]
          \node (x2) at (1.2,.5)  {$\cdot$};
          \draw (x1) -- (x2);
          \node (x3) at (2.5,.7)  {$\bcdots$};
          \draw (x2) -- (x3);
          \node (x4) at  (3.8,.5) {$\cdot$};
          \draw (x3) -- (x4);
          \draw (x4) -- (x5);
        \end{scope}
      \end{scope}
      \draw [decorate,decoration={brace,amplitude=10pt,raise=4ex}] (x1) -- (x5) node[midway,yshift=3.5em]{$m$};
      \draw [decorate,decoration={brace,amplitude=10pt,mirror,raise=4ex}] (x1) -- (x5) node[midway,yshift=-3.5em]{$n$};
    \end{tikzpicture}
  \]
\end{itemize}



As in \cref{sec:twowords}, we refer to 2-cells of shape $\ttw^1_1$
as \textbf{bigons}:
\[
  \begin{tikzpicture}[->, scale=.65]
    \node (a) at (-1, 0) {};
    \node (b) at (1, 0) {};
    \node at (a) {$\cdot$};
    \node at (b) {$\cdot$};
    \node (z) at (0, 0) {$\Downarrow$};
    \draw (a) [bend left=45] to (b);
    \draw (a) [bend right=45] to (b);
  \end{tikzpicture}
\]
A 2-computad in which all 2-cells are bigons is called a
\textbf{2-graph} (a.k.a.\ 2-globular set). We denote this full
subcategory of \tcptd by \tgph, also a functor category with domain a
full subcategory of $\lC_2$:
\[
  \ttw \toto \ton \toto \tz.
\]
(composition laws as in $\lC_2$, where $\ttw \coloneqq \ttw^1_1$).

The category \tgph is also a comma category
$(\bSet / \ocptd(\paral, -))$, so we have a functor
from $\tcptd = (\bSet / \ocptd(\paral, T_1-))$ to $\tgph$ given by applying
$T_1$ to the 1-cells, which reinterprets all of the 2-cells in a 2-computad as bigons
between paths.
  \[
    \begin{tikzpicture}[->, scale=.2]
      \node (a) at (-6, 0) {};
      \node (b) at (6, 0) {};
      \node at (a) {$\cdot$};
      \node at (b) {$\cdot$};
      \node (z) at (0, 0) {$\!\!\phantom{{}_\alpha}\Downarrow_\alpha$};
      \node (s1) at (-1, 2.5) {};
      \node (s2) at (1, 2.5) {};
      \node (t1) at (-1, -2.5) {};
      \node (t2) at (1, -2.5) {};
      \node (sh) at (0,2.25) {$\bcdots$};
      \node (th) at (0,-2.25) {$\bcdots$};
      \draw (a) -- node [auto,anchor=south,pos=.4] {$s_1$} (s1);
      \draw (s2) -- node [auto,anchor=south,pos=.6] {$s_m$} (b);
      \draw (a) -- node [auto,anchor=north,pos=.4] {$t_1$} (t1);
      \draw (t2) -- node [auto,anchor=north,pos=.6] {$t_n$} (b);
    \end{tikzpicture}
    \qquad
    \mapsto
    \qquad
    \begin{tikzpicture}[->, scale=.2]
      \node (a) at (-6, 0) {};
      \node (b) at (6, 0) {};
      \node at (a) {$\cdot$};
      \node at (b) {$\cdot$};
      \node (z) at (0, 0) {$\!\!\phantom{{}_\alpha}\Downarrow_\alpha$};
      \draw (a) [bend left=35] to node [auto,anchor=south] {$s_1\cdots s_m$} (b);
      \draw (a) [bend right=35] to node [auto,anchor=north] {$t_1 \cdots t_n$} (b);
    \end{tikzpicture}
  \]
This is more precisely a functor $\tflat \colon \tcptd \to \ocattgph$
where the codomain is 2-graphs equipped with 1-category structure on
1-cells. Note that this category \ocattgph is evidently monadic over
\tgph.

The functor \tflat is pseudomonic; its image consists of 2-graphs
equipped with \emph{free} 1-category structure and maps sending
generating 1-cells to generating 1-cells. Thus 2-computads are equivalently such structured 2-graphs.

The category \tcat of (small, strict) 2-categories is also monadic
over \tgph, essentially by definition (as a 2-graph equipped with
various operations). The forgetful right adjoint evidently factors
through an intermediate right adjoint $\tcat \to \ocattgph$, which is
also monadic by the following lemma.
\begin{lem}[{\cite[Propositions 4 and 5]{bourn:ldgeochoice}}]\label{thm:mnd-cancel}
  If $G_3 = G_2\circ G_1$, where $G_2$ and $G_3$ are monadic and all
  three functors have left adjoints, then $G_1$ is also monadic.\qed
\end{lem}
In the next section we will see that \tcat is monadic over \tcptd as
well, but this is less straightforward.
(Street~\cite{street:catval-2lim} asserted this by a monadicity
theorem, but it seems nontrivial to verify the hypotheses.)

It is time to move on to double computads.  Here the roles of
1-computads and 1-categories are played by structures which we call
\ovo-computads and \ovo-categories; these are like double categories
but without any 2-cells.

\begin{defn}
  A \textbf{\ovo-computad} $X$ consists of two 1-computads (directed
  graphs) with the same set of 0-cells (vertices) $X_0$. We refer to the two kinds of 1-cell
  as \emph{horizontal} and \emph{vertical} and draw
  them accordingly. The category \ovocptd of \ovo-computads is a
  functor category $[\lC_{\ovo},\bSet]$, with domain $\lC_{\ovo}$
  given by the category
  \[
    \th\! \toto \tz \leftleftarrows \tv.
  \]
\end{defn}

\begin{rmk}\label{thm:ovocptd-slice}
  This category $\lC_{\ovo}$ is the category of elements of the
  1-computad $A\colon \lC_1\to \bSet$ defined by $A(0) = \{ 0 \}$ and
  $A(1) = \{ \th,\tv \}$. Thus 
  we can also
  write $\ovocptd = \ocptd / A$. There are hence projection functors
  \[\oproj \colon \lC_{\ovo} \to \lC_1\qquad\text{and}\qquad\oproj_!
    \colon \ovocptd \to \ocptd\] which forget the distinction between
  horizontal and vertical arrows.
\end{rmk}

Similarly, a \textbf{\ovo-category} consists of two categories with
the same set of objects; \ovo-categories are monadic over
\ovo-computads via an adjunction
\[
  \vcenter{\xymatrix@C=12mm{\ovocptd \ar@{}[r]|{\;\;\bot}
      \ar@<2mm>[r]^-{\cF_{\ovo}} & \ovocat \ar@<2mm>[l]^-{\cU_{\ovo}}}}
\]
with induced monad $T_{\ovo}$.  Let \sq denote the \ovo-computad with
four objects 
and two arrows of each sort,
forming a square:
\[
  \begin{tikzpicture}[->, scale=.4]
    \node (a) at (-1, 1) {};
    \node (b) at (1, 1) {};
    \node (c) at (-1, -1) {};
    \node (d) at (1, -1) {};
    \node at (a) {$\cdot$};
    \node at (b) {$\cdot$};
    \node at (c) {$\cdot$};
    \node at (d) {$\cdot$};
    \draw (a) -- (b);
    \draw (b) -- (d);
    \draw (a) -- (c);
    \draw (c) -- (d);
  \end{tikzpicture}
\]


\begin{defn}\label{defn:dblcptd}
  A \textbf{double computad} consists of a \ovo-computad $X_{\leovo}$,
  together with a set $X_2$ of \emph{2-cells} and a function $\del$
  sending each 2-cell to a square of paths in $X_{\le 1}$ (its
  boundary):
  \[\del \colon X_2 \too \ovocptd\big(\sq,\,T_{\ovo} X_{\leovo}\big). \]
  We write \obocptd for the category of double computads, the comma
  category of \bSet over $\ovocptd(\sq,\,T_{\ovo} -)$.
\end{defn}


Like $T_1$, the monad $T_{\ovo}$ is a parametric right adjoint. Thus,
by \cref{thm:prafunctors}, \obocptd is also a functor category
$[\lC_{\obo},\bSet]$. We describe $\lC_{\obo}$ by the same process
we used to describe $\lC_2$. We find
that the objects are $\tz$, $\th$, $\tv$, and $\td{a}{b}{c}{d}$ for
natural numbers $a,b,c,d \in \mathbb{N}$, and the morphisms are as
follows:
\begin{itemize}
\item The full subcategory of objects $\tz$, $\th$, and $\tv$ is $\lC_{\ovo}$.
\item The only arrows into the objects $\td{a}{b}{c}{d}$ are identities.
\item For $a,b,c,d \in \mathbb{N}$, the homsets from
  $\td{a}{b}{c}{d}$ into $\tz$, $\th$, and $\tv$, acted on by
  composing arrows in $\lC_{\obo}$, determine the $\ovo$-computad
  representing a square of paths of lengths $a$ (top), $b$
  (right), $c$ (left), and $d$ (bottom):
  \[
    \begin{tikzpicture}[scale=.28]
      \node (a) at (-2.5, 2.5) {$\cdot$};
      \node (b) at (2.5, 2.5) {$\cdot$};
      \node (c) at (-2.5, -2.5) {$\cdot$};
      \node (d) at (2.5, -2.5) {$\cdot$};
      \node (uu) at (0,2.5) {$\bcdots$};
      \node (rr) at (2.5,0) {$\vcdots$};
      \node (ll) at (-2.5,0) {$\vcdots$};
      \node (dd) at (0,-2.5) {$\bcdots$};
      \begin{scope}[->,shorten <=-2pt,shorten >=-3pt]
        \draw (a) -- (uu);
        \draw (uu) -- (b);
        \draw (b) -- (rr);
        \draw (rr) -- (d);
        \draw (a) -- (ll);
        \draw (ll) -- (c);
        \draw (c) -- (dd);
        \draw (dd) -- (d);
      \end{scope}
      \draw [decorate,decoration={brace,amplitude=5,raise=.625em}] (a) -- (b) node[midway,yshift=1.75em]{$a$};
      \draw [decorate,decoration={brace,amplitude=5,raise=.625em}] (b) -- (d) node[midway,xshift=1.75em]{$b$};
      \draw [decorate,decoration={brace,amplitude=5,mirror,raise=.625em}] (a) -- (c) node[midway,xshift=-1.75em]{$c$};
      \draw [decorate,decoration={brace,amplitude=5,mirror,raise=.625em}] (c) -- (d) node[midway,yshift=-1.75em]{$d$};
    \end{tikzpicture}
  \]
\end{itemize}



\begin{rmk}
  We also have that $\lC_\d$ is the category of elements of a certain
  2-computad $B \colon \lC_2\to\bSet$, which we can see in the
  following way.

  Composing $\oproj_! \colon \ovocptd \to \ocptd$ from
  \cref{thm:ovocptd-slice} with
  $\ocptd(\paral, T_1 -) : \ocptd \to \bSet$ yields a functor
  $\ovocptd \to \bSet$, which sends a \ovo-computad to the set of
  pairs of parallel paths of 1-cells of either sort.  We also have the
  functor $\ovocptd(\sq, T_{\ovo} -)$, which 
  sends a \ovo-computad to the set of parallel pairs of paths where
  the first consists of horizontal 1-cells followed by vertical
  1-cells and the second consists of vertical 1-cells followed by
  horizontal 1-cells.

  Forgetting this requirement on the pairs of paths yields a natural
  transformation
  $\alpha \colon \ovocptd(\sq, T_{\ovo} -) \hookrightarrow
  \ocptd(\paral, T_1 \oproj_! -)$.  This transformation is
  \emph{cartesian}, i.e.\ its naturality squares are pullbacks. In this
  case, cartesianness corresponds to the fact that whether an element
  of $\ocptd(\paral, T_1 \oproj_!X)$ lifts to
  $\ovocptd(\sq, T_{\ovo} X)$ is determined solely by its ``shape'',
  i.e.\ the induced element of $\ocptd(\paral, T_1 \oproj_!1)$ (a pair
  of sequences of the values \th and \tv).

  By the following lemma, we have $\obocptd= \tcptd/B$,
  where $B$ is the 2-computad in $\ovocptd= \ocptd/A$ corresponding to
  $\alpha_1 \colon \ovocptd(\sq, T_{\ovo} 1) \to \ocptd(\paral, T_1A)$.
\end{rmk}

\begin{lem}\label{lem:cart}
  If $\alpha$ is a cartesian natural transformation
  \[
    \begin{tikzcd}
      C / c \ar[rr, "F"] \ar[rd] &\ar[d, phantom, pos=.4, "\!\!\phantom{{}_\alpha}\Downarrow_\alpha"]& D \\
      & C \ar[ru, "G"']
    \end{tikzcd}
  \]
  then the comma category $(D / F)$ is a slice category of the comma
  category $(D / G)$.

  Namely, $(D / F) \cong (D / G) / \alpha_1$, the slice over the
  object $\alpha_1 \colon F(1) \to G(c)$.
\end{lem}
\begin{proof}
  Since $\alpha$ is cartesian, for any object $f \colon c' \to c$ of
  $C / c$ we have a pullback
  \[
    \begin{tikzcd}
      F(f) \ar[r, "\alpha_f"] \ar[d, "F(f)"'] \drpullback & G(c') \ar[d,
      "G(f)"] \\
      F(1) \ar[r, "\alpha_1"'] & G(c)
    \end{tikzcd}
  \]
  Now, an object of the comma category $(D / F)$ consists of an object
  $d$ of $D$, an object $f \colon c' \to c$ of $C / c$, and an arrow
  $d \to F(f)$. By the universal property of the above pullback, to
  give such a $d \to F(f)$ is to give a commutative square
  \[
    \begin{tikzcd}
      d \ar[r] \ar[d] & G(c') \ar[d, "G(f)"] \\
      F(1) \ar[r, "\alpha_1"'] & G(c)
    \end{tikzcd}
  \]
  And this is precisely an object of $(D / G) / \alpha_1$. The
  morphisms are also the same.
\end{proof}

Explicitly, in this case we have $\ovocptd= \ocptd/B$ where
$B\colon \lC_2\to\bSet$ is defined by:
\begin{align*}
  B(\tz) &= \setof{\tz}\\
  B(\ton) &= \setof{\th,\tv}\\
  B(\ttw^m_n) &= \setof{\td{a}{b}{c}{d}
                | a, b, c, d \in \mathbb{N}, \; a + b = m, \; c + d = n}\\
  B(s_{i})(\td{a}{b}{c}{d})
         &=
           \begin{cases}
             \th & \quad\text{if } i \le a\\
             \tv & \quad\text{if } i > a
           \end{cases}\\
  B(t_{j})(\td{a}{b}{c}{d})
         &=
           \begin{cases}
             \tv & \quad\text{if } j \le c\\
             \th & \quad\text{if } j > c
           \end{cases}
\end{align*}
(the action of all other arrows being trivial).
The category $\lC_{\obo}$ is the category of elements of this $B$.

\begin{rmk}
  We have a commutative diagram (moreover, a pullback square)
  \[
    \begin{tikzcd}[row sep=large,every label/.append style = {font = \small}]
      \lC_{\ovo} \ar[r, "\!\!\!\oproj"] \ar[d,hook] \drpullback & \lC_1 \ar[d,hook] \\
      \lC_{\obo}\ar[r, "\tproj"'] & \lC_2
    \end{tikzcd}
  \]
  where each horizontal functor is the projection of a category of
  elements onto its domain, and the vertical functors are the obvious
  inclusions (each of which, incidentally, may also be viewed as
  projection of a category of elements onto its domain). We thereby
  obtain a similar diagram of functor categories:
  \[
    \vcenter{\xymatrix@R=12mm@C=10mm{\ovocptd \ar@<-2mm>[d]_{\sk}
        \ar@{}[r]|{\;\;\bot} \ar@<2mm>[r]^-{\oproj_!} &
        \ocptd \ar@<-2mm>[d]_{\sk} \ar@<2mm>[l]^-{\oproj^*}\\
        \obocptd \ar@{}[u]|{\dashv}
        \ar@<-2mm>[u]_-{\tau} \ar@{}[r]|{\;\;\bot} \ar@<2mm>[r]^-{\tproj_!} &
        \tcptd \ar@{}[u]|{\dashv} \ar@<-2mm>[u]_-{\tau} \ar@<2mm>[l]^-{\tproj^*}}}
  \]
  Here $\tproj^*$, $\oproj^*$, and both functors denoted \tau are
  restrictions (\tau\ means ``truncation''); $\tproj_!$, $\oproj_!$,
  and both functors denoted \sk are left Kan extensions (\sk\ means
  ``skeleton'').  We have the obvious commutativities
  $\oproj^* \tau \cong \tau \tproj^*$ and
  $\sk \oproj_! \cong \tproj_! \sk$, and the Beck-Chevalley property
  also holds, giving isomorphisms $\oproj_! \tau \cong \tau \tproj_!$
  and $\sk \oproj^* \cong \tproj^* \sk$.
  
  Viewing the left Kan extensions as slice category projections
  \[\oproj_! \colon \ocptd / A \to \ocptd \qquad \text{and} \qquad
    \tproj_! \colon \tcptd / B \to \tcptd\] we have that the right
  adjoints $\oproj^*$ and $\tproj^*$ are respectively given by product
  with $A$ and $B$ (pulling back $\ocptd = \ocptd / 1$ along $A \to 1$
  and $\tcptd = \tcptd / 1$ along $B \to 1$).  Explicitly, $\tproj^*$
  sends a 2-computad to a double computad whose 2-cells of shape
  $\td{a}{b}{c}{d}$ are the 2-cells of shape $\ttw^{a + b}_{c + d}$
  therein (a.k.a.\ ``quintets'').
  
    \[
      \begin{tikzpicture}[->,shorten <=-2pt,shorten >=-3pt, scale=.15]
        \node (a) at (-5, 5) {$\cdot$};
        \node (b) at (5, 5) {$\cdot$};
        \node (c) at (-5, -5) {$\cdot$};
        \node (d) at (5, -5) {$\cdot$};
        \node (uu) at (0,5) {$\bcdots$};
        \node (rr) at (5,0) {$\vcdots$};
        \node (ll) at (-5,0) {$\vcdots$};
        \node (dd) at (0,-5) {$\bcdots$};
        \node (z) at (0, 0) {$\alpha$};
        \draw (a) -- node[above] {$s^H_1$} (uu);
        \draw (uu) -- node [above] {$s^H_a$} (b);
        \draw (b) -- node [right] {$t^V_1$} (rr);
        \draw (rr) -- node [right] {$t^V_b$} (d);
        \draw (a) -- node [left] {$s^V_1$} (ll);
        \draw (ll) -- node [left] {$s^V_c$} (c);
        \draw (c) -- node [below] {$t^H_1$} (dd);
        \draw (dd) -- node [below] {$t^H_d$} (d);
      \end{tikzpicture}
      \qquad\;\;\; \mapsto \qquad
      \begin{tikzpicture}[->,shorten <=-2pt,shorten >=-3pt, xscale=.15,yscale=.045]
        \node (a) at (-10, 0) {$\cdot$};
        \node (b) at (0, 10) {$\cdot$};
        \node (c) at (0, -10) {$\cdot$};
        \node (d) at (10, 0) {$\cdot$};
        \node (uu) at (-5,5) {$\bcdots$};
        \node (rr) at (5,5) {$\bcdots$};
        \node (ll) at (-5,-5) {$\bcdots$};
        \node (dd) at (5,-5) {$\bcdots$};
        \node (z) at (0, 0) {$\!\!\phantom{{}_\alpha}\Downarrow_\alpha$};
        \draw (a) -- node[above] {$s^H_1$} (uu);
        \draw (uu) -- node [above] {$s^H_a$} (b);
        \draw (b) -- node [above] {$t^V_1$} (rr);
        \draw (rr) -- node [above] {$t^V_b$} (d);
        \draw (a) -- node [below] {$s^V_1$} (ll);
        \draw (ll) -- node [below] {$s^V_c$} (c);
        \draw (c) -- node [below] {$t^H_1$} (dd);
        \draw (dd) -- node [below] {$t^H_d$} (d);
      \end{tikzpicture}
    \]
\end{rmk}


We refer to 2-cells of shapes \td1001, \td0110, and \td1111 in a
double computad respectively as \textbf{horizontal bigons},
\textbf{vertical bigons}, and \textbf{squares}. We call a double
computad in which all 2-cells are squares a \textbf{double graph}. We
denote this full subcategory of \bDblCptd by \bDblGph, also a functor
category with domain a full subcategory of $\lC_\d$:
\[\vcenter{\xymatrix@R=5mm@C=5mm{\ttw \ar@<1mm>[r] \ar@<-1mm>[r]
      \ar@<-1mm>[d] \ar@<1mm>[d] &  \tv \ar@<-1mm>[d] \ar@<1mm>[d]\\
      \th \ar@<1mm>[r] \ar@<-1mm>[r] & \tz }}
\]
(composition laws as in $\lC_\d$, where $\ttw \coloneqq \td1111$).

The category \bDblGph is also a comma category
$(\bSet / \ovocptd(\sq, -))$. Hence we additionally have a functor
from $\bDblCptd = (\bSet / \ovocptd(\sq, T_{\ovo}-))$ to $\bDblGph$ by
applying $T_{\ovo}$, which reinterprets all of the 2-cells in a double
computad as squares of paths.
  \[
    \begin{tikzpicture}[->, scale=.15]
      \node (a) at (-5, 5) {};
      \node (b) at (5, 5) {};
      \node (c) at (-5, -5) {};
      \node (d) at (5, -5) {};
      \node at (a) {$\cdot$};
      \node at (b) {$\cdot$};
      \node at (c) {$\cdot$};
      \node at (d) {$\cdot$};
      \node (z) at (0, 0) {$\alpha$};
      \node (sh) at (0,5) {$\bcdots$};
      \node (th) at (0,-5) {$\bcdots$};
      \node (tv) at (5,0) {$\vcdots$};
      \node (sv) at (-5,0) {$\vcdots$};
      \draw (a) -- node [auto,anchor=south] {$s^H_1$} (sh);
      \draw (sh) -- node [auto,anchor=south] {$s^H_a$} (b);
      \draw (b) -- node [auto,anchor=west] {$t^V_1$} (tv);
      \draw (tv) -- node [auto,anchor=west] {$t^V_b$} (d);
      \draw (a) -- node [auto,anchor=east] {$s^V_1$} (sv);
      \draw (sv) -- node [auto,anchor=east] {$s^V_c$} (c);
      \draw (c) -- node [auto,anchor=north] {$t^H_1$} (th);
      \draw (th) -- node [auto,anchor=north] {$t^H_d$} (d);
    \end{tikzpicture}
    \qquad\;\;\;
    \mapsto
    \;\;\;\quad
    \begin{tikzpicture}[->, scale=.15]
      \node (a) at (-5, 5) {};
      \node (b) at (5, 5) {};
      \node (c) at (-5, -5) {};
      \node (d) at (5, -5) {};
      \node at (a) {$\cdot$};
      \node at (b) {$\cdot$};
      \node at (c) {$\cdot$};
      \node at (d) {$\cdot$};
      \node (z) at (0, 0) {$\alpha$};
      \draw (a) -- node [auto,anchor=south] {$s^H_1\cdots s^H_a$} (b);
      \draw (b) -- node [auto,anchor=west] {$t^V_1 \cdots t^V_b$} (d);
      \draw (a) -- node [auto,anchor=east] {$s^V_1 \cdots s^V_c$} (c);
      \draw (c) -- node [auto,anchor=north] {$t^H_1 \cdots t^H_d$} (d);
    \end{tikzpicture}
  \]
This is more precisely a functor
$\dflat \colon \bDblCptd \to \ovocatbDblGph$ where the codomain is
double graphs equipped with \ovo-category structure on 1-cells. Note
that this category \ovocatbDblGph is evidently monadic over \bDblGph.

The functor \dflat is pseudomonic; its image consists of double graphs
equipped with \emph{free} \ovo-category structure and maps sending
generating 1-cells to generating 1-cells. Thus double computads are equivalently such structured double graphs.

The category \bDblCat of (small, strict) double categories is also
monadic over \bDblGph, essentially by definition (as a double graph
equipped with various operations). The forgetful right adjoint
evidently factors through an intermediate right adjoint
$\bDblCat \to \ovocatbDblGph$, which is also monadic by
\cref{thm:mnd-cancel}. In the next section we will see that
\bDblCat is monadic over \bDblCptd as well.


\section{Algebraic definitions}
\label{sec:doublewords}

Now we are able to describe {\twowords} and {\doublewords}
(\cref{sec:twowords,sec:shortcut}) as algebras of monads on
the presheaf categories \tcptd and \bDblCptd respectively, confirming
their essentially algebraic nature.

In \cref{sec:double-computads}, we encountered several essentially algebraic
structures presented by operations and equations (such as
categories, strict 2-categories, and strict double categories), and we
tacitly interpreted these as monads on presheaf categories.  But we will soon need presentations of monads in
more general situations, so we review a general method for presenting monads, following~\cite[\S5]{steve:companion}.
(Our definitions of {\twowords} and {\doublewords} in this section
will just be presentations of monads on presheaf categories as usual, but
in \cref{sec:2-monads} we
will also be interested in presenting 2-monads on non-presheaf categories.)

Let \sV be a locally finitely presentable (l.f.p.)\ monoidal category whose subcategory of finitely presentable objects $\sV_f$ is closed under the monoidal structure, so we have a good theory of l.f.p\ \sV-enriched categories as in~\cite{kelly:enr-lfp}; we will use $\sV=\bSet$ and $\sV=\bCat$.
Let \sK be an l.f.p.\ \sV-category.
Then by~\cite{lack:finitary}, the category $\bMnd_f(\sK)$ of finitary monads on \sK is monadic over the category $[\nob\sK_f,\sK]$ of families of objects of \sK indexed by the set of finitely presentable objects of \sK.
Thus, we can \emph{present} such monads using free monads generated by such families and colimits in $\bMnd_f(\sK)$; and because these free monads and colimits are \emph{algebraic}~\cite{kelly:transfinite}, such a presentation also determines the algebras for the presented monad.
Specifically, given $A\in [\nob\sK_f,\sK]$, an algebra for the free finitary monad $FA$ it generates is an object $X\in \sK$ with a family of maps $\sK(c,X) \to \sK(Ac,X)$ for all $c\in \sK_f$; 
and an algebra for a colimit of finitary monads is an object with a compatible family of algebra structures for those monads.

As an example, we start with a definition of \twowords.
\begin{defn}\label{defn:wordmonads}
  {\Aan} \textbf{\twoword} is a 2-computad $X$ equipped with
  \begin{itemize}
  \item horizontal composition operations
    \[X(\ttw^m_n) \times_0 X(\ttw^{m'}_{n'}) \to X(\ttw^{m+m'}_{n+n'})\]
    (where the target 0-cell of the first factor is identified with the source
    0-cell of the second factor), 
  \item vertical composition operations
    \[X(\ttw^m_x) \times_1 X(\ttw^x_n) \to X(\ttw^m_n)\]
    (where the target 1-cell path of the first factor is identified with the
    source 1-cell path of the second factor),
    and
  \item identity operations
    \[\overbrace{X(\ton)\times_0 \cdots \times_0 X(\ton)}^n \to X(\ttw^n_n)\]
    (where the domain is length $n$ paths of 1-cells)
  \end{itemize}
  satisfying source and target laws, associativity and unit laws, and
  interchange laws.
\end{defn}

To go from this definition to a monad on \tcptd whose algebras are \twowords, we start with the following family $A\in [\nob\,\tcptd_f,\tcptd]$, where we identify objects of $\lC_1$ with their corresponding representable functors in \tcptd:
\[
  Ac =
  \begin{cases}
    \ttw^{m+m'}_{n+n'} &\text{if } c = \ttw^m_n \sqcup_0 \ttw^{m'}_{n'}\\
    \ttw^m_n &\text{if } c = \ttw^m_x \sqcup_1 \ttw^x_n\\
    \ttw^n_n &\text{if } c = \overbrace{\ton \sqcup_0 \cdots \sqcup_0 \ton}^n
  \end{cases}
\]
(Note that all representables are finitely presentable, and pushouts of finitely presentable objects are finitely presentable.)
Then an $FA$-algebra is a 2-computad $X$ equipped with three families of maps.
The first consists of maps
\[ \tcptd(\ttw^m_n \sqcup_0 \ttw^{m'}_{n'}, X) \to \tcptd(\ttw^{m+m'}_{n+n'}, X) \]
But by the universal property of colimits and the Yoneda lemma, this is equivalent to a map
\[X(\ttw^m_n) \times_0 X(\ttw^{m'}_{n'}) \to X(\ttw^{m+m'}_{n+n'})\]
as in \cref{defn:wordmonads} above.
The other two families similarly correspond to the other families of operations in \cref{defn:wordmonads}.
An $FA$-algebra is then a 2-computad equipped with all these operations, but not satisfying any axioms.

To impose the axioms on such a structure, we specify another family $B\in [\nob\,\tcptd_f,\tcptd]$ and a pair of morphisms $B \toto UFA$ in $[\nob\,\tcptd_f,\tcptd]$, where $U$ is the forgetful right adjoint to $F$.
For instance, the contribution to $B$ for associativity of vertical composition is
\[ B(\ttw^m_x \sqcup_1 \ttw^x_y \sqcup_1 \ttw^y_n) = \ttw^m_n. \]
We must then specify two morphisms $\ttw^m_n \to FA(\ttw^m_x \sqcup_1 \ttw^x_y \sqcup_1 \ttw^y_n)$, which is to say two 2-cells of shape $\ttw^m_n$ in the free $FA$-algebra on a trio of 2-cells that could be composed to give one of shape $\ttw^m_n$.
In an $FA$-algebra, there are two ways to bracket the composition of such a trio that are not equal; we take these two bracketed compositions as the two desired 2-cells.
All the other axioms are treated similarly.

Finally, we let $T_2^{\ad}$ be the coequalizer of the two maps $FB \toto FA$ in $\bMnd_f(\tcptd)$.
Then a $T_2^{\ad}$-algebra is an $FA$-algebra $X$ whose two underlying $FB$-algebra structures are equal.
In the case of associativity, this says precisely that the two possible composites of a vertically composable trio are equal in $X$, i.e.\ that $X$ obeys the associativity axiom; and similarly for the other axioms.
Thus, $T_2^{\ad}$-algebras are precisely \twowords\ as defined above.

As usual, we could give an equivalent ``unbiased''
definition using $n$-ary compositions, rather than just binary and nullary composition.
This would lead to a different presentation, but an isomorphic monad.

The double-categorical case is entirely analogous, leading to a monad $T_{\d}^{\ad}$ on \obocptd whose algebras are \doublewords.

\begin{defn}
  {\Aan} \textbf{\doubleword} is a double computad $X$
  with
  \begin{itemize}
  \item horizontal composition operations
    \[X(\td{a}{x}{c}{d}) \times_1 X(\td{a'}{b'}{x}{d'}) \to X(\td{a+a'}{b'}{c}{d+d'})\]
    (where the vertical target 1-cell path of the first factor is
    identified with the vertical source 1-cell path of the second
    factor),
  \item horizontal identity operations
    \[X(\tv)\times_0 \cdots \times_0 X(\tv) \to X(\td{0}{n}{n}{0})\]
    (where the domain is length $n$ paths of vertical 1-cells),
  \item vertical composition operations
    \[X(\td{a}{b}{c}{x}) \times_1 X(\td{x}{b'}{c'}{d'}) \to X(\td{a}{b+b'}{c+c'}{d'})\]
    (where the horizontal target 1-cell path of the first factor
    is identified with the horizontal source 1-cell path of the
    second factor), and
  \item vertical identity operations
    \[X(\th)\times_0 \cdots \times_0 X(\th) \to X(\td{n}{0}{0}{n})\]
    (where the domain is length $n$ paths of horizontal 1-cells)
  \end{itemize}
  satisfying source and target laws, associativity and unit laws, and
  interchange laws.
\end{defn}

These definitions agree
with those of \cref{sec:twowords,sec:shortcut},
since we have observed that 2-computads and double computads can be identified with 2-graphs
and double graphs equipped with free category structure via the functors $\tflat$ and $\dflat$, and the
2-cell operations and laws given here exactly enhance this to 2-category or double category
structure.

\begin{rmk}\label{thm:free-implicit}
  We can also describe these monads in a more conceptual way. Observe that the free 2-category
  monad on \ocattgph (2-graphs equipped with 1-category structure)
  restricts to the subcategory \tcptd (2-graphs equipped with free
  1-category structure and maps sending generating 1-cells to
  generating 1-cells); indeed, this free 2-category monad acts as
  identity on underlying 1-category structure.
  The algebras of this monad on \tcptd are simply algebras of the monad
  on \ocattgph that lie within the subcategory \tcptd, namely those
  2-categories with free underlying 1-categories; algebra morphisms
  are restricted to those that lie within the subcategory \tcptd,
  namely those sending generating 1-cells to generating 1-cells. But
  these are precisely {\twowords} and their functors as defined in
  \cref{sec:twowords}, so the monad is the same as $T_2^\ad$
  constructed above whose category of algebras is \tword.

  Similarly, the free double category monad on \ovocatbDblGph (double
  graphs equipped with horizontal and vertical 1-category structure)
  restricts to the subcategory \bDblCptd (double graphs equipped with
  free 1-category structure and maps sending generating 1-cells to
  generating 1-cells). This induced monad on \bDblCptd is $T_\d^\ad$,
  whose category of algebras is \bDblWord.
\end{rmk}

To upgrade these to definitions of bicategories and doubly weak double
categories, we need only introduce the following additional
operations.

\begin{defn}\label{defn:representeds}
  A \textbf{represented} {\twoword} $X$ is equipped with
  \begin{itemize}
    \item 1-cell composition 2-cells
      \[X(\ton) \times_0 X(\ton) \to X(\ttw^{2}_{1}) \qquad\text{and}\qquad X(\ton) \times_0 X(\ton) \to X(\ttw^{1}_{2})\]
      (where the domain is length 2 paths of 1-cells) and
    \item 1-cell identity 2-cells
      \[X(\tz) \to X(\ttw^{0}_{1}) \qquad\text{and}\qquad X(\tz) \to X(\ttw^{1}_{0})\]
  \end{itemize}
  satisfying laws that ensure these 2-cells form inverse pairs from
  and to the given 1-cell paths.

  Similarly, a \textbf{represented} {\doubleword} $X$ is equipped with
  \begin{itemize}  
  \item 1-cell composition 2-cells
    \begin{align*}
      X(\th) \times_0 X(\th) &\to X(\td{2}{0}{0}{1}),&  X(\th) \times_0 X(\th) &\to X(\td{1}{0}{0}{2}),\\ X(\tv) \times_0 X(\tv) &\to X(\td{0}{1}{2}{0}), & X(\tv) \times_0 X(\tv) &\to X(\td{0}{2}{1}{0})
    \end{align*}
    (where the domains are length 2 paths of horizontal or vertical 1-cells) and
  \item 1-cell identity creation 2-cells
    \begin{align*}
      X(\tz) &\to X(\td{0}{0}{0}{1}), &  X(\tz) &\to X(\td{1}{0}{0}{0}),\\ X(\tz) &\to X(\td{0}{1}{0}{0}), & X(\tz) &\to X(\td{0}{0}{1}{0})
    \end{align*}
  \end{itemize}
  satisfying laws that ensure these 2-cells form inverse pairs from
  and to the given 1-cell paths.
\end{defn}

In \cref{sec:twowords,sec:shortcut}
respectively we characterized bicategories and doubly weak double
categories as represented {\adjective} 2-categories and double
categories. Hence, by the above algebraic definitions:

\begin{prop}\label{prop:tmonadic}
  The category \bBicatst of bicategories and strict functors is
  monadic over the category \tcptd of 2-computads.
  
  Likewise, the category \bWDblCatst of doubly weak double categories
  and strict functors is monadic over the category \bDblCptd of
  double computads.\qed
\end{prop}

Now by the cancellation lemma (\cref{thm:mnd-cancel}), since \tword
is also monadic over \tcptd, we have that \bBicatst is furthermore
monadic over \tword; similarly, $\bWDblCatst$ is monadic over
\bDblWord. However, let us also say how to \emph{present} these monads
on \tword and \bDblWord; we do this because in the next section, we
will obtain 2-monads from the same presentations.

Since the category of algebras for a finitary monad on an l.f.p.\ category is again l.f.p., we can just apply the machinery of presentations of
monads again with $\sK=\tword$ and \bDblWord.  Thus, considering
the double case explicitly for concreteness and variety, we start with $A\in [\nob\,\bDblWord_f, \bDblWord]$ defined by
\[
  A(c) =
  \begin{cases}
    \td{2}{0}{0}{1} \sqcup \td{1}{0}{0}{2} &\text{if } c = \th \sqcup_0 \th\\
    \td{0}{1}{2}{0} \sqcup \td{0}{2}{1}{0} &\text{if } c = \tv\sqcup_0 \tv\\
    \td{0}{0}{0}{1}\sqcup \td{1}{0}{0}{0} \sqcup \td{0}{1}{0}{0} \sqcup \td{0}{0}{1}{0}
    &\text{if } c = \tz
  \end{cases}
\]
where we implicitly identify the representable objects in $\obocptd$ with their images under the free functor in \bDblWord.
Then an $FA$-algebra is {\aan} {\doubleword} equipped with the 1-cell composition and identity creation 2-cell operations as specified above.
We then describe another $B\in [\nob\,\bDblWord_f, \bDblWord]$ with two maps $B \toto UFA$ and consider the coequalizer in $\bMnd_f(\bDblWord)$ of the induced parallel pair $FB \toto FA$, to obtain a monad $T_{\d}^{\wk}$ on $\bDblWord$ whose algebras are represented implicit double categories.
Similarly, we get a monad $T_2^{\wk}$ on $\tword$ whose algebras are represented implicit 2-categories.

We can also describe the free algebras of these monads more directly.

\begin{prop}\label{rmk:freedescription}
  The free bicategory on {\aan} {\twoword} $\cat{X}$ admits the following description.
  \begin{itemize}
  \item Its 0-cells are those of $\cat{X}$.
  \item Its 1-cells are freely generated from those of $\cat{X}$ by binary composition and identities.
  \item Its 2-cells with a given boundary are those in $\cat{X}$ with boundary given by erasing parentheses and identities, with composition as in $\cat{X}$.
  \end{itemize}
  Similarly, the free doubly weak double category on {\aan} {\doubleword} $\cat{X}$ admits the following description.
  \begin{itemize}
  \item Its 0-cells are those of $\cat{X}$.
  \item Its 1-cells of both sorts are freely generated from those of $\cat{X}$ by binary composition and identities.
  \item Its 2-cells with a given boundary are those in $\cat{X}$ with boundary given by erasing parentheses and identities, with composition as in $\cat{X}$.
  \end{itemize}
\end{prop}
\begin{proof}
  We describe the 2-category case; the double-category case is
  similar.  First note that given a path $f_1, \ldots, f_n$ from $A$
  to $B$ in {\aan} {\twoword} $\cat{X}$, the {\twoword} obtained from
  $\cat{X}$ by freely adjoining a 1-cell $f\colon A \to B$ and an
  isomorphism $f_1, \ldots, f_n \cong f$ is described as follows: its
  0-cells and 1-cells are those of $\cat{X}$ plus the 1-cell $f$, and
  the 2-cells in $\cat{X}'$ with a given boundary are those in
  $\cat{X}$ with boundary obtained by replacing all occurrences of $f$
  with $f_1, \ldots, f_n$. It is easy to verify this {\twoword}
  satisfies the claimed universal property.  Similarly, we can adjoin
  any number of such 1-cells with isomorphisms.

  Now the free represented {\twoword} (equivalently, bicategory) on
  {\aan} {\twoword} defined as in \cref{defn:representeds} is a
  sequential colimit of such steps of adjoining isomorphisms.
  Specifically, starting from $\cat{X}_0 = \cat{X}$, we adjoin a
  1-cell as above for \emph{every} path in $\cat{X}_0$ of length 2 or
  0, obtaining a new {\twoword} $\cat{X}_1$.  We then repeat for every
  path of length 2 or 0 in $\cat{X}_1$, obtaining $\cat{X}_2$, and so
  on.  This yields a chain of inclusions
  \[ \cat{X}_0 \to \cat{X}_1 \to \cat{X}_2 \to \cdots.
  \]
  Since the monad on 2-computads for {\twowords} is finitary, the
  colimit $\cat{X}_\infty$ of this chain in \tword\ is its colimit in
  \tcptd\ equipped with the evident composition structure.  And since
  \tcptd\ is a presheaf category and this chain consists of
  monomorphisms, its colimit in \tcptd\ is its ``union'' in a
  straightforward sense, giving the explicit description as stated in
  the proposition.  Finally, it is straightforward to check that
  $\cat{X}_\infty$ is represented, and that any map from $\cat{X}$ to
  a represented {\twoword} factors uniquely through $\cat{X}_\infty$.
\end{proof}  

\begin{cor}\label{cor:freedescription}
  The free bicategory on a 2-computad $\cat{X}$ has 1-cells freely
  generated from those of $\cat{X}$ by binary composition and
  identities, and 2-cells as in the free strict 2-category with
  boundary given by erasing parentheses and identities.  Similarly,
  the free doubly weak double category on a double computad $\cat{X}$
  has 1-cells of both types freely generated from those of $\cat{X}$
  by binary composition and identities, and 2-cells as in the free
  strict double category with boundary given by erasing parentheses
  and identities.
\end{cor}
\begin{proof}
  Combine \cref{rmk:freedescription,thm:free-implicit}.
\end{proof}

Finally, in \cref{sec:twowords,sec:shortcut} we also characterized
strict 2-categories, pseudo double categories, and strict double categories by
imposing associativity and unit laws. These axioms can be added to the
monad presentations, so we have:

\begin{prop}\label{thm:strict-monadic}
  The category \tcat of 2-categories (and strict functors) is monadic over the category
  \tcptd of 2-computads.
  
  Likewise, the categories \bDblCat and \bPsDblCatst of strict double categories and
  pseudo double categories (both with strict
  functors) are monadic over the category \bDblCptd of double
  computads.\qed
\end{prop}

The situation is summarized by chains of forgetful functors
\[\tcat \to \bBicatst \to \tword \to \tcptd\]
and
\[\bDblCat \to \bPsDblCatst \to \bWDblCatst \to \bDblWord \to \bDblCptd\]
all compositions of which are monadic, using
\cref{thm:mnd-cancel}.

\begin{rmk}
  The left adjoint $\tword \to \tcat$
  is in fact the obvious subcategory inclusion, sending {\twowords}
  to their path 2-categories. The left adjoint $\bDblWord
  \to \bDblCat$ is similar.

  The composite $\bBicatst \to \tword \to \tcat$ (forget then free) is
  the usual strictification functor for bicategories, which we
  described explicitly in \cref{prop:bitotwo}. Analogously, the
  composite $\bWDblCatst \to \bDblWord \to \bDblCat$ provides a
  strictification functor for doubly weak double categories; in the next section we will show that every doubly weak double category is equivalent to its strictification in a suitable sense.
\end{rmk}

\section{Icons and 2-monads}
\label{sec:2-monads}

In this section we will see that \tword and \bDblWord can be enhanced
to 2-categories. (One furthermore expects the instances of a
two-dimensional categorical structure to be objects in a
\emph{three}-dimensional categorical structure; transformations and
modifications of {\adjective} 2-categories are discussed in
\cref{sec:twowords-hom}.)

As is standard in the theory of bicategories, we cannot directly
define a (weak or strict) 2-category of bicategories, pseudofunctors,
and transformations: vertical composition of transformations is not
strictly associative.
But there is an alternative notion of 2-cell will gives us a
2-category after all, called an \emph{icon}~\cite{lack:icons}.


When $F$ and $G$ are pseudofunctors of bicategories, an icon from $F$ to $G$ is equivalent to a \emph{colax} transformation whose components are identity 1-cells.
(A \emph{lax} transformation from $F$ to $G$ whose components are identity 1-cells can be identified with an icon from $G$ to $F$; the reason one chooses the colax ones to be primary is that it is in that case that the 2-cell components point \emph{from} the value of $F$ on a 1-cell \emph{to} the value of $G$ on that 1-cell.)

We may define an icon of {\twoword} functors to be simply an icon of
the associated 2-functors between path 2-categories. Unpacking this,
we get the following:

\begin{defn}\label{defn:2icon}
  Let $\cat{C}$ and $\cat{D}$ be {\twowords}, and let
  $F, G \colon \cat{C} \to \cat{D}$ be functors \emph{that agree on
    0-cells}. An \textbf{icon} $\theta$ between $F$ and $G$ consists
  of, for each 1-cell $f \colon A \to B$ in $\cat{C}$, a 2-cell (bigon)
  $\theta_f$ in $\cat{D}$:
  \[
    \begin{tikzpicture}[->]
      \node (x1) at (-1,0) {};
      \node (x2) at (1,0) {};
      \node at (-1.2,.25) {\footnotesize$FA$};
      \node at (-1.2,0) {\footnotesize$=$};
      \node at (-1.2,-.25) {\footnotesize$GA$};
      \node at (1.2,.25) {\footnotesize$FB$};
      \node at (1.2,0) {\footnotesize$=$};
      \node at (1.2,-.25) {\footnotesize$GB$};
      \draw (x1) to[bend left=35] node[above] {\small$Ff$} (x2);
      \draw (x1) to[bend right=35] node[below] {\small$Gf$} (x2);
      \node at (0,0) {$\theta_f$};
    \end{tikzpicture}
  \]
  such that for each 2-cell $\alpha$ in $\cat{C}$, we have
  {\allowdisplaybreaks
    \begin{align*}
      \begin{tikzpicture}[<-,xscale=-.75,yscale=.95,shorten <=-3pt,shorten >=-2pt]
        \node (a) at (-2,0) {$\cdot$};
        \node (bl) at (-.4,.8) {};
        \node at (0,.8) {$\bcdots$};
        \node (br) at (.4,.8) {};
        \node (cl) at (-.4,-1) {};
        \node at (0,-1) {$\bcdots$};
        \node (cr) at (.4,-1) {};
        \node (d) at (2,0) {$\cdot$};
        \node at (0,-.1) {$F\alpha$};
        \draw (a) [bend left=12.5] to node [above right=-2,pos=.65] {\tiny$Fs_m$} (bl);
        \draw (br) [bend left=12.5] to node [above left=-2,pos=.35] {\tiny$Fs_1$} (d);
        \draw (a) [bend left=40] to node [above left=-2,pos=.25] {\tiny$Ft_n$} (cl);
        \draw (cr) [bend left=40] to node [above right=-2,pos=.75] {\tiny$Ft_1$} (d);
        \draw (a) [bend right=40] to node [below right=-2] {\tiny$Gt_n$} (cl);
        \draw (cr) [bend right=40] to node [below left=-2] {\tiny$Gt_1$} (d);
        \node at (-1.25,-.5) {$\sigma_{t_n}$};
        \node at (1.25,-.5) {$\sigma_{t_1}$};
      \end{tikzpicture}
      \,\;
      &=
        \;\,
        \begin{tikzpicture}[->,xscale=.75,yscale=-.95,shorten <=-2pt,shorten >=-3pt]
          \node (a) at (-2,0) {$\cdot$};
          \node (bl) at (-.4,.8) {};
          \node at (0,.8) {$\bcdots$};
          \node (br) at (.4,.8) {};
          \node (cl) at (-.4,-1) {};
          \node at (0,-1) {$\bcdots$};
          \node (cr) at (.4,-1) {};
          \node (d) at (2,0) {$\cdot$};
          \node at (0,-.1) {$G\alpha$};
          \draw (a) [bend left=12.5] to node [below left=-2,pos=.65] {\tiny$Gt_1$} (bl);
          \draw (br) [bend left=12.5] to node [below right=-2,pos=.35] {\tiny$Gt_n$} (d);
          \draw (a) [bend left=40] to node [below right=-2,pos=.25] {\tiny$Gs_1$} (cl);
          \draw (cr) [bend left=40] to node [below left=-2,pos=.75] {\tiny$Gs_m$} (d);
          \draw (a) [bend right=40] to node [above left=-2] {\tiny$Fs_1$} (cl);
          \draw (cr) [bend right=40] to node [above right=-2] {\tiny$Fs_m$} (d);
          \node at (-1.25,-.5) {$\sigma_{s_1}$};
          \node at (1.25,-.5) {$\sigma_{s_m}$};
        \end{tikzpicture}
    \end{align*}
  }
\end{defn}


We define \textbf{compositions} of icons componentwise. Likewise
\textbf{identity} icons are identities componentwise. We can also
\textbf{whisker} an icon with a functor (i.e.\ compose a functor
$C' \to C$ with an icon of functors $C \to D$ to obtain an icon of
functors $C' \to D$; or compose an icon of functors $C \to D$ with a
functor $D \to D'$ to obtain an icon of functors $C \to D'$) by using
the icon components at the image of the functor or by applying the
functor to the icon components, as usual.

\begin{prop}\label{prop:icon-two}
  There is a strict 2-category \ctword\ of implicit 2-categories, functors, and icons.\qed
\end{prop}


This is just the locally full sub-2-category of the 2-category of strict 2-categories, 2-functors, and icons in the ordinary sense.

The definition for {\doublewords} is similar, but there is an added subtlety: we have to choose directions for both the horizontal and vertical component bigons, and these choices can be independent.
Thus in principle we get four different notions of icon, and which one we regard as going ``from'' $F$ ``to'' $G$ depends on our beliefs about which direction the squares in a double category ``point''.
There are also four possibilities for this, which we may name cardinally as \textbf{northwest} $\Nwarrow$, \textbf{northeast} $\Nearrow$, \textbf{southeast} $\Searrow$, and \textbf{southwest} $\Swarrow$.

For the most part we will choose the \emph{southeast} view, which has the advantage that squares point in the same direction as all the arrows on their boundaries:
\[
  \begin{tikzpicture}[->, scale=.4]
    \node (a) at (-1, 1) {};
    \node (b) at (1, 1) {};
    \node (c) at (-1, -1) {};
    \node (d) at (1, -1) {};
    \node (0, 0) {$\Searrow$};
    \node at (a) {$\cdot$};
    \node at (b) {$\cdot$};
    \node at (c) {$\cdot$};
    \node at (d) {$\cdot$};
    \draw (a) -- (b);
    \draw (b) -- (d);
    \draw (a) -- (c);
    \draw (c) -- (d);
  \end{tikzpicture}
\]
This has the consequence that horizontal bigons point from top to bottom, while vertical bigons point from left to right.
However, it should be noted that this is not compatible with the ``quintets'' construction of a double category from a 2-category, which requires picking either the northeast or southwest view.
Fortunately, the four kinds of icon are interchanged by the symmetry operations of double categories, so all of them provide equivalent 2-categories of double categories in the end.
Moreover, \emph{invertible} icons are the same no matter which definition we pick.

\begin{defn}\label{defn:dblicon}
  Let $F, G \colon \cat{C} \to \cat{D}$ be functors of {\doublewords} \emph{that agree on
    0-cells}. A \textbf{southeast icon} $\theta$ between $F$ and $G$ consists
  of
  \begin{itemize}
  \item for each horizontal $f \colon A \to B$ in $\cat{C}$, a
    2-cell (horizontal bigon) $\theta_f$ in $\cat{D}$:
    \[
      \begin{tikzpicture}[->]
        \node (x1) at (-1,0) {};
        \node (x2) at (1,0) {};
        \node at (-1.2,.25) {\footnotesize$FA$};
        \node at (-1.2,0) {\footnotesize$=$};
        \node at (-1.2,-.25) {\footnotesize$GA$};
        \node at (1.2,.25) {\footnotesize$FB$};
        \node at (1.2,0) {\footnotesize$=$};
        \node at (1.2,-.25) {\footnotesize$GB$};
        \draw (x1) to[bend left=35] node[above] {\small$Ff$} (x2);
        \draw (x1) to[bend right=35] node[below] {\small$Gf$} (x2);
        \node at (0,0) {$\theta_f$};
      \end{tikzpicture}
    \]
  \item for each vertical $g \colon A \to B$ in $\cat{C}$, a
    2-cell (vertical bigon) $\theta_g$ in $\cat{D}$:
    \[
      \begin{tikzpicture}[->]
        \node (x1) at (0,1) {};
        \node (x2) at (0,-1) {};
        \node at (0,1.05) {\footnotesize$FA\!=\!GA$};
        \node at (0,-1.05) {\footnotesize$FB\!=\!GB$};
        \draw (x1) to[bend right=35] node[left] {\small$Fg$} (x2);
        \draw (x1) to[bend left=35] node[right] {\small$Gg$} (x2);
        \node at (0,0) {$\theta_g$};
      \end{tikzpicture}
    \]
  \end{itemize}
  such that for each 2-cell $\alpha$ in $\cat{C}$, we have
  {\allowdisplaybreaks
    \begin{align*}
      \begin{tikzpicture}[->,scale=1]
        \path (0, 2.75) -- (0,-2.75);
        \node (a) at (-1.5, 1.5) {};
        \node (b) at (1.5, 1.5) {};
        \node (c) at (-1.5, -1.5) {};
        \node (d) at (1.5, -1.5) {};
        \node (hu1) at (-.2, 1.5) {};
        \node (hu2) at (.2, 1.5) {};
        \node (vr1) at (1.5, .2) {};
        \node (vr2) at (1.5, -.2) {};
        \node (vl1) at (-1.5, .2) {};
        \node (vl2) at (-1.5, -.2) {};
        \node (hd1) at (-.2, -1.5) {};
        \node (hd2) at (.2, -1.5) {};
        \node at (a) {$\cdot$};
        \node at (b) {$\cdot$};
        \node at (c) {$\cdot$};
        \node at (d) {$\cdot$};
        \node at (0,1.5) {$\bcdots$};
        \node at (-1.5,0) {$\vcdots$};
        \node at (1.5,0) {$\vcdots$};
        \node at (0,-1.5) {$\bcdots$};
        \node at (0, 0) {$F\alpha$};
        \node at (1.8,.85) {$\theta_{t_1^V}$};
        \node at (1.8,-.85) {\,$\theta_{t_b^V}$};
        \node at (-.85,-1.8) {$\theta_{t_1^H}$};
        \node at (.85,-1.8) {$\theta_{t_d^H}$};
        \draw (a) -- node[above]{\tiny$Fs_1^H$} (hu1);
        \draw (hu2) -- node[above]{\tiny$Fs_a^H$} (b);
        \draw (b) to[out=-20,in=20,looseness=1.5] node[right]{\tiny$Gt_1^V$} (vr1);
        \draw (b) -- node[left]{\tiny$Ft_1^V$} (vr1);
        \draw (vr2) to[out=-20,in=20,looseness=1.5] node[right]{\tiny$Gt_b^V$} (d);
        \draw (vr2) -- node[left,pos=.4]{\tiny$Ft_b^V$\!} (d);
        \draw (a) -- node[left]{\tiny$Fs_1^V$} (vl1);
        \draw (vl2) -- node[left]{\tiny$Fs_c^V$} (c);
        \draw (c) to[out=-70,in=-110,looseness=1.5] node[below]{\tiny$Gt_1^H$} (hd1);
        \draw (c) -- node[above]{\tiny$Ft_1^H$} (hd1);
        \draw (hd2) to[out=-70,in=-110,looseness=1.5] node[below]{\tiny$Gt_d^H$} (d);
        \draw (hd2) -- node[above]{\tiny$Ft_d^H$} (d);
      \end{tikzpicture}
      \quad
      &=
        \quad
        \begin{tikzpicture}[->,scale=1]
          \path (0, 2.75) -- (0,-2.75);
          \node (a) at (-1.5, 1.5) {};
          \node (b) at (1.5, 1.5) {};
          \node (c) at (-1.5, -1.5) {};
          \node (d) at (1.5, -1.5) {};
          \node (hu1) at (-.2, 1.5) {};
          \node (hu2) at (.2, 1.5) {};
          \node (vr1) at (1.5, .2) {};
          \node (vr2) at (1.5, -.2) {};
          \node (vl1) at (-1.5, .2) {};
          \node (vl2) at (-1.5, -.2) {};
          \node (hd1) at (-.2, -1.5) {};
          \node (hd2) at (.2, -1.5) {};
          \node at (a) {$\cdot$};
          \node at (b) {$\cdot$};
          \node at (c) {$\cdot$};
          \node at (d) {$\cdot$};
          \node at (0,1.5) {$\bcdots$};
          \node at (-1.5,0) {$\vcdots$};
          \node at (1.5,0) {$\vcdots$};
          \node at (0,-1.5) {$\bcdots$};
          \node at (0, 0) {$G\alpha$};
          \node at (-.85,1.8) {$\theta_{s_1^H}$};
          \node at (.85,1.8) {$\theta_{s_a^H}$};
          \node at (-1.8,.85) {$\theta_{s_1^V}$};
          \node at (-1.8,-.85) {$\theta_{s_c^V}$};
          \draw (a) to[out=70,in=110,looseness=1.5] node[above]{\tiny$Fs_1^H$} (hu1);
          \draw (a) -- node[below]{\tiny$Gs_1^H$} (hu1);
          \draw (hu2) to[out=70,in=110,looseness=1.5] node[above]{\tiny$Fs_a^H$} (b);
          \draw (hu2) -- node[below]{\tiny$Gs_a^H$} (b);
          \draw (b) -- node[right]{\tiny$Gt_1^V$} (vr1);
          \draw (vr2) -- node[right]{\tiny$Gt_b^V$} (d);
          \draw (a) to[out=-160,in=160,looseness=1.5] node[left]{\tiny$Fs_1^V$} (vl1);
          \draw (a) -- node[right,pos=.6]{\tiny$Gs_1^V$} (vl1);
          \draw (vl2) to[out=-160,in=160,looseness=1.5] node[left]{\tiny$Fs_c^V$} (c);
          \draw (vl2) -- node[right]{\tiny$Gs_c^V$} (c);
          \draw (c) -- node[below]{\tiny$Gt_1^H$} (hd1);
          \draw (hd2) -- node[below]{\tiny$Gt_d^H$} (d);
        \end{tikzpicture}     
    \end{align*}
  }
\end{defn}

\begin{prop}\label{prop:icon-dbl}
  There is a strict 2-category \cDblWord of implicit double categories, functors, and (southeast) icons.\qed
\end{prop}

Now since $\ctword$ and $\cDblWord$ are 2-categories, we can hope to
enhance the monads on these categories to 2-monads.  This is not
possible for our monads on \tcptd\ and \obocptd, as these are not
2-categories in any obvious way.

\begin{rmk}\label{rmk:cptd-bigon}
  There is also another category between \tword\ and \tcptd\ that can
  be extended to a 2-category: its objects are 2-computads equipped
  with composition operations allowing arbitrary 2-cells to be
  composed only with bigons. (In other words, the bigons form
  categories which compatibly act on other 2-cells.)  The
  double-categorical case is similar.  However, for reasons of space
  we will not treat these categories.
\end{rmk}

\begin{lem}
  These 2-categories \ctword\ and \cDblWord\ are locally finitely
  presentable as 2-categories (that is, \bCat-enriched categories).
\end{lem}
\begin{proof}
  By~\cite[Proposition 7.5]{kelly:enr-lfp}, a cocomplete 2-category
  \sK\ is locally finitely presentable if and only if its underlying
  ordinary category $\sK_0$ is locally finitely presentable and
  whenever $X\in \sK$ is finitely presentable in $\sK_0$ (that is,
  $\sK_0(X,-): \sK_0 \to \bSet$ preserves filtered colimits) then it
  is also \bCat-finitely-presentable in \sK (that is,
  $\sK(X,-): \sK \to \bCat$ preserves filtered colimits).  For this,
  in turn, it suffices to show that $\sK_0$ has a strongly generating
  set of finitely presentable objects that are also finitely
  presentable in $\sK$.

  We consider \cDblWord; the case of \ctword\ is analogous.  For
  cocompleteness, since the underlying 1-category \bDblWord\ is
  cocomplete, it suffices by~\cite[\S3.8]{kelly:enriched} to show that
  \cDblWord\ has powers by small categories.  As for other
  2-categories of icons, these can be constructed ``hom-wise''.  The
  power $X^\lJ$ has the same objects as $X$, its vertical arrows from
  $x$ to $y$ are \lJ-shaped diagrams in the category of such vertical
  arrows of $X$, and similarly for horizontal arrows, while its
  2-cells are families of 2-cells in $X$ indexed by the objects of \lJ
  that are ``natural'' with respect to their boundaries.

  Now an evident strongly generating set of objects in the 1-category
  \bDblWord\ consists of the images of the representables $\tz$,
  $\th$, $\tv$, and $\td{a}{b}{c}{d}$, so it suffices to show that
  these are also finitely presentable in the 2-category, in other
  words that icons mapping out of them preserve filtered colimits.
  Now, there are no nontrivial icons with domain $\tz$, while icons
  with domain $\th$ and $\tv$ are simply horizontally or vertically
  globular 2-cells, and icons with domain $\td{a}{b}{c}{d}$ are
  commutative ``cylinders'' relating two 2-cells of shape
  $\td{a}{b}{c}{d}$ by globular 2-cells on their boundaries.  But all
  of these are finitary structures, and hence are preserved in
  filtered colimits.
\end{proof}

Therefore, we can use the machinery sketched in \cref{sec:doublewords} to present 2-monads
on \ctword\ and \cDblWord.  Moreover, since the finitary objects are
the same whether we regard them as 1-categories or 2-categories,
exactly the same presentation as before actually presents a
2-monad.

We immediately deduce that $\bBicatst$ and $\bWDblCatst$ can also be
enhanced to 2-categories $\cBicatst$ and $\cWDblCatst$, namely the
2-categories of strict algebras and strict morphisms for these 2-monads.
We also obtain immediately notions of
pseudo, lax, and colax morphism between bicategories and doubly weak
double categories.  Moreover, the ``endomorphism monad of a morphism''
$\{f,f\}$ from~\cite[\S2]{kl:property-like} (see
also~\cite[\S5.1]{steve:companion}) implies that the definitions of
these more general morphisms can also be deduced algebraically from
the presentation.

In general, suppose $FA$ is the free 2-monad on
$A\in [\nob\sK_f,\sK]$, for some locally finitely presentable
2-category \sK, so that an $FA$-algebra $X$ is determined by maps
$\sK(c,X) \to \sK(Ac,X)$.  Then a pseudo $FA$-morphism $f:X\to Y$ is
determined by natural isomorphisms
\[
  \begin{tikzcd}
    \sK(c,X) \ar[r] \ar[d] \ar[dr,phantom,"\cong"] & \sK(Ac,X) \ar[d] \\
    \sK(c,Y) \ar[r] & \sK(Ac,Y).
  \end{tikzcd}
\]
Similarly, if $T$ is the coequalizer of the maps $FB \toto FA$, a
pseudo $T$-morphism is a pseudo $FA$-morphism (as above) that
restricts to the same pseudo $FB$-morphism along the two given maps.
In our case, this specializes to the following:

\begin{lem}
  Let $\cat{C}$ and $\cat{D}$ be doubly weak double categories.  A
  pseudo $T_{\d}^{\wk}$-morphism $F:\cat{C}\to \cat{D}$ is a functor
  of implicit double categories together with
  \begin{itemize}
  \item For each pair of composable horizontal 1-cells $f:A\to B$ and $g:B\to C$ in $\cat{C}$, an invertible horizontal bigon in $\cat{D}$:
    \[
      \begin{tikzpicture}[->,yscale=1.2]
        \node (x1) at (-1,0) {};
        \node (x2) at (1,0) {};
        \node at (-1.2,0) {\footnotesize$FA$};
        \node at (1.2,0) {\footnotesize$FC$};
        \draw (x1) to[bend left=35] node[above] {\small$F\bcom{f}{g}$} (x2);
        \draw (x1) to[bend right=35] node[below] {\small$\ubcom{(Ff)}{(Fg)}$} (x2);
        \node at (0,0) {$\phi^H_{f,g}$};
      \end{tikzpicture}
    \]
    that commutes with the representability isomorphisms:
    \[
      \begin{tikzpicture}[->,xscale=.8]
        \node (x1) at (-1.5,0) {};
        \node (x2) at (1.5,0) {};
        \node (x3) at (0,-1) {};
        \node at (-1.7,0) {\footnotesize$FA$};
        \node at (1.7,0) {\footnotesize$FC$};
        \node at (0,-1.15) {\footnotesize$FB$};
        \draw[shorten <=2,shorten >=2] (x1) .. controls +(.75,1.25) and +(-.75,1.25) .. node[above] {\small$F\bcom{f}{g}$} (x2);
        \draw (x1) -- node[fill=white] {\scriptsize$\ubcom{(Ff)}{(Fg)}$} (x2);
        \draw (x1) -- node[below left] {\small$Ff$} (x3);
        \draw (x3) -- node[below right] {\small$Fg$} (x2);
        \node at (0,.55) {\small$\phi^H_{f,g}$};
        \node at (0,-.4) {\small$\choseniso$};
      \end{tikzpicture}
      \quad = \quad
      \begin{tikzpicture}[->,xscale=.8]
        \node (x1) at (-1.5,0) {};
        \node (x2) at (1.5,0) {};
        \node (x3) at (0,-1) {};
        \node at (-1.7,0) {\footnotesize$FA$};
        \node at (1.7,0) {\footnotesize$FC$};
        \node at (0,-1.15) {\footnotesize$FB$};
        \draw[shorten <=2,shorten >=2] (x1) .. controls +(.75,1.25) and +(-.75,1.25) .. node[above] {\small$F\bcom{f}{g}$} (x2);
        \draw (x1) -- node[below left] {\small$Ff$} (x3);
        \draw (x3) -- node[below right] {\small$Fg$} (x2);
        \node at (0,0) {$F(\choseniso)$};
      \end{tikzpicture}
    \]
  \item For each pair of composable vertical 1-cells $f:A\to B$ and $g:B\to C$ in $\cat{C}$, an invertible vertical bigon in $\cat{D}$:
    \[
      \begin{tikzpicture}[->,xscale=1.2]
        \node (x1) at (0,1) {};
        \node (x2) at (0,-1) {};
        \node at (0,1.1) {\footnotesize$FA$};
        \node at (0,-1.1) {\footnotesize$FC$};
        \draw (x1) to[bend right=35] node[left=9] {\small$F($} node[left=5.4] {\tiny$\uvcom{f}{g}$} node[left] {\small$)$} (x2);
        \draw (x1) to[bend left=35] node[right=1] {\tiny$\uvcom{Ff}{Fg}$} (x2);
        \node at (0,0) {$\phi^V_{f,g}$};
      \end{tikzpicture}
    \]
    that commutes with the representability isomorphisms:
    \[
      \begin{tikzpicture}[<-,rotate=90,xscale=.8]
        \node (x1) at (-1.5,0) {};
        \node (x2) at (1.5,0) {};
        \node (x3) at (0,-1) {};
        \node at (-1.65,0) {\footnotesize$FA$};
        \node at (1.65,0) {\footnotesize$FC$};
        \node at (0,-1.2) {\footnotesize$FB$};
        \draw[shorten <=2,shorten >=2] (x1) .. controls +(.75,1.25) and +(-.75,1.25) .. node[left=9] {\small$F($} node[left=5.4] {\tiny$\uvcom{f}{g}$} node[left] {\small$)$} (x2);
        \draw (x1) to node[fill=white] {\tiny$\uvcom{Ff}{Fg}$} (x2);
        \draw (x1) -- node[below right=-1.5] {\small$Ff$} (x3);
        \draw (x3) -- node[above right=-1.5] {\small$Fg$} (x2);
        \node at (0,.55) {\small$\phi^H_{f,g}$};
        \node at (0,-.4) {\;\small$\choseniso$};
      \end{tikzpicture}
      \quad = \quad
      \begin{tikzpicture}[<-,rotate=90,xscale=.8]
        \node (x1) at (-1.5,0) {};
        \node (x2) at (1.5,0) {};
        \node (x3) at (0,-1) {};
        \node at (-1.65,0) {\footnotesize$FA$};
        \node at (1.65,0) {\footnotesize$FC$};
        \node at (0,-1.2) {\footnotesize$FB$};
        \draw[shorten <=2,shorten >=2] (x1) .. controls +(.75,1.25) and +(-.75,1.25) .. node[left=9] {\small$F($} node[left=5.4] {\tiny$\uvcom{f}{g}$} node[left] {\small$)$} (x2);
        \draw (x1) -- node[below right=-1.5] {\small$Ff$} (x3);
        \draw (x3) -- node[above right=-1.5] {\small$Fg$} (x2);
        \node at (0,0) {\!\!$F(\choseniso)$};
      \end{tikzpicture}
    \]
  \item For each object $A\in \cat{C}$, invertible horizontal and vertical bigons:
    \[
      \begin{tikzpicture}[->,yscale=1.2]
        \node (x1) at (-1,0) {};
        \node (x2) at (1,0) {};
        \node at (-1.2,0) {\footnotesize$FA$};
        \node at (1.2,0) {\footnotesize$FA$};
        \draw (x1) to[bend left=35] node[above] {\small$F\bi{A}$} (x2);
        \draw (x1) to[bend right=35] node[below] {\small$\bi{FA}$} (x2);
        \node at (0,0) {$\phi^H_{A}$};
      \end{tikzpicture}
      \qquad\qquad\qquad\quad
      \begin{tikzpicture}[->,xscale=1.2]
        \node (x1) at (0,1) {};
        \node (x2) at (0,-1) {};
        \node at (0,1.1) {\footnotesize$FA$};
        \node at (0,-1.1) {\footnotesize$FA$};
        \draw (x1) to[bend right=35] node[left] {\small$F\bi{A}$}(x2);
        \draw (x1) to[bend left=35] node[right] {\small$\bi{FA}$} (x2);
        \node at (0,0) {$\phi^V_{A}$};
      \end{tikzpicture}
    \]
    that commute with the representability isomorphisms:
    \[
      \begin{tikzpicture}[->,xscale=.5,yscale=.8]
        \node (a) at (0,0) {\footnotesize$FA$};
        \begin{scope}[yscale=.9]
          \draw (a) .. controls +(-1.25,.65) and +(-1,0) .. (0,1.25) node [fill=white,inner sep=0] {\footnotesize$\bi{FA}$} .. controls +(1,0) and +(1.25,.65) .. (a);
          \node at (0,.7) {$\choseniso$};
        \end{scope}
        \begin{scope}[scale=2,shift={(0,-.125)}]
          \node (a) at (0,0) {\phantom{\footnotesize$FA$}};
          \draw[shorten <=2,shorten >=2] (a) .. controls +(-1.25,.65) and +(-1,0) .. (0,1.25) node [above] {\small$F\bi{A}$} .. controls +(1,0) and +(1.25,.65) .. (a);
        \end{scope}
        \node at (0,1.7) {$\phi^H_{A}$};
      \end{tikzpicture}
      \!\!=\!\!
      \begin{tikzpicture}[->,xscale=.5,yscale=.8]
        \node (a) at (0,0) {\footnotesize$FA$};
        \begin{scope}[scale=2,shift={(0,-.125)}]
          \node (a) at (0,0) {\phantom{\footnotesize$FA$}};
          \node at (0,.7) {$F(\choseniso)$};
          \draw[shorten <=2,shorten >=2] (a) .. controls +(-1.25,.65) and +(-1,0) .. (0,1.25) node [above] {\small$F\bi{A}$} .. controls +(1,0) and +(1.25,.65) .. (a);
        \end{scope}
      \end{tikzpicture}
      \qquad\;\;\;
      \begin{tikzpicture}[<-,rotate=90,xscale=.5,yscale=.8]
        \node (a) at (0,0) {\footnotesize$FA$};
        \begin{scope}[yscale=.9]
          \draw (a) .. controls +(-1.25,.65) and +(-1,0) .. (0,1.25) node [fill=white,inner sep=.5] {\footnotesize\;$\bi{FA}$} .. controls +(1,0) and +(1.25,.65) .. (a);
          \node at (0,.7) {\;$\choseniso$};
        \end{scope}
        \begin{scope}[scale=2,shift={(0,-.125)}]
          \node (a) at (0,0) {\phantom{\footnotesize$FA$}};
          \draw[shorten <=2,shorten >=2] (a) .. controls +(-1.25,.65) and +(-1,0) .. (0,1.25) node [left] {\small$F\bi{A}$} .. controls +(1,0) and +(1.25,.65) .. (a);
        \end{scope}
        \node at (0,1.7) {\!\!$\phi^H_{A}$};
      \end{tikzpicture}
      \!\!=
      \begin{tikzpicture}[<-,rotate=90,xscale=.5,yscale=.8]
        \node (a) at (0,0) {\footnotesize$FA$};
        \begin{scope}[scale=2,shift={(0,-.125)}]
          \node (a) at (0,0) {\phantom{\footnotesize$FA$}};
          \node at (0,.7) {$F(\choseniso)$};
          \draw[shorten <=2,shorten >=2] (a) .. controls +(-1.25,.65) and +(-1,0) .. (0,1.25) node [left] {\small$F\bi{A}$} .. controls +(1,0) and +(1.25,.65) .. (a);
        \end{scope}
      \end{tikzpicture}
    \]
  \end{itemize}\qed
\end{lem}

However, since the representability cells are also isomorphisms, the
conditions required above uniquely determine each invertible cell
$\phi$ (as the composite of two representability cells). The case of
bicategories is similar. Thus the pseudo-morphisms are simply functors
of the underlying {\adjective} structures, recovering the categories
\bBicat and \bWDblCat from \cref{sec:twowords} and
\cref{sec:shortcut}:

\begin{prop}
  If $X$ and $Y$ are bicategories, then every functor $F:X\to Y$ of
  \emph{implicit} 2-categories has a unique structure of pseudo
  $T_2^{\wk}$-morphism.

  Similarly, if $X$ and $Y$ are doubly weak double categories, then every
  functor $F:X\to Y$ of \emph{implicit} double categories has a unique
  structure of pseudo $T_{\d}^{\wk}$-morphism.\qed
\end{prop}

\begin{cor}
  The 2-monads $T_2^{\wk}$ on \ctword, and $T_{\d}^{\wk}$ on \cDblWord, are
  pseudo-idempotent.  Therefore, an icon between bicategories or
  doubly weak double categories is nothing more than an icon between
  their underlying implicit 2-categories or implicit double
  categories.
\end{cor}
\begin{proof}
  The first statement is by definition of ``pseudo-idempotent''.
  The second follows from~\cite[Proposition 6.7]{kl:property-like}.
\end{proof}

\begin{rmk}
  In particular, every lax or colax $T_2^{\wk}$- or
  $T_{\d}^{\wk}$-morphism is automatically pseudo.  We could obtain
  nontrivial notions of lax and colax functors by using the
  alternative base 2-category suggested in \cref{rmk:cptd-bigon}.
\end{rmk}

\begin{rmk}
  The same arguments apply for the 2-monads whose algebras are strict
  2-categories, strict double categories, and pseudo double categories.
  In the fully strict case it is also sensible to consider \emph{pseudo
    algebras}; these yield ``unbiased'' bicategories and a similar
  notion of ``unbiased doubly weak double category''.  General 2-monadic
  coherence techniques as
  in~\cite{power:coherence,lack:codescent-coh,shulman:psalg} can be
  adapted to show that every such unbiased structure is equivalent to a
  strict one.
\end{rmk}

We end this section by characterizing the relevant equivalences more
explicitly, and proving a coherence theorem for (biased) doubly weak
double categories.

\begin{lem}\label{lem:equivalence}
  A functor of {\doublewords} $F \colon \cat{C} \to \cat{D}$ is an
  equivalence in the 2-category \cDblWord if and only if it is
  \begin{itemize}
  \item bijective on 0-cells,
  \item locally essentially surjective on horizontal and vertical
    1-cells, and
  \item bijective on 2-cells per boundary of 1-cells in $\cat{C}$.
  \end{itemize}
  Therefore, a functor of doubly weak double categories is an
  equivalence in the 2-category \cWDblCat if and only if it satisfies
  these same conditions.
\end{lem}
\begin{proof}
  Suppose $F$ is an equivalence, so there exists
  $G \colon \cat{D} \to \cat{C}$ with invertible icons
  $1_\cat{C} \cong G \circ F$ and $1_\cat{D} \cong F \circ G$. For
  these icons to exist forces $F$ and $G$ to be inverse on 0-cells. We
  also have that $F$ is surjective on isomorphism classes of 1-cells,
  since $g \cong FGg$ for any 1-cell $g$ in $\cat{D}$. Finally, any
  2-cell $\alpha$ is related to $FG\alpha$ by composing with
  invertible icon components, so $FG$ is bijective on 2-cells per
  boundary of 1-cells; likewise so is $GF$, and therefore so
  must be $F$ and $G$.

  Conversely, suppose $F$ satisfies the conditions above. We
  nonconstructively define a functor $G \colon \cat{D} \to
  \cat{C}$. On 0-cells $G$ is inverse to $F$. For each 1-cell $g$ in
  $\cat{D}$, we pick a 1-cell $Gg$ in $\cat{C}$ with an isomorphism
  $g \cong FGg$. Now to define $G$ on a 2-cell in $\cat{D}$, we
  compose on all sides with these chosen isomorphisms or their
  inverses, then apply the inverse of the bijection on 2-cells given
  by $F$. Functoriality of $G$ so defined follows from functoriality
  of $F$, and we have an invertible icon $1_\cat{D} \cong F \circ G$
  by construction. To define the icon $1_\cat{C} \cong G \circ F$ at a
  1-cell $f$ in $\cat{C}$, we take the chosen isomorphism in $\cat{D}$
  at $Ff$, then apply the inverse of the bijection on 2-cells given by
  $F$. Naturality of this icon also follows from functoriality of $F$.
\end{proof}

\begin{prop}\label{prop:strict}
  Every doubly weak double category is equivalent to a strict one.
\end{prop}
\begin{proof}
  A doubly weak double category is defined as a representable
  {\doubleword}, and {\aan} {\doubleword} is in turn defined as a
  strict double category with free 1-cells. Hence every doubly weak
  double category has an associated strict double category (the path
  double category), its ``strictification''.
  On the other hand, in \cref{cor:strict-double}, we saw that strict
  double categories in the usual sense are identified with doubly weak
  double categories that happen to be strict. Thus the strictification
  of a doubly weak double category $\cat{C}$ determines another doubly
  weak double category $\st \cat{C}$, which is strict. We will show
  that $\cat{C}$ and $\st \cat{C}$ are equivalent {\doublewords}.

  Under the correspondence of \cref{prop:main}, we obtain the
  following description of $\st \cat{C}$: 0-cells in $\st \cat{C}$ are
  0-cells in $\cat{C}$, horizontal or vertical 1-cells in
  $\st \cat{C}$ are \emph{paths} of horizontal or vertical 1-cells in
  $\cat{C}$, and a 2-cell in $\st \cat{C}$ (bordered by paths of
  paths) is a 2-cell in $\cat{C}$ (bordered by the concatenations).

  There is an evident functor $F \colon \cat{C} \to \st \cat{C}$
  sending 1-cells to corresponding length 1 paths. This $F$ is clearly
  bijective on 0-cells and bijective on 2-cells per boundary of
  1-cells in $\cat{C}$. Moreover, $F$ is surjective on isomorphism
  classes of 1-cells, since each 1-cell in $\st \cat{C}$ (a path in
  $\cat{C}$) is isomorphic to a 1-cell in the image of $F$ (a length 1
  path, a composite of the path in $\cat{C}$). Hence $F$ is an
  equivalence by \cref{lem:equivalence}. (Moreover an equivalence in
  the other direction can be constructed explicitly by choosing a
  preferred way of associating compositions of paths.)
\end{proof}


\section{Double bicategories}\label{sec:tidy}

Our last goal in this paper is to give finite axiomatizations of
doubly weak double categories.  There are actually many such
definitions, and we struggled with choosing which ones to present in
detail.  In this section we give a definition that clarifies the
relationship to Verity's double bicategories (and \cref{defn:tidier}
reduces it to a definition only involving cells of \emph{square}
shape); in \cref{sec:cubical} we give a definition that clarifies the
relationship to Garner's cubical bicategories; and finally in
\cref{sec:finax} we give a monadic presentation using only finitely
many of the shapes of a double computad.

A \textbf{double graph with bigons} is a double computad whose only
2-cells are squares, horizontal bigons, and vertical bigons:

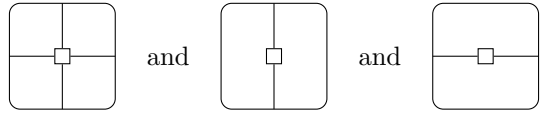
\[
  \begin{tikzpicture}[scale=.7]
    \draw [rect] (0,0) rectangle (2,2);
    \node (x1) [ov] at (1,0) {};
    \node (x2) [ov] at (0,1) {};
    \node (x3) [ov] at (1,2) {};
    \node (x4) [ov] at (2,1) {};
    \node (a) [ivs] at (1,1) {};
    \draw (x1) -- (a) -- (x3);
    \draw (x2) -- (a) -- (x4);

    \node at (3,1) {and};

    \draw [rect] (4,0) rectangle (6,2);
    \node (x1') [ov] at (5,0) {};
    \node (x3') [ov] at (5,2) {};
    \node (a') [ivs] at (5,1) {};
    \draw (x1') -- (a') -- (x3');

    \node at (7,1) {and};

    \draw [rect] (8,0) rectangle (10,2);
    \node (x1') [ov] at (8,1) {};
    \node (x3') [ov] at (10,1) {};
    \node (a') [ivs] at (9,1) {};
    \draw (x1') -- (a') -- (x3');
  \end{tikzpicture}
\]

The category \bBiDblGph of double graphs with bigons can be identified
with a functor category whose domain is a suitable full subcategory of
$\lC_\d$:
\[
  \vcenter{\xymatrix@R=5mm@C=5mm{& & \td{0}{1}{1}{0} \ar@<-1mm>[d] \ar@<1mm>[d]\\
      & \td{1}{1}{1}{1} \ar@<1mm>[r] \ar@<-1mm>[r] \ar@<-1mm>[d] \ar@<1mm>[d] &
      \tv \ar@<-1mm>[d] \ar@<1mm>[d]\\
      \td{1}{0}{0}{1} \ar@<1mm>[r] \ar@<-1mm>[r] & \th \ar@<1mm>[r]
      \ar@<-1mm>[r] & \tz}}
\]
(composition laws as in $\lC_\d$). Hence the forgetful functor
$\bDblCptd \to \bDblGph$ factors through \bBiDblGph.

We now recall the definition of double bicategory, writing out all the operations explicitly for reference.

\begin{defn}[{\cite{verity:base-change}}]
  A \textbf{double bicategory} consists of:
  \begin{itemize}
  \item A double graph with bigons. (That is, collections of 0-cells,
    horizontal and vertical 1-cells, and horizontal bigon 2-cells,
    vertical bigon 2-cells, and square 2-cells, related appropriately by
    various source and target maps.)
  \item The operations of a bicategory on the horizontal 1-cells and
    bigons. Likewise, the operations of a bicategory on the vertical
    1-cells and bigons.
  \item A top bigon-on-square action operation sending compatible
    pairs of horizontal bigons and squares (where the bottom 1-cell of
    the bigon is the same as the top 1-cell of the square) to squares.
    
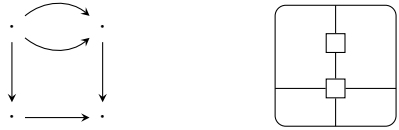
\[
      \begin{tikzpicture}[->,xscale=.6,yscale=.6]
        \node (au) at (-2,1) {$\cdot$};
        \node (bu) at (0,1) {$\cdot$};
        \node (ad) at (-2,-1) {$\cdot$};
        \node (bd) at (0,-1) {$\cdot$};
        \draw (au) to [bend left=40] (bu);
        \draw (au) to [bend right=40] (bu);
        \draw (ad) -- (bd);
        \draw (au) -- (ad);
        \draw (bu) -- (bd);
      \end{tikzpicture}
      \qquad\qquad\qquad
      \begin{tikzpicture}[scale=0.4,rotate=-90]
        \draw [rect] (-2,-2) rectangle (2,2);
        \node (a) [ivs, inner sep=0,minimum width=10] at (-.75,0) {};
        \node (b) [ivs, inner sep=0,minimum width=10] at (.75,0) {};
        \draw (.75,2) -- (b) -- (.75,-2);
        \draw (-2,0) -- (a) -- (b) -- (2,0);
      \end{tikzpicture}
    \]

    Likewise bottom, left, and right bigon-on-square action operations.
  \item A horizontal identity square operation sending vertical
    1-cells to squares.

    Likewise, a vertical identity square operation sending horizontal
    1-cells to squares.
  \item A horizontal composition operation sending compatible pairs of
    squares (where the right 1-cell of the first square is the same as
    the left 1-cell of the second square) to squares.

    Likewise, a vertical composition operation for squares.
  \end{itemize}
  Furthermore, the following laws hold:
  \begin{itemize}
  \item Appropriate source and target laws for all ways of composing
    bigons and squares.
  \item The laws of a bicategory for horizontal 1-cells and bigons, and
    likewise for vertical 1-cells and bigons.
  \item Identity, associativity, and mutual commutativity laws making
    the left, right, top, and bottom bigon-on-square operations into
    four compatible actions.
  \item For any vertical bigon $\beta$, the identity square
    commutativity law
    \[\ubcom{\beta}{\justi} = \ubcom{\justi}{\beta}\]
    (where the left hand side is the left action of $\beta$ on the
    identity square of its codomain, and the right hand side is the
    right action of $\beta$ on the identity square of its domain).

    Likewise, the analogous identity square commutativity law for
    horizontal bigons.
  \item For any compatible horizontal string consisting of a vertical
    bigon $\beta$ sandwiched between two squares $\zeta, \xi$,
    the associativity law
    \[\ubcom{\bcom{\zeta}{\beta}}{\xi} = \ubcom{\zeta}{\bcom{\beta}{\xi}}.\]
    Likewise, the analogous vertical sandwiching associativity law.
  \item For any compatible horizontal string consisting of a vertical
    bigon $\beta$ to the left of two squares $\zeta, \xi$, the
    associativity law
    \[\ubcom{\bcom{\beta}{\zeta}}{\xi} = \ubcom{\beta}{\bcom{\zeta}{\xi}}.\]
    Likewise, the analogous horizontal associativity law on the right,
    and the analogous vertical associativity laws on the top and bottom.
  \item An interchange law that says the two possible ways of composing
    two horizontal bigons side by side atop two horizontally adjacent
    squares are equal.
    
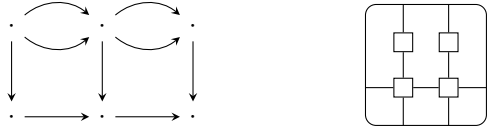
\[
      \begin{tikzpicture}[->,xscale=.6,yscale=.6]
        \node (au) at (-2,1) {$\cdot$};
        \node (bu) at (0,1) {$\cdot$};
        \node (cu) at (2,1) {$\cdot$};
        \node (ad) at (-2,-1) {$\cdot$};
        \node (bd) at (0,-1) {$\cdot$};
        \node (cd) at (2,-1) {$\cdot$};
        \draw (au) to [bend left=40] (bu);
        \draw (bu) to [bend left=40] (cu);
        \draw (au) to [bend right=40] (bu);
        \draw (bu) to [bend right=40] (cu);
        \draw (ad) -- (bd);
        \draw (bd) -- (cd);
        \draw (au) -- (ad);
        \draw (bu) -- (bd);
        \draw (cu) -- (cd);
      \end{tikzpicture}
      \qquad\qquad\qquad
      \begin{tikzpicture}[scale=0.4]
        \draw [rect] (-2,-2) rectangle (2,2);
        \node (lu) [ivs, inner sep=0,minimum width=10] at (-.75,.75) {};
        \node (ru) [ivs, inner sep=0,minimum width=10] at (.75,.75) {};
        \node (ld) [ivs, inner sep=0,minimum width=10] at (-.75,-.75) {};
        \node (rd) [ivs, inner sep=0,minimum width=10] at (.75,-.75) {};
        \draw (-.75,2) -- (lu) -- (ld) -- (-.75,-2);
        \draw (.75,2) -- (ru) -- (rd) -- (.75,-2);
        \draw (-2,-.75) -- (ld) -- (rd) -- (2,-.75);
      \end{tikzpicture}
    \]
    Likewise, the analogous interchange laws for horizontal bigons below
    horizontally adjacent squares, and for vertical bigons to the left
    and to the right of vertically stacked squares.
  \item A horizontal left unitor naturality law for squares $\zeta$:
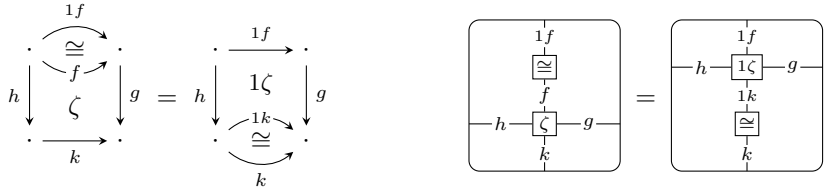
    \[
      \begin{tikzpicture}[->,xscale=.6,yscale=.6]
        \node (au) at (-2,1) {$\cdot$};
        \node (bu) at (0,1) {$\cdot$};
        \node (ad) at (-2,-1) {$\cdot$};
        \node (bd) at (0,-1) {$\cdot$};
        \node at (-1,-.25) {$\zeta$};
        \node at (-1,1) {$\choseniso$};
        \draw (au) to [bend left=40] node [above] {\tiny$\ubcom{\justi}{f}$} (bu);
        \draw (au) to [bend right=40] node [ed] {\scriptsize$f$} (bu);
        \draw (ad) -- node [below] {\scriptsize$k$} (bd);
        \draw (au) -- node [left] {\scriptsize$h$} (ad);
        \draw (bu) -- node [right] {\scriptsize$g$} (bd);
        \path (0,2.5) -- (0,-2.5);
      \end{tikzpicture}
      =
      \begin{tikzpicture}[->,xscale=.6,yscale=.6]
        \node (au) at (-2,1) {$\cdot$};
        \node (bu) at (0,1) {$\cdot$};
        \node (ad) at (-2,-1) {$\cdot$};
        \node (bd) at (0,-1) {$\cdot$};
        \node at (-1,.25) {\small$\ubcom{\justi}{\zeta}$};
        \node at (-1,-1) {$\choseniso$};
        \draw (au) -- node [above] {\tiny$\ubcom{\justi}{f}$} (bu);
        \draw (ad) to [bend left=40] node [ed] {\tiny$\ubcom{\justi}{k}$} (bd);
        \draw (ad) to [bend right=40] node [below] {\scriptsize$k$} (bd);
        \draw (au) -- node [left] {\scriptsize$h$} (ad);
        \draw (bu) -- node [right] {\scriptsize$g$} (bd);
        \path (0,2.5) -- (0,-2.5);
      \end{tikzpicture}
      \qquad\quad\qquad
      \begin{tikzpicture}[xscale=0.5,yscale=-0.5,rotate=90]
        \draw [rect] (-2,-2) rectangle (2,2);
        \node (a) [ivs, inner sep=0,minimum width=12.5] at (-.75,0) {$\choseniso$};
        \node (b) [ivs, inner sep=0,minimum width=12.5] at (.75,0) {$\zeta$};
        \draw (.75,2) -- node [ed] {$h$} (b) -- node [ed] {$g$} (.75,-2);
        \draw (-2,0) -- node [ed] {\tiny$\ubcom{\justi}{f}$} (a) -- node [ed] {$f$} (b) -- node [ed] {$k$} (2,0);
      \end{tikzpicture}
      \;=\;
      \begin{tikzpicture}[xscale=0.5,yscale=-0.5,rotate=90]
        \draw [rect] (-2,-2) rectangle (2,2);
        \node (a) [inner sep=2, draw] at (-.75,0) {\tiny$\ubcom{\justi}{\zeta}$};
        \node (b) [ivs, inner sep=0,minimum width=12.5] at (.75,0) {$\choseniso$};
        \draw (-.75,2) -- node [ed] {$h$} (a) -- node [ed] {$g$} (-.75,-2);
        \draw (-2,0) -- node [ed] {\tiny$\ubcom{\justi}{f}$} (a) -- node [ed] {\tiny$\ubcom{\justi}{k}$} (b) -- node [ed] {$k$} (2,0);
      \end{tikzpicture}
    \]
    where $\choseniso$ denotes the appropriate left unitor
    isomorphism bigons.
    
    Likewise, the analogous horizontal right unitor naturality law, and
    the analogous top and bottom (i.e.\ vertical left and right) unitor
    naturality laws.
  \item A horizontal associator naturality law for squares $\zeta,
    \xi, \psi$:
    \[
      \begin{tikzpicture}[->,xscale=.85,yscale=.6]
        \node (au) at (-2,1) {$\cdot$};
        \node (bu) at (0,1) {$\cdot$};
        \node (ad) at (-2,-1) {$\cdot$};
        \node (bd) at (0,-1) {$\cdot$};
        \node at (-1,-.25) {\footnotesize$\ubcom{\bcom{\zeta}{\xi}}{\psi}$};
        \node at (-1,1) {$\choseniso$};
        \draw (au) to [bend left=40] node [above] {\tiny$\ubcom{f}{\bcom{g}{h}}$} (bu);
        \draw (au) to [bend right=40] node [ed] {\tiny$\ubcom{\bcom{f}{g}}{h}$} (bu);
        \draw (ad) -- node [below] {\tiny$\ubcom{\bcom{p}{q}}{r}$} (bd);
        \draw (au) -- node [left] {\scriptsize$s$} (ad);
        \draw (bu) -- node [right] {\scriptsize$t$} (bd);
        \path (0,2.5) -- (0,-2.5);
      \end{tikzpicture}
      =
      \begin{tikzpicture}[->,xscale=.85,yscale=.6]
        \node (au) at (-2,1) {$\cdot$};
        \node (bu) at (0,1) {$\cdot$};
        \node (ad) at (-2,-1) {$\cdot$};
        \node (bd) at (0,-1) {$\cdot$};
        \node at (-1,.25) {\footnotesize$\ubcom{\zeta}{\bcom{\xi}{\psi}}$};
        \node at (-1,-1) {$\choseniso$};
        \draw (au) -- node [above] {\tiny$\ubcom{f}{\bcom{g}{h}}$} (bu);
        \draw (ad) to [bend left=40] node [ed] {\tiny$\ubcom{p}{\bcom{q}{r}}$} (bd);
        \draw (ad) to [bend right=40] node [below] {\tiny$\ubcom{\bcom{p}{q}}{r}$} (bd);
        \draw (au) -- node [left] {\scriptsize$s$} (ad);
        \draw (bu) -- node [right] {\scriptsize$t$} (bd);
        \path (0,2.5) -- (0,-2.5);
      \end{tikzpicture}
      \qquad\quad\qquad
      \begin{tikzpicture}[xscale=0.5,yscale=-0.5,rotate=90]
        \draw [rect] (-2,-2) rectangle (2,2);
        \node (a) [ivs, inner sep=0,minimum width=12.5] at (-.75,0) {$\choseniso$};
        \node (b) [inner sep=2, draw] at (.75,0) {\tiny$\ubcom{\bcom{\zeta}{\xi}}{\psi}$};
        \draw (.75,2) -- node [ed] {$s$} (b) -- node [ed] {$t$} (.75,-2);
        \draw (-2,0) -- node [ed] {\tiny$\ubcom{f}{\bcom{g}{h}}$} (a) -- node [ed] {\tiny$\ubcom{\bcom{f}{g}}{h}$} (b) -- node [ed] {\tiny$\ubcom{\bcom{p}{q}}{r}$} (2,0);
      \end{tikzpicture}
      \;=\;
      \begin{tikzpicture}[xscale=0.5,yscale=-0.5,rotate=90]
        \draw [rect] (-2,-2) rectangle (2,2);
        \node (a) [inner sep=2, draw] at (-.75,0) {\tiny$\ubcom{\zeta}{\bcom{\xi}{\psi}}$};
        \node (b) [ivs, inner sep=0,minimum width=12.5] at (.75,0) {$\choseniso$};
        \draw (-.75,2) -- node [ed] {$s$} (a) -- node [ed] {$t$} (-.75,-2);
        \draw (-2,0) -- node [ed] {\tiny$\ubcom{f}{\bcom{g}{h}}$} (a) -- node [ed] {\tiny$\ubcom{p}{\bcom{q}{r}}$} (b) -- node [ed] {\tiny$\ubcom{\bcom{p}{q}}{r}$} (2,0);
      \end{tikzpicture}
    \] where $\choseniso$ denotes the appropriate associator
    isomorphism bigons.
    
    Likewise, the analogous vertical associator naturality law.
  \item The interchange laws for squares as in a double category.

    Specifically, the identity compatibility law states that vertical
    identity squares on horizontal identity 1-cells agree with
    horizontal identity squares on vertical 1-cells; the identity
    interchange laws state that horizontal compositions of vertical
    identity squares are vertical identity squares and vice versa; and
    the square composition interchange law states that the two
    possible ways of composing a two by two grid of compatible squares
    are equal.
  \end{itemize}
\end{defn}

We will show that doubly weak double categories are equivalent to double bicategories satisfying an extra ``tidiness'' condition.

\begin{defn}\label{defn:tidy}
  A \textbf{tidy double bicategory} is a double bicategory in which
  the canonical map that sends \emph{2-cells in the horizontal
    bicategory} to \emph{squares whose vertical source and target are
    identities} is bijective
  \[
    \begin{tikzpicture}[->,xscale=.6,yscale=.7]
      \node (a) at (-1,0) {$\cdot$};
      \node (b) at (1,0) {$\cdot$};
      \draw (a) to [bend left=40] node [above] {\scriptsize$f$} (b);
      \draw (a) to [bend right=40] node [below] {\scriptsize$g$} (b);
    \end{tikzpicture}
    \;\;\overset{\cong}{\mapsto}\;\;
    \begin{tikzpicture}[->,xscale=.5,yscale=.5]
      \node (au) at (-1,1) {$\cdot$};
      \node (bu) at (1,1) {$\cdot$};
      \node (ad) at (-1,-1) {$\cdot$};
      \node (bd) at (1,-1) {$\cdot$};
      \draw (au) -- node [above] {\scriptsize$f$} (bu);
      \draw (ad) -- node [below] {\scriptsize$g$} (bd);
      \draw (au) -- node [left] {\scriptsize$\justi$} (ad);
      \draw (bu) -- node [right] {\scriptsize$\justi$} (bd);
    \end{tikzpicture}
    \qquad\qquad\qquad
    \begin{tikzpicture}[scale=0.4]
      \draw [rect] (-2,-2) rectangle (2,2);
      \node (a) [ivs] at (0,0) {};
      \draw (0,2) node [ov] {} -- node [ed] {$f$} (a) -- node [ed] {$g$} (0,-2) node [ov] {};
    \end{tikzpicture}
    \quad\overset{\cong}{\mapsto}\quad
    \begin{tikzpicture}[scale=0.4]
      \draw [rect] (-2,-2) rectangle (2,2);
      \node (a) [ivs] at (0,0) {};
      \draw (-2,0) node [ov] {} -- node [ed] {$\justi$} (a) -- node [ed] {$\justi$} (2,0) node [ov] {};
      \draw (0,2) node [ov] {} -- node [ed] {$f$} (a) -- node [ed] {$g$} (0,-2) node [ov] {};
    \end{tikzpicture}
  \]
  and analogously for 2-cells in the vertical bicategory and squares
  whose horizontal source and target are identities.

  Explicitly, this means a tidy double bicategory has:
  \begin{itemize}
  \item A conversion operation sending squares whose top and bottom
    1-cells are identities to vertical bigons.
    
    Likewise, a conversion operation sending squares whose left and
    right 1-cells are identities to horizontal bigons.
  \end{itemize}
  and the following laws are satisfied:
  \begin{itemize}
  \item Appropriate source and target laws for the degenerate square
    to bigon conversion operations.
  \item The horizontally degenerate square to vertical bigon
    conversion operation is inverse to the map that sends each
    vertical bigon $\beta$ to the square
    \[\ubcom{\beta}{\justi} = \ubcom{\justi}{\beta}.\]

    Likewise, the analogous correspondence holds between vertically
    degenerate squares and horizontal bigons.
  \end{itemize}
\end{defn}

\begin{rmk}
  Tidiness already appears, without a name, in~\cite[Lemma 1.4.9]{verity:base-change}.
  In~\cite{rwan:univalent-doublecats} it is called \emph{saturation}.
\end{rmk}

\begin{rmk}\label{rmk:tidy-nonmonadic}
  Double bicategories are monadic over double graphs, essentially by
  construction.  But \emph{tidy} double bicategories are not, since
  the domains of the additional square-to-bigon conversion operations
  are not objects of \bBiDblGph: there is no double graph with bigons
  representing, say, a ``square whose vertical source and target are
  identities''.
  %
\end{rmk}

All of the operations and laws in a (tidy) double bicategory are
readily derived from those in a doubly weak double category, and so
there is a forgetful functor $U \colon \bWDblCatst \to \bDblBicatst$,
where \bDblBicatst denotes the category of double bicategories and
strict functors, i.e.\ homomorphisms of the algebraic structure.  In
the other direction, we have a functor described as follows (similarly
to \cref{prop:bitotwo}), which will turn out to be left adjoint to
this forgetful functor.

\begin{prop}\label{prop:strictify}
  Given a double bicategory $\bicat{C}$, the following data amount to a
  doubly weak double category $F\bicat{C}$:
  \begin{itemize}
  \item The 0-cells and 1-cells (horizontal and vertical) are as in
    $\bicat{C}$.
  \item A 2-cell with a given boundary is a family consisting of a
    choice of square in $\bicat{C}$ for every possible bracketing of
    the boundary, such that these squares are related by composing
    with the appropriate rebracketing coherence isomorphism bigons.
  \item Composition (and identity) for 2-cells is induced by
    composition of squares in $\bicat{C}$.
  \item The composition isomorphisms are given by identity squares.
  \end{itemize}
\end{prop}
\begin{proof}
  Due to the compatibilities of the bigon actions, the coherence
  theorem for bicategories guarantees that each square with bracketed
  paths along its boundary determines, by composing with coherence
  isomorphisms, a unique corresponding square for every rebracketing
  of the boundary.  Thus composition of 2-cells is well-defined, since
  rebracketing then composing squares is the same as composing then
  rebracketing as appropriate.

  Finally, composition of 2-cells is horizontally and vertically associative
  and unital by the naturality conditions relating associators and
  unitors with squares. It satisfies interchange laws because the
  square composition operations do.
\end{proof}

\begin{rmk}
  The only use of bigons in this definition is to rebracket
  squares. Hence this construction discards the two bicategories of
  bigons; only when the double bicategory is tidy can these two
  bicategories be recovered from the bracketed squares and their
  composition.  Surprisingly, however, although it forgets this
  information it is still left adjoint to the forgetful functor.
\end{rmk}

\begin{lem}\label{lem:clique}
  Any doubly weak double category $\cat{C}$ is isomorphic to $FU\cat{C}$.
\end{lem}
\begin{proof}
  By composing with chosen isomorphisms, the 2-cells with arbitrary
  boundary are in composition-respecting correspondence with bracketed
  squares.
\end{proof}

\begin{lem}\label{lem:bicatfunctor}
  In any double bicategory $\bicat{C}$, the canonical map converting
  horizontal bigons to squares induces a strict functor from the
  horizontal bicategory of $\bicat{C}$ to the horizontal bicategory of
  $F\bicat{C}$. (Likewise for the vertical bicategory.) Hence in the
  case of a \emph{tidy} double bicategory, this is a strict
  isomorphism of bicategories.

  Moreover, this assignment preserves the action of bigons on squares.
\end{lem}
\begin{proof}
  The canonical map from horizontal bigons to squares is by composing
  with a vertical identity square; the resulting square is bordered by
  vertical identities, and so corresponds to a bigon in $F\bicat{C}$.

  The double bicategory laws of associativity, identity commutativity,
  and unitor naturality ensure this mapping preserves vertical
  bicategorical composition (i.e.\ that vertically composing bigons
  then converting to a square is the same as converting then
  vertically composing squares, up to rebracketing with unitors). The
  unit laws for bigon-on-square action ensure preservation of
  identities. The identity interchange and bigon-square interchange
  laws ensure preservation of horizontal composition. Coherence
  isomorphisms are preserved because in $F\bicat{C}$ they are defined
  (see \cref{prop:twotobi}) as compositions of morphisms related to
  identities by composing coherence isomorphisms.

  Moreover, the action of bigons on squares is preserved by
  associativity and unitor naturality laws.
\end{proof}

\begin{lem}\label{lem:freed}
  If $\bicat{C}$ is a double bicategory, then $F \bicat{C}$ is the
  free doubly weak double category on $\bicat{C}$.
\end{lem}
\begin{proof}
  Let $\cat{D}$ be a doubly weak double category. A strict functor
  $\bicat{C} \to U\cat{D}$ induces a strict functor
  $F\bicat{C} \to \cat{D} \cong FU\cat{D}$, since, using
  \cref{lem:clique}, the latter amounts to functorially mapping
  families of bracketed squares in $\bicat{C}$ to families of
  bracketed squares in $\cat{D}$. Conversely, by
  \cref{lem:bicatfunctor} such a map of squares
  $F\bicat{C} \to \cat{D}$ also induces action-respecting strict
  functors from the horizontal and vertical bicategories of \bicat{C}
  to those of \cat{D}, in total determining a strict functor
  $\bicat{C} \to U\cat{D}$. Moreover, these processes of translation
  are inverse.
\end{proof}




\begin{prop}\label{prop:tidy}
  The adjunction
  \[
    \vcenter{\xymatrix@C=12mm{\bDblBicatst \ar@{}[r]|{\!\!\!\!\bot}
        \ar@<2mm>[r]^-{F} & \bWDblCatst \ar@<2mm>[l]^-{U}}}
  \]
  restricts to an equivalence of categories between \bWDblCatst and
  the full subcategory of \bDblBicatst consisting of \emph{tidy} double
  bicategories.
\end{prop}
\begin{proof}
  The counit is an isomorphism, via \cref{lem:clique}. The unit is
  an isomorphism at tidy double bicategories, via
  \cref{lem:bicatfunctor} (additionally noting that squares and
  their composition in $F\bicat{C}$ are also as in $\bicat{C}$).
\end{proof}

\begin{cor}\label{cor:fftidy}
  The forgetful functor $\bWDblCatst \to \bDblBicatst$ is fully faithful.\qed
\end{cor}

\begin{cor}
  The forgetful functor $\bWDblCatst \to \bBiDblGph$ is faithful and
  conservative.\qed
\end{cor}

Thus, we can still regard a doubly weak double category as
``structure'' on an underlying double graph with bigons, though that
structure is not monadic.

\begin{rmk}\label{rmk:tidysquare}
  Conversely, a
  double bicategory is equivalently a doubly weak double category
  $\cat{C}$ together with two bicategories with strict functors into the
  horizontal and vertical bicategories of $\cat{C}$ that are the identity
  on 1-cells.
  Thus, we may alternatively identify doubly weak double
  categories with double bicategories in which the bicategories are
  freely generated by the 1-cells and their incoherent operations.
\end{rmk}

The equivalence of \cref{prop:tidy} can also be extended to pseudofunctors.
For double bicategories, these are the morphisms in Verity's category $\smash{\underline{\underline{\mathcal{H}\mathit{oriz}}}}_{SH}$, whose definition is obtained by combining~\cite[Definition 1.4.7, the definition preceding Lemma 1.4.9, and the definition preceding Observation 1.4.10]{verity:base-change}.

\begin{defn}\label{defn:dblbicat-psfr}
  Let $\cat{C}$ and $\cat{D}$ be double bicategories.  A \textbf{double pseudofunctor} $\cat{C}\to\cat{D}$ consists of:
  \begin{itemize}
  \item Two pseudofunctors from the vertical and horizontal bicategories of $\cat{C}$ to those of $\cat{D}$, which are the same on objects.
  \item A function from squares of $\cat{C}$ to squares of $\cat{D}$ that acts on boundaries as the 1-cell action of the horizontal and vertical pseudofunctors.
  \item The top, bottom, left, and right actions of bigons on squares are preserved.
  \item The horizontal and vertical square composition and identities are preserved, modulo the coherence cells for the horizontal and vertical pseudofunctors.
  \end{itemize}
  These are the morphisms of a category \bDblBicat.
\end{defn}

\begin{lem}
  Any pseudofunctor between doubly weak double categories induces a
  double pseudofunctor between their underlying double bicategories.
\end{lem}
\begin{proof}
  Just like \cref{prop:funtops}.
\end{proof}

\begin{lem}
  If $G:\cat{C}\to\cat{D}$ is a double pseudofunctor between double
  bicategories, the following defines a pseudofunctor of doubly weak
  double categories $FG:F\cat{C}\to F\cat{D}$, where $F$ is as in
  \cref{prop:strictify}.
  \begin{itemize}
  \item The action on 0-cells and 1-cells is as for $G$.
  \item Given a 2-cell with some boundary, its component with a given
    bracketing of the boundary is sent to the image of that 2-cell
    under $G$, acted on all four sides by the coherence isomorphisms
    for that bracketing induced by the horizontal and vertical
    pseudofunctor parts of $G$.
  \end{itemize}
\end{lem}
\begin{proof}
  Coherence for pseudofunctors implies that the operation on 2-cells
  is well-defined, and preserves composition of 2-cells.
\end{proof}

\begin{prop}
  The equivalence of \cref{prop:tidy} extends to an equivalence
  between \bWDblCat and the full subcategory of \bDblBicat determined
  by the tidy double bicategories.\qed
\end{prop}

\begin{rmk}
  If $\cat{C}$ and $\cat{D}$ are strict double categories regarded as
  double bicategories, then a double pseudofunctor as in
  \cref{defn:dblbicat-psfr} specializes to the notion of double
  pseudofunctor from~\cite[Definition 6.1]{shulman:dblderived}.
\end{rmk}

  

Finally, we can further clarify the relationship between doubly weak
double categories and ``untidy'' double bicategories as follows.

\begin{lem}\label{lem:veritymonad}
  The algebras of the monad on \bBiDblGph induced by the forgetful
  functor $\bWDblCatst \to \bDblGph$ are precisely double
  bicategories.
\end{lem}
\begin{proof}
  First we observe that the free doubly weak double category on a
  double graph with bigons is such that the 1-cells are bracketed
  paths, and the 2-cells are grids of squares with sequences of
  vertical or horizontal bigons placed at the vertical and horizontal
  edges, matching along 1-cells, with boundaries bracketed
  arbitrarily.

  Using \cref{cor:freedescription}, to see this it suffices to give a
  similar description of free strict double categories, where the
  1-cells are instead simply paths. Such bigon-accessorized grids
  indeed form a strict double category (where 2-cells bordered by
  identities are given by zero-width or zero-height grids), and we may
  check its universal property. Namely, given a double graph with
  bigons $X$, a strict double category $\cat{C}$, and a map
  $X \to U\cat{C}$ (where $U\cat{C}$ is the underlying double graph
  with bigons of $\cat{C}$), there is a unique extension to a strict
  double functor from the free strict double category
  $FX \to \cat{C}$. Each 2-cell in $FX$ may be composed from the
  generators $X$, for example by horizontally composing the rows
  consisting of squares and vertical bigons; horizontally composing
  (whiskering) horizontal 1-cells and vertical compositions of
  horizontal bigons between the rows; and finally vertically composing
  all these horizontal composites. Hence we obtain a map
  $FX \to \cat{C}$ sending cells in $FX$ to the corresponding
  composites in $\cat{C}$. Functoriality is shown using the
  associativity and interchange laws.


  Now by \cref{prop:tidy}, in order to see that the two monads on
  \bBiDblGph agree, it is enough to see that the underlying
  bicategories of a free double bicategory and those of a free doubly
  weak double category both constitute the free bicategories on the
  underlying 2-graphs. For double bicategories this is clear because
  the only operations giving bigons are the bicategory operations; for
  doubly weak double categories this follows from the description in
  the previous paragraph (and the similar description of free
  bicategories on 2-graphs).
\end{proof}

\begin{prop}
  The forgetful functor $\bWDblCatst \to \bBiDblGph$ is not monadic.
  (That is to say, doubly weak double categories are distinct from
  double bicategories.)
\end{prop}
\begin{proof}
  % ...
  %
  %
  By \cref{lem:veritymonad}, it suffices to exhibit a double
  bicategory that does not arise from any doubly weak double category.
  In a doubly weak double category, there is a bijection between
  2-cells of shapes
  
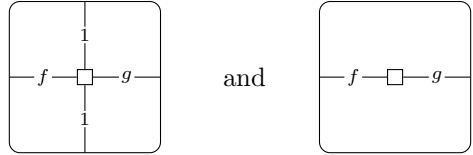
\begin{center}
    \begin{tikzpicture}[scale=0.5]
      \draw [rect] (-2,-2) rectangle (2,2);
      \node (a) [ivs] at (0,0) {};
      \draw (-2,0) node [ov] {} -- node [ed] {$f$} (a) -- node [ed] {$g$} (2,0) node [ov] {};
      \draw (0,-2) node [ov] {} -- node [ed] {$\justi$} (a) -- node [ed] {$\justi$} (0,2) node [ov] {};
    \end{tikzpicture}
    \qquad and\qquad
    \begin{tikzpicture}[scale=0.5]
      \draw [rect] (-2,-2) rectangle (2,2);
      \node (a) [ivs] at (0,0) {};
      \draw (-2,0) node [ov] {} -- node [ed] {$f$} (a) -- node [ed] {$g$} (2,0) node [ov] {};
    \end{tikzpicture}
  \end{center}
  obtained by composing on the top and bottom with the isomorphisms
  
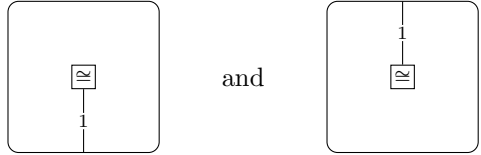
\begin{center}
    \begin{tikzpicture}
      \draw [rect] (-1,-1) rectangle (1,1);
      \node (a) [ivs,inner sep=0] at (0,0) {$\cong$};
      \draw (0,-1) node [ov] {} -- node [ed] {$\justi$} (a);
    \end{tikzpicture}
    \qquad and \qquad
    \begin{tikzpicture}
      \draw [rect] (-1,-1) rectangle (1,1);
      \node (a) [ivs,inner sep = 0] at (0,0) {$\cong$};
      \draw (0,1) node [ov] {} -- node [ed] {$\justi$} (a);
    \end{tikzpicture}\,.
  \end{center}
  
  We now construct a double bicategory without this property. Given
  any monoid $M$, let the double bicategory $\bicat{C}_M$ have two
  0-cells $A$ and $B$, one nonidentity vertical 1-cell
  $f \colon A \to B$, a vertical bigon
  
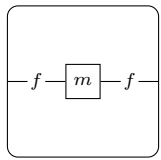
\begin{center}
    \begin{tikzpicture}[scale=0.5]
      \draw [rect] (-2,-2) rectangle (2,2);
      \node (a) [ivs] at (0,0) {$m$};
      \draw (-2,0) node [ov] {} -- node [ed] {$f$} (a) -- node [ed] {$f$} (2,0) node [ov] {};
    \end{tikzpicture}
  \end{center}
  for each $m\in M$, no nonidentity horizontal arrows or bigons, and
  no nonidentity squares.  The only square that the nontrivial
  vertical bigons can act on is the identity square
  
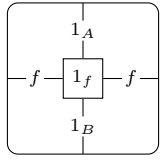
\begin{center}
    \begin{tikzpicture}[scale=0.5]
      \draw [rect] (-2,-2) rectangle (2,2);
      \node (a) [ivs] at (0,0) {$\ti{f}$};
      \draw (-2,0) node [ov] {} -- node [ed] {$f$} (a) -- node [ed] {$f$} (2,0) node [ov] {};
      \draw (0,-2) node [ov] {} -- node [ed] {$\hi{B}$} (a) -- node [ed] {$\hi{A}$} (0,2) node [ov] {};
    \end{tikzpicture}
  \end{center}
  and we can simply say that it is fixed by this action.  Thus, if $M$
  is nontrivial, then in $\bicat{C}_M$ there is no bijection between 2-cells
  of shapes
  
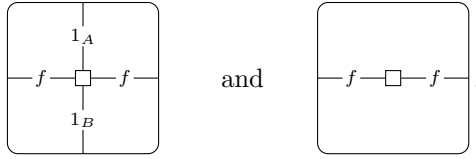
\begin{center}
    \begin{tikzpicture}[scale=0.5]
      \draw [rect] (-2,-2) rectangle (2,2);
      \node (a) [ivs] at (0,0) {};
      \draw (-2,0) node [ov] {} -- node [ed] {$f$} (a) -- node [ed] {$f$} (2,0) node [ov] {};
      \draw (0,-2) node [ov] {} -- node [ed] {$\hi{B}$} (a) -- node [ed] {$\hi{A}$} (0,2) node [ov] {};
    \end{tikzpicture}
    \qquad and\qquad
    \begin{tikzpicture}[scale=0.5]
      \draw [rect] (-2,-2) rectangle (2,2);
      \node (a) [ivs] at (0,0) {};
      \draw (-2,0) node [ov] {} -- node [ed] {$f$} (a) -- node [ed] {$f$} (2,0) node [ov] {};
    \end{tikzpicture}\,.
  \end{center}
  Hence $\bicat{C}_M$ cannot arise from any doubly weak double category.
\end{proof}

\begin{rmk}\label{rmk:premonadic}
  Given \cref{lem:veritymonad}, the functor $\bWDblCatst\to\bDblBicatst$
  is the canonical comparison functor to the category of algebras for
  the induced monad on \bBiDblGph.  When such a comparison functor is fully faithful
  (as it is in this case, by \cref{cor:fftidy}), the right adjoint
  forgetful functor (here $\bWDblCatst\to \bBiDblGph$) is said to be
  \emph{of descent type}~\cite{bw:ttt} or
  \emph{premonadic}~\cite{tholen:thesis}.  There are many other
  equivalent characterizations of this property, which are summarized
  in~\cite[Theorem 2.4]{finitary-monads}; perhaps the most
  interesting is that \emph{every doubly weak double category has a
    canonical presentation as a coequalizer of maps between doubly
    weak double categories that are freely generated by double graphs
    with bigons.}
\end{rmk}

The definition of tidy double bicategory is convenient because it is
finite.  However, it still contains redundancies that can be
eliminated.  If we pare it down to the bones, we obtain our most
concise definition of doubly weak double category.

\begin{prop}\label{defn:tidier}
  A doubly weak double category amounts to:
  \begin{itemize}
  \item a double graph,
  \item horizontal and vertical 1-cell composition and identity operations (as in a double category),
  \item horizontal and vertical square composition and identity operations (as in a double category), and
  \item horizontal and vertical associator and unitor squares (and their putative inverses) with identity 1-cells as their vertical and horizontal boundaries, respectively,
  \end{itemize}
  with appropriate sources and targets, such that
  \begin{itemize}
  \item
        the canonical map induced by composing with an identity square (in any of the four directions)
    \[
      \begin{tikzpicture}[->,xscale=.5,yscale=.5]
        \node (au) at (-1,1) {$\cdot$};
        \node (bu) at (1,1) {$\cdot$};
        \node (ad) at (-1,-1) {$\cdot$};
        \node (bd) at (1,-1) {$\cdot$};
        \draw (au) -- node [above] {\scriptsize$f$} (bu);
        \draw (ad) -- node [below] {\scriptsize$k$} (bd);
        \draw (au) -- node [left] {\scriptsize$h$} (ad);
        \draw (bu) -- node [right] {\scriptsize$g$} (bd);
      \end{tikzpicture}
      \quad\overset{\cong}{\mapsto}\quad
      \begin{tikzpicture}[->,xscale=.5,yscale=.5]
        \node (au) at (-1,1) {$\cdot$};
        \node (bu) at (1,1) {$\cdot$};
        \node (ad) at (-1,-1) {$\cdot$};
        \node (bd) at (1,-1) {$\cdot$};
        \draw (au) -- node [above] {\scriptsize$f$} (bu);
        \draw (ad) -- node [below] {\scriptsize$k$} (bd);
        \draw (au) -- node [left=2] {\scriptsize$\uvcom{\justi}{h}$} (ad);
        \draw (bu) -- node [right=2] {\scriptsize$\uvcom{\justi}{g}$} (bd);
      \end{tikzpicture}
    \]
    is a bijection, per boundary, and 
  \item
    if we define a vertical (resp.\ horizontal) bigon to be a square whose vertical (resp.\ horizontal)
    boundaries are identities:
    \[
      \begin{tikzcd}[arrow style=tikz]
        A \ar[d,"f"'] \ar[r,"1_A^H"] \ar[dr,phantom,"\alpha"] & A \ar[d,"g"] \\
        B \ar[r,"1_B^H"'] & B
      \end{tikzcd}
      \hspace{2cm}
      \begin{tikzcd}[arrow style=tikz]
        A \ar[r,"f"] \ar[d,"1_A^V"'] \ar[dr,phantom,"\beta"] & B \ar[d,"1_B^V"]\\
        A \ar[r,"g"'] & B
      \end{tikzcd}
    \]
    then these data with the derived bigon identity,
    composition, and action operations
    \[
      \qquad\qquad\quad
      \begin{tikzpicture}[->,xscale=.5,yscale=.5]
        \node (au) at (-2,1) {$\cdot$};
        \node (bu) at (0,1) {$\cdot$};
        \node (ad) at (-2,-1) {$\cdot$};
        \node (bd) at (0,-1) {$\cdot$};
        \node at (-1,0) {\small$\justi$};
        \draw (au) -- node[above] {\scriptsize$f$} (bu);
        \draw (ad) -- node[below] {\scriptsize$f$} (bd);
        \draw (au) -- node[left] {\scriptsize$\justi$} (ad);
        \draw (bu) -- node[right] {\scriptsize$\justi$} (bd);
      \end{tikzpicture}
      \qquad\qquad\qquad\quad\hspace{.9em}
      \begin{tikzpicture}[->,xscale=.5,yscale=.5,rotate=-90]
        \node (au) at (-2,1) {$\cdot$};
        \node (bu) at (0,1) {$\cdot$};
        \node (cu) at (2,1) {$\cdot$};
        \node (ad) at (-2,-1) {$\cdot$};
        \node (bd) at (0,-1) {$\cdot$};
        \node (cd) at (2,-1) {$\cdot$};
        \node at (-1,0) {\small$\alpha$};
        \node at (1,0) {\small$\beta$};
        \draw (au) -- node[right] {\scriptsize$\justi$} (bu);
        \draw (bu) -- node[right] {\scriptsize$\justi$} (cu);
        \draw (ad) -- node[left] {\scriptsize$\justi$} (bd);
        \draw (bd) -- node[left] {\scriptsize$\justi$} (cd);
        \draw[<-] (au) -- node[above] {\scriptsize$f$} (ad);
        \draw[<-] (bu) -- node[fill=white,inner sep=2] {\scriptsize$x$} (bd);
        \draw[<-] (cu) -- node[below] {\scriptsize$g$} (cd);
      \end{tikzpicture}
      \;\mapsto\;
      \begin{tikzpicture}[->,xscale=.5,yscale=.5,rotate=-90]
        \node (au) at (-2,1) {$\cdot$};
        \node (bu) at (0,1) {$\cdot$};
        \node (ad) at (-2,-1) {$\cdot$};
        \node (bd) at (0,-1) {$\cdot$};
        \node at (-1,0) {\small$\uvcom{\alpha}{\beta}$};
        \draw (au) -- node[right=2] {\scriptsize$\uvcom{\justi}{\justi}$} (bu);
        \draw (ad) -- node[left=2] {\scriptsize$\uvcom{\justi}{\justi}$} (bd);
        \draw[<-] (au) -- node[above] {\scriptsize$f$} (ad);
        \draw[<-] (bu) -- node[below] {\scriptsize$g$} (bd);
      \end{tikzpicture}
      \;\overset{\cong}{\mapsto}\;
      \begin{tikzpicture}[->,xscale=.5,yscale=.5,rotate=-90]
        \node (au) at (-2,1) {$\cdot$};
        \node (bu) at (0,1) {$\cdot$};
        \node (ad) at (-2,-1) {$\cdot$};
        \node (bd) at (0,-1) {$\cdot$};
        \node at (-1,0) {\small$\uvcom{\alpha}{\beta}$};
        \draw (au) -- node[right] {\scriptsize$\justi$} (bu);
        \draw (ad) -- node[left] {\scriptsize$\justi$} (bd);
        \draw[<-] (au) -- node[above] {\scriptsize$f$} (ad);
        \draw[<-] (bu) -- node[below] {\scriptsize$g$} (bd);
      \end{tikzpicture}
    \]
    \[
      \begin{tikzpicture}[->,xscale=.5,yscale=.5]
        \node (au) at (-2,1) {$\cdot$};
        \node (bu) at (0,1) {$\cdot$};
        \node (cu) at (2,1) {$\cdot$};
        \node (ad) at (-2,-1) {$\cdot$};
        \node (bd) at (0,-1) {$\cdot$};
        \node (cd) at (2,-1) {$\cdot$};
        \node at (-1,0) {\small$\alpha$};
        \node at (1,0) {\small$\beta$};
        \draw (au) -- node[above] {\scriptsize$f$} (bu);
        \draw (bu) -- node[above] {\scriptsize$h$} (cu);
        \draw (ad) -- node[below] {\scriptsize$g$} (bd);
        \draw (bd) -- node[below] {\scriptsize$k$} (cd);
        \draw (au) -- node[left] {\scriptsize$\justi$} (ad);
        \draw (bu) -- node[fill=white,inner sep=2] {\scriptsize$\justi$} (bd);
        \draw (cu) -- node[right] {\scriptsize$\justi$} (cd);
      \end{tikzpicture}
      \;\mapsto\;
      \begin{tikzpicture}[->,xscale=.5,yscale=.5]
        \node (au) at (-2,1) {$\cdot$};
        \node (bu) at (0,1) {$\cdot$};
        \node (ad) at (-2,-1) {$\cdot$};
        \node (bd) at (0,-1) {$\cdot$};
        \node at (-1,0) {\small$\ubcom{\alpha}{\beta}$};
        \draw (au) -- node[above] {\scriptsize$\ubcom{f}{h}$} (bu);
        \draw (ad) -- node[below] {\scriptsize$\ubcom{g}{k}$} (bd);
        \draw (au) -- node[left] {\scriptsize$\justi$} (ad);
        \draw (bu) -- node[right] {\scriptsize$\justi$} (bd);
      \end{tikzpicture}
      \qquad\quad
      \begin{tikzpicture}[->,xscale=.5,yscale=.5,rotate=-90]
        \node (au) at (-2,1) {$\cdot$};
        \node (bu) at (0,1) {$\cdot$};
        \node (cu) at (2,1) {$\cdot$};
        \node (ad) at (-2,-1) {$\cdot$};
        \node (bd) at (0,-1) {$\cdot$};
        \node (cd) at (2,-1) {$\cdot$};
        \node at (-1,0) {\small$\alpha$};
        \node at (1,0) {\small$\zeta$};
        \draw (au) -- node[right] {\scriptsize$\justi$} (bu);
        \draw (bu) -- node[right] {\scriptsize$g$} (cu);
        \draw (ad) -- node[left] {\scriptsize$\justi$} (bd);
        \draw (bd) -- node[left] {\scriptsize$h$} (cd);
        \draw[<-] (au) -- node[above] {\scriptsize$f$} (ad);
        \draw[<-] (bu) -- node[fill=white,inner sep=2] {\scriptsize$x$} (bd);
        \draw[<-] (cu) -- node[below] {\scriptsize$k$} (cd);
      \end{tikzpicture}
      \;\mapsto\;
      \begin{tikzpicture}[->,xscale=.5,yscale=.5,rotate=-90]
        \node (au) at (-2,1) {$\cdot$};
        \node (bu) at (0,1) {$\cdot$};
        \node (ad) at (-2,-1) {$\cdot$};
        \node (bd) at (0,-1) {$\cdot$};
        \node at (-1,0) {\small$\uvcom{\alpha}{\zeta}$};
        \draw (au) -- node[right=2] {\scriptsize$\uvcom{\justi}{g}$} (bu);
        \draw (ad) -- node[left=2] {\scriptsize$\uvcom{\justi}{h}$} (bd);
        \draw[<-] (au) -- node[above] {\scriptsize$f$} (ad);
        \draw[<-] (bu) -- node[below] {\scriptsize$k$} (bd);
      \end{tikzpicture}
      \;\overset{\cong}{\mapsto}\;
      \begin{tikzpicture}[->,xscale=.5,yscale=.5,rotate=-90]
        \node (au) at (-2,1) {$\cdot$};
        \node (bu) at (0,1) {$\cdot$};
        \node (ad) at (-2,-1) {$\cdot$};
        \node (bd) at (0,-1) {$\cdot$};
        \node at (-1,0) {\small$\uvcom{\alpha}{\zeta}$};
        \draw (au) -- node[right] {\scriptsize$g$} (bu);
        \draw (ad) -- node[left] {\scriptsize$h$} (bd);
        \draw[<-] (au) -- node[above] {\scriptsize$f$} (ad);
        \draw[<-] (bu) -- node[below] {\scriptsize$k$} (bd);
      \end{tikzpicture}
    \]
    (and similarly in other directions) satisfy the laws of a double bicategory.
  \end{itemize}
\end{prop}
(Here one could use either of the two inverse bijections to define
composition of bigons; it does not matter.)
\begin{proof}
  The double bicategory so-defined is automatically tidy.  Conversely,
  given any tidy double bicategory, we obtain an isomorphic one by
  replacing all the sets of bigons by the sets of squares to which
  they are in bijection by tidiness.  After this replacement, the
  tidiness isomorphisms become identities, and all the composition
  operations on bigons become equal to the corresponding ones on
  squares; thus we have a structure as described in the statement.
  The two processes are evidently inverse up to isomorphism.
\end{proof}

This definition can be convenient when constructing examples that
do not start with a given bicategory.

\begin{eg}\label{eg:fundamental2}
  As in \cref{eg:fundamental}, let $X$ be a topological space, let the
  0-cells be points of $X$, the 1-cells be continuous paths
  $p:[0,1] \to X$, and the 2-cells be homotopy classes of continuous
  maps $[0,1]\times [0,1] \to X$ rel their boundaries.  We take the
  composition operations on these data to be the usual ones, and the
  associator and unitor squares to be the usual reparametrizing
  homotopies.  It is then straightforward to verify the axioms.
\end{eg}

We will also see a worked example putting this definition to use in
the next section.

\section{Cubical bicategories}\label{sec:cubical}

Next, we compare our definition of doubly weak double category with
Garner's definition of cubical bicategory, which he described as
follows~\cite{garner:cubical-bicats}:
\begin{quote}\small
  Definition. A cubical bicategory is given by sets of objects, of
  vertical arrows, of horizontal arrows and of squares, satisfying the
  obvious source and target criteria, together with operations of
  identity and binary composition for vertical and horizontal arrows,
  satisfying no laws at all; and finally, for every
  $n \mathord{\times} m$ grid of squares (where possibly $n$ or $m$
  are zero), and every way of composing up the horizontal and vertical
  boundaries using the nullary and binary compositions, a composite
  square with those boundaries. The coherence axioms which this
  structure must satisfy say that any two ways of composing up a
  diagram of squares must give the same answer.
\end{quote}

Just like Verity's definition, Garner's definition can be derived from
ours by ignoring some of the structure of a double computad. But
first, let us elaborate on the subtler points of this definition.

The condition that ``any two ways of composing up a diagram of squares
must give the same answer'' a priori constitutes infinitely many
axioms involving grids of squares nested arbitrarily deeply. In
particular, we have infinitely many axioms involving arbitrarily many
$n \times 0$ and $0 \times m$ grids nested within other grids. For
example, all of the following composites are made equal, since they
represent the same formal $2 \times 2$ grid of squares:

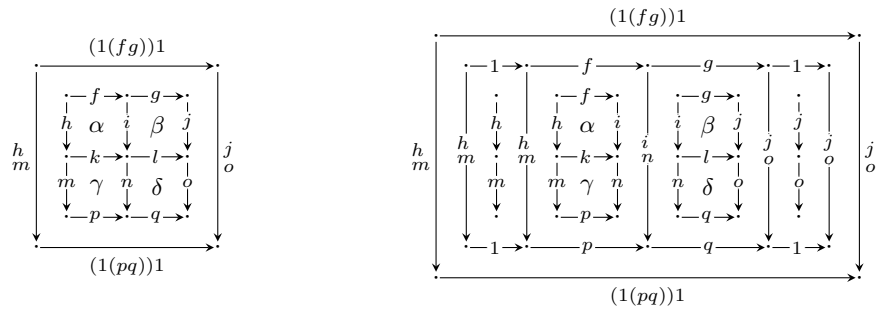
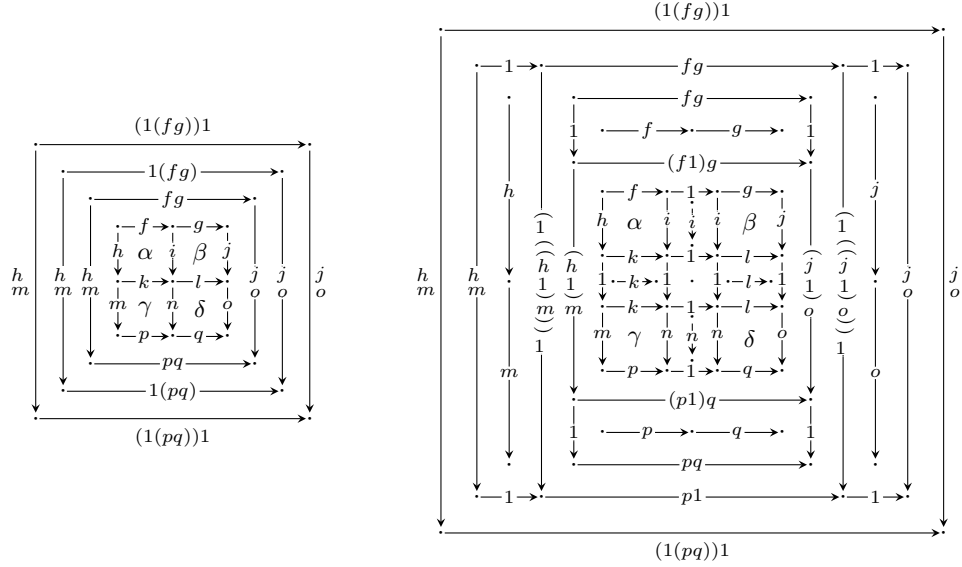
\[\;
  \begin{tikzpicture}[->,scale=.8,shorten <=-3.5pt,shorten >=-3.5pt]
    \node (lu) at (-1,1) {$\cdot$};
    \node (mu) at (0,1) {$\cdot$};
    \node (ru) at (1,1) {$\cdot$};
    \node (lm) at (-1,0) {$\cdot$};
    \node (mm) at (0,0) {$\cdot$};
    \node (rm) at (1,0) {$\cdot$};
    \node (ld) at (-1,-1) {$\cdot$};
    \node (md) at (0,-1) {$\cdot$};
    \node (rd) at (1,-1) {$\cdot$};
    \node (lluu) at (-1.5,1.5) {$\cdot$};
    \node (rruu) at (1.5,1.5) {$\cdot$};
    \node (lldd) at (-1.5,-1.5) {$\cdot$};
    \node (rrdd) at (1.5,-1.5) {$\cdot$};
    \node at (-.5,.5) {\small$\alpha$};
    \node at (.5,.5) {\small$\beta$};
    \node at (-.5,-.5) {\small$\gamma$};
    \node at (.5,-.5) {\small$\delta$};
    \draw (lu) -- node[ed,pos=.45] {\scriptsize$f$} (mu);
    \draw (mu) -- node[ed,pos=.45] {\scriptsize$g$} (ru);
    \draw (lm) -- node[ed,pos=.45] {\scriptsize$k$} (mm);
    \draw (mm) -- node[ed,pos=.45] {\scriptsize$l$} (rm);
    \draw (ld) -- node[ed,pos=.45] {\scriptsize$p$} (md);
    \draw (md) -- node[ed,pos=.45] {\scriptsize$q$} (rd);
    \draw (lu) -- node[ed,pos=.35] {\scriptsize$h$} (lm);
    \draw (lm) -- node[ed,pos=.35] {\scriptsize$m$} (ld);
    \draw (mu) -- node[ed,pos=.35] {\scriptsize$i$} (mm);
    \draw (mm) -- node[ed,pos=.35] {\scriptsize$n$} (md);
    \draw (ru) -- node[ed,pos=.35] {\scriptsize$j$} (rm);
    \draw (rm) -- node[ed,pos=.35] {\scriptsize$o$} (rd);
    \draw (lluu) -- node[above] {\scriptsize$\uhcom{\hcom{\justi}{\hcom{f}{g}}}{\justi}$} (rruu);
    \draw (lldd) -- node[below] {\scriptsize$\uhcom{\hcom{\justi}{\hcom{p}{q}}}{\justi}$} (rrdd);
    \draw (lluu) -- node[left] {\scriptsize$\uvcom{h}{m}$} (lldd);
    \draw (rruu) -- node[right] {\scriptsize$\uvcom{j}{o}$} (rrdd);
  \end{tikzpicture}
  \qquad\qquad\qquad\;
  \begin{tikzpicture}[->,scale=.8,shorten <=-3.5pt,shorten >=-3.5pt]
    \node (lu) at (-1.5,1) {$\cdot$};
    \node (lhu) at (-.5,1) {$\cdot$};
    \node (rhu) at (.5,1) {$\cdot$};
    \node (ru) at (1.5,1) {$\cdot$};
    \node (lm) at (-1.5,0) {$\cdot$};
    \node (lhm) at (-.5,0) {$\cdot$};
    \node (rhm) at (.5,0) {$\cdot$};
    \node (rm) at (1.5,0) {$\cdot$};
    \node (ld) at (-1.5,-1) {$\cdot$};
    \node (lhd) at (-.5,-1) {$\cdot$};
    \node (rhd) at (.5,-1) {$\cdot$};
    \node (rd) at (1.5,-1) {$\cdot$};
    \node (lluuh) at (-3,1.5) {$\cdot$};
    \node (llddh) at (-3,-1.5) {$\cdot$};
    \node (llhuuh) at (-2,1.5) {$\cdot$};
    \node (llhddh) at (-2,-1.5) {$\cdot$};
    \node (muuh) at (0,1.5) {$\cdot$};
    \node (mddh) at (0,-1.5) {$\cdot$};
    \node (rrhuuh) at (2,1.5) {$\cdot$};
    \node (rrhddh) at (2,-1.5) {$\cdot$};
    \node (rruuh) at (3,1.5) {$\cdot$};
    \node (rrddh) at (3,-1.5) {$\cdot$};
    \node (lluu) at (-3.5,2) {$\cdot$};
    \node (rruu) at (3.5,2) {$\cdot$};
    \node (lldd) at (-3.5,-2) {$\cdot$};
    \node (rrdd) at (3.5,-2) {$\cdot$};
    \node (l0) at (-2.5,1) {$\cdot$};
    \node (l1) at (-2.5,0) {$\cdot$};
    \node (l2) at (-2.5,-1) {$\cdot$};
    \node (r0) at (2.5,1) {$\cdot$};
    \node (r1) at (2.5,0) {$\cdot$};
    \node (r2) at (2.5,-1) {$\cdot$};
    \node at (-1,.5) {\small$\alpha$};
    \node at (1,.5) {\small$\beta$};
    \node at (-1,-.5) {\small$\gamma$};
    \node at (1,-.5) {\small$\delta$};
    \draw (l0) -- node[ed,pos=.35] {\scriptsize$h$} (l1);
    \draw (l1) -- node[ed,pos=.35] {\scriptsize$m$} (l2);
    \draw (r0) -- node[ed,pos=.35] {\scriptsize$j$} (r1);
    \draw (r1) -- node[ed,pos=.35] {\scriptsize$o$} (r2);
    \draw (lu) -- node[ed,pos=.45] {\scriptsize$f$} (lhu);
    \draw (rhu) -- node[ed,pos=.45] {\scriptsize$g$} (ru);
    \draw (lm) -- node[ed,pos=.45] {\scriptsize$k$} (lhm);
    \draw (rhm) -- node[ed,pos=.45] {\scriptsize$l$} (rm);
    \draw (ld) -- node[ed,pos=.45] {\scriptsize$p$} (lhd);
    \draw (rhd) -- node[ed,pos=.45] {\scriptsize$q$} (rd);
    \draw (lu) -- node[ed,pos=.35] {\scriptsize$h$} (lm);
    \draw (lm) -- node[ed,pos=.35] {\scriptsize$m$} (ld);
    \draw (lhu) -- node[ed,pos=.35] {\scriptsize$i$} (lhm);
    \draw (lhm) -- node[ed,pos=.35] {\scriptsize$n$} (lhd);
    \draw (rhu) -- node[ed,pos=.35] {\scriptsize$i$} (rhm);
    \draw (rhm) -- node[ed,pos=.35] {\scriptsize$n$} (rhd);
    \draw (ru) -- node[ed,pos=.35] {\scriptsize$j$} (rm);
    \draw (rm) -- node[ed,pos=.35] {\scriptsize$o$} (rd);
    \draw (lluuh) -- node[ed,pos=.45] {\scriptsize$\justi$} (llhuuh);
    \draw (llhuuh) -- node[ed] {\scriptsize$f$} (muuh);
    \draw (muuh) -- node[ed] {\scriptsize$g$} (rrhuuh);
    \draw (rrhuuh) -- node[ed,pos=.45] {\scriptsize$\justi$} (rruuh);
    \draw (llddh) -- node[ed,pos=.45] {\scriptsize$\justi$} (llhddh);
    \draw (llhddh) -- node[ed] {\scriptsize$p$} (mddh);
    \draw (mddh) -- node[ed] {\scriptsize$q$} (rrhddh);
    \draw (rrhddh) -- node[ed,pos=.45] {\scriptsize$\justi$} (rrddh);
    \draw (lluuh) -- node[ed,pos=.45] {\scriptsize$\uvcom{h}{m}$} (llddh);
    \draw (llhuuh) -- node[ed,pos=.45] {\scriptsize$\uvcom{h}{m}$} (llhddh);
    \draw (muuh) -- node[ed,pos=.45] {\scriptsize$\uvcom{i}{n}$} (mddh);
    \draw (rrhuuh) -- node[ed,pos=.45] {\scriptsize$\uvcom{j}{o}$} (rrhddh);
    \draw (rruuh) -- node[ed,pos=.45] {\scriptsize$\uvcom{j}{o}$} (rrddh);
    \draw (lluu) -- node[above] {\scriptsize$\uhcom{\hcom{\justi}{\hcom{f}{g}}}{\justi}$} (rruu);
    \draw (lldd) -- node[below] {\scriptsize$\uhcom{\hcom{\justi}{\hcom{p}{q}}}{\justi}$} (rrdd);
    \draw (lluu) -- node[left] {\scriptsize$\uvcom{h}{m}$} (lldd);
    \draw (rruu) -- node[right] {\scriptsize$\uvcom{j}{o}$} (rrdd);
  \end{tikzpicture}
\]
\[
  \begin{tikzpicture}[->,scale=.725,shorten <=-3.5pt,shorten >=-3.5pt]
    \node (lu) at (-1,1) {$\cdot$};
    \node (mu) at (0,1) {$\cdot$};
    \node (ru) at (1,1) {$\cdot$};
    \node (lm) at (-1,0) {$\cdot$};
    \node (mm) at (0,0) {$\cdot$};
    \node (rm) at (1,0) {$\cdot$};
    \node (ld) at (-1,-1) {$\cdot$};
    \node (md) at (0,-1) {$\cdot$};
    \node (rd) at (1,-1) {$\cdot$};
    \node (lluu) at (-1.5,1.5) {$\cdot$};
    \node (rruu) at (1.5,1.5) {$\cdot$};
    \node (lldd) at (-1.5,-1.5) {$\cdot$};
    \node (rrdd) at (1.5,-1.5) {$\cdot$};
    \node (llluuu) at (-2,2) {$\cdot$};
    \node (rrruuu) at (2,2) {$\cdot$};
    \node (lllddd) at (-2,-2) {$\cdot$};
    \node (rrrddd) at (2,-2) {$\cdot$};
    \node (lllluuuu) at (-2.5,2.5) {$\cdot$};
    \node (rrrruuuu) at (2.5,2.5) {$\cdot$};
    \node (lllldddd) at (-2.5,-2.5) {$\cdot$};
    \node (rrrrdddd) at (2.5,-2.5) {$\cdot$};
    \node at (-.5,.5) {\small$\alpha$};
    \node at (.5,.5) {\small$\beta$};
    \node at (-.5,-.5) {\small$\gamma$};
    \node at (.5,-.5) {\small$\delta$};
    \draw (lu) -- node[ed,pos=.45] {\scriptsize$f$} (mu);
    \draw (mu) -- node[ed,pos=.45] {\scriptsize$g$} (ru);
    \draw (lm) -- node[ed,pos=.45] {\scriptsize$k$} (mm);
    \draw (mm) -- node[ed,pos=.45] {\scriptsize$l$} (rm);
    \draw (ld) -- node[ed,pos=.45] {\scriptsize$p$} (md);
    \draw (md) -- node[ed,pos=.45] {\scriptsize$q$} (rd);
    \draw (lu) -- node[ed,pos=.35] {\scriptsize$h$} (lm);
    \draw (lm) -- node[ed,pos=.35] {\scriptsize$m$} (ld);
    \draw (mu) -- node[ed,pos=.35] {\scriptsize$i$} (mm);
    \draw (mm) -- node[ed,pos=.35] {\scriptsize$n$} (md);
    \draw (ru) -- node[ed,pos=.35] {\scriptsize$j$} (rm);
    \draw (rm) -- node[ed,pos=.35] {\scriptsize$o$} (rd);
    \draw (lluu) -- node[ed] {\scriptsize$\uhcom{f}{g}$} (rruu);
    \draw (lldd) -- node[ed] {\scriptsize$\uhcom{p}{q}$} (rrdd);
    \draw (lluu) -- node[ed] {\scriptsize$\uvcom{h}{m}$} (lldd);
    \draw (rruu) -- node[ed] {\scriptsize$\uvcom{j}{o}$} (rrdd);
    \draw (llluuu) -- node[ed] {\scriptsize$\uhcom{\justi}{\hcom{f}{g}}$} (rrruuu);
    \draw (lllddd) -- node[ed] {\scriptsize$\uhcom{\justi}{\hcom{p}{q}}$} (rrrddd);
    \draw (llluuu) -- node[ed] {\scriptsize$\uvcom{h}{m}$} (lllddd);
    \draw (rrruuu) -- node[ed] {\scriptsize$\uvcom{j}{o}$} (rrrddd);
    \draw (lllluuuu) -- node[above] {\scriptsize$\uhcom{\hcom{\justi}{\hcom{f}{g}}}{\justi}$} (rrrruuuu);
    \draw (lllldddd) -- node[below] {\scriptsize$\uhcom{\hcom{\justi}{\hcom{p}{q}}}{\justi}$} (rrrrdddd);
    \draw (lllluuuu) -- node[left] {\scriptsize$\uvcom{h}{m}$} (lllldddd);
    \draw (rrrruuuu) -- node[right] {\scriptsize$\uvcom{j}{o}$} (rrrrdddd);
  \end{tikzpicture}
  \qquad\quad
  \begin{tikzpicture}[->,scale=.95,shorten <=-3.5pt,shorten >=-3.5pt]
    \node (lu) at (-1.25,1.25) {$\cdot$};
    \node (lhu) at (-.35,1.25) {$\cdot$};
    \node (rhu) at (.35,1.25) {$\cdot$};
    \node (ru) at (1.25,1.25) {$\cdot$};
    \node (luh) at (-1.25,.35) {$\cdot$};
    \node (ldh) at (-1.25,-.35) {$\cdot$};
    \node (lhuh) at (-.35,.35) {$\cdot$};
    \node (lhdh) at (-.35,-.35) {$\cdot$};
    \node (rhuh) at (.35,.35) {$\cdot$};
    \node (rhdh) at (.35,-.35) {$\cdot$};
    \node (ruh) at (1.25,.35) {$\cdot$};
    \node (rdh) at (1.25,-.35) {$\cdot$};
    \node (ld) at (-1.25,-1.25) {$\cdot$};
    \node (lhd) at (-.35,-1.25) {$\cdot$};
    \node (rhd) at (.35,-1.25) {$\cdot$};
    \node (rd) at (1.25,-1.25) {$\cdot$};
    \node (lluu) at (-1.65,1.65) {$\cdot$};
    \node (rruu) at (1.65,1.65) {$\cdot$};
    \node (lldd) at (-1.65,-1.65) {$\cdot$};
    \node (rrdd) at (1.65,-1.65) {$\cdot$};
    \node (lluuuh) at (-1.65,2.55) {$\cdot$};
    \node (rruuuh) at (1.65,2.55) {$\cdot$};
    \node (lldddh) at (-1.65,-2.55) {$\cdot$};
    \node (rrdddh) at (1.65,-2.55) {$\cdot$};
    \node (lluuu) at (-2.1,3) {$\cdot$};
    \node (rruuu) at (2.1,3) {$\cdot$};
    \node (llddd) at (-2.1,-3) {$\cdot$};
    \node (rrddd) at (2.1,-3) {$\cdot$};
    \node (lllhuuu) at (-3,3) {$\cdot$};
    \node (rrrhuuu) at (3,3) {$\cdot$};
    \node (lllhddd) at (-3,-3) {$\cdot$};
    \node (rrrhddd) at (3,-3) {$\cdot$};
    \node (llluuu) at (-3.5,3.5) {$\cdot$};
    \node (rrruuu) at (3.5,3.5) {$\cdot$};
    \node (lllddd) at (-3.5,-3.5) {$\cdot$};
    \node (rrrddd) at (3.5,-3.5) {$\cdot$};
    \node (u0) at (0,1.1) {$\cdot$};
    \node (u1) at (0,.5) {$\cdot$};
    \node (d0) at (0,-.5) {$\cdot$};
    \node (d1) at (0,-1.1) {$\cdot$};
    \node (l0) at (-1.1,0) {$\cdot$};
    \node (l1) at (-.5,0) {$\cdot$};
    \node (r0) at (.5,0) {$\cdot$};
    \node (r1) at (1.1,0) {$\cdot$};
    \node (uu0) at (-1.25,2.1) {$\cdot$};
    \node (uu1) at (0,2.1) {$\cdot$};
    \node (uu2) at (1.25,2.1) {$\cdot$};
    \node (dd0) at (-1.25,-2.1) {$\cdot$};
    \node (dd1) at (0,-2.1) {$\cdot$};
    \node (dd2) at (1.25,-2.1) {$\cdot$};
    \node (ll0) at (-2.55,2.55) {$\cdot$};
    \node (ll1) at (-2.55,0) {$\cdot$};
    \node (ll2) at (-2.55,-2.55) {$\cdot$};
    \node (rr0) at (2.55,2.55) {$\cdot$};
    \node (rr1) at (2.55,0) {$\cdot$};
    \node (rr2) at (2.55,-2.55) {$\cdot$};
    \node at (-.8,.8) {\small$\alpha$};
    \node at (.8,.8) {\small$\beta$};
    \node at (-.8,-.8) {\small$\gamma$};
    \node at (.8,-.8) {\small$\delta$};
    \node at (0,0) {\small$\cdot$};
    \draw (l0) -- node[ed,pos=.45] {\scriptsize$k$} (l1);
    \draw (r0) -- node[ed,pos=.45] {\scriptsize$l$} (r1);
    \draw (u0) -- node[ed,pos=.35] {\scriptsize$i$} (u1);
    \draw (d0) -- node[ed,pos=.35] {\scriptsize$n$} (d1);
    \draw (lu) -- node[ed,pos=.45] {\scriptsize$f$} (lhu);
    \draw (uu0) -- node[ed] {\scriptsize$f$} (uu1);
    \draw (uu1) -- node[ed] {\scriptsize$g$} (uu2);
    \draw (dd0) -- node[ed] {\scriptsize$p$} (dd1);
    \draw (dd1) -- node[ed] {\scriptsize$q$} (dd2);
    \draw (ll0) -- node[ed] {\scriptsize$h$} (ll1);
    \draw (ll1) -- node[ed] {\scriptsize$m$} (ll2);
    \draw (rr0) -- node[ed] {\scriptsize$j$} (rr1);
    \draw (rr1) -- node[ed] {\scriptsize$o$} (rr2);
    \draw (lhu) -- node[ed,pos=.45] {\scriptsize$\justi$} (rhu);
    \draw (rhu) -- node[ed,pos=.45] {\scriptsize$g$} (ru);
    \draw (luh) -- node[ed,pos=.45] {\scriptsize$k$} (lhuh);
    \draw (lhuh) -- node[ed,pos=.45] {\scriptsize$\justi$} (rhuh);
    \draw (rhuh) -- node[ed,pos=.45] {\scriptsize$l$} (ruh);
    \draw (ldh) -- node[ed,pos=.45] {\scriptsize$k$} (lhdh);
    \draw (lhdh) -- node[ed,pos=.45] {\scriptsize$\justi$} (rhdh);
    \draw (rhdh) -- node[ed,pos=.45] {\scriptsize$l$} (rdh);
    \draw (ld) -- node[ed,pos=.45] {\scriptsize$p$} (lhd);
    \draw (lhd) -- node[ed,pos=.45] {\scriptsize$\justi$} (rhd);
    \draw (rhd) -- node[ed,pos=.45] {\scriptsize$q$} (rd);
    \draw (lu) -- node[ed,pos=.35] {\scriptsize$h$} (luh);
    \draw (luh) -- node[ed,pos=.35] {\scriptsize$\justi$} (ldh);
    \draw (ldh) -- node[ed,pos=.35] {\scriptsize$m$} (ld);
    \draw (lhu) -- node[ed,pos=.35] {\scriptsize$i$} (lhuh);
    \draw (lhuh) -- node[ed,pos=.35] {\scriptsize$\justi$} (lhdh);
    \draw (lhdh) -- node[ed,pos=.35] {\scriptsize$n$} (lhd);
    \draw (rhu) -- node[ed,pos=.35] {\scriptsize$i$} (rhuh);
    \draw (rhuh) -- node[ed,pos=.35] {\scriptsize$\justi$} (rhdh);
    \draw (rhdh) -- node[ed,pos=.35] {\scriptsize$n$} (rhd);
    \draw (ru) -- node[ed,pos=.35] {\scriptsize$j$} (ruh);
    \draw (ruh) -- node[ed,pos=.35] {\scriptsize$\justi$} (rdh);
    \draw (rdh) -- node[ed,pos=.35] {\scriptsize$o$} (rd);
    \draw (lluu) -- node[ed] {\scriptsize$\uhcom{\hcom{f}{\justi}}{g}$} (rruu);
    \draw (lldd) -- node[ed] {\scriptsize$\uhcom{\hcom{p}{\justi}}{q}$} (rrdd);
    \draw (lluu) -- node[ed] {\scriptsize$\uvcoms{\vcoms{h}{\justi}}{m}$} (lldd);
    \draw (rruu) -- node[ed] {\scriptsize$\uvcoms{\vcoms{j}{\justi}}{o}$} (rrdd);
    \draw (lluuuh) -- node[ed] {\scriptsize$\uhcom{f}{g}$} (rruuuh);
    \draw (lldddh) -- node[ed] {\scriptsize$\uhcom{p}{q}$} (rrdddh);
    \draw (lluuuh) -- node[ed] {\scriptsize$\justi$} (lluu);
    \draw (rruuuh) -- node[ed] {\scriptsize$\justi$} (rruu);
    \draw (lldd) -- node[ed] {\scriptsize$\justi$} (lldddh);
    \draw (rrdd) -- node[ed] {\scriptsize$\justi$} (rrdddh);
    \draw (lluuu) -- node[ed] {\scriptsize$\uhcom{f}{g}$} (rruuu);
    \draw (llddd) -- node[ed] {\scriptsize$\uhcom{p}{\justi}$} (rrddd);
    \draw (lluuu) -- node[ed] {\scriptsize$\uvcoms{\vcoms{\justi}{\vcoms{\vcoms{h}{\justi}}{m}}}{\justi}$} (llddd);
    \draw (rruuu) -- node[ed] {\scriptsize$\uvcoms{\vcoms{\justi}{\vcoms{\vcoms{j}{\justi}}{o}}}{\justi}$} (rrddd);
    \draw (lllhuuu) -- node[ed] {\scriptsize$\justi$} (lluuu);
    \draw (rruuu) -- node[ed] {\scriptsize$\justi$} (rrrhuuu);
    \draw (lllhddd) -- node[ed] {\scriptsize$\justi$} (llddd);
    \draw (rrddd) -- node[ed] {\scriptsize$\justi$} (rrrhddd);
    \draw (lllhuuu) -- node[ed] {\scriptsize$\uvcom{h}{m}$} (lllhddd);
    \draw (rrrhuuu) -- node[ed] {\scriptsize$\uvcom{j}{o}$} (rrrhddd);
    \draw (llluuu) -- node[above] {\scriptsize$\uhcom{\hcom{\justi}{\hcom{f}{g}}}{\justi}$} (rrruuu);
    \draw (lllddd) -- node[below] {\scriptsize$\uhcom{\hcom{\justi}{\hcom{p}{q}}}{\justi}$} (rrrddd);
    \draw (llluuu) -- node[left] {\scriptsize$\uvcom{h}{m}$} (lllddd);
    \draw (rrruuu) -- node[right] {\scriptsize$\uvcom{j}{o}$} (rrrddd);
  \end{tikzpicture}
\] 

To put it another way, equality between composite squares in the free
cubical bicategory on a double graph is checked by comparing the
boundaries and the induced grids of generating squares appearing in
the composites. This leads to the following observation.

Recall that $\bDblGph$ denotes the category of double graphs.

\begin{lem}\label{lem:garnermonad}
  The algebras of the monad on \bDblGph induced by the forgetful
  functor $\bWDblCatst \to \bDblGph$ are precisely cubical
  bicategories.
\end{lem}
\begin{proof}
  The above definition is obtained from the characterization in
  \cref{cor:freedescription} of the free doubly weak double category
  on a double computad, specialized to the case of double graphs. (The
  2-cells in free \emph{strict} double categories on double graphs are
  given by grids of squares; this is well-known and also follows as a
  special case of the description of free strict double categories in
  the proof of \cref{lem:veritymonad}.)
\end{proof}

\begin{prop}
  The forgetful functor $\bWDblCatst \to \bDblGph$ is not
  monadic. (That is to say, doubly weak double categories are distinct
  from cubical bicategories.)
\end{prop}
\begin{proof}
  By \cref{lem:garnermonad}, it suffices to exhibit a cubical
  bicategory that does not arise from any doubly weak double category.
  In a doubly weak double category, there is a bijection between
  2-cells of shapes
  
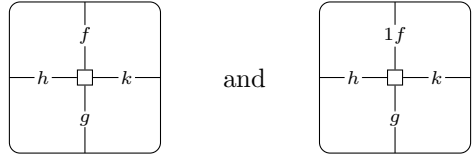
\begin{center}
    \begin{tikzpicture}[scale=0.1666]
      \draw [rect] (-6,-6) rectangle (6,6);
      \node (a) [ivs] at (0,0) {};
      \draw (-6,0) node [ov] {} -- node [ed] {$h$} (a) -- node [ed] {$k$} (6,0) node [ov] {};
      \draw (0,-6) node [ov] {} -- node [ed] {$g$} (a) -- node [ed] {$f$} (0,6) node [ov] {};
    \end{tikzpicture}
    \qquad and\qquad
    \begin{tikzpicture}[scale=0.1666]
      \draw [rect] (-6,-6) rectangle (6,6);
      \node (a) [ivs] at (0,0) {};
      \draw (-6,0) node [ov] {} -- node [ed] {$h$} (a) -- node [ed] {$k$} (6,0) node [ov] {};
      \draw (0,-6) node [ov] {} -- node [ed] {$g$} (a) -- node [ed] {$\udcom{\justi}{f}$} (0,6) node [ov] {};
    \end{tikzpicture}
  \end{center}
  obtained by composing on the top with the unitor isomorphism
  
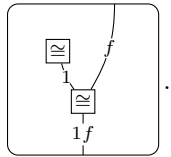
\begin{center}
    \begin{tikzpicture}[xscale=.1666,yscale=-.1666]
      \draw [rect] (-6,-6) rectangle (6,6);
      \node [ov] (v) at (2.5, -6) {};
      \node [ivs, inner sep=0] (a) at (-2, -2.25) {$\choseniso$};
      \node [ivs, inner sep=0] (c) at (0, 1.75) {$\choseniso$};
      \node [ov] (y) at (0, 6) {};
      \draw (v) [bend right=15] to node [ed] {$f$} (c);
      \draw (a) [bend left=10] to node [ed] {$\justi$} (c);
      \draw (c) -- node [ed] {$\ubcom{\justi}{f}$} (y);
    \end{tikzpicture}\,.
  \end{center}
  We now construct a cubical bicategory without this property. Given
  any commutative monoid $M$ with identity $0_M$, let the cubical
  bicategory $\cat{C}_M$ have one 0-cell, and let the horizontal and
  vertical 1-cells both be freely generated, i.e.\ given by bracketed
  strings of $1$. Let there be one 2-cell bordered on all sides by
  $1$, which we label $0_M$
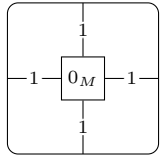
  \begin{center}
    \begin{tikzpicture}[scale=0.1666]
      \draw [rect] (-6,-6) rectangle (6,6);
      \node (a) [ivs, inner sep=0] at (0,0) {$0_M$};
      \draw (-6,0) node [ov] {} -- node [ed] {$1$} (a) -- node [ed] {$1$} (6,0) node [ov] {};
      \draw (0,-6) node [ov] {} -- node [ed] {$1$} (a) -- node [ed] {$1$} (0,6) node [ov] {};
    \end{tikzpicture}
  \end{center}
  and let the 2-cells having any other particular boundary be
  identified with $M$. The composite of any grid of 2-cells will be
  given by simply adding up the elements of $M$ occurring in it.
  
  Now if $M$ is nontrivial, then in $\cat{C}_M$ there is no bijection
  between 2-cells of shapes
  
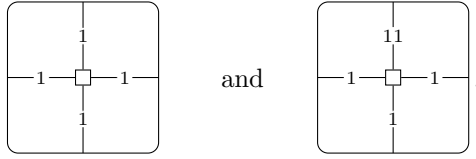
\begin{center}
    \begin{tikzpicture}[scale=0.1666]
      \draw [rect] (-6,-6) rectangle (6,6);
      \node (a) [ivs] at (0,0) {};
      \draw (-6,0) node [ov] {} -- node [ed] {$1$} (a) -- node [ed] {$1$} (6,0) node [ov] {};
      \draw (0,-6) node [ov] {} -- node [ed] {$1$} (a) -- node [ed] {$1$} (0,6) node [ov] {};
    \end{tikzpicture}
    \qquad and\qquad
    \begin{tikzpicture}[scale=0.1666]
      \draw [rect] (-6,-6) rectangle (6,6);
      \node (a) [ivs] at (0,0) {};
      \draw (-6,0) node [ov] {} -- node [ed] {$1$} (a) -- node [ed] {$1$} (6,0) node [ov] {};
      \draw (0,-6) node [ov] {} -- node [ed] {$1$} (a) -- node [ed] {$\udcom{1}{1}$} (0,6) node [ov] {};
    \end{tikzpicture}\,.
  \end{center}
  Hence $\cat{C}_M$ cannot arise from any doubly weak double category.
\end{proof}

However, \cref{lem:garnermonad} does also give us:

\begin{cor}\label{cor:cubical-cmp}
  There is a canonical functor from doubly weak double categories to cubical bicategories.
\end{cor}
\begin{proof}
  This is the standard comparison functor from the domain of any right
  adjoint to the category of algebras for the monad induced by the
  adjunction.
\end{proof}

We now show that, as was the case for double bicategories, this
comparison functor is fully faithful, and we characterize the
image. (It is possible to quickly see that the comparison functor is
fully faithful using \cref{defn:tidier}, but it will take us some
additional work to establish the following simple characterization of
the image.)

\begin{defn}\label{defn:tidycubical}
  A \textbf{tidy cubical bicategory} is a cubical bicategory such that
  the canonical map induced by composing with an identity square (in
  any of the four directions)
  
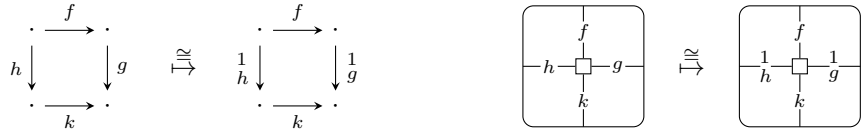
\[
    \begin{tikzpicture}[->,xscale=.5,yscale=.5]
      \node (au) at (-1,1) {$\cdot$};
      \node (bu) at (1,1) {$\cdot$};
      \node (ad) at (-1,-1) {$\cdot$};
      \node (bd) at (1,-1) {$\cdot$};
      \draw (au) -- node [above] {\scriptsize$f$} (bu);
      \draw (ad) -- node [below] {\scriptsize$k$} (bd);
      \draw (au) -- node [left] {\scriptsize$h$} (ad);
      \draw (bu) -- node [right] {\scriptsize$g$} (bd);
    \end{tikzpicture}
    \quad\overset{\cong}{\mapsto}\quad
    \begin{tikzpicture}[->,xscale=.5,yscale=.5]
      \node (au) at (-1,1) {$\cdot$};
      \node (bu) at (1,1) {$\cdot$};
      \node (ad) at (-1,-1) {$\cdot$};
      \node (bd) at (1,-1) {$\cdot$};
      \draw (au) -- node [above] {\scriptsize$f$} (bu);
      \draw (ad) -- node [below] {\scriptsize$k$} (bd);
      \draw (au) -- node [left=2] {\scriptsize$\uvcom{\justi}{h}$} (ad);
      \draw (bu) -- node [right=2] {\scriptsize$\uvcom{\justi}{g}$} (bd);
    \end{tikzpicture}
    \qquad\qquad\qquad
    \begin{tikzpicture}[scale=0.4]
      \draw [rect] (-2,-2) rectangle (2,2);
      \node (a) [ivs] at (0,0) {};
      \draw (-2,0) node [ov] {} -- node [ed] {$h$} (a) -- node [ed] {$g$} (2,0) node [ov] {};
      \draw (0,2) node [ov] {} -- node [ed] {$f$} (a) -- node [ed] {$k$} (0,-2) node [ov] {};
    \end{tikzpicture}
    \quad\overset{\cong}{\mapsto}\quad
    \begin{tikzpicture}[scale=0.4]
      \draw [rect] (-2,-2) rectangle (2,2);
      \node (a) [ivs] at (0,0) {};
      \draw (-2,0) node [ov] {} -- node [ed] {$\uvcom{\justi}{h}$} (a) -- node [ed] {$\uvcom{\justi}{g}$} (2,0) node [ov] {};
      \draw (0,2) node [ov] {} -- node [ed] {$f$} (a) -- node [ed] {$k$} (0,-2) node [ov] {};
    \end{tikzpicture}
  \]
  is a bijection, per boundary.  In terms of operations and laws, this
  means a tidy cubical bicategory is additionally equipped with four
  conversion operations, defined on squares of forms
  \[
    \begin{tikzpicture}[->,xscale=.5,yscale=.5]
      \node (au) at (-1,1) {$\cdot$};
      \node (bu) at (1,1) {$\cdot$};
      \node (ad) at (-1,-1) {$\cdot$};
      \node (bd) at (1,-1) {$\cdot$};
      \draw (au) -- node [above] {\scriptsize$f$} (bu);
      \draw (ad) -- node [below] {\scriptsize$k$} (bd);
      \draw (au) -- node [left=2] {\scriptsize$\uvcom{\justi}{h}$} (ad);
      \draw (bu) -- node [right=2] {\scriptsize$\uvcom{\justi}{g}$} (bd);
    \end{tikzpicture}
    \qquad
    \begin{tikzpicture}[->,xscale=.5,yscale=.5]
      \node (au) at (-1,1) {$\cdot$};
      \node (bu) at (1,1) {$\cdot$};
      \node (ad) at (-1,-1) {$\cdot$};
      \node (bd) at (1,-1) {$\cdot$};
      \draw (au) -- node [above] {\scriptsize$f$} (bu);
      \draw (ad) -- node [below] {\scriptsize$k$} (bd);
      \draw (au) -- node [left=2] {\scriptsize$\uvcom{h}{\justi}$} (ad);
      \draw (bu) -- node [right=2] {\scriptsize$\uvcom{g}{\justi}$} (bd);
    \end{tikzpicture}
    \qquad
    \begin{tikzpicture}[->,xscale=.5,yscale=.5]
      \node (au) at (-1,1) {$\cdot$};
      \node (bu) at (1,1) {$\cdot$};
      \node (ad) at (-1,-1) {$\cdot$};
      \node (bd) at (1,-1) {$\cdot$};
      \draw (au) -- node [above] {\scriptsize$\ubcom{\justi}{f}$} (bu);
      \draw (ad) -- node [below] {\scriptsize$\ubcom{\justi}{k}$} (bd);
      \draw (au) -- node [left=2] {\scriptsize$h$} (ad);
      \draw (bu) -- node [right=2] {\scriptsize$g$} (bd);
    \end{tikzpicture}
    \qquad
    \begin{tikzpicture}[->,xscale=.5,yscale=.5]
      \node (au) at (-1,1) {$\cdot$};
      \node (bu) at (1,1) {$\cdot$};
      \node (ad) at (-1,-1) {$\cdot$};
      \node (bd) at (1,-1) {$\cdot$};
      \draw (au) -- node [above] {\scriptsize$\ubcom{f}{\justi}$} (bu);
      \draw (ad) -- node [below] {\scriptsize$\ubcom{k}{\justi}$} (bd);
      \draw (au) -- node [left=2] {\scriptsize$h$} (ad);
      \draw (bu) -- node [right=2] {\scriptsize$g$} (bd);
    \end{tikzpicture}
  \]
  satisfying laws that ensure these are sent to squares of the form
  \[
    \begin{tikzpicture}[->,xscale=.5,yscale=.5]
      \node (au) at (-1,1) {$\cdot$};
      \node (bu) at (1,1) {$\cdot$};
      \node (ad) at (-1,-1) {$\cdot$};
      \node (bd) at (1,-1) {$\cdot$};
      \draw (au) -- node [above] {\scriptsize$f$} (bu);
      \draw (ad) -- node [below] {\scriptsize$k$} (bd);
      \draw (au) -- node [left=2] {\scriptsize$h$} (ad);
      \draw (bu) -- node [right=2] {\scriptsize$g$} (bd);
    \end{tikzpicture}
  \]
  and that these operations are inverse to composing with identities.
\end{defn}

\begin{prop}
  The comparison functor of \cref{cor:cubical-cmp} is an equivalence
  onto the full subcategory of tidy cubical bicategories.
\end{prop}
\begin{proof}
  Suppose given a tidy cubical bicategory.
  We will construct a tidy double bicategory using the squares-only
  definition from \cref{defn:tidier}. That is, we require a double
  graph equipped with binary composition and identity operations, such
  that the canonical maps induced by composing with identities are
  bijections per boundary, and the squares and ``bigons'' (squares
  bordered appropriately by identities) with the induced operations
  have the structure of a double bicategory.

  Any cubical bicategory has an underlying double graph with binary
  composition operations and identities (among other more general
  composition operations). In particular, an identity square for (say)
  vertical composition is obtained by composing a $0 \times 1$ grid
  using single identity 1-cells as the composites of the nullary left
  and right boundaries. A \emph{tidy} cubical bicategory moreover by
  definition has the same identity square cancellation condition of
  \cref{defn:tidier}.

  As in \cref{defn:tidier}, we define horizontal (vertical) bigons to
  be squares bordered by vertical (horizontal) identity 1-cells, and
  we define the bigon-on-square and bigon-on-bigon composition
  operations of a double bicategory by composing squares then applying
  the given identity square cancellation bijection. We display this again
  here for convenience:
  \[
    \begin{tikzpicture}[->,xscale=.5,yscale=.5,rotate=-90]
      \node (au) at (-2,1) {$\cdot$};
      \node (bu) at (0,1) {$\cdot$};
      \node (cu) at (2,1) {$\cdot$};
      \node (ad) at (-2,-1) {$\cdot$};
      \node (bd) at (0,-1) {$\cdot$};
      \node (cd) at (2,-1) {$\cdot$};
      \node at (-1,0) {\small$\alpha$};
      \node at (1,0) {\small$\zeta$};
      \draw (au) -- node[right] {\scriptsize$\justi$} (bu);
      \draw (bu) -- node[right] {\scriptsize$g$} (cu);
      \draw (ad) -- node[left] {\scriptsize$\justi$} (bd);
      \draw (bd) -- node[left] {\scriptsize$h$} (cd);
      \draw[<-] (au) -- node[above] {\scriptsize$f$} (ad);
      \draw[<-] (bu) -- node[fill=white,inner sep=2] {\scriptsize$x$} (bd);
      \draw[<-] (cu) -- node[below] {\scriptsize$k$} (cd);
    \end{tikzpicture}
    \;\mapsto\;
    \begin{tikzpicture}[->,xscale=.5,yscale=.5,rotate=-90]
      \node (au) at (-2,1) {$\cdot$};
      \node (bu) at (0,1) {$\cdot$};
      \node (ad) at (-2,-1) {$\cdot$};
      \node (bd) at (0,-1) {$\cdot$};
      \node at (-1,0) {\small$\uvcom{\alpha}{\zeta}$};
      \draw (au) -- node[right=2] {\scriptsize$\uvcom{\justi}{g}$} (bu);
      \draw (ad) -- node[left=2] {\scriptsize$\uvcom{\justi}{h}$} (bd);
      \draw[<-] (au) -- node[above] {\scriptsize$f$} (ad);
      \draw[<-] (bu) -- node[below] {\scriptsize$k$} (bd);
    \end{tikzpicture}
    \;\overset{\cong}{\mapsto}\;
    \begin{tikzpicture}[->,xscale=.5,yscale=.5,rotate=-90]
      \node (au) at (-2,1) {$\cdot$};
      \node (bu) at (0,1) {$\cdot$};
      \node (ad) at (-2,-1) {$\cdot$};
      \node (bd) at (0,-1) {$\cdot$};
      \node at (-1,0) {\small$\uvcom{\alpha}{\zeta}$};
      \draw (au) -- node[right] {\scriptsize$g$} (bu);
      \draw (ad) -- node[left] {\scriptsize$h$} (bd);
      \draw[<-] (au) -- node[above] {\scriptsize$f$} (ad);
      \draw[<-] (bu) -- node[below] {\scriptsize$k$} (bd);
    \end{tikzpicture}
  \] 
  
  Now we observe that the structure of a cubical bicategory does
  contain coherence 2-cells bounded by identities, as in the structure
  of a double bicategory.  Any sequence of (say) horizontal 1-cells
  \[
    \begin{tikzcd}[arrow style=tikz]
      \cdot \ar[r,"f_1"] & \cdot \ar[r,"f_2"]  & \cdots \ar[r,"f_{n-1}"] & \cdot \ar[r,"f_n"] & \cdot
    \end{tikzcd}
  \]
  can be regarded as a $0\times n$ grid of composable squares.
  Therefore, given any two ways of bracketing a composite of these
  1-cells (perhaps including insertion of identities), we can take
  those to be the top and bottom composites for this grid, use single
  identity 1-cells as the composites of the nullary left and right
  boundaries, and obtain a coherence 2-cell. We will write all of
  these coherence 2-cells as ``$\choseniso$'', save for the identity
  squares written as ``$\justi$'' (which, observe, are a special case
  of coherence 2-cells), and we often write elongated $=$ signs for
  identity 1-cells.  For instance, here is horizontal associativity:
  \[
    \begin{tikzcd}[arrow style=tikz]
      \cdot \ar[r,"f(gh)"] \ar[d,-,double] \ar[dr,phantom,"\choseniso"] & \cdot \ar[d,-,double] \\
      \cdot \ar[r,"(fg)h"'] & \cdot
    \end{tikzcd}
  \]

  Our discussion at the beginning of this section implies that
  two formal composites, i.e.\ squares in a free cubical bicategory,
  constructed from the same grid of squares are equal if and only if
  they have the same boundary. (By definition, the 2-cells in a free
  cubical bicategory on a double graph are compatible grids of squares
  with bracketed boundaries.)  In particular, any formal composite
  featuring only coherence cells is itself a coherence cell, since
  there is at most one formal composite with any given boundary
  featuring \emph{no} squares.

  We next verify the double bicategory laws. The double-categorical
  interchange laws are automatic from the cubical bicategory
  structure.
  To show the remaining laws, note that in a \emph{tidy} cubical
  bicategory, we have cancellation with respect to composing with
  identities. Therefore one strategy to show an equation between two
  squares is to compose both of them with identities and then to
  express the resulting two squares as formal composites derived from
  the same grid. (Then since we know these squares must be equal, by
  cancellation the original squares are equal.)

  Let us start with the unitor naturality laws. We must show that the
  following compositions with coherence bigons are equal:
  \[

  \]
  since each is a formal composite constructed from the same
  $1 \times 1$ grid $\zeta$. Hence by definition of bigon composition we
  have
  \[
    %
  \]
  and therefore by cancelling identities
  \[
    %
  \]
  The other unitor naturality laws
  are analogous, as well as the associator naturality laws, where we use
  \[
    %
  \]
  constructed from the same $1 \times 3$ grid (and similarly in the
  vertical case, with a $3 \times 1$ grid). We also have that the inverse
  pairs of coherence cells do behave as such:
  \[
    %
  \]
  Similarly the pentagon and triangle laws of a bicategory are satisfied
  because all formal compositions of coherence cells agree, as noted
  above.

  The next law we show is the identity square commutativity law of a
  double bicategory. Observe for any square $\alpha$, we have the
  equations
  \[
    %
  \]
  since both sides of each equation have the same boundary and are
  formal composites constructed from the $1 \times 1$ grid
  $\alpha$. (Of course, whenever we compose a grid, we must choose
  some bracketing of its boundary, but we will omit such
  annotations from our diagrams, trusting the reader to supply
  suitable choices.)

  When $\alpha$ is moreover a \emph{bigon} (bordered on either side by identities), we get
  \[
    %
  \]
  (The composite in the middle agrees the two from above since there
  is a unique coherence cell for any bracketed boundary of a
  $0 \times 0$ grid.) Hence by cancelling the identities on the left
  and right, we obtain
  \[
    %
  \]
  Horizontal identity square commutativity is similar.

  The bigon identity laws are trivial. We also have the associativity laws for
  composing bigons (with squares or bigons):
  \[
    %
  \]
  and the action compatibility laws:

  \[
    %
  \]
  
  Last we have the bigon-square sandwiching laws, associativity laws, and interchange laws:
  \[
    %
  \]

  The reader is warned that the above calculations are somewhat
  subtle, as the visual representations omit detail. Formally, each
  diagram is to be decomposed into particular nested parenthesized
  grids. For example, in the above proof of the bigon-square
  sandwiching law, we start with a $5 \times 2$ grid. Then we
  reinterpret this as a $1 \times 2$ grid nested within the middle of
  a $3 \times 1$ grid nested within the middle of a $3 \times 1$ grid
  (using that both this and the previous composite represent the same
  formal $1 \times 2$ grid). Then we reinterpret the middle
  $3 \times 1$ grid as a $1 \times 3$ grid, and the rest of the
  argument proceeds symmetrically. Also note that in certain cases, a
  correct choice of parenthesization along the boundary allows the
  shown steps to be performed, whereas an incorrect choice of
  parenthesization does not. For example, in the same proof, the
  bracketing $(\justi((\justi\tikz[baseline=-2.5]{ \draw[->] (0,0) -- (.3,0); })\justi))\justi$ for the left and right boundaries
  works, whereas the bracketing $((\justi\justi)\tikz[baseline=-2.5]{ \draw[->] (0,0) -- (.3,0); })(\justi\justi)$ does not.
  
  Finally, we observe that the processes of translation between tidy
  double bicategories and tidy cubical bicategories are inverse. It is
  clear that a tidy double bicategory is recovered from the cubical
  bicategory structure of its underlying doubly weak double category,
  since all the data of \cref{defn:tidier} are included in the
  structure. Conversely, all the structure of a tidy cubical
  bicategory is determined by the underlying tidy double category
  structure, since an arbitrary grid composition operation is obtained
  by binarily composing the grid and acting with coherence isomorphism
  bigons to rebracket the boundary as desired.
\end{proof}

\begin{cor}\label{thm:cub-ff}
  The forgetful functor from doubly weak double categories to
  cubical bicategories is fully faithful.\qed
\end{cor}

\begin{cor}
  The forgetful functor $\bWDblCatst \to \bDblGph$ is faithful and
  conservative.\qed
\end{cor}

Thus we can still regard a doubly weak double category as
``structure'' on an underlying double graph, though that structure is
not monadic.

\begin{rmk}
  Similarly to \cref{rmk:premonadic}, \cref{thm:cub-ff}
  says that the forgetful functor $\bWDblCat \to \bDblGph$ is \emph{of
    descent type} or \emph{premonadic}, and this implies that
  \emph{every doubly weak double category has a canonical presentation
    as a coequalizer of maps between doubly weak double categories that
    are freely generated by double graphs}.
\end{rmk}


\section{A finite axiomatization}\label{sec:finax}

Tidy double bicategories do constitute a finite axiomatization of
doubly weak double categories: they are essentially algebraic
(presenting a finite limit theory) with finitely many types, finitely
many operations, and finitely many equations.

However, they do not share the good property of the infinitary
definition in \cref{sec:doublewords} of being presented as monadic
over a presheaf category in which \emph{pseudofunctors} can also be
represented as presheaf maps.
%
%
%
(A tidy double bicategory requires operations whose domains involve
identity 1-cells; however, identity 1-cells are not strictly preserved
by pseudofunctors.)


We now present another finitary definition, exhibiting doubly weak
double categories as monadic over a presheaf category with domain a
finite subcategory of that of double computads.  The practical use of
this particular presentation is questionable, but the point is to
illustrate that something like it can be done. There are many axioms,
but most of them are adaptations of the axioms for double
bicategories.

A \textbf{monogon} in a double computad is a 2-cell of shape
\td{1}{0}{0}{0}, \td{0}{1}{0}{0}, \td{0}{0}{1}{0}, or \td{0}{0}{0}{1}.
A \textbf{double graph with monogons} is a double computad in which
all 2-cells are monogons or squares.
Let \bMoDblGph denote the category of
double graphs with monogons, a functor category whose domain is a
suitable full subcategory of $\lC_\d$:
\[
  \vcenter{\xymatrix@R=5mm@C=5mm{& \td{0}{0}{1}{0} \ar[dr] & \td{0}{1}{0}{0} \ar[d]\\
      \td{1}{0}{0}{0} \ar[dr] & \td{1}{1}{1}{1} \ar@<1mm>[r]
      \ar@<-1mm>[r] \ar@<-1mm>[d] \ar@<1mm>[d] &
      \tv \ar@<-1mm>[d] \ar@<1mm>[d]\\
      \td{0}{0}{0}{1} \ar[r] & \th \ar@<1mm>[r] \ar@<-1mm>[r] & \tz}}
\]

\begin{defn}\label{def:weakmono}
  A \textbf{weak composition structure} on a double graph with monogons
  consists of the following operations.
  \begin{itemize}
  \item Horizontal and vertical binary composition and identity
    operations for 1-cells and squares, as in a double bicategory.
  \item Four 2-cell composition operations sending two compatible
    squares and two compatible monogons to a square:
    \[

    \]
  \item Four 2-cell composition operations sending a square and three compatible monogons
    to a monogon:
    \[
      %
    \]
  \item A 2-cell composition operation sending four compatible
    monogons (one of each type) to a square:
    \[
      %
    \]
  \item Operations sending horizontal 1-cells to left and right unitor
    squares and their inverses, and likewise for vertical 1-cells:
    \[
      %
    \]
  \item Operations sending length three paths of horizontal 1-cells to
    associator squares and their inverses, and likewise for vertical
    1-cells:
    \[
      %
    \]
  \item Operations sending 0-cells to horizontal and vertical identity
    composition monogons and their inverses:
    \[
      %
    \]
  \end{itemize}
  Moreover, these operations must satisfy the following laws.
  \begin{itemize}
  \item Source and target laws for horizontal and vertical identity
    and binary composition operations of 1-cells and squares, as in a
    double bicategory.
  \item Source and target laws for unitor and associator squares,
    2-cell composition operations, and identity composition monogons,
    as appropriate.
  \item Identity laws:
    \[
      %
    \]
  \item Associativity laws that say the three possible ways of composing
    each of the following diagram shapes are equal:
    \[
      %
    \]
  \item Horizontal unitor and associator invertibility laws:
    \[
      %
    \]
    Likewise, analogous laws for vertical unitors and associators.
  \item
    Horizontal unitor and associator naturality laws:
    \[
      %
    \]
    Likewise, analogous laws for vertical unitors and associators.
  \item Horizontal bicategory triangle and pentagon laws:
    \[
      %
    \]
    (By the associativity laws above, we can use any of the three
    possible ways to compose the right hand side of the pentagon
    equation. Here and elsewhere we do not annotate how each diagram
    is built up from the basic composition operations, trusting the
    reader to compose the diagrams up in a suitable way.)
    
    Likewise, analogous vertical bicategory pentagon and triangle
    laws.
  \item The square interchange laws as in a double category (the
    identity compatibility law, the identity interchange laws, and the
    square composition interchange law).
  \item Interchange laws involving monogons and horizontal composition of squares:
    \[
      \bcom{\;

        \;}
    \]
    Likewise, the three other analogous (rotated) interchange laws.
  \item A law ensuring that identity composition monogons correspond
    to identities:
    \[
      %
    \]
  \item Associativity laws that say the two possible ways of composing
    each of the following diagram shapes are equal:
    \[
      %
    \]
  \item Identity commutativity laws:
    \[
      %
    \]
    (By the associativity laws above, we can use either of the two
    possible ways to compose the middle diagram.)

    Likewise, analogous (rotated) laws for horizontal identities.
  \item Associativity laws for sandwiching monogons between squares:
    \[
      \bcom{\;
        %
        \;}
    \]
    Likewise, the other analogous (rotated) associativity law for
    vertical composition.
  \item Associativity laws for squares composed beside monogons:
    \[
      \bcom{\;
        %
    \]
    Likewise, the three other analogous (rotated) associativity laws.
  \item
    Associativity laws that say the two possible ways of composing
    each of the following diagram shapes are equal:
    \[
      %
    \]
  \item Laws ensuring that the canonical map from monogons to squares
    is undone by the canonical map from appropriately degenerate
    squares to monogons:
    \[
      %
    \]
    Likewise, three other analogous (rotated) laws.
  \end{itemize}
\end{defn}

Any doubly weak double category has an underlying double graph with
monogons, equipped with weak composition structure. Conversely, we
have the following.

\begin{prop}
  Any double graph with monogons having a weak composition structure $\cat{X}$ has an underlying
  tidy double bicategory:
  \begin{itemize}
  \item The 0-cells, 1-cells, squares, 1-cell identities and
    composition, and square identities and composition are as in
    $\cat{X}$.
  \item The horizontal bigons are the squares in $\cat{X}$ bordered by
    vertical identities. The vertical bigons are the squares in
    $\cat{X}$ bordered by horizontal ientities.
  \item Horizontal compostion of horizontal bigons, horizontal
    unitors, and horizontal associators are as in $\cat{X}$. The top
    and bottom actions of horizontal bigons $\alpha$ on squares
    $\zeta$ are defined as
    
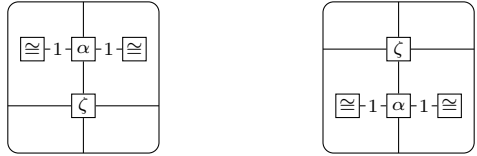
\[
      \begin{tikzpicture}[xscale=0.5,yscale=-0.5,rotate=90]
        \draw [rect] (-2,-2) rectangle (2,2);
        \node (a) [ivs, inner sep=0,minimum width=12.5] at (-.75,0) {$\alpha$};
        \node (b) [ivs, inner sep=0,minimum width=12.5] at (.75,0) {$\zeta$};
        \node (c) [ivs, inner sep=0,minimum width=12.5] at (-.75,1.35) {$\choseniso$};
        \node (d) [ivs, inner sep=0,minimum width=12.5] at (-.75,-1.35) {$\choseniso$};
        \draw (.75,2) -- (b) -- (.75,-2);
        \draw (c) -- node [ed] {$\justi$} (a) -- node [ed] {$\justi$} (d);
        \draw (-2,0) -- (a) -- (b) -- (2,0);
      \end{tikzpicture}
      \qquad\qquad\qquad
      \begin{tikzpicture}[xscale=0.5,yscale=-0.5,rotate=90]
        \draw [rect] (-2,-2) rectangle (2,2);
        \node (a) [ivs, inner sep=0,minimum width=12.5] at (-.75,0) {$\zeta$};
        \node (b) [ivs, inner sep=0,minimum width=12.5] at (.75,0) {$\alpha$};
        \node (c) [ivs, inner sep=0,minimum width=12.5] at (.75,1.35) {$\choseniso$};
        \node (d) [ivs, inner sep=0,minimum width=12.5] at (.75,-1.35) {$\choseniso$};
        \draw (c) -- node [ed] {$\justi$} (b) -- node [ed] {$\justi$} (d);
        \draw (-.75,2) -- (a) -- (-.75,-2);
        \draw (-2,0) -- (a) -- (b) -- (2,0);
      \end{tikzpicture}
    \]
    and vertical composition of bigons $\alpha$ and $\beta$ is defined as
    
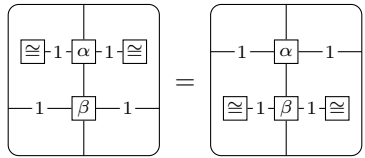
\[
      \begin{tikzpicture}[xscale=0.5,yscale=-0.5,rotate=90]
        \draw [rect] (-2,-2) rectangle (2,2);
        \node (a) [ivs, inner sep=0,minimum width=12.5] at (-.75,0) {$\alpha$};
        \node (b) [ivs, inner sep=0,minimum width=12.5] at (.75,0) {$\beta$};
        \node (c) [ivs, inner sep=0,minimum width=12.5] at (-.75,1.35) {$\choseniso$};
        \node (d) [ivs, inner sep=0,minimum width=12.5] at (-.75,-1.35) {$\choseniso$};
        \draw (.75,2) -- node [ed] {$\justi$} (b) -- node [ed] {$\justi$} (.75,-2);
        \draw (c) -- node [ed] {$\justi$} (a) -- node [ed] {$\justi$} (d);
        \draw (-2,0) -- (a) -- (b) -- (2,0);
      \end{tikzpicture}
      \;=\;
      \begin{tikzpicture}[xscale=0.5,yscale=-0.5,rotate=90]
        \draw [rect] (-2,-2) rectangle (2,2);
        \node (a) [ivs, inner sep=0,minimum width=12.5] at (-.75,0) {$\alpha$};
        \node (b) [ivs, inner sep=0,minimum width=12.5] at (.75,0) {$\beta$};
        \node (c) [ivs, inner sep=0,minimum width=12.5] at (.75,1.35) {$\choseniso$};
        \node (d) [ivs, inner sep=0,minimum width=12.5] at (.75,-1.35) {$\choseniso$};
        \draw (c) -- node [ed] {$\justi$} (b) -- node [ed] {$\justi$} (d);
        \draw (-.75,2) -- node [ed] {$\justi$} (a) -- node [ed] {$\justi$} (-.75,-2);
        \draw (-2,0) -- (a) -- (b) -- (2,0);
      \end{tikzpicture}
    \]
    Similarly for the vertical bicategory.
  \end{itemize}
\end{prop}
\begin{proof}
  Notice that the two ways of defining vertical composition of bigons
  do in fact agree, using identity laws, identity composition
  monogons, and identity commutativity:
  
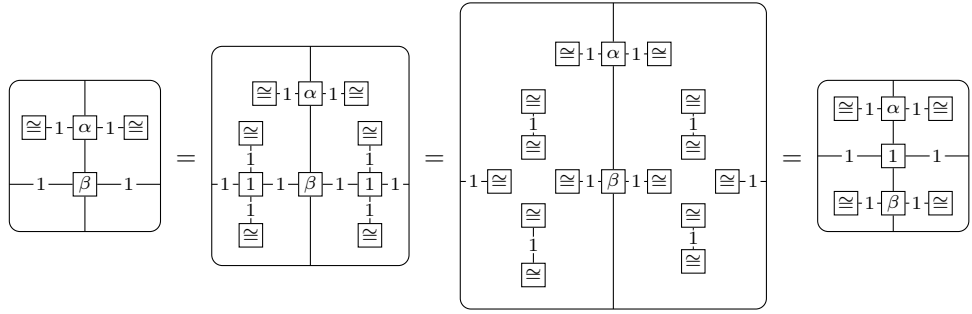
\[
    \begin{tikzpicture}[xscale=0.5,yscale=-0.5,rotate=90]
      \draw [rect] (-2,-2) rectangle (2,2);
      \node (a) [ivs, inner sep=0,minimum width=12.5] at (-.75,0) {$\alpha$};
      \node (b) [ivs, inner sep=0,minimum width=12.5] at (.75,0) {$\beta$};
      \node (c) [ivs, inner sep=0,minimum width=12.5] at (-.75,1.35) {$\choseniso$};
      \node (d) [ivs, inner sep=0,minimum width=12.5] at (-.75,-1.35) {$\choseniso$};
      \draw (.75,2) -- node [ed] {$\justi$} (b) -- node [ed] {$\justi$} (.75,-2);
      \draw (c) -- node [ed] {$\justi$} (a) -- node [ed] {$\justi$} (d);
      \draw (-2,0) -- (a) -- (b) -- (2,0);
    \end{tikzpicture}
    \;=\;
    \begin{tikzpicture}[xscale=0.45,yscale=-0.5,rotate=90]
      \draw [rect] (-2.9,-2.9) rectangle (2.9,2.9);
      \node (a) [ivs, inner sep=0,minimum width=12.5] at (-1.65,0) {$\alpha$};
      \node (b) [ivs, inner sep=0,minimum width=12.5] at (.75,0) {$\beta$};
      \node (c) [ivs, inner sep=0,minimum width=12.5] at (-1.65,1.35) {$\choseniso$};
      \node (d) [ivs, inner sep=0,minimum width=12.5] at (-1.65,-1.35) {$\choseniso$};
      \node (e) [ivs, inner sep=0,minimum width=12.5] at (.75,1.75) {$\justi$};
      \node (f) [ivs, inner sep=0,minimum width=12.5] at (.75,-1.75) {$\justi$};
      \node (el) [ivs, inner sep=0,minimum width=12.5] at (-.6,1.75) {$\choseniso$};
      \node (fl) [ivs, inner sep=0,minimum width=12.5] at (-.6,-1.75) {$\choseniso$};
      \node (er) [ivs, inner sep=0,minimum width=12.5] at (2.1,1.75) {$\choseniso$};
      \node (fr) [ivs, inner sep=0,minimum width=12.5] at (2.1,-1.75) {$\choseniso$};
      \draw (.75,2.9) -- node [ed] {$\justi$} (e) -- node [ed] {$\justi$} (b) -- node [ed] {$\justi$} (f) -- node [ed] {$\justi$} (.75,-2.9);
      \draw (c) -- node [ed] {$\justi$} (a) -- node [ed] {$\justi$} (d);
      \draw (el) -- node [ed] {$\justi$} (e) -- node [ed] {$\justi$} (er);
      \draw (fl) -- node [ed] {$\justi$} (f) -- node [ed] {$\justi$} (fr);
      \draw (-2.9,0) -- (a) -- (b) -- (2.9,0);
    \end{tikzpicture}
    \;=\;
    \begin{tikzpicture}[scale=0.45,rotate=90]
      \draw [rect] (-4.5,-4.5) rectangle (4.5,4.5);
      \node (ld) [ivs, inner sep=0,minimum width=12.5] at (-.75,-1.35) {$\choseniso$};
      \node (rd) [ivs, inner sep=0,minimum width=12.5] at (3,-1.35) {$\choseniso$};
      \node (dl) [ivs, inner sep=0,minimum width=12.5] at (-1.75,-2.35) {$\choseniso$};
      \node (dr) [ivs, inner sep=0,minimum width=12.5] at (.25,-2.35) {$\choseniso$};
      \node (dll) [ivs, inner sep=0,minimum width=12.5] at (-3.1,-2.35) {$\choseniso$};
      \node (drr) [ivs, inner sep=0,minimum width=12.5] at (1.6,-2.35) {$\choseniso$};
      \node (dd) [ivs, inner sep=0,minimum width=12.5] at (-.75,-3.35) {$\choseniso$};
      \node (lu) [ivs, inner sep=0,minimum width=12.5] at (-.75,1.35) {$\choseniso$};
      \node (ru) [ivs, inner sep=0,minimum width=12.5] at (3,1.35) {$\choseniso$};
      \node (ul) [ivs, inner sep=0,minimum width=12.5] at (-1.75,2.35) {$\choseniso$};
      \node (ur) [ivs, inner sep=0,minimum width=12.5] at (.25,2.35) {$\choseniso$};
      \node (ull) [ivs, inner sep=0,minimum width=12.5] at (-3.5,2.35) {$\choseniso$};
      \node (urr) [ivs, inner sep=0,minimum width=12.5] at (1.6,2.35) {$\choseniso$};
      \node (uu) [ivs, inner sep=0,minimum width=12.5] at (-.75,3.35) {$\choseniso$};
      \node (lm) [ivs, inner sep=0,minimum width=12.5] at (-.75,0) {$\beta$};
      \node (rm) [ivs, inner sep=0,minimum width=12.5] at (3,0) {$\alpha$};
      \draw (dl) -- node [ed] {$1$} (dll);
      \draw (dr) -- node [ed] {$1$} (drr);
      \draw (dd) -- node [ed] {$1$} (-.75,-4.5);
      \draw (ul) -- node [ed] {$1$} (ull);
      \draw (ur) -- node [ed] {$1$} (urr);
      \draw (uu) -- node [ed] {$1$} (-.75,4.5);
      \draw (lu) -- node [ed] {$1$} (lm) -- node [ed] {$1$} (ld);
      \draw (ru) -- node [ed] {$1$} (rm) -- node [ed] {$1$} (rd);
      \draw (-4.5,0) -- (lm) -- (rm) -- (4.5,0);
    \end{tikzpicture}
    \;=\;
    \begin{tikzpicture}[scale=0.5,rotate=-90]
      \draw [rect] (-2,-2) rectangle (2,2);
      \node (l) [ivs, inner sep=0,minimum width=12.5] at (-1.25,0) {$\alpha$};
      \node (m) [ivs, inner sep=0,minimum width=12.5] at (0,0) {$\justi$};
      \node (r) [ivs, inner sep=0,minimum width=12.5] at (1.25,0) {$\beta$};
      \node (lu) [ivs, inner sep=0,minimum width=12.5] at (-1.25,1.25) {$\choseniso$};
      \node (ld) [ivs, inner sep=0,minimum width=12.5] at (-1.25,-1.25) {$\choseniso$};
      \node (ru) [ivs, inner sep=0,minimum width=12.5] at (1.25,1.25) {$\choseniso$};
      \node (rd) [ivs, inner sep=0,minimum width=12.5] at (1.25,-1.25) {$\choseniso$};
      \draw (lu) -- node [ed] {$\justi$} (l) -- node [ed] {$\justi$} (ld);
      \draw (ru) -- node [ed] {$\justi$} (r) -- node [ed] {$\justi$} (rd);
      \draw (0,2) -- node [ed] {$\justi$} (m) --  node [ed] {$\justi$}(0,-2);
      \draw (-2,0) -- (l) -- (m) -- (r) -- (2,0);
    \end{tikzpicture}
  \]
  
  The interchange laws for the bicategories and for the
  bigon-on-square actions come straightforwardly from the monogon
  interchange laws, identity composition monogon law, and identity
  laws.

  All other laws of a double bicategory correspond directly to laws of
  weak composition structure.
\end{proof}

\begin{prop}
  The category of double graphs with monogons equipped with weak
  composition structure (and homomorphisms) is equivalent to the
  category of doubly weak double categories (and strict functors)
  $\bWDblCatst$.
\end{prop}
\begin{proof}
  First, observe that the canonical maps between squares
  bordered by identities on three sides and monogons
  \[

  \]
  (with the other direction stipulated as a law in the definition).

  Now all of the operations are determined by the double bicategory
  structure. Indeed, we have
  \[
    %
  \]
  The right hand sides can be interpreted in the tidy double
  bicategory. It follows that so can
  \[
    %
  \]
\end{proof}

\begin{cor}
  The forgetful functor $\bWDblCatst \to \bMoDblGph$ is monadic.\qed
\end{cor}

\begin{rmk}
  Alternatively, the 4-ary monogon to square operation could be
  replaced without difficulty by two operations that send two monogons
  and a compatible 1-cell to a square:
  \[
    %
  \]
  
  We also note that two of the operations combining two squares and
  two monogons can be derived from the others, e.g.:
  \[
    %
      \;}\,.
  \]
\end{rmk}

\begin{rmk}
  Less minimal than squares and monogons, but perhaps more natural, is
  the full subcategory $\lE \into \lC_\d$ including $\tz$, $\th$,
  $\tv$, and $\td{a}{b}{c}{d}$ for all $a, b, c, d \leq 1$, so that
  $[\lE, \bSet]$ gives the ``subunary'' double computads. An
  axiomatization for doubly weak double categories presenting a monad
  on $[\lE, \bSet]$ could presumably be given involving a large number
  of binary 2-cell composition operations, removing the need for the
  unusual 4-ary operations we have given.  As a middle ground, one
  could also give a definition using monogons, bigons, and squares
  (involving both binary and ternary 2-cell composition operations).

  It is tempting to conjecture that the forgetful functor
  $\bWDblCatst \to [\lE, \bSet]$ will be monadic when $\lE$ is any
  full subcategory of $\lC_\d$ including the 0-cells, 1-cells,
  monogons, and squares. However, this appears not to be true:
  consider the case where $\lE$ consists of only these and the 2-cell
  shape $\td{2}{0}{0}{0}$.
\end{rmk}


\appendix
\crefalias{section}{appendix}

\section{Transformations and modifications}\label{sec:twowords-hom}

In this section we discuss transformations and modifications of
{\adjective} structures. We will see that when $\cat{C}$ and $\cat{D}$
are {\twowords}, we obtain {\aan} {\twoword} $\Hom(\cat{C}, \cat{D})$;
in the case $\cat{C}$ and $\cat{D}$ are representable, this is the
usual bicategory of transformations and modifications of
bicategories. More than this, we will see that the \emph{lax and
  colax} transformations of {\twowords} (resp.\ bicategories) assemble
into {\aan} \emph{\doubleword} (resp.\ \emph{doubly weak double
  category}) $\Homl(\cat{C}, \cat{D})$, providing a natural source of
examples of doubly weak double categories.

It is also true that when $\cat{C}$ and $\cat{D}$ are {\doublewords}
(resp.\ doubly weak double categories), we have {\aan} {\doubleword}
(resp.\ doubly weak double category) $\Hom(\cat{C},
\cat{D})$. However, we will focus on the 2-categorical case. This is
for reasons of space and also because we are unable to provide
motivation for studying transformations and modifications of doubly
weak double categories (we have no examples). Still, all of the
definitions in this section readily generalize to double-categorical
analogues.


To figure out what the content of $\Hom(\cat{C}, \cat{D})$ ought to
be, recall the defining property of an internal hom: it is universal
such that $\cat{C} \otimes \Hom(\cat{C}, \cat{D})$ maps into
$\cat{D}$. However, this leaves us to wonder what the monoidal product
$\otimes$ ought to be. In ordinary 2-category theory, the relevant
monoidal product is the \emph{Gray tensor product}~\cite{gray:formal},
which composes 2-categories as if they were the homs in a semistrict
tricategory (so that closure for $\otimes$ induces a semistrict
tricategory of 2-categories).

This composition can be represented very cleanly using string diagrams, as described in~\cite{morehouse:two}.
Namely, a string diagram for $\cat{C} \otimes \cat{D}$ consists of a
string diagram for $\cat{C}$ superimposed over a string diagram for
$\cat{D}$.
For example, diagrams in
$\cat{C} \cong \Hom(\cat{1}, \cat{C})$ can be composed with
diagrams in $\Hom(\cat{C}, \cat{D})$ to yield diagrams in
$\cat{D} \cong \Hom(\cat{1}, \cat{D})$:

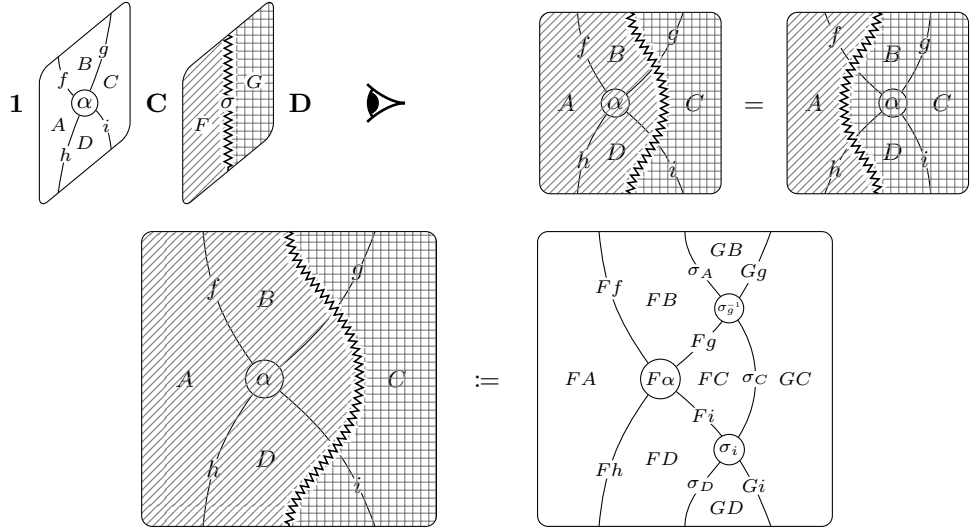
\[\cat{1} \;\;
  \begin{tikzpicture}[xscale=.1,yscale=.15,yslant=.55]
    \draw [rect] (-6,-6) rectangle (6,6);
    \node [ov] (u) at (-3.5, 6) {};
    \node [ov] (v) at (3.5, 6) {};
    \node [iv] (a) at (0, 0) {\small$\alpha$};
    \node [ov] (x) at (-3.5, -6) {};
    \node [ov] (y) at (3.5, -6) {};
    \node [ed] at (-3.5, 0) {$A$};
    \node [ed] at (0, 3.5) {$B$};
    \node [ed] at (3.5, 0) {$C$};
    \node [ed] at (0, -3.5) {$D$};
    \draw (u) [bend right=15] to node [ed] {$f$} (a);
    \draw (v) [bend left=15] to node [ed] {$g$} (a);
    \draw (a) [bend right=15] to node [ed] {$h$} (x);
    \draw (a) [bend left=15] to node [ed] {$i$} (y);
  \end{tikzpicture}
  \;\; \cat{C} \;\;
  \begin{tikzpicture}[xscale=.1,yscale=.15,yslant=.55]
    \begin{scope}
      \clip [rect] (-6,-6) rectangle (6,6);
      \fill [fun, pattern=north east lines](-6,-6) rectangle (0,6);
      \fill [fun, pattern=grid](0,-6) rectangle (6,6);
      \node [ed] at (-3.5, 0) {$F$};
      \node [ed] at (3.5, 0) {$G$};
      \node [ov] (u) at (0, 6) {};
      \node [ov] (x) at (0, -6) {};
      \draw [nat] (u) to node [nated] {\small$\sigma$} (x);
    \end{scope}
    \draw [rect] (-6,-6) rectangle (6,6);
  \end{tikzpicture}
  \;\; \cat{D} \qquad
  \begin{tikzpicture}
    \begin{scope}[scale=.25,line width=.35mm]
      \draw (-1, 1) [out=-45, in=180,looseness=.75] to (1, 0) [out=180, in=45,looseness=.75] to (-1, -1);
      \clip (-1, 1) [out=-45, in=180,looseness=.75] to (1, 0)[out=180, in=45,looseness=.75] to (-1, -1) [bend left=10] to cycle;
      \draw (.3,0) circle (1.25);
      \clip (.3,0) circle (1.25);
      \draw [fill=black] (-1,0) circle (.55);
    \end{scope}
  \end{tikzpicture}
  \qquad\qquad\quad
  \begin{tikzpicture}[scale=.2]
    \begin{scope}
      \clip [rect] (-6,-6) rectangle (6,6);
      \node [ov] (u) at (-3.5, 6) {};
      \node [ov] (v) at (3.5, 6) {};
      \node [iv] (a) at (-1, 0) {\normalsize$\alpha$};
      \node [ov] (x) at (-3.5, -6) {};
      \node [ov] (y) at (3.5, -6) {};
      \node [ed] at (-4.25, 0) {\small$A$};
      \node [ed] at (-.95, 3.25) {\small$B$};
      \node [ed] at (4.25, 0) {\small$C$};
      \node [ed] at (-.95, -3.25) {\small$D$};
      \draw (u) [bend right=15] to node [ed,pos=.4] {\small$f$} (a);
      \draw (v) [bend left=15] to node [ed,pos=.25] {\small$g$} (a);
      \draw (a) [bend right=15] to node [ed,pos=.6] {\small$h$} (x);
      \draw (a) [bend left=15] to node [ed,pos=.75] {\small$i$} (y);
      
      \node [ov] (u) at (0, 6) {};
      \node [ov] (x) at (0, -6) {};
      \node [ov] (a) at (2.25, 0) {};
      \fill [fun, pattern=north east lines](-6,6) -- (u) .. controls (0, 4) and (2.25, 3) .. (a) .. controls (2.25, -3) and (0, -4) .. (x) -- (-6,-6) -- cycle;
      \fill [fun, pattern=grid](6,6) -- (u) .. controls (0, 4) and (2.25, 3) .. (a) .. controls (2.25, -3) and (0, -4) .. (x) -- (6,-6) -- cycle;
      \draw [nat] (u) .. controls (0, 4) and (2.25, 3) .. (a) .. controls (2.25, -3) and (0, -4) .. (x);
    \end{scope}
    \draw [rect] (-6,-6) rectangle (6,6);
  \end{tikzpicture}
  \;\;=\;\;
  \begin{tikzpicture}[scale=.2]
    \begin{scope}
      \clip [rect] (-6,-6) rectangle (6,6);
      \node [ov] (u) at (-3.5, 6) {};
      \node [ov] (v) at (3.5, 6) {};
      \node [iv] (a) at (1, 0) {\normalsize$\alpha$};
      \node [ov] (x) at (-3.5, -6) {};
      \node [ov] (y) at (3.5, -6) {};
      \node [ed] at (-4.25, 0) {\small$A$};
      \node [ed] at (.95, 3.25) {\small$B$};
      \node [ed] at (4.25, 0) {\small$C$};
      \node [ed] at (.95, -3.25) {\small$D$};
      \draw (u) [bend right=15] to node [ed,pos=.25] {\small$f$} (a);
      \draw (v) [bend left=15] to node [ed,pos=.4] {\small$g$} (a);
      \draw (a) [bend right=15] to node [ed,pos=.75] {\small$h$} (x);
      \draw (a) [bend left=15] to node [ed,pos=.6] {\small$i$} (y);
      
      \node [ov] (u) at (0, 6) {};
      \node [ov] (x) at (0, -6) {};
      \node [ov] (a) at (-2.25, 0) {};
      \fill [fun, pattern=north east lines](-6,6) -- (u) .. controls (0, 4) and (-2.25, 3) .. (a) .. controls (-2.25, -3) and (0, -4) .. (x) -- (-6,-6) -- cycle;
      \fill [fun, pattern=grid](6,6) -- (u) .. controls (0, 4) and (-2.25, 3) .. (a) .. controls (-2.25, -3) and (0, -4) .. (x) -- (6,-6) -- cycle;
      \draw [nat] (u) .. controls (0, 4) and (-2.25, 3) .. (a) .. controls (-2.25, -3) and (0, -4) .. (x);
    \end{scope}
    \draw [rect] (-6,-6) rectangle (6,6);
  \end{tikzpicture}
\]

\[
  \begin{tikzpicture}[scale=.325]
    \begin{scope}
      \clip [rect] (-6,-6) rectangle (6,6);
      \node [ov] (u) at (-3.5, 6) {};
      \node [ov] (v) at (3.5, 6) {};
      \node [iv, inner sep=2.3] (a) at (-1, 0) {\normalsize$\alpha$};
      \node [ov] (x) at (-3.5, -6) {};
      \node [ov] (y) at (3.5, -6) {};
      \node [ed] at (-4.25, 0) {\small$A$};
      \node [ed] at (-.95, 3.25) {\small$B$};
      \node [ed] at (4.4, 0) {\small$C$};
      \node [ed] at (-.95, -3.25) {\small$D$};
      \draw (u) [bend right=15] to node [ed,pos=.4] {\small$f$} (a);
      \draw (v) [bend left=15] to node [ed,pos=.25] {\small$g$} (a);
      \draw (a) [bend right=15] to node [ed,pos=.6] {\small$h$} (x);
      \draw (a) [bend left=15] to node [ed,pos=.75] {\small$i$} (y);
      
      \node [ov] (u) at (0, 6) {};
      \node [ov] (x) at (0, -6) {};
      \node [ov] (a) at (2.875, 0) {};
      \fill [fun, pattern=north east lines](-6,6) -- (u) .. controls (0, 4) and (2.875, 3) .. (a) .. controls (2.875, -3) and (0, -4) .. (x) -- (-6,-6) -- cycle;
      \fill [fun, pattern=grid](6,6) -- (u) .. controls (0, 4) and (2.875, 3) .. (a) .. controls (2.875, -3) and (0, -4) .. (x) -- (6,-6) -- cycle;
      \draw [nat] (u) .. controls (0, 4) and (2.875, 3) .. (a) .. controls (2.875, -3) and (0, -4) .. (x);
    \end{scope}
    \draw [rect] (-6,-6) rectangle (6,6);
  \end{tikzpicture}
  \;\;\;\;\coloneqq\;\;\;\;
  \begin{tikzpicture}[scale=.325]
    \begin{scope}
      \clip [rect] (-6,-6) rectangle (6,6);
      \node [ov] (u) at (-3.5, 6) {};
      \node [ov] (v) at (3.5, 6) {};
      \node [iv] (a) at (-1, 0) {$F\alpha$};
      \node [ov] (x) at (-3.5, -6) {};
      \node [ov] (y) at (3.5, -6) {};
      \node [ed] at (-4.25, 0) {$FA$};
      \node [ed] at (-.95, 3.25) {$FB$};
      \node [ed] at (4.5, 0) {$GC$};
      \node [ed] at (-.95, -3.25) {$FD$};
      \node [ed] at (1.15, 0) {$FC$};
      \node [ed] at (1.7, 5.25) {$GB$};
      \node [ed] at (1.7, -5.25) {$GD$};
      \draw (u) [bend right=15] to node [ed,pos=.4] {$Ff$} (a);
      \draw (v) [bend left=15] to node [ed,pos=.25] {$Gg$} node [ed,pos=.8] {$Fg$} (a);
      \draw (a) [bend right=15] to node [ed,pos=.6] {$Fh$} (x);
      \draw (a) [bend left=15] to node [ed,pos=.2] {$Fi$} node [ed,pos=.75] {$Gi$} (y);
      
      \node [ov] (u) at (0, 6) {};
      \node [ov] (x) at (0, -6) {};
      \node [ov] (a) at (2.875, 0) {};
      \draw (u) .. controls (0, 4) and (2.875, 3) .. node [ed,pos=.31] {$\sigma_A$} (a) node [ed] {$\sigma_C$} .. controls (2.875, -3) and (0, -4) .. node [ed,pos=.69] {$\sigma_D$} (x);
      \node [iv,fill=white,inner sep=.3] at (1.8, 2.875) {\scalebox{.75}{$\sigma_g\!\!\!\scalebox{.75}{${}^{-1}$}$}};
      \node [iv,fill=white] at (1.8, -2.875) {$\sigma_i$};
    \end{scope}
    \draw [rect] (-6,-6) rectangle (6,6);
  \end{tikzpicture}
\]

The Gray tensor product is easy to express in terms of {\adjective}
structures. Recall that a \textbf{shuffle} of linearly ordered sets
is a compatible linear order on their disjoint union.

\begin{defn}\label{defn:tgray}
  Let $\cat{C}$ and $\cat{D}$ be {\twowords}. The \textbf{Gray tensor product} of
  $\cat{C}$ and $\cat{D}$, denoted $\cat{C} \otimes \cat{D}$, is {\aan} {\twoword} defined as
  follows.
  \begin{itemize}
  \item A 0-cell in $\cat{C} \otimes \cat{D}$ is a pair $\grpair{c}{d}$ of a
    0-cell $c$ in $\cat{C}$ and a 0-cell $d$ in $\cat{D}$.
  \item A  1-cell in $\cat{C} \otimes \cat{D}$ is \emph{either}
    \begin{itemize}
    \item a pair
      $\grpair{f}{d} \colon \grpair{c}{d} \to \grpair{c'}{d}$ of a
      1-cell $f \colon c \to c'$ in $\cat{C}$ and a 1-cell $d$ in $\cat{D}$,
      \emph{or}
    \item a pair
      $\grpair{c}{g} \colon \grpair{c}{d} \to \grpair{c}{d'}$ of a
      0-cell $c$ in $\cat{C}$ and a 1-cell $g \colon d \to d'$ in $\cat{D}$.
    \end{itemize}

    Equivalently, a path of 1-cells in $\cat{C} \otimes \cat{D}$ is a
    \emph{shuffle} of a path of 1-cells in $\cat{C}$ and a path of
    1-cells in $\cat{D}$.
  \item A 2-cell in $\cat{C} \otimes \cat{D}$, with source and target
    each a shuffle of a path in $\cat{C}$ and a path in $\cat{D}$, is
    a pair $\grpair{\alpha}{\beta}$ of a 2-cell $\alpha$ with the
    source and target paths in $\cat{C}$ and a 2-cell $\beta$ with the
    source and target paths in $\cat{D}$.
  \item Composition of 2-cells is by composition in $\cat{C}$ and
    $\cat{D}$.
  \end{itemize}
\end{defn}

We also define $\otimes$ on functors in the obvious way: if
$F \colon \cat{C} \to \cat{D}$ and $G \colon \cat{C}' \to \cat{D}'$
are functors of {\twowords}, then $F \otimes G$ sends each cell called
$\grpair{x}{y}$ to the cell called $\grpair{F(x)}{G(y)}$ with
appropriate boundary.

\begin{rmk}
  This is the usual Gray tensor product of strict 2-categories,
  specialized to {\twowords} (i.e.\ the path 2-category of the Gray tensor product of
  {\twowords} is the usual Gray tensor product of their path
  2-categories). The description of the 2-cells given here follows
  from the equivalence (see e.g.~\cite[Corollary 3.22]{gurski})
  between the Gray tensor product of 2-categories
  $\cat{C} \otimes \cat{D}$ and the cartesian product of 2-categories
  $\cat{C} \times \cat{D}$.
\end{rmk}

\begin{rmk}\label{rmk:grayunit-two}
  The above definition easily generalizes from a binary product to an
  $n$-ary product, by replacing pairs and binary shuffles with $n$-tuples
  and $n$-ary shuffles. In particular, observe that the empty Gray
  tensor product defined in this way is {\aan} {\twoword} with one
  0-cell denoted $\grempty$ and no other non-identity cells.
\end{rmk}

\begin{prop}\label{prop:symmetric-two}
  \tword is symmetric monoidal with respect to $\otimes$.
\end{prop}
\begin{proof}[Sketch of proof]
  Functoriality of $\otimes$ is immediate from the
  definition. Moreover, $\otimes$ is associative, unital
  (\cref{rmk:grayunit-two}), and symmetric up to coherent natural
  isomorphism, by reparenthesizing and reordering the names of tuples.
\end{proof}

In \cref{sec:twowords} we defined {\aan} {\twoword} as a strict
2-category whose 1-cells are free, and we defined a functor of
{\twowords} as a 2-functor sending the generating 1-cells to
generating 1-cells. Now we define a \textbf{(lax or colax)
  transformation} of {\twoword} functors as a (lax or colax) natural
transformation of 2-functors whose components are generating 1-cells,
and we define a \textbf{modification} of {\twoword} transformations as
a modification of (compositions of) these 2-category natural
transformations. We spell out the details below.

These definitions are appropriate in that they provide closure for the
Gray tensor product (to be shown in \cref{prop:gray-two}), and they
exactly give the usual notions of transformations and modifications in
bicategories, under the correspondence between representable
{\twowords} and bicategories (to be shown in
\cref{prop:tranmain-two}).

\begin{defn}\label{defn:tran-two}
  Let $F$ and $G$ be functors between {\twowords} $\cat{C}$ and
  $\cat{D}$. A \textbf{colax transformation} $\sigma\colon F \to G$ consists of
  \begin{itemize}
  \item for each 0-cell $A$ in $\cat{C}$, a 1-cell $\sigma_A$ in $\cat{D}$:
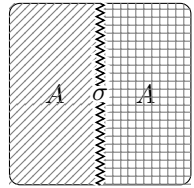
    \[
      \quad
      \begin{tikzpicture}[->,xscale=.75,yscale=-.75]
        \begin{scope}
          \node (a) at (-1, 0) {$FA$};
          \node (b) at (1, 0) {$GA$};
          \draw (a) -- node[above]{\small$\sigma_A$} (b);
        \end{scope}
      \end{tikzpicture}
      \quad
      \qquad\qquad\qquad
      \begin{tikzpicture}[scale=.3]
        \draw [rect] (-4,-4) rectangle (4,4);
        \node at (-2, 0) {$A$};
        \node at (2, 0) {$A$};
        \fill [fun, pattern=north east lines](-4,-4) rectangle (0,4);
        \fill [fun, pattern=grid](0,-4) rectangle (4,4);
        \node [ov] (s) at (0, 4) {};
        \node [ov] (t) at (0, -4) {};
        \draw (s) [nat] -- node [nated] {$\sigma$} (t);
      \end{tikzpicture}
    \]
  \item for each 1-cell $f \colon A \to B$ in $\cat{C}$,
    a 2-cell $\sigma_f$ in $\cat{D}$:
    \[
      \begin{tikzpicture}[->,xscale=.75,yscale=-.75]
        \begin{scope}[rotate=45]
          \node (a) at (-1, 1) {$FA$};
          \node (b) at (1, 1) {$GA$};
          \node (c) at (-1, -1) {$FB$};
          \node (d) at (1, -1) {$GB$};
          \node at (0, 0) {$\sigma_f$};
          \draw (a) -- node[below left]{\small$\sigma_A$} (b);
          \draw (b) -- node[below right]{\small$Gf$} (d);
          \draw (a) -- node[above left]{\small$Ff$} (c);
          \draw (c) -- node[above right]{\small$\sigma_B$} (d);
        \end{scope}
      \end{tikzpicture}
      \qquad\qquad\qquad
      \begin{tikzpicture}[scale=-.3]
        \draw [rect] (-4,-4) rectangle (4,4);
        \node at (-2, 0) {$B$};
        \node at (2, 0) {$A$};
        \node [ov] (s) at (0, 4) {};
        \node [ov] (t) at (0, -4) {};
        \draw (s) -- (t);
        \node [ed] at (0, -2.5) {$f$};
        \node [ed] at (0, 2.5) {$f$};
        \fill [fun, pattern=grid](-4,-4) -- (-4,4) --
        (2,4) .. controls (2,2) and (-2,-2) .. (-2, -4) -- cycle;
        \fill [fun, pattern=north east lines](4,4) -- (4,-4) --
        (-2,-4) .. controls (-2,-2) and (2,2) .. (2, 4) -- cycle;
        \node [ov] (sa) at (2, 4) {};
        \node [ov] (ta) at (-2, -4) {};
        \draw (sa) [nat] .. controls (2,2) and (-2,-2) .. node [nated, pos=.2] {$\sigma$} node [nated,pos=.8] {$\sigma$} (ta);
      \end{tikzpicture}
    \]
  \end{itemize}
  such that for each 2-cell $\alpha$ in $\cat{C}$, we have
  {\allowdisplaybreaks
    \begin{align*}
      \begin{tikzpicture}[<-,xscale=-.75,yscale=.65,shorten <=-3pt,shorten >=-2pt]
        \node (a) at (-2,0) {$\cdot$};
        \node (bl) at (-.4,1) {};
        \node at (0,1) {$\bcdots$};
        \node (br) at (.4,1) {};
        \node (cl) at (-.4,-1) {};
        \node at (0,-1) {$\bcdots$};
        \node (cr) at (.4,-1) {};
        \node (d) at (2,0) {$\cdot$};
        \node at (0,0) {$F\alpha$};
        \draw (a) [bend left=12.5] to node [above right=-2,pos=.65] {\tiny$Fs_m$} (bl);
        \draw (br) [bend left=12.5] to node [above left=-2,pos=.35] {\tiny$Fs_1$} (d);
        \draw (a) [bend right=12.5] to node [above left=-3,pos=.25] {\tiny$Ft_n$} (cl);
        \draw (cr) [bend right=12.5] to node [above right=-3,pos=.75] {\tiny$Ft_1$} (d);
        \begin{scope}[shift={(-.5,-1.5)}]
          \node (a1) at (-2,0) {$\cdot$};
          \node (cl1) at (-.4,-1) {};
          \node at (0,-1) {$\bcdots$};
          \node (cr1) at (.4,-1) {};
          \node (d1) at (2,0) {$\cdot$};
          \draw (a1) [bend right=12.5] to node [below right=-2,pos=.65] {\tiny$Gt_n$} (cl1);
          \draw (cr1) [bend right=12.5] to node [below left=-2,pos=.35] {\tiny$Gt_1$} (d1);
        \end{scope}
        \begin{scope}[shift={(-.25,-.75)}]
          \node at (0,-1) {$\bcdots$};
        \end{scope}
        \draw (a1) -- node [above right=-2] {\tiny$\sigma$} (a);
        \draw (cl1) -- (cl);
        \draw (cr1) -- (cr);
        \draw (d1) -- node [below left=-2] {\tiny$\sigma$} (d);
        \node at (-1.45,-1.35) {$\sigma_{t_n}$};
        \node at (.95,-1.35) {$\sigma_{t_1}$};
      \end{tikzpicture}
      \,\;
      &=
        \;\,
        \begin{tikzpicture}[->,xscale=.75,yscale=-.65,shorten <=-2pt,shorten >=-3pt]
          \node (a) at (-2,0) {$\cdot$};
          \node (bl) at (-.4,1) {};
          \node at (0,1) {$\bcdots$};
          \node (br) at (.4,1) {};
          \node (cl) at (-.4,-1) {};
          \node at (0,-1) {$\bcdots$};
          \node (cr) at (.4,-1) {};
          \node (d) at (2,0) {$\cdot$};
          \node at (0,0) {$G\alpha$};
          \draw (a) [bend left=12.5] to node [below left=-2,pos=.65] {\tiny$Gt_1$} (bl);
          \draw (br) [bend left=12.5] to node [below right=-2,pos=.35] {\tiny$Gt_n$} (d);
          \draw (a) [bend right=12.5] to node [below right=-3,pos=.25] {\tiny$Gs_1$} (cl);
          \draw (cr) [bend right=12.5] to node [below left=-3,pos=.75] {\tiny$Gs_m$} (d);
          \begin{scope}[shift={(-.5,-1.5)}]
            \node (a1) at (-2,0) {$\cdot$};
            \node (cl1) at (-.4,-1) {};
            \node at (0,-1) {$\bcdots$};
            \node (cr1) at (.4,-1) {};
            \node (d1) at (2,0) {$\cdot$};
            \draw (a1) [bend right=12.5] to node [above left=-2,pos=.65] {\tiny$Fs_1$} (cl1);
            \draw (cr1) [bend right=12.5] to node [above right=-2,pos=.35] {\tiny$Fs_m$} (d1);
          \end{scope}
          \begin{scope}[shift={(-.25,-.75)}]
            \node at (0,-1) {$\bcdots$};
          \end{scope}
          \draw (a1) -- node [below left=-2] {\tiny$\sigma$} (a);
          \draw (cl1) -- (cl);
          \draw (cr1) -- (cr);
          \draw (d1) -- node [above right=-2] {\tiny$\sigma$} (d);
          \node at (-1.45,-1.35) {$\sigma_{s_1}$};
          \node at (.95,-1.35) {$\sigma_{s_m}$};
        \end{tikzpicture}
      \\[1em]
      \begin{tikzpicture}[scale=-.3]
        \draw [rect] (-5,-5) rectangle (5,5);
        \node [ov] (u) at (-2, 5) {};
        \node [ov] (v) at (2, 5) {};
        \node [iv,inner sep=2pt,font=\small] (a) at (0, 0) {$\alpha$};
        \node [ov] (x) at (-2, -5) {};
        \node [ov] (y) at (2, -5) {};
        \draw (u) [bend right=20] to node [ed] {$t_n$} (a);
        \draw (v) [bend left=20] to node [ed] {$t_1$} (a);
        \draw (a) [bend right=20] to node [ed] {$s_m$} (x);
        \draw (a) [bend left=20] to node [ed] {$s_1$} (y);
        \node [ed,font=\large] at (0,-2.5) {$\bcdots$};
        \fill [fun, pattern=grid](-5,-5) -- (-5,5) -- (3.5,5) .. controls (3.5,3.6) and (-3.5,3) .. (-3.5, -5) -- cycle;
        \fill [fun, pattern=north east lines](5,5) -- (5,-5) -- (-3.5,-5) .. controls (-3.5,3) and (3.5,3.6) .. (3.5, 5) -- cycle;
        \node [ov] (sa) at (3.5, 5) {};
        \node [ov] (ta) at (-3.5, -5) {};
        \draw (sa) [nat] .. controls (3.5,3.6) and (-3.5,3) .. node [nated, pos=.75] {$\sigma$} (ta);
        \node [circle,fill=white,inner sep=0pt,font=\large] at (0,2.5) {$\bcdots$};
        \node [ed,font=\large] at (0,5) {$\bcdots$};
        \node [ed,font=\large] at (0,-5) {$\bcdots$};
      \end{tikzpicture}
      \quad
      &=
        \quad
        \begin{tikzpicture}[scale=-.3]
          \draw [rect] (-5,-5) rectangle (5,5);
          \node [ov] (u) at (-2, 5) {};
          \node [ov] (v) at (2, 5) {};
          \node [iv,inner sep=2pt,font=\small] (a) at (0, 0) {$\alpha$};
          \node [ov] (x) at (-2, -5) {};
          \node [ov] (y) at (2, -5) {};
          \draw (u) [bend right=20] to node [ed] {$t_n$} (a);
          \draw (v) [bend left=20] to node [ed] {$t_1$} (a);
          \draw (a) [bend right=20] to node [ed] {$s_m$} (x);
          \draw (a) [bend left=20] to node [ed] {$s_1$} (y);
          \node [ed,font=\large] at (0,2.5) {$\bcdots$};
          \fill [fun, pattern=grid](-5,-5) -- (-5,5) -- (3.5,5) .. controls (3.5,-3) and (-3.5,-3.6) .. (-3.5, -5) -- cycle;
          \fill [fun, pattern=north east lines](5,5) -- (5,-5) -- (-3.5,-5) .. controls (-3.5,-3.6) and (3.5,-3) .. (3.5, 5) -- cycle;
          \node [ov] (sa) at (3.5, 5) {};
          \node [ov] (ta) at (-3.5, -5) {};
          \draw (sa) [nat] .. controls (3.5,-3) and (-3.5,-3.6) .. node [nated, pos=.25] {$\sigma$} (ta);
          \node [circle,fill=white,inner sep=0pt,font=\large] at (0,-2.5) {$\bcdots$};
          \node [ed,font=\large] at (0,5) {$\bcdots$};
          \node [ed,font=\large] at (0,-5) {$\bcdots$};
        \end{tikzpicture}
    \end{align*}
  }
  A \textbf{lax transformation} is defined dually, with
  transformation component 1-cells at the northwest and southeast
  corners of diagrams (diagrams mirrored).
  
  When the $\sigma_f$ 2-cells are all invertible, we call $\sigma$
  simply a \textbf{transformation}.
  \[
    \begin{tikzpicture}[->,xscale=.75,yscale=-.75]
      \begin{scope}[rotate=45]
        \node (a) at (-1, 1) {$FA$};
        \node (b) at (1, 1) {$FB$};
        \node (c) at (-1, -1) {$GA$};
        \node (d) at (1, -1) {$GB$};
        \node at (0, 0) {$\sigma_f^{-1}$};
        \draw (a) -- node[below left]{\small$Ff$} (b);
        \draw (b) -- node[below right]{\small$\sigma_B$} (d);
        \draw (a) -- node[above left]{\small$\sigma_A$} (c);
        \draw (c) -- node[above right]{\small$Gf$} (d);
      \end{scope}
    \end{tikzpicture}
    \qquad\qquad\qquad
    \begin{tikzpicture}[xscale=-.3,yscale=.3]
      \draw [rect] (-4,-4) rectangle (4,4);
      \node at (-2, 0) {$B$};
      \node at (2, 0) {$A$};
      \node [ov] (s) at (0, 4) {};
      \node [ov] (t) at (0, -4) {};
      \draw (s) -- (t);
      \node [ed] at (0, -2.5) {$f$};
      \node [ed] at (0, 2.5) {$f$};
      \fill [fun, pattern=grid](-4,-4) -- (-4,4) --
      (2,4) .. controls (2,2) and (-2,-2) .. (-2, -4) -- cycle;
      \fill [fun, pattern=north east lines](4,4) -- (4,-4) --
      (-2,-4) .. controls (-2,-2) and (2,2) .. (2, 4) -- cycle;
      \node [ov] (sa) at (2, 4) {};
      \node [ov] (ta) at (-2, -4) {};
      \draw (sa) [nat] .. controls (2,2) and (-2,-2) .. node [nated, pos=.2] {$\sigma$} node [nated,pos=.8] {$\sigma$} (ta);
    \end{tikzpicture}
  \]
  \[
    \begin{tikzpicture}[scale=.25]
      \draw [rect] (-4,-4) rectangle (4,4);
      \node [ov] (s) at (0, 4) {};
      \node [ov] (t) at (0, -4) {};
      \draw (s) -- (t);
      \node [ed] at (0, 0) {$f$};
      \fill [fun, pattern=north east lines](-4,-4) -- (-4,4) -- (2, 4) .. controls (2,3) and (-2,1) .. (-2, 0) .. controls (-2,-1) and (2,-3) .. (2, -4) -- cycle;
      \fill [fun, pattern=grid](4,4) -- (4,-4) -- (2, -4) .. controls (2,-3) and (-2,-1) .. (-2, 0) .. controls (-2,1) and (2,3) .. (2, 4) -- cycle;
      \node [ov] (sa) at (2, 4) {};
      \node [ov] (ma) at (-2, 0) {};
      \node [ov] (ta) at (2, -4) {};
      \draw (sa) [nat] .. controls (2,3) and (-2,1) .. (ma) node [nated] {$\sigma$} .. controls (-2,-1) and (2,-3) .. (ta);
    \end{tikzpicture}
    \quad=\quad
    \begin{tikzpicture}[scale=.25]
      \draw [rect] (-4,-4) rectangle (4,4);
      \node [ov] (s) at (0, 4) {};
      \node [ov] (t) at (0, -4) {};
      \draw (s) -- (t);
      \node [ed] at (0, 0) {$f$};
      \fill [fun, pattern=north east lines](-4,-4) -- (-4,4) -- (2, 4) -- (2, -4) -- cycle;
      \fill [fun, pattern=grid](4,4) -- (4,-4) -- (2, -4) -- (2, 4) -- cycle;
      \node [ov] (sa) at (2, 4) {};
      \node [ov] (ma) at (-2, 0) {};
      \node [ov] (ta) at (2, -4) {};
      \draw (sa) [nat] -- node [nated,pos=.5] {$\sigma$} (ta);
    \end{tikzpicture}
    \qquad\qquad\qquad
    \begin{tikzpicture}[xscale=-.25, yscale=.25]
      \draw [rect] (-4,-4) rectangle (4,4);
      \node [ov] (s) at (0, 4) {};
      \node [ov] (t) at (0, -4) {};
      \draw (s) -- (t);
      \node [ed] at (0, 0) {$f$};
      \fill [fun, pattern=grid](-4,-4) -- (-4,4) -- (2, 4) .. controls (2,3) and (-2,1) .. (-2, 0) .. controls (-2,-1) and (2,-3) .. (2, -4) -- cycle;
      \fill [fun, pattern=north east lines](4,4) -- (4,-4) -- (2, -4) .. controls (2,-3) and (-2,-1) .. (-2, 0) .. controls (-2,1) and (2,3) .. (2, 4) -- cycle;
      \node [ov] (sa) at (2, 4) {};
      \node [ov] (ma) at (-2, 0) {};
      \node [ov] (ta) at (2, -4) {};
      \draw (sa) [nat] .. controls (2,3) and (-2,1) .. (ma) node [nated] {$\sigma$} .. controls (-2,-1) and (2,-3) .. (ta);
    \end{tikzpicture}
    \quad=\quad
    \begin{tikzpicture}[xscale=-.25, yscale=.25]
      \draw [rect] (-4,-4) rectangle (4,4);
      \node [ov] (s) at (0, 4) {};
      \node [ov] (t) at (0, -4) {};
      \draw (s) -- (t);
      \node [ed] at (0, 0) {$f$};
      \fill [fun, pattern=grid](-4,-4) -- (-4,4) -- (2, 4) -- (2, -4) -- cycle;
      \fill [fun, pattern=north east lines](4,4) -- (4,-4) -- (2, -4) -- (2, 4) -- cycle;
      \node [ov] (sa) at (2, 4) {};
      \node [ov] (ma) at (-2, 0) {};
      \node [ov] (ta) at (2, -4) {};
      \draw (sa) [nat] -- node [nated,pos=.5] {$\sigma$} (ta);
    \end{tikzpicture}
  \]
\end{defn}

\begin{rmk}\label{rmk:lax-colax}
  A transformation is both a colax transformation and a lax
  transformation: given a colax transformation where the $\sigma_f$
  2-cells are all invertible, the inverse 2-cells $\sigma_f^{-1}$ are
  components of a lax transformation, and vice versa.
  \[
    \begin{tikzpicture}[xscale=.25,yscale=-.25]
      \draw [rect] (-5,-5) rectangle (5,5);
      \node [ov] (u) at (-1.65, 5) {};
      \node [ov] (v) at (1.65, 5) {};
      \node [iv,inner sep=4pt] (a) at (0, 0) {};
      \node [ov] (x) at (-1.65, -5) {};
      \node [ov] (y) at (1.65, -5) {};
      \draw (u) [bend right=16.5,pos=.525] to (a);
      \draw (v) [bend left=16.5,pos=.525] to (a);
      \draw (a) [bend right=16.5,pos=.475] to (x);
      \draw (a) [bend left=16.5,pos=.475] to (y);
      \begin{scope}
        \node [ed] at (0,2.5) {$\bcdots$};
        \fill [fun, pattern=north east lines](-5,-5) -- (-5,5) -- (3.5,5) .. controls (3.5,-4) and (-3.5,-2.5) .. (-3.5, -5) -- cycle;
        \fill [fun, pattern=grid](5,5) -- (5,-5) -- (-3.5,-5) .. controls (-3.5,-4) and (3.5,-2.5) .. (3.5, 5) -- cycle;
        \node [ov] (sa) at (3.5, 5) {};
        \node [ov] (ta) at (-3.5, -5) {};
        \draw (sa) [nat] .. controls (3.5,-4) and (-3.5,-2.5) .. (ta);
        \node [circle,fill=white,inner sep=0pt] at (0,-2.5) {$\bcdots$};
      \end{scope}
      \node [ed] at (0,5) {$\bcdots$};
      \node [ed] at (0,-5) {$\bcdots$};
    \end{tikzpicture}
    \, = \,
    \begin{tikzpicture}[xscale=.25,yscale=-.25]
      \draw [rect] (-5,-5) rectangle (5,5);
      \node [ov] (u) at (-1.65, 5) {};
      \node [ov] (v) at (1.65, 5) {};
      \node [iv,inner sep=4pt] (a) at (0, 0) {};
      \node [ov] (x) at (-1.65, -5) {};
      \node [ov] (y) at (1.65, -5) {};
      \draw (u) [bend right=16.5,pos=.525] to (a);
      \draw (v) [bend left=16.5,pos=.525] to (a);
      \draw (a) [bend right=16.5,pos=.475] to (x);
      \draw (a) [bend left=16.5,pos=.475] to (y);
      \begin{scope}
        \fill [fun, pattern=north east lines](-5,-5) -- (-5,5) -- (3.5,5) .. controls (3.5,3) and (-3,4.5) .. (-3,3) .. controls (-3,1.5) and (2.5,3) .. (2.5,0) .. controls (2.5,-3) and (-3.5,-2.5) .. (-3.5,-5) -- cycle;
        \fill [fun, pattern=grid](5,-5) -- (5,5) -- (3.5,5) .. controls (3.5,3) and (-3,4.5) .. (-3,3) .. controls (-3,1.5) and (2.5,3) .. (2.5,0) .. controls (2.5,-3) and (-3.5,-2.5) .. (-3.5,-5) -- cycle;
        \node [ov] (sa) at (3.5, 5) {};
        \node [ov] (ta) at (-3.5, -5) {};
        \draw (sa) [nat] .. controls (3.5,3) and (-3,4.5) .. (-3,3) .. controls (-3,1.5) and (2.5,3) .. (2.5,0) .. controls (2.5,-3) and (-3.5,-2.5) .. (ta);
        \node [circle,fill=white,inner sep=0pt] at (0,2) {$\bcdots$};
        \node [circle,fill=white,inner sep=0pt] at (0,3.75) {$\bcdots$};
        \node [circle,fill=white,inner sep=0pt] at (0,-2.5) {$\bcdots$};
      \end{scope}
      \node [ed] at (0,5) {$\bcdots$};
      \node [ed] at (0,-5) {$\bcdots$};
    \end{tikzpicture}
    \, = \,
    \begin{tikzpicture}[xscale=.25,yscale=-.25]
      \draw [rect] (-5,-5) rectangle (5,5);
      \node [ov] (u) at (-1.65, 5) {};
      \node [ov] (v) at (1.65, 5) {};
      \node [iv,inner sep=4pt] (a) at (0, 0) {};
      \node [ov] (x) at (-1.65, -5) {};
      \node [ov] (y) at (1.65, -5) {};
      \draw (u) [bend right=16.5,pos=.525] to (a);
      \draw (v) [bend left=16.5,pos=.525] to (a);
      \draw (a) [bend right=16.5,pos=.475] to (x);
      \draw (a) [bend left=16.5,pos=.475] to (y);
      \begin{scope}[scale=-1]
        \fill [fun, pattern=grid](-5,-5) -- (-5,5) -- (3.5,5) .. controls (3.5,3) and (-3,4.5) .. (-3,3) .. controls (-3,1.5) and (2.5,3) .. (2.5,0) .. controls (2.5,-3) and (-3.5,-2.5) .. (-3.5,-5) -- cycle;
        \fill [fun, pattern=north east lines](5,-5) -- (5,5) -- (3.5,5) .. controls (3.5,3) and (-3,4.5) .. (-3,3) .. controls (-3,1.5) and (2.5,3) .. (2.5,0) .. controls (2.5,-3) and (-3.5,-2.5) .. (-3.5,-5) -- cycle;
        \node [ov] (sa) at (3.5, 5) {};
        \node [ov] (ta) at (-3.5, -5) {};
        \draw (sa) [nat] .. controls (3.5,3) and (-3,4.5) .. (-3,3) .. controls (-3,1.5) and (2.5,3) .. (2.5,0) .. controls (2.5,-3) and (-3.5,-2.5) .. (ta);
        \node [circle,fill=white,inner sep=0pt] at (0,2) {$\bcdots$};
        \node [circle,fill=white,inner sep=0pt] at (0,3.75) {$\bcdots$};
        \node [circle,fill=white,inner sep=0pt] at (0,-2.5) {$\bcdots$};
      \end{scope}
      \node [ed] at (0,5) {$\bcdots$};
      \node [ed] at (0,-5) {$\bcdots$};
    \end{tikzpicture}
    \, = \,
    \begin{tikzpicture}[xscale=.25,yscale=-.25]
      \draw [rect] (-5,-5) rectangle (5,5);
      \node [ov] (u) at (-1.65, 5) {};
      \node [ov] (v) at (1.65, 5) {};
      \node [iv,inner sep=4pt] (a) at (0, 0) {};
      \node [ov] (x) at (-1.65, -5) {};
      \node [ov] (y) at (1.65, -5) {};
      \draw (u) [bend right=16.5,pos=.525] to (a);
      \draw (v) [bend left=16.5,pos=.525] to (a);
      \draw (a) [bend right=16.5,pos=.475] to (x);
      \draw (a) [bend left=16.5,pos=.475] to (y);
      \begin{scope}[scale=-1]
        \node [ed] at (0,2.5) {$\bcdots$};
        \fill [fun, pattern=grid](-5,-5) -- (-5,5) -- (3.5,5) .. controls (3.5,-4) and (-3.5,-2.5) .. (-3.5, -5) -- cycle;
        \fill [fun, pattern=north east lines](5,5) -- (5,-5) -- (-3.5,-5) .. controls (-3.5,-4) and (3.5,-2.5) .. (3.5, 5) -- cycle;
        \node [ov] (sa) at (3.5, 5) {};
        \node [ov] (ta) at (-3.5, -5) {};
        \draw (sa) [nat] .. controls (3.5,-4) and (-3.5,-2.5) .. (ta);
        \node [circle,fill=white,inner sep=0pt] at (0,-2.5) {$\bcdots$};
      \end{scope}
      \node [ed] at (0,5) {$\bcdots$};
      \node [ed] at (0,-5) {$\bcdots$};
    \end{tikzpicture}
  \]
\end{rmk}



Just as a transformation is a morphism of functors, a modification is
a morphism of transformations. The most commonly seen definition of
modification goes between two (lax or colax) transformations. However,
there is a more general definition of modification that involves both
\emph{lax and colax} transformations. We actually get a (\adjective)
\emph{double category} $\Homl(\cat{C}, \cat{D})$ where the horizontal
arrows are lax transformations and the vertical arrows are colax
transformations.

\begin{defn}\label{defn:mod-two}
  A \textbf{modification}
  \[

  \]
  where $\pi_i$ and $\tau_i$ are lax transformations and $\rho_i$ and
  $\sigma_i$ are colax transformations of functors
  $\cat{C} \to \cat{D}$ consists of for each 0-cell $A$ in $\cat{C}$ a
  2-cell $\Gamma_A$ in $\cat{D}$:
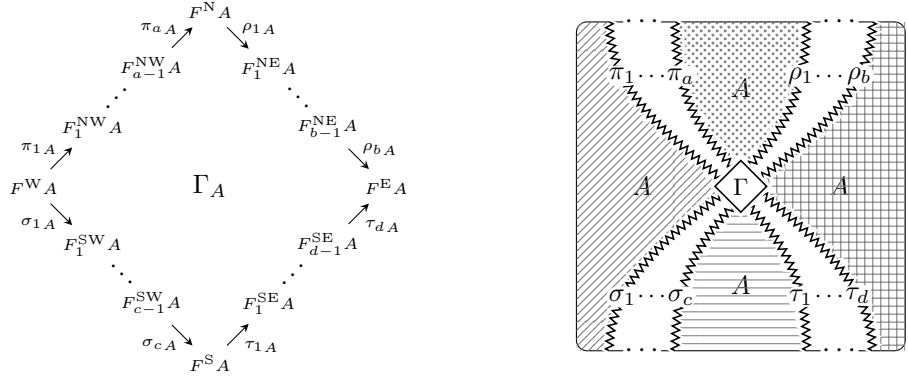
  \[
    %
  \]
  such that
  for any 1-cell $f\colon A \to B$ in $\cat{C}$, we have
  {\allowdisplaybreaks
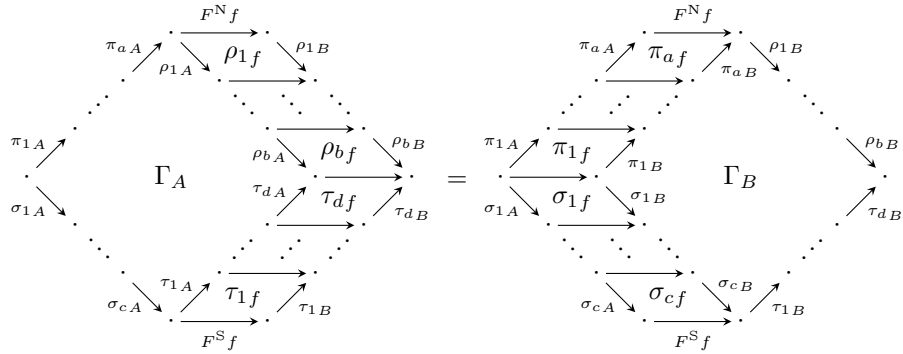
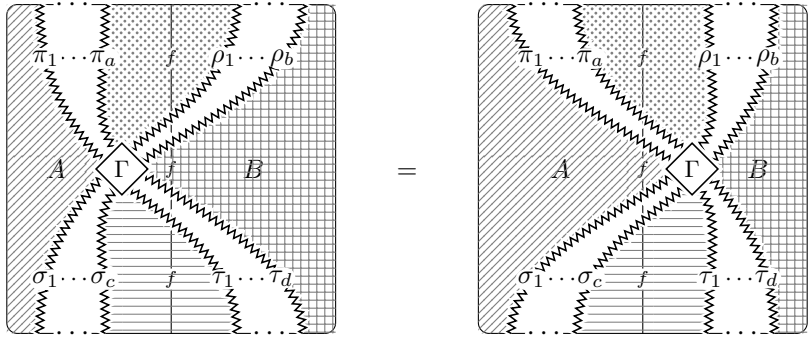
    \begin{align*}
      %
    \end{align*}
  }
\end{defn}

We define \textbf{horizontal compositions} and \textbf{vertical
  compositions} of modifications
componentwise.
Likewise \textbf{horizontal (lax) identity} and
\textbf{vertical (colax) identity} modifications are identities
componentwise.

\begin{prop}
  Functors, lax and colax transformations, and modifications between
  $\cat{C}$ and $\cat{D}$ form {\aan} {\doubleword}
  $\Homl(\cat{C}, \cat{D})$ (via composition of modifications).
\end{prop}
\begin{proof}
  The associativity, unit, and interchange laws are inherited from the
  2-cells in $\cat{D}$.
\end{proof}

We denote by $\Hom(\cat{C}, \cat{D})$ the {\twoword} whose 0-cells are
functors $\cat{C} \to \cat{D}$, 1-cells are \emph{transformations},
and 2-cells are modifications between these.

\begin{rmk}
  Given a colax transformation of {\twoword} functors, if every
  component 1-cell is a left adjoint, we obtain (upon choosing
  adjunctions) a lax transformation in the other direction (where the
  new component 2-cells are the \emph{mates} of the old ones).
  %
  %
  A \emph{conjoint pair} in $\Homl(\cat{C}, \cat{D})$ is such a pair
  of colax and lax transformations, with component 1-cells in left and
  right adjoint pairs.
  
  On the other hand, as noted in \cref{rmk:lax-colax}, given a
  colax transformation, if every component 2-cell is invertible, we
  obtain a lax transformation in the same direction; this is the
  content of a (non lax or colax) transformation. A \emph{companion
    pair} in $\Homl(\cat{C}, \cat{D})$ is (up to isomorphism) such a
  transformation.

  In general, {\twowords} may be identified with {\doublewords} having
  horizontal and vertical 1-cells in assigned companion pairs. (It is
  the same as in the strict case; the translation from (\adjective)
  2-categories to such (\adjective) double categories is the
  ``squares'' or ``quintets'' construction of \cref{eg:bi-dbl}.)
  The {\twoword} $\Hom(\cat{C}, \cat{D})$ is then embedded in
  $\Homl(\cat{C}, \cat{D})$ as the 1-cells with companions.
  (The former is recovered up to equivalence from the latter through
  the right adjoint to the quintets construction.)
\end{rmk}

It still remains to verify that $\Hom(\cat{C}, \cat{D})$ in fact
provides an internal hom for the Gray tensor product. In other words,
$\cat{C} \otimes \cat{D}$ is universal with a map
$\cat{C} \to \Hom(\cat{D}, \cat{C} \otimes \cat{D})$:

\begin{prop}\label{prop:gray-two}
  \tword is closed with respect to $\otimes$.

  In particular, the Gray tensor product $\cat{C} \otimes \cat{D}$ is
  the free {\twoword} on the following data and laws:
  \begin{itemize}
  \item For every 0-cell $c$ of $\cat{C}$, there is a functor
    $\grpair{c}{-} \colon \cat{D} \to \cat{C} \otimes \cat{D}$.
  \item For every 1-cell $f \colon c \to d$ of $\cat{C}$,
    there is a transformation $\grpair{f}{1} \colon
    \grpair{c}{-} \to \grpair{d}{-}$.
  \item For every 2-cell $\alpha$ of $\cat{C}$, there is a
    modification $\grpair{\alpha}{1}$ between the associated
    transformations.
  \item Such modifications compose as in $\cat{C}$, with identities as
    in $\cat{C}$.
  \end{itemize}
\end{prop}
\begin{proof}
  Note first that the construction $\Hom(\cat{D}, \cat{X})$ is
  functorial in $\cat{X}$ (since functors, transformations,
  modifications, and their compositions are shapes consisting of cells
  and equations in $\cat{X}$), and a map from $\cat{C}$ into
  $\Hom(\cat{D}, \cat{X})$ is precisely the data in $\cat{X}$ as
  described above.
  
  It is easy to see that $\cat{C} \otimes \cat{D}$ contains such
  data. Now suppose $\cat{X}$ also contains such data. We must check
  that the induced map on the putative generating cells extends to a
  unique functor $\cat{C} \otimes \cat{D} \to \cat{X}$.

  \begin{figure}
    \[
      \begin{tikzpicture}[->, xscale=1.1,yscale=.7]
        \begin{scope}
          \node (a1) at (-2, 0) {};
          \node (b1) at (0, 0) {};
          \node at (a1) {$\cdot$};
          \node at (b1) {$\cdot$};
          \node at (a1) {$\cdot$};
          \node at (b1) {$\cdot$};
          \node (z1) at (-1, 0) {$\grpair{\alpha}{1}$};
          \node at (-1,.85) {$\bcdots$};
          \node at (-1,-.85) {$\bcdots$};
          \node (sh1) at (-1,1) {$\phantom{\bcdots}$};
          \node (th1) at (-1,-1) {$\phantom{\bcdots}$};
          \draw (a1) to [bend left=10] (sh1);
          \draw (sh1) to [bend left=10] (b1);
          \draw (a1) to [bend right=10] (th1);
          \draw (th1) to [bend right=10] (b1);
        \end{scope}
        \begin{scope}[shift={(2,0)}]
          \node (a2) at (-2, 0) {};
          \node (b2) at (0, 0) {};
          \node at (b2) {$\cdot$};
          \node (z2) at (-1, 0) {$\grpair{1}{\beta}$};
          \node at (-1,.85) {$\bcdots$};
          \node at (-1,-.85) {$\bcdots$};
          \node (sh2) at (-1,1) {$\phantom{\bcdots}$};
          \node (th2) at (-1,-1) {$\phantom{\bcdots}$};
          \draw (a2) to [bend left=10] (sh2);
          \draw (sh2) to [bend left=10] (b2);
          \draw (a2) to [bend right=10] (th2);
          \draw (th2) to [bend right=10] (b2);
        \end{scope}
        \begin{scope}[shift={(0,1)}]
          \node (ua) at (-1, .8) {$\bcdots$};
          \node (ub) at (1, .8) {$\bcdots$};
          \node (um) at (0, 1.3) {$\bcdots$};
          \node (z4) at (0, .4) {$\grpair{1}{1}$};
          \draw (a1) to [bend left=11] (ua);
          \draw (ua) to [bend left=6] (um);
          \draw (um) to [bend left=6] (ub);
          \draw (ub) to [bend left=11] (b2);
        \end{scope}
        \begin{scope}[shift={(0,-1)}]
          \node (ba) at (-1, -.8) {$\bcdots$};
          \node (bb) at (1, -.8) {$\bcdots$};
          \node (bm) at (0, -1.3) {$\bcdots$};
          \node (z4) at (0, -.4) {$\grpair{1}{1}$};
          \draw (a1) to [bend right=11] (ba);
          \draw (ba) to [bend right=6] (bm);
          \draw (bm) to [bend right=6] (bb);
          \draw (bb) to [bend right=11] (b2);
        \end{scope}
      \end{tikzpicture}
      \qquad\qquad\;\;
      \begin{tikzpicture}[scale=.49]
        \draw [rect] (-4.5,-3.3) rectangle (4.5,3.3);
        \draw (2.6,2.3) -- (2.6, 1.6) .. controls +(0,-.5) and +(.5,.5) .. (1.75,0) .. controls +(.5,-.5) and +(0,.5) .. (2.6,-1.6) -- (2.6,-2.3);
        \draw (.9,2.3) -- (.9, 1.6) .. controls +(0,-.5) and +(-.5,.5) .. (1.75,0) .. controls +(-.5,-.5) and +(0,.5) .. (.9,-1.6) -- (.9,-2.3);
        \node [ed] at (-1.75,1.3) {\large$\bcdots$};
        \node [ed] at (1.75,-1.3) {\large$\bcdots$};
        \node [ed] at (1.75,1.3) {\large$\bcdots$};
        \node [ed] at (-1.75,-1.3) {\large$\bcdots$};
        \node [iv, minimum width=6mm,fill=white] at (1.75, 0) {\large$\beta$};
        \fill [fun, pattern=north east lines] (-4.5,3.3) -- (-3.4,3.3) -- (-3.4,2.3) -- (-2.6,2.3) -- (-2.6, 1.6) .. controls +(0,-.5) and +(-.5,.5) .. (-1.75,0) .. controls +(-.5,-.5) and +(0,.5) .. (-2.6,-1.6) -- (-2.6,-2.3) -- (-3.4,-2.3) -- (-3.4,-3.3) -- (-4.5,-3.3) -- cycle;
        \fill [fun, pattern=grid] (4.5,3.3) -- (3.4,3.3) -- (3.4,2.3) -- (-.9,2.3) -- (-.9, 1.6) .. controls +(0,-.5) and +(.5,.5) .. (-1.75,0) .. controls +(.5,-.5) and +(0,.5) .. (-.9,-1.6) -- (-.9,-2.3) -- (3.4,-2.3) -- (3.4,-3.3) -- (4.5,-3.3) -- cycle;
        \draw [nat] (-2.6,2.3) -- (-2.6, 1.6) .. controls +(0,-.5) and +(-.5,.5) .. (-1.75,0) .. controls +(-.5,-.5) and +(0,.5) .. (-2.6,-1.6) -- (-2.6,-2.3);
        \draw [nat] (-.9,2.3) -- (-.9, 1.6) .. controls +(0,-.5) and +(.5,.5) .. (-1.75,0) .. controls +(.5,-.5) and +(0,.5) .. (-.9,-1.6) -- (-.9,-2.3);
        \node [modc, minimum width=6mm] at (-1.75, 0) {\large$\alpha$};
        \draw (-2.6,2.3) -- (-2.6, 3.3);
        \draw [nat] (-.9,2.3) -- (-.9, 3.3);
        \draw (.9,2.3) -- (.9, 3.3);
        \draw [nat] (2.6,2.3) -- (2.6, 3.3);
        \draw (-2.6,-2.3) -- (-2.6, -3.3);
        \draw [nat] (-.9,-2.3) -- (-.9, -3.3);
        \draw (.9,-2.3) -- (.9, -3.3);
        \draw [nat] (2.6,-2.3) -- (2.6, -3.3);
        \node [rectangle, rounded corners, fill=white, draw, dashed, inner sep=0, minimum height=5.2mm, minimum width=36mm, font=\large\em] at (0, 2.3) {s \, h \, u \, f \, f \, l \, e};
        \node [rectangle, rounded corners, fill=white, draw, dashed, inner sep=0, minimum height=5.2mm, minimum width=36mm,font=\large\em] at (0, -2.3) {s \, h \, u \, f \, f \, l \, e};
        \node [ed] at (-3.4,3.3) {\large$\bcdots$};
        \node [ed] at (-1.75,3.3) {\large$\bcdots$};
        \node [ed] at (0,3.3) {\large$\bcdots$};
        \node [ed] at (1.75,3.3) {\large$\bcdots$};
        \node [ed] at (3.4,3.3) {\large$\bcdots$};
        \node [ed] at (-3.4,-3.3) {\large$\bcdots$};
        \node [ed] at (-1.75,-3.3) {\large$\bcdots$};
        \node [ed] at (0,-3.3) {\large$\bcdots$};
        \node [ed] at (1.75,-3.3) {\large$\bcdots$};
        \node [ed] at (3.4,-3.3) {\large$\bcdots$};
      \end{tikzpicture}
    \]
    \caption{A generic 2-cell $\grpair{\alpha}{\beta}$ in $\cat{C} \otimes \cat{D}$.}\label{fig:grpair-two}
  \end{figure}
  
  All cells in $\cat{C} \otimes \cat{D}$ are indeed compositions of
  these generating cells: see \cref{fig:grpair-two}. Here each 2-cell
  written $\grpair{1}{1}$, or ``shuffle'', may be composed in a
  canonical way (up to associativity) from the transformation
  component 2-cells
  $\grpair{f}{d}, \grpair{c'}{g} \to \grpair{c}{g}, \grpair{f}{d'}$ or
  their inverses, by constructing the induced permutation out of
  transpositions. We accordingly extend the map
  $\cat{C} \otimes \cat{D} \to \cat{X}$ to arbitrary cells, sending
  each 2-cell written as a composite of the generating 2-cells to the
  corresponding composite in $\cat{X}$.

  To show functoriality, consider 2-cells in the image of this
  extended map $\cat{C} \otimes \cat{D} \to \cat{X}$, i.e.\ those
  built as in \cref{fig:grpair-two}. Vertical composites reduce to the
  desired form by transformation component 2-cells cancelling with
  their inverses; horizontal composites are put into the desired form
  using the naturality and modification laws.

  It is then easy to see that the left adjoint acts as
  $- \otimes \cat{D}$ on morphisms as well.

  Alternatively, we could skip this argument by appealing to existing
  knowledge about the Gray tensor product of 2-categories, of which
  the Gray tensor product of {\twowords} may be viewed as a special
  case; the Gray tensor product of strict 2-categories has a
  presentation like the above since its internal homs are given by
  2-functors, pseudonatural transformations, and modifications of
  strict 2-categories.
\end{proof}

\begin{rmk}\label{rmk:laxgray-two}
  Replacing the transformations in \cref{prop:gray-two} with
  \emph{(co)lax} transformations, we obtain the \textbf{(co)lax Gray
    tensor product}~\cite{gray:formal} as the presented
  structure. (The lax Gray tensor product is then the reverse of the
  colax Gray tensor product.)  However, it is perhaps less obvious
  that this definition gives a (non-symmetric) monoidal product.
\end{rmk}

\begin{rmk}\label{rmk:nocartesian}
  In contrast, \tword is not cartesian closed. For example, let
  $\cat{C}$, $\cat{D}_1$, $\cat{D}_2$ and $\cat{I}$ be
  respectively free on
  
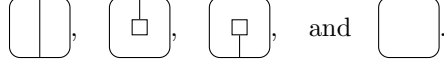
\[
    \begin{tikzpicture}[scale=.1]
      \draw [rect] (-4,-4) rectangle (4,4);
      \node [ov] (s) at (0, 4) {};
      \node [ov] (t) at (0, -4) {};
      \draw (s) -- (t);
    \end{tikzpicture}
    ,\quad
    \begin{tikzpicture}[scale=.1]
      \draw [rect] (-4,-4) rectangle (4,4);
      \node [ov] (s) at (0, 4) {};
      \node [ivs] (m) at (0, 0) {};
      \draw (m) -- (s);
    \end{tikzpicture}
    ,\quad
    \begin{tikzpicture}[scale=.1]
      \draw [rect] (-4,-4) rectangle (4,4);
      \node [ov] (t) at (0, -4) {};
      \node [ivs] (m) at (0, 0) {};
      \draw (m) -- (t);
    \end{tikzpicture}
    ,\quad \text{and}\quad
    \begin{tikzpicture}[scale=.1]
      \draw [rect] (-4,-4) rectangle (4,4);
    \end{tikzpicture}
    .
  \]
  Now the pushout of the unique functors $\cat{C} \to \cat{D}_1$ and
  $\cat{C} \to \cat{D}_2$ is not preserved by $- \times \cat{I}$.
  (Cartesian products in \tword are calculated using the essentially
  algebraic definition of {\twowords} in \cref{sec:doublewords}; note
  this does not agree with the cartesian product in \tcat.) Indeed,
  this pushout has nontrivial composite 2-cells $\alpha$ with nullary
  source and target, so its product with $\cat{I}$ likewise has
  nontrivial 2-cells $(\alpha, \justi)$. On the other hand $\cat{D}_1$
  and $\cat{D}_2$ have no nontrivial 2-cells with nullary source and
  target, so the products with $\cat{I}$ are simply $\cat{I}$, as is
  the pushout of these.
\end{rmk}

The next proposition implies in particular that if $\cat{C}$ and
$\cat{D}$ are bicategories, then $\Homl(\cat{C}, \cat{D})$ is a doubly
weak double category.

\begin{prop}
  If $\cat{C}$ and $\cat{D}$ are {\twowords} and $\cat{D}$ is
  represented, then $\Homl(\cat{C}, \cat{D})$ (and hence in
  particular $\Hom(\cat{C}, \cat{D})$) is represented.
\end{prop}
\begin{proof}
  We define binary composites of colax transformations
  $\sigma \colon F \to G$ and $\rho \colon G \to H$ and identity
  transformations (nullary composites) componentwise on 1-cells, and
  with 2-cell components:
  
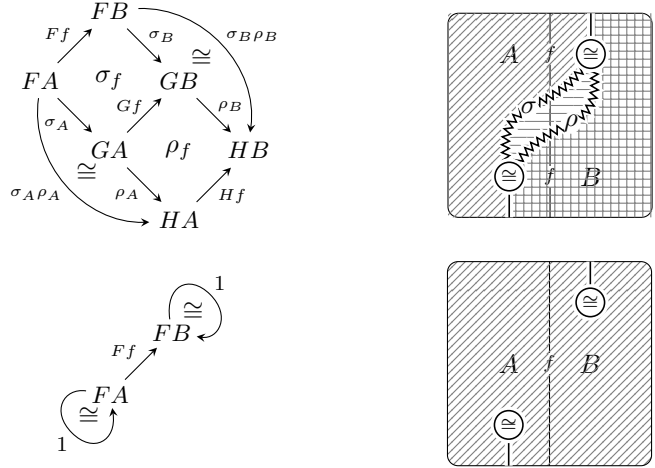
\[
    \begin{tikzpicture}[->,xscale=.65,yscale=-.65]
      \begin{scope}[rotate=45]
        \node (a) at (-1, 1) {\small$FA$};
        \node (b) at (1, 1) {\small$GA$};
        \node (c) at (-1, -1) {\small$FB$};
        \node (d) at (1, -1) {\small$GB$};
        \node at (0, 0) {$\sigma_f$};
        \draw (a) -- node[below left=-2]{\tiny$\sigma_A$} (b);
        \draw (b) -- node[above left=-3]{\tiny$Gf$} (d);
        \draw (a) -- node[above left=-2]{\tiny$Ff$} (c);
        \draw (c) -- node[above right=-2]{\tiny$\sigma_B$} (d);
        \begin{scope}[shift={(2,0)}]
          \node (b2) at (1, 1) {\small$HA$};
          \node (d2) at (1, -1) {\small$HB$};
          \node at (0, 0) {$\rho_f$};
          \draw (b) -- node[below left=-2]{\tiny$\rho_A$} (b2);
          \draw (b2) -- node[below right=-2]{\tiny$Hf$} (d2);
          \draw (d) -- node[above right=-2]{\tiny$\rho_B$} (d2);
        \end{scope}
        \node at (1, 1.65) {\choseniso};
        \node at (1, -1.65) {\choseniso};
        \draw (a) [bend left=50] to node[below left=-2]{\tiny$\ubcom{\sigma_A}{\rho_A}$} (b2);
        \draw (c) [bend right=50] to node[above right=-2]{\tiny$\ubcom{\sigma_B}{\rho_B}$} (d2);
      \end{scope}
    \end{tikzpicture}
    \qquad\qquad\qquad
    \begin{tikzpicture}[xscale=.27,yscale=-.27]
      \draw [rect] (-5,-5) rectangle (5,5);
      \node [ov] (s) at (0, 5) {};
      \node [ov] (t) at (0, -5) {};
      \draw (s) -- node [ed,pos=.8] {$f$} node [ed,pos=.2] {$f$} (t);
      \node at (-2, -3) {$A$};
      \node at (2, 3) {$B$};
      \coordinate (si) at (-2, 3) {};
      \coordinate (ti) at (2, -3) {};
      \fill [fun, pattern=north east lines] (-5,-5) -- (2,-5) -- (ti)  .. controls +(-1.75,3.5) and +(-1,-4.5) .. (si) -- (-2, 5) -- (-5,5) -- cycle;
      \fill [fun, pattern=grid] (5,5) -- (-2,5) -- (si) .. controls +(1.75,-3.5) and +(1,4.5) .. (ti) -- (2, -5) -- (5,-5) -- cycle;
      \fill [fun, pattern=horizontal lines] (si) .. controls +(1.75,-3.5) and +(1,4.5) .. (ti) .. controls +(-1.75,3.5) and +(-1,-4.5) .. cycle;
      \node [ov] (sa) at (-2, 5) {};
      \node [ov] (ta) at (2, -5) {};
      \draw (sa) [edf] -- (si);
      \draw (ta) [edf] -- (ti);
      \draw (si) [nat] .. controls +(-1,-4.5) and +(-1.75,3.5) .. node [nated] {$\sigma$} (ti);
      \draw (si) [nat] .. controls +(1.75,-3.5) and +(1,4.5) .. node [nated] {$\rho$} (ti);
      \node [modc] at (si) {$\choseniso$};
      \node [modc] at (ti) {$\choseniso$};
    \end{tikzpicture}
  \]
  \[
    \,\;\;\quad\begin{tikzpicture}[->,shorten <=-1.5pt,shorten >=-1.5pt,rotate=135]
      \begin{scope}[xscale=-.4,yscale=.575]
        \node (a) at (0,1) {\small$FA$};
        \node (b) at (0,-1) {\small$FB$};
        \node at (0,1.7) {$\choseniso$};
        \node at (0,-1.7) {$\choseniso$};
        \draw (a) -- node [above left=-2] {\tiny$Ff$} (b);
        \draw (a) .. controls +(-1.25,.65) and +(-1,0) .. (0,2.25) node [below left=-2] {\footnotesize$\justi$} .. controls +(1,0) and +(1.25,.65) .. (a);
        \draw (b) .. controls +(-1.25,-.65) and +(-1,0) .. (0,-2.25) node [above right=-2] {\footnotesize$\justi$} .. controls +(1,0) and +(1.25,-.65) .. (b);
      \end{scope}
    \end{tikzpicture}
    \qquad\qquad\qquad\qquad
    \begin{tikzpicture}[xscale=.27,yscale=-.27]
      \node [ov] (s) at (0, 5) {};
      \node [ov] (t) at (0, -5) {};
      \draw (s) -- node [ed] {$f$} (t);
      \node at (-2, 0) {$A$};
      \node at (2, 0) {$B$};
      \draw [rect, fun, pattern=north east lines] (-5,-5) rectangle (5,5);
      \node [ov] (sa) at (-2, 5) {};
      \node (si) at (-2, 3) {};
      \node (ti) at (2, -3) {};
      \node [ov] (ta) at (2, -5) {};
      \draw (sa) [edf] -- (si);
      \draw (ta) [edf] -- (ti);
      \node [modc] at (si) {$\choseniso$};
      \node [modc] at (ti) {$\choseniso$};
    \end{tikzpicture}
  \]
  These are easily checked to be horizontal transformations. Moreover,
  the composition 2-cells in $\cat{D}$ are components of invertible
  modifications. Lax transformations are similar.
\end{proof}

\begin{rmk}\label{rmk:notrep}
  The Gray tensor product of two representable {\twowords} is usually
  \emph{not} representable: if $f \colon c \to c'$ is an arrow in
  $\cat{C}$ and $g \colon d \to d'$ is an arrow in $\cat{D}$, there is
  no composite 1-cell of the compatible $\grpair{f}{d}$ and
  $\grpair{c'}{g}$ in $\cat{C} \otimes \cat{D}$.
\end{rmk}



Next we observe that our notions of transformation, modification, and
icon correspond to the usual notions for bicategories.


\begin{prop}\label{prop:tranmain-two}
  Identifying represented {\twowords} and functors with bicategories
  and pseudofunctors (\cref{prop:main-two}) respects (co)lax
  transformations, modifications, and icons, as well as their
  composition.
\end{prop}
\begin{proof}
  Suppose $\sigma \colon F \to G$ is a colax transformation of
  {\twoword} functors. We define a colax natural transformation of the
  underlying pseudofunctors as follows.
  \begin{itemize}
  \item The component 1-cell at 0-cell $A$ is $\sigma_A$.
  \item The component 2-cell at 1-cell $f\colon A \to B$ is $\sigma_f$
    converted to a bigon:
    \[

    \]
  \end{itemize}
  The axioms of a colax natural transformation then follow from
  composition isomorphisms cancelling with their inverses and
  applications of the colax transformation naturality axiom.
  
  Conversely, suppose $\sigma$ is a colax natural transformation of
  pseudofunctors. We define a colax transformation of the underlying
  {\twoword} functors as follows.
  \begin{itemize}
  \item The component 1-cell at 0-cell $A$ is $\sigma_A$.
  \item The component 2-cell is $\sigma_f$ converted to a
    (2,2)-ary 2-cell:
    \[
      %
    \]
  \end{itemize}
  
  When translated into a statement about corresponding cells in the
  underlying {\twoword}, the naturality axiom yields
  \[\quad
    %
  \]
  for all bigons $\alpha$, and the coherence axioms yield
  \begin{align*}
    %
  \end{align*}
  for all chosen composition isomorphisms.  We obtain the {\twoword}
  colax transformation naturality axiom for an arbitrary 2-cell
  $\alpha$ by bracketing up its 1-cells and moving $\sigma$ across the
  resulting bigon, as shown in \cref{fig:i2ctna}.
  \begin{figure}
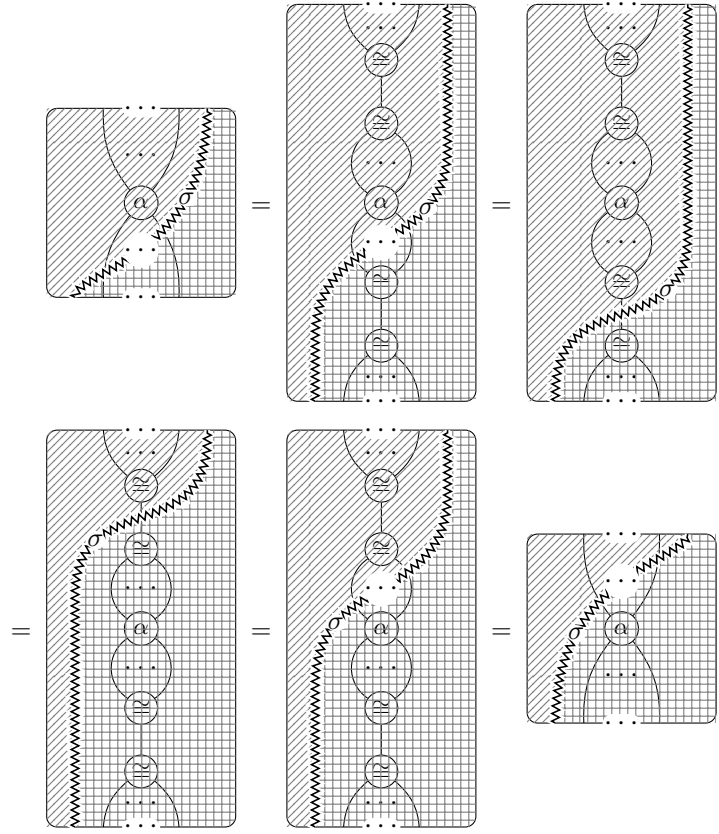

    \centering
    \begin{align*}
      %
    \end{align*}
    \caption{The colax transformation naturality axiom}
    \label{fig:i2ctna}
  \end{figure}
  These translation processes are clearly inverse. Moreover, it is
  easy to see that identities, compositions, and whiskerings are sent
  to identities, compositions, and whiskerings, as defined in
  e.g.~\cite{johnson-yau}.

  Our general notion of modification between lax and colax
  transformations of {\twowords} corresponds to a notion for
  bicategories defined in the same way, and it is easy to see that the
  specialization to modifications between only lax or only colax
  transformations (and their composition) coincides with the usual
  definition, as in e.g.~\cite{johnson-yau}.

  Finally, icons in a represented {\twoword} are in one-to-one
  correspondence with colax transformations whose components are
  identities, by composing the naturality 2-cells with nullary
  composition isomorphisms:\[
    \!
    \begin{tikzpicture}[->]
      \node at (-1.2,.25) {\footnotesize$FA$};
      \node at (-1.2,0) {\footnotesize$=$};
      \node at (-1.2,-.25) {\footnotesize$GA$};
      \node at (1.2,.25) {\footnotesize$FB$};
      \node at (1.2,0) {\footnotesize$=$};
      \node at (1.2,-.25) {\footnotesize$GB$};
      \begin{scope}[rotate=-90]
        \node at (0,0) {$\sigma_f$};
        \node at (0,.65) {$\choseniso$};
        \node at (0,-.65) {$\choseniso$};
        \begin{scope}[shift={(0,-1.25)},xscale=.35,yscale=.75]
          \draw [shorten >=12.5, shorten <=12.5] (0,0) .. controls +(-1.25,.65) and +(-1,0) .. (0,1.25) .. controls +(1,0) and +(1.25,.65) .. node [below,pos=.25] {\tiny$\justi$} (0,0);
        \end{scope}
        \begin{scope}[shift={(0,1.25)},xscale=.35,yscale=.75]
          \draw [shorten >=12.5, shorten <=12.5] (0,0) .. controls +(-1.25,-.65) and +(-1,0) .. node [above,pos=.75] {\tiny$\justi$} (0,-1.25) .. controls +(1,0) and +(1.25,-.65) .. (0,0);
        \end{scope}
      \end{scope}
      \draw [shorten >=12.5, shorten <=12.5] (-1.2, .1) to [bend left=50] node [above] {\tiny $Ff$} (1.2, .1);
      \draw [shorten >=12.5, shorten <=12.5] (-1.2, -.1) to [bend right=50] node [below] {\tiny $Gf$} (1.2, -.1);
    \end{tikzpicture}
    \quad\leftrightarrow\quad
    \begin{tikzpicture}[->,xscale=.75]
      \node (x1) at (-1,0) {};
      \node (x2) at (1,0) {};
      \node at (-1.2,.25) {\footnotesize$FA$};
      \node at (-1.2,0) {\footnotesize$=$};
      \node at (-1.2,-.25) {\footnotesize$GA$};
      \node at (1.2,.25) {\footnotesize$FB$};
      \node at (1.2,0) {\footnotesize$=$};
      \node at (1.2,-.25) {\footnotesize$GB$};
      \draw (x1) to[bend left=35] node[above] {\tiny$Ff$} (x2);
      \draw (x1) to[bend right=35] node[below] {\tiny$Gf$} (x2);
      \node at (0,0) {$\theta_f$};
      \node at (1.2,.6) {$\choseniso$};
      \node at (-1.2,-.6) {$\choseniso$};
      \begin{scope}[shift={(1.2,0)},xscale=.5,yscale=.75]
        \draw [shorten >=12.5, shorten <=12.5] (0,0) .. controls +(-1.25,.65) and +(-1,0) .. (0,1.25) node [above] {\tiny$\justi$} .. controls +(1,0) and +(1.25,.65) .. (0,0);
      \end{scope}
      \begin{scope}[shift={(-1.2,0)},xscale=.5,yscale=.75]
        \draw [shorten >=12.5, shorten <=12.5] (0,0) .. controls +(-1.25,-.65) and +(-1,0) .. (0,-1.25) node [below] {\tiny$\justi$} .. controls +(1,0) and +(1.25,-.65) .. (0,0);
      \end{scope}
    \end{tikzpicture}
  \]
  Composition and whiskering for icons are also as in~\cite{lack:icons}.
\end{proof}

\begin{rmk}
  It is easy to generalize most of the results of this section to
  double-categorical versions, with a few caveats. We refer the reader
  to~\cite{bohm:gray} for definitions of horizontal and vertical
  pseudonatural transformations, modifications, and Gray tensor
  products of strict double categories; see
  also~\cite{morehouse:double} for definitions of horizontal and
  vertical lax and colax transformations.
  
  A maximally general definition of modification between both lax and
  colax horizontal and vertical transformations of (\adjective) double
  categories can be formulated by placing transformation component
  1-cells at all possible corners of the diagram:
  
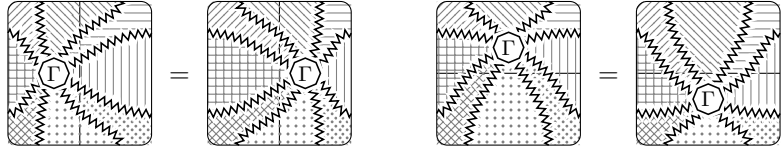
\[
    \begin{tikzpicture}[scale=.19,rotate=90]
      \draw (5.0,0) -- (-5.0,0);
      \fill [fun, pattern=crosshatch] (-5.0,5.0) -- (-2.8,5.0) .. controls +(0,-1.5) and +(-1.25,2.5) .. (0,1.8).. controls +(-2.0,1.0) and +(1.5,0) .. (-5.0,2.8) -- cycle;
      \fill [fun, pattern=grid] (-2.8,5.0) .. controls +(0,-1.5) and +(-1.25,2.5) .. (0,1.8) .. controls +(1.25,2.5) and +(0,-1.5) .. (2.8,5.0) -- cycle;
      \fill [fun, pattern=north east lines] (5.0,5.0) -- (2.8,5.0) .. controls +(0,-1.5) and +(1.25,2.5) .. (0,1.8) .. controls +(2.0,1.0) and +(-1.5,0) .. (5.0,2.8) -- cycle;
      \fill [fun, pattern=north west lines] (5.0,2.8) .. controls +(-1.5,0) and +(2.0,1.0) .. (0,1.8) .. controls +(4.0,-1.5) and +(-1.5,0) .. (5.0,-2.8) -- cycle; 
      \fill [fun, pattern=dots] (-5.0,2.8) .. controls +(1.5,0) and +(-2.0,1.0) .. (0,1.8) .. controls +(-4.0,-1.5) and +(1.5,0) .. (-5.0,-2.8) -- cycle;
      \fill [fun, pattern=crosshatch dots] (-5.0,-5.0) -- (-5.0,-2.8) .. controls +(1.5,0) and +(-4.0,-1.5) .. (0,1.8) .. controls +(-1.5,-3) and +(0,2.5) .. (-2.8,-5.0) -- cycle;
      \fill [fun, pattern=vertical lines] (-2.8,-5.0) .. controls +(0,2.5) and +(-1.5,-3) .. (0,1.8) .. controls +(1.5,-3) and +(0,2.5) .. (2.8,-5.0) -- cycle;
      \fill [fun, pattern=horizontal lines] (5.0,-5.0) -- (5.0,-2.8) .. controls +(-1.5,0) and +(4.0,-1.5) .. (0,1.8) .. controls +(1.5,-3) and +(0,2.5) .. (2.8,-5.0) -- cycle;
      \draw [nat] (-2.8,5.0) .. controls +(0,-1.5) and +(-1.25,2.5) .. (0,1.8);
      \draw [nat] (2.8,5.0) .. controls +(0,-1.5) and +(1.25,2.5) .. (0,1.8);
      \draw [nat] (5.0,2.8) .. controls +(-1.5,0) and +(2.0,1.0) .. (0,1.8);
      \draw [nat] (5.0,-2.8) .. controls +(-1.5,0) and +(4.0,-1.5) .. (0,1.8);
      \draw [nat] (-5.0,2.8) .. controls +(1.5,0) and +(-2.0,1.0) .. (0,1.8);
      \draw [nat] (-5.0,-2.8) .. controls +(1.5,0) and +(-4.0,-1.5) .. (0,1.8);
      \draw [nat] (-2.8,-5.0) .. controls +(0,2.5) and +(-1.5,-3) .. (0,1.8);
      \draw [nat] (2.8,-5.0) .. controls +(0,2.5) and +(1.5,-3) .. (0,1.8);
      \node [modo,minimum size=0pt,inner sep=1pt] at (0,1.8) {\footnotesize$\Gamma$};
      \draw [rect] (-5.0,-5.0) rectangle (5.0,5.0);
    \end{tikzpicture}
    \; = \;
    \begin{tikzpicture}[scale=-.19,rotate=90]
      \draw (5.0,0) -- (-5.0,0);
      \fill [fun, pattern=horizontal lines] (-5.0,5.0) -- (-2.8,5.0) .. controls +(0,-1.5) and +(-1.25,2.5) .. (0,1.8).. controls +(-2.0,1.0) and +(1.5,0) .. (-5.0,2.8) -- cycle;
      \fill [fun, pattern=vertical lines] (-2.8,5.0) .. controls +(0,-1.5) and +(-1.25,2.5) .. (0,1.8) .. controls +(1.25,2.5) and +(0,-1.5) .. (2.8,5.0) -- cycle;
      \fill [fun, pattern=crosshatch dots] (5.0,5.0) -- (2.8,5.0) .. controls +(0,-1.5) and +(1.25,2.5) .. (0,1.8) .. controls +(2.0,1.0) and +(-1.5,0) .. (5.0,2.8) -- cycle;
      \fill [fun, pattern=dots] (5.0,2.8) .. controls +(-1.5,0) and +(2.0,1.0) .. (0,1.8) .. controls +(4.0,-1.5) and +(-1.5,0) .. (5.0,-2.8) -- cycle; 
      \fill [fun, pattern=north west lines] (-5.0,2.8) .. controls +(1.5,0) and +(-2.0,1.0) .. (0,1.8) .. controls +(-4.0,-1.5) and +(1.5,0) .. (-5.0,-2.8) -- cycle;
      \fill [fun, pattern=north east lines] (-5.0,-5.0) -- (-5.0,-2.8) .. controls +(1.5,0) and +(-4.0,-1.5) .. (0,1.8) .. controls +(-1.5,-3) and +(0,2.5) .. (-2.8,-5.0) -- cycle;
      \fill [fun, pattern=grid] (-2.8,-5.0) .. controls +(0,2.5) and +(-1.5,-3) .. (0,1.8) .. controls +(1.5,-3) and +(0,2.5) .. (2.8,-5.0) -- cycle;
      \fill [fun, pattern=crosshatch] (5.0,-5.0) -- (5.0,-2.8) .. controls +(-1.5,0) and +(4.0,-1.5) .. (0,1.8) .. controls +(1.5,-3) and +(0,2.5) .. (2.8,-5.0) -- cycle;
      \draw [nat] (-2.8,5.0) .. controls +(0,-1.5) and +(-1.25,2.5) .. (0,1.8);
      \draw [nat] (2.8,5.0) .. controls +(0,-1.5) and +(1.25,2.5) .. (0,1.8);
      \draw [nat] (5.0,2.8) .. controls +(-1.5,0) and +(2.0,1.0) .. (0,1.8);
      \draw [nat] (5.0,-2.8) .. controls +(-1.5,0) and +(4.0,-1.5) .. (0,1.8);
      \draw [nat] (-5.0,2.8) .. controls +(1.5,0) and +(-2.0,1.0) .. (0,1.8);
      \draw [nat] (-5.0,-2.8) .. controls +(1.5,0) and +(-4.0,-1.5) .. (0,1.8);
      \draw [nat] (-2.8,-5.0) .. controls +(0,2.5) and +(-1.5,-3) .. (0,1.8);
      \draw [nat] (2.8,-5.0) .. controls +(0,2.5) and +(1.5,-3) .. (0,1.8);
      \node [modo,minimum size=0pt,inner sep=1pt] at (0,1.8) {\footnotesize$\Gamma$};
      \draw [rect] (-5.0,-5.0) rectangle (5.0,5.0);
    \end{tikzpicture}
    \qquad\quad
    \begin{tikzpicture}[scale=.19]
      \draw (5.0,0) -- (-5.0,0);
      \fill [fun, pattern=north east lines] (-5.0,5.0) -- (-2.8,5.0) .. controls +(0,-1.5) and +(-1.25,2.5) .. (0,1.8).. controls +(-2.0,1.0) and +(1.5,0) .. (-5.0,2.8) -- cycle;
      \fill [fun, pattern=north west lines] (-2.8,5.0) .. controls +(0,-1.5) and +(-1.25,2.5) .. (0,1.8) .. controls +(1.25,2.5) and +(0,-1.5) .. (2.8,5.0) -- cycle;
      \fill [fun, pattern=horizontal lines] (5.0,5.0) -- (2.8,5.0) .. controls +(0,-1.5) and +(1.25,2.5) .. (0,1.8) .. controls +(2.0,1.0) and +(-1.5,0) .. (5.0,2.8) -- cycle;
      \fill [fun, pattern=vertical lines] (5.0,2.8) .. controls +(-1.5,0) and +(2.0,1.0) .. (0,1.8) .. controls +(4.0,-1.5) and +(-1.5,0) .. (5.0,-2.8) -- cycle; 
      \fill [fun, pattern=grid] (-5.0,2.8) .. controls +(1.5,0) and +(-2.0,1.0) .. (0,1.8) .. controls +(-4.0,-1.5) and +(1.5,0) .. (-5.0,-2.8) -- cycle;
      \fill [fun, pattern=crosshatch] (-5.0,-5.0) -- (-5.0,-2.8) .. controls +(1.5,0) and +(-4.0,-1.5) .. (0,1.8) .. controls +(-1.5,-3) and +(0,2.5) .. (-2.8,-5.0) -- cycle;
      \fill [fun, pattern=dots] (-2.8,-5.0) .. controls +(0,2.5) and +(-1.5,-3) .. (0,1.8) .. controls +(1.5,-3) and +(0,2.5) .. (2.8,-5.0) -- cycle;
      \fill [fun, pattern=crosshatch dots] (5.0,-5.0) -- (5.0,-2.8) .. controls +(-1.5,0) and +(4.0,-1.5) .. (0,1.8) .. controls +(1.5,-3) and +(0,2.5) .. (2.8,-5.0) -- cycle;
      \draw [nat] (-2.8,5.0) .. controls +(0,-1.5) and +(-1.25,2.5) .. (0,1.8);
      \draw [nat] (2.8,5.0) .. controls +(0,-1.5) and +(1.25,2.5) .. (0,1.8);
      \draw [nat] (5.0,2.8) .. controls +(-1.5,0) and +(2.0,1.0) .. (0,1.8);
      \draw [nat] (5.0,-2.8) .. controls +(-1.5,0) and +(4.0,-1.5) .. (0,1.8);
      \draw [nat] (-5.0,2.8) .. controls +(1.5,0) and +(-2.0,1.0) .. (0,1.8);
      \draw [nat] (-5.0,-2.8) .. controls +(1.5,0) and +(-4.0,-1.5) .. (0,1.8);
      \draw [nat] (-2.8,-5.0) .. controls +(0,2.5) and +(-1.5,-3) .. (0,1.8);
      \draw [nat] (2.8,-5.0) .. controls +(0,2.5) and +(1.5,-3) .. (0,1.8);
      \node [modo,minimum size=0pt,inner sep=1pt] at (0,1.8) {\footnotesize$\Gamma$};
      \draw [rect] (-5.0,-5.0) rectangle (5.0,5.0);
    \end{tikzpicture}
    \; = \;
    \begin{tikzpicture}[scale=-.19]
      \draw (5.0,0) -- (-5.0,0);
      \fill [fun, pattern=crosshatch dots] (-5.0,5.0) -- (-2.8,5.0) .. controls +(0,-1.5) and +(-1.25,2.5) .. (0,1.8).. controls +(-2.0,1.0) and +(1.5,0) .. (-5.0,2.8) -- cycle;
      \fill [fun, pattern=dots] (-2.8,5.0) .. controls +(0,-1.5) and +(-1.25,2.5) .. (0,1.8) .. controls +(1.25,2.5) and +(0,-1.5) .. (2.8,5.0) -- cycle;
      \fill [fun, pattern=crosshatch] (5.0,5.0) -- (2.8,5.0) .. controls +(0,-1.5) and +(1.25,2.5) .. (0,1.8) .. controls +(2.0,1.0) and +(-1.5,0) .. (5.0,2.8) -- cycle;
      \fill [fun, pattern=grid] (5.0,2.8) .. controls +(-1.5,0) and +(2.0,1.0) .. (0,1.8) .. controls +(4.0,-1.5) and +(-1.5,0) .. (5.0,-2.8) -- cycle; 
      \fill [fun, pattern=vertical lines] (-5.0,2.8) .. controls +(1.5,0) and +(-2.0,1.0) .. (0,1.8) .. controls +(-4.0,-1.5) and +(1.5,0) .. (-5.0,-2.8) -- cycle;
      \fill [fun, pattern=horizontal lines] (-5.0,-5.0) -- (-5.0,-2.8) .. controls +(1.5,0) and +(-4.0,-1.5) .. (0,1.8) .. controls +(-1.5,-3) and +(0,2.5) .. (-2.8,-5.0) -- cycle;
      \fill [fun, pattern=north west lines] (-2.8,-5.0) .. controls +(0,2.5) and +(-1.5,-3) .. (0,1.8) .. controls +(1.5,-3) and +(0,2.5) .. (2.8,-5.0) -- cycle;
      \fill [fun, pattern=north east lines] (5.0,-5.0) -- (5.0,-2.8) .. controls +(-1.5,0) and +(4.0,-1.5) .. (0,1.8) .. controls +(1.5,-3) and +(0,2.5) .. (2.8,-5.0) -- cycle;
      \draw [nat] (-2.8,5.0) .. controls +(0,-1.5) and +(-1.25,2.5) .. (0,1.8);
      \draw [nat] (2.8,5.0) .. controls +(0,-1.5) and +(1.25,2.5) .. (0,1.8);
      \draw [nat] (5.0,2.8) .. controls +(-1.5,0) and +(2.0,1.0) .. (0,1.8);
      \draw [nat] (5.0,-2.8) .. controls +(-1.5,0) and +(4.0,-1.5) .. (0,1.8);
      \draw [nat] (-5.0,2.8) .. controls +(1.5,0) and +(-2.0,1.0) .. (0,1.8);
      \draw [nat] (-5.0,-2.8) .. controls +(1.5,0) and +(-4.0,-1.5) .. (0,1.8);
      \draw [nat] (-2.8,-5.0) .. controls +(0,2.5) and +(-1.5,-3) .. (0,1.8);
      \draw [nat] (2.8,-5.0) .. controls +(0,2.5) and +(1.5,-3) .. (0,1.8);
      \node [modo,minimum size=0pt,inner sep=1pt] at (0,1.8) {\footnotesize$\Gamma$};
      \draw [rect] (-5.0,-5.0) rectangle (5.0,5.0);
    \end{tikzpicture}
  \]
  
  One then expects to assemble some two-dimensional categorical
  structure, analogous to $\Homl(\cat{C}, \cat{D})$, in which 0-cells
  are functors, 1-cells are lax and colax transformations, and 2-cells
  are these generalized modifications. But here there are four
  different sorts of 1-cells, apparently requiring an analogue of a
  ({\adjective}) double category with octagon-shaped rather than
  square 2-cells.
\end{rmk}

\begin{rmk}
  There is a relationship between double categories and (co)lax
  transformations of 2-categories. Let $H\cat{C}$ denote the
  vertically trivial (\adjective) double category with horizontal
  (\adjective) 2-category $\cat{C}$, let $V\cat{D}$ denote the
  horizontally trivial (\adjective) double category with vertical
  (\adjective) 2-category $\cat{D}$, and let $Q\cat{X}$ denote the
  (\adjective) double category of ``quintets'' of (\adjective)
  2-category $\cat{X}$.

  By comparing presentations, we can see that a (\adjective)
  2-category functor from the \emph{lax} Gray tensor product
  (\cref{rmk:laxgray-two}) of $\cat{C}$ and $\cat{D}$ into $\cat{X}$
  is the same as a (\adjective) double category functor
  $H\cat{C} \otimes V\cat{D} \to Q\cat{X}$. (Here the
  double-categorical Gray tensor product $H \cat{C} \otimes V\cat{D}$
  simply agrees with the cartesian product of strict double
  categories, due to lack of nontrivial 1-cells of each type in some
  factor. This is (the transpose of) the ``external product'' of
  2-categories from~\cite[Definition 2.6]{fiore-paoli-pronk}.) In
  other words, the lax Gray tensor product of (\adjective)
  2-categories is given by $F(H(-) \otimes V(-))$, where $F$ is the
  left adjoint to $Q$.

  In particular, as can also be seen directly, lax and colax
  transformations valued in a (\adjective) 2-category can be described
  as horizontal and vertical transformations valued in its associated
  (\adjective) double category.
\end{rmk}

\bibliographystyle{alpha}
\bibliography{doublyweak}

\end{document}